\documentclass[11pt]{amsart}
\usepackage{amsmath,amssymb,amscd}

\usepackage{enumerate}
\usepackage[all]{xy}
\usepackage{amssymb}
\usepackage{mathrsfs}
\usepackage{stmaryrd}
\usepackage{enumerate}
\usepackage{epsfig}
\usepackage{layout}
\usepackage{fullpage}

\usepackage{xcolor}
\usepackage{url}
\usepackage{hyperref}
\hypersetup{
    colorlinks,%
    citecolor=blue,%
    filecolor=black,%
    linkcolor=red,%
   urlcolor=MyDarkBlue
}
\definecolor{lightgray}{gray}{0.65}
\definecolor{darkgray}{gray}{0.5}
\definecolor{MyDarkBlue}{rgb}{0,0.1,1}

\usepackage{tikz}
\usetikzlibrary{calc,decorations.pathmorphing,shapes}
\newcounter{sarrow}
\newcommand\xrsquigarrow[1]{%
\stepcounter{sarrow}%
\mathrel{\begin{tikzpicture}[baseline= {( $ (current bounding box.south) + (0,-0.5ex) $ )}]
\node[inner sep=.5ex] (\thesarrow) {$\scriptstyle #1$};
\path[draw,<-,decorate,
  decoration={zigzag,amplitude=0.7pt,segment length=1.2mm,pre=lineto,pre length=4pt}]
    (\thesarrow.south east) -- (\thesarrow.south west);
\end{tikzpicture}}%
}
\newcommand\xxrsquigarrow[1]{%
\stepcounter{sarrow}%
\mathrel{\begin{tikzpicture}[baseline= {( $ (current bounding box.south) + (0,-0.5ex) $ )}]
\node[inner sep=.5ex] (\thesarrow) {$\scriptstyle #1$};
\path[draw,<-,decorate,
  decoration={zigzag,amplitude=0.7pt,segment length=1.2mm,pre=lineto,pre length=4pt}]
  (0.5,-.5) -- (0,0);
\end{tikzpicture}}%
}
\newcommand\xxlsquigarrow[1]{%
\stepcounter{sarrow}%
\mathrel{\begin{tikzpicture}[baseline= {( $ (current bounding box.south) + (0,-0.5ex) $ )}]
\node[inner sep=.5ex] (\thesarrow) {$\scriptstyle #1$};
\path[draw,<-,decorate,
  decoration={zigzag,amplitude=0.7pt,segment length=1.2mm,pre=lineto,pre length=4pt}]
  (0,-.5) -- (0.5,0);
\end{tikzpicture}}%
}

\numberwithin{equation}{section}
\newtheorem{thm}{Theorem}[section]
\newtheorem{theoremalpha}{Theorem}

\newtheorem{convention}[thm]{}
\newtheorem{prop}[thm]{Proposition}
\newtheorem{coro}[thm]{Corollary}
\newtheorem{lemma}[thm]{Lemma}

\newtheorem{question}[thm]{Question}

\newtheorem{problemalpha}{Problem}

\newtheorem{question*}[]{Question}

\newtheorem{fact}[thm]{Fact}

\theoremstyle{definition}

\newtheorem{rmk}[thm]{Remark}
\newtheorem{cons}[thm]{Construction}

\newtheorem{defi}[thm]{Definition}

\newenvironment{sis}{\left\{\begin{aligned}}{\end{aligned}\right.}

\def\w{\widetilde}
\def\wt{\widetilde}
\def\wh{\widehat}
\def\mbb{\mathbb}
\def\mcl{\mathcal}

\def\un{\underline}
\def\ov{\overline}

\def\inj{\hookrightarrow}
\def\surj{\twoheadrightarrow}

\def\A{\mbb A}

\def\bZ{\mbb Z}
\def\Z{\mbb Z}

\def\N{\mbb N}

\def\cC{\mcl C}

\def\cO{\mcl O}

\def\cX{\mcl X}
\def\cY{\mcl Y}
\def \cF{\mcl F}

\def \cZ{\mcl Z}
\def\cP{\mcl P}

\def\cL{\mcl L}
\def\cN{\mcl N}
\def \cI{\mcl I}
\def \cJ{\mcl J}
\def\cM{\mcl M}

\def\cE{\mcl E}

\def\deg{{\rm deg}}

\def\lra{\longrightarrow}

\def \Id{{\rm Id}}
\def \id{{\rm id}}
\def \In{{\rm in}}

\DeclareMathOperator{\proj}{Proj}
\def\ord{\tu{ord\,}}

\def \Ext{{\rm Ext}}
\def \Def{{\rm Def}}
\def \ord{{\rm ord}}
\def \Stab{{\rm Stab}}
\def\Jac{\rm{Jac}}
\def\SET{\rm SET}
\def\SCH{\rm SCH}
\def \Sing{{\rm Sing}}

\def\P{\mathbb{P}}

\def \Mg{\bar{M}_g}
\def\Mgp{{\overline{M}}^{\rm p}_g}

\def \MMgwp {{\overline{\mathcal M}}_g^{\rm wp}}
\def \MMg {{\overline{\mathcal M}}_g}
\def \MMgp {{\overline{\mathcal M}}_g^{\rm p}}

\def \Hilb{{\rm Hilb}}
\def \Chow{{\rm Chow}}
\def \Ch{{\rm Ch}}

\def \Gr{{\rm Gr}}
\def\SL{{\rm{SL}}}
\def\GL{{\rm{GL}}}
\def\Pic{{\rm{Pic}}}
\def\PGL{{\rm{PGL}}}
\def\Gm{\mathbb{G}_m}
\def\Ga{\mathbb{G}_a}

\def\Aut{{\rm Aut}}
\def \Orb{{\rm Orb}}

\def \OO{\mcl O}
\def \Sym{{\rm Sym}}
\def \Im{{\rm Im}}
\def \ker{{\rm Ker}}

\DeclareMathOperator{\ps}{ps}
\DeclareMathOperator{\wps}{wps}
\DeclareMathOperator{\s}{s}

\def \exc{X_{\rm exc}}

\def \can{{\rm can}}

\def \B{\w{B}}

\def \sp{\rightsquigarrow}
\def \spel{\stackrel{{\rm el}}{\rightsquigarrow}}

\def \JJst{\ov{\mathcal J}_{d,g}}
\def \Jst{{\ov{J}}_{d,g}}

\def \JJwp{\ov{\mathcal J}_{d,g}^{\rm wp}}
\def \Jwp{{\ov{J}}_{d,g}^{\rm wp}}

\def \JJps{\ov{\mathcal J}_{d,g}^{\rm ps}}
\def \Jps{{\ov{J}}_{d,g}^{\rm ps}}

\def \SSst{\ov{\mathcal S}_{d,g}}

\def \SSwp{\ov{\mathcal S}_{d,g}^{\rm wp}}

\def \SSps{\ov{\mathcal S}_{d,g}^{\rm ps}}

\def\ra{\rightarrow}

\def\bar{\overline}

\input xy
\xyoption{all}

\begin{document}

\normalsize

%\vskip 1in

\title[GIT for polarized curves]{GIT for polarized curves}

\author{Gilberto Bini, Fabio Felici, Margarida Melo, Filippo Viviani}

\email{gilberto.bini@unimi.it} \curraddr{{\sc Dipartimento di Matematica \\ Universit\`a degli Studi di Milano \\ Via C. Saldini 50 \\ 20133 Milano \\ Italy.}}

\email{felici@mat.uniroma3.it} \curraddr{{\sc Dipartimento di Matematica \\ Universit\`a Roma Tre \\ Largo S. Leonardo Murialdo 1 \\ 00146 Roma \\ Italy.} }

\email{mmelo@mat.uc.pt}\curraddr{{\sc Departamento de Matem\'atica \\ Universidade de Coimbra\\
Largo D. Dinis, Apartado 3008 \\ 3001 Coimbra \\Portugal.}}

\email{viviani@mat.uniroma3.it} \curraddr{{\sc Dipartimento di Matematica \\ Universit\`a Roma Tre \\ Largo S. Leonardo Murialdo 1 \\ 00146 Roma \\ Italy.} }

%\date{\today}

%\keywords{Geometric invariant theory, moduli of curves, compactified Jacobians.}


\begin{abstract}
We investigate the GIT quotients of the Hilbert and Chow schemes of curves of degree $d$ and genus $g$ in a projective space of dimension $d-g$, as the degree $d$ decreases with respect to the genus $g$. We prove that the first three values of $d$ at which the GIT quotients change are given by $d=4(2g-2)$, $d=\frac{7}{2}(2g-2)$ and $d=2(2g-2)$.
In the range $d>4(2g-2)$, we show that the previous results of L. Caporaso hold true both for the Hilbert and Chow semistability.
In the range $\frac{7}{2}(2g-2)< d <4(2g-2)$, the Hilbert semistable locus coincides with the Chow semistable locus and it maps to the moduli stack of weakly-pseudo-stable curves.
In the range $2(2g-2)<d<\frac{7}{2}(2g-2)$, the Hilbert and Chow semistable loci coincide and they map to the moduli stack of pseudo-stable curves.
We also analyze in detail the first two critical values $d=4(2g-2)$ and $d=\frac{7}{2}(2g-2)$, where the Hilbert semistable locus is strictly smaller than the Chow semistable locus.
As an application of our results, we get two new compactifications of the universal Jacobian over the moduli space of weakly-pseudo-stable and pseudo-stable curves, respectively.

\end{abstract}

%\thanks{G. Bini has been partially supported by ``FIRST'' Universit\`a di Milano, by MIUR--PRIN \textit{Variet\`a algebriche: geometria, aritmetica e strutture di Hodge} and by the MIUR--FIRB project  \textit{Spazi di moduli e applicazioni}. M. Melo was supported by the FCT project \textit{Espa\c cos de Moduli em Geometria Alg\'ebrica} (PTDC/MAT/111332/2009), by the FCT project \textit{Geometria Alg\'ebrica em Portugal} (PTDC/MAT/099275/2008) and by the Funda\c c\~ao Calouste Gulbenkian program ``Est\'imulo \`a investiga\c c\~ao 2010''. F. Viviani is a member of the research center CMUC (University of Coimbra) and he was supported by  the FCT project \textit{Espa\c cos de Moduli em Geometria Alg\'ebrica} (PTDC/ MAT/111332/2009) and by MIUR--FIRB project \textit{Spazi di moduli e applicazioni}. }

\maketitle

\tableofcontents

\section{Introduction}

\subsection{Motivation and previous related works}

One of the first successful applications of Geometric Invariant Theory (GIT for short), and perhaps one of the major motivations for its development by Mumford and his co-authors
(see \cite{GIT}),  was the construction of the moduli space $M_g$ of smooth curves of genus $g\geq 2$ and its compactification $\Mg$ via \emph{stable curves} (i.e. connected nodal projective curves with finite automorphism group), carried out by Mumford (\cite{Mum}) and Gieseker (\cite{Gie}).
Indeed, the moduli space of stable curves was constructed as a GIT quotient of a locally closed subset of a suitable Hilbert scheme
(as in \cite{Gie}) or  Chow scheme (as in \cite{Mum})
parametrizing $n$-canonically embedded curves, for $n$ sufficiently large.  More precisely, Mumford in \cite{Mum} works under the assumption that $n\geq 5$ and
Gieseker in \cite{Gie} requires the more restrictive assumption that $n\geq 10$. However, it was later discovered that Gieseker's approach can also be extended to the case $n\geq 5$ (see
\cite[Chap. 4, Sec. C]{HM} or \cite[Sec. 3]{Mo}).

Recently, there has been a lot of interest in extending the above GIT analysis to smaller values of $n$, especially in connection with the so called Hassett-Keel program whose ultimate goal
is to find the minimal model of $M_g$ via successive constructions of modular birational models of $\Mg$ (see \cite{FS} and \cite{AH} for nice overviews).

The first work in this direction is due to Schubert, who described in \cite{Sch}  the GIT quotient of the locus of $3$-canonically embedded curves (of genus $g\geq 3$) in the Chow scheme
as the coarse moduli space $\Mgp$ of pseudo-stable curves (or \emph{p-stable curves} for short). These are connected projective curves with finite automorphism group, whose only singularities are nodes and cusps, and which have no elliptic tails (i.e. connected subcurves of arithmetic genus one meeting the rest of the curve in one point).
Since the GIT quotient analyzed by Schubert is geometric (i.e. there are no
strictly semistable objects), one gets exactly the same description working with $3$-canonically embedded curves inside the Hilbert scheme
%\footnote{with respect to the asymptotic linearization, see Section \ref{S:GIT} for more details.}
(see \cite[Prop. 3.13]{HH2}).
Later, Hassett-Hyeon constructed in \cite{HH1} a modular map $T:\Mg\to \Mgp$ which on geometric points sends a stable curve onto the p-stable curve obtained by contracting
all its elliptic tails to cusps. Moreover, the authors of loc. cit. identified the map $T$ with the first contraction in the Hassett-Keel program for $\Mg$.

The case of $4$-canonical curves was worked out by Hyeon-Morrison in \cite{HMo}. The Hilbert GIT-semistable points turn out to correspond again to p-stable curves, while the Chow GIT-semistable locus is strictly bigger and it consists of
weakly-pseudo-stable curves (or \emph{wp-stable curves} for short), which are connected projective curves with finite automorphism group, whose only singularities are nodes and cusps (and having possibly elliptic tails).
However, Hyeon-Morrison also proved that the GIT quotient for the Chow scheme turns out to be again isomorphic to
the moduli space $\Mgp$ of p-stable curves, a fact that can be reinterpreted as saying that the non-separated stack of wp-stable curves and its open and proper substack of p-stable curves have the same moduli space (see \S~\ref{S:wp-stable} for more details).

Finally, the case of $2$-canonical curves was studied by Hassett-Hyeon in \cite{HH2}, where the authors described the Hilbert GIT quotient  $\ov{M}_g^h$ and the Chow GIT quotient
$\ov{M}_g^c$
(they are now different), as moduli spaces of $h$-semistable (resp. $c$-semistable) curves; see loc. cit. for the precise description. Moreover, they constructed a small contraction
$\Psi:\Mgp\to \ov{M}_g^c$ and identified the natural map 
$\Psi^+:\ov{M}_g^h\to \ov{M}_g^c$ as the flip of $\Psi$.
These maps are then interpreted as further steps in the Hassett-Keel program for $\Mg$.

For some partial results on the GIT
quotient for the Hilbert scheme of $1$-canonically  embedded curves, we refer the reader to the work of Alper, Fedorchuk and Smyth (see \cite{AFS}).

\vspace{0.2cm}

From the point of view of constructing new projective birational models of $\Mg$, it is of course natural to restrict the GIT analysis to the locally closed subset inside the Hilbert or Chow scheme
parametrizing $n$-canonical embedded curves. However, the problem of describing the whole GIT quotient seems very natural and interesting too. The first result in this direction is the
pioneering work of Caporaso \cite{Cap}, where the author describes the GIT quotient of the Hilbert scheme of connected curves of genus $g\geq 3$ and degree $d\geq 10(2g-2)$ in $\P^{d-g}$. The GIT quotient
% that she obtains
obtained by Caporaso in loc. cit. is indeed a modular compactification of the universal Jacobian  $J_{d,g}$, which is
the moduli scheme parametrizing pairs $(C,L)$ where $C$ is a smooth curve of genus $g$ and $L$ is a line bundle on $C$
of degree $d$. Note that recently Li and Wang in \cite{LW} have studied Chow (semi-)stability of polarized nodal curves of sufficiently high degree, giving in particular a different proof of Caporaso's result for $d\gg 0$\footnote{Notice that Li-Wang worked more generally with polarized pointed weighted nodal curves.}.

Our work is motivated by the following

\vspace{0,2cm}
\noindent \textbf{Problem:}
\emph{Describe the GIT quotient for the Hilbert and Chow scheme of curves of genus $g$ and degree $d$ in $\P^{d-g}$,
as $d$ decreases with  respect to $g$.}

\subsection{Our results}

%The motivation of this work is to study the GIT quotients for Hilbert and Chow schemes of curves, as the degree gets smaller compared with the genus.

In order to describe our results, we need to introduce some notation. Fix an integer $g\geq 2$.  For any natural number  $d$, denote by $\Hilb_d$ the Hilbert scheme of curves
of degree $d$ and arithmetic genus $g$ in $\P^{d-g}$; denote by $\Chow_d$ the Chow scheme of $1$-cycles of degree $d$ in $\P^{d-g}$ and by
$$\Ch:\Hilb_d\to \Chow_d$$
the map sending a one dimensional subscheme $[X\subset \P^{d-g}]\in \Hilb_d$ to its $1$-cycle. The linear algebraic group $\SL_{d-g+1}$ acts naturally on
$\Hilb_d$ and $\Chow_d$ so that $\Ch$ is an equivariant map; moreover, these actions are naturally linearized (see Section \ref{Sec:Hilb-Chow} for details
\footnote{In particular, when working with $\Hilb_d$, we will always consider the $m$-linearization for $m \gg 0$;
see Section \ref{Sec:Hilb-Chow} for details.}), so it makes
sense to talk about (GIT) (semi-,poly-)stability of a point in $\Hilb_d$ and $\Chow_d$.

\vspace{0.2cm}

The aim of this work is to give a complete characterization of the (semi-,poly-)stable points $[X\subset \P^{d-g}]\in \Hilb_d$ or of its image $\Ch([X\subset \P^{d-g}])\in \Chow_d$, provided that $d>2(2g-2)$.
Our characterization  of Hilbert or Chow (semi-, poly-)stability will require some conditions on the singularities of $X$ and some conditions on the multidegree of the line bundle $\OO_X(1)$.
Let us introduce the relevant definitions.

A curve $X$ is said to be \emph{quasi-stable}  if it is obtained from a stable curve $Y$ by bubbling some of its nodes, i.e. by taking the partial normalization
of $Y$ at some of its nodes and inserting a $\P^1$ connecting the two branches of each node.
A curve $X$ is said to be \emph{quasi-p-stable} (resp. \emph{quasi-wp-stable}) if it is obtained from a p-stable curve
(resp. a wp-stable curve) $Y$ by bubbling some of its nodes (as before) and bubbling  some of its cusps, i.e.
by taking the partial normalization  of $Y$ at some of its cusps and inserting a $\P^1$ tangent to the branch point of each cusp (the singularity that one gets by bubbling  a cusp is called \emph{tacnode with a line}). Note that quasi-stable and quasi-p-stable curves are special cases of quasi-wp-stable curves: the quasi-stable curves are exactly the quasi-wp-stable curves without cusps nor tacnodes with a line; the quasi-p-stable curves are exactly the quasi-wp-stable curves without elliptic tails.
Given a quasi-wp-stable curve $X$, we call the $\P^1$'s inserted by bubbling  nodes or cusps of $Y$ the \emph{exceptional components}, and we denote by $\exc\subset X$
the union of all of them.

A line bundle $L$ of degree $d$ on a quasi-wp-stable curve $X$ of genus $g$ is said to be \emph{balanced} if for each subcurve $Z\subset X$ the following inequality (called the basic inequality) is satisfied
\begin{equation*}
\left| \deg_ZL - \frac{d}{2g-2}\deg_Z(\omega_X)\right|\leq \frac{|Z\cap Z^c|}{2},\tag{*}
\end{equation*}
where $|Z\cap Z^c|$ denotes the length of the $0$-dimensional subscheme of $X$ obtained as the scheme-theoretic intersection of $Z$ with the complementary subcurve $Z^c:=\ov{X\setminus Z}$. A balanced line bundle
$L$ on $X$ is said to be \emph{properly balanced} if  the degree of $L$ on each exceptional component of $X$ is $1$.
Moreover, a properly balanced line bundle $L$
is said to be \emph{strictly balanced} (resp. \emph{stably balanced}) if the basic inequality (*) is strict except possibly  for the subcurves $Z$ such that $Z\cap Z^c\subset \exc$
(resp. such that $Z$ or $Z^c$ is entirely contained in $\exc$).

The last definition concerns the behavior of irreducible elliptic tails of $X$ (i.e. irreducible components of $X$ of arithmetic genus one and meeting the rest of the curve in one point) with respect to a line bundle on $X$. Let $F$ be an irreducible elliptic tail of $X$ and let $p$ denote the intersection point between $F$ and the complementary subcurve.
%Let $L$ be an ample line bundle on $X$ and denote by $d_F=\deg_F L$ the degree of $L$ on $F$.
Given a line bundle $L$ on $X$, we can write $L_{|F}=\OO_F((d_F-1)p+q)$, where $d_F=\deg_F L$ denotes the degree of $L$ on $F$, for a uniquely determined smooth point $q$ of $F$.
We say that $F$ is \emph{special} with respect to $L$ (or simply special when the line bundle $L$ is clear from
the context) and \emph{non-special} (with respect to $L$) otherwise.

\vspace{0.2cm}

Now, we can state the main theorems proved in this manuscript.
Our first main result extends the description of semistable (resp. polystable, resp. stable) points  $[X\subset \P^{d-g}]\in \Hilb_d$ given by Caporaso in \cite{Cap} to the case $d>4(2g-2)$ and also to the Chow scheme.
% is an extension of the result of Caporaso \cite{Cap} to the case $d>4(2g-2)$ and also considering the Chow scheme.

\begin{theoremalpha} \label{T:MainThm1}
Consider a point $[X\subset \P^{d-g}]\in \Hilb_d$ with $d>4(2g-2)$; assume moreover that $X$ is connected.
Then the following conditions are equivalent:
\begin{enumerate}[(i)]
\item $[X\subset \P^{d-g}] $ is semistable (resp. polystable, resp. stable);
\item $\Ch([X\subset \P^{d-g}])$ is semistable (resp. polystable, resp. stable);
\item $X$ is quasi-stable and $\OO_X(1)$ is balanced (resp. strictly balanced, resp. stably balanced).
\end{enumerate}
In each of the above cases, $X\subset \P^{d-g}$ is non-degenerate and linearly normal, and $\OO_X(1)$ is  non-special.

\end{theoremalpha}

Theorem \ref{T:MainThm1} follows by combining Theorem \ref{T:semistable}\eqref{T:semistable1}, Corollary \ref{C:polystable}\eqref{C:polystable1} and Corollary \ref{C:stable}\eqref{C:stable1}.

When $d=4(2g-2)$, the description of the semistable locus in Theorem \ref{T:MainThm1} breaks down and we get that the Hilbert and Chow semistable loci admit a different description.

\begin{theoremalpha}\label{T:1crit}
Consider a point $[X\subset \P^{d-g}]\in \Hilb_d$ with $d=4(2g-2)$ and $g\geq 3$;
assume moreover that $X$ is connected. Then the following holds:
\begin{enumerate}[(i)]
\item \label{T:1crit1} $[X\subset \P^{d-g}] $ is semistable if and only if $X$ is quasi-wp-stable without tacnodes nor special elliptic tails (with respect to $\OO_X(1)$) and $\OO_X(1)$ is balanced.
\item \label{T:1crit2} $\Ch([X\subset \P^{d-g}])$ is semistable if and only if $X$ is quasi-wp-stable without tacnodes  and $\OO_X(1)$ is balanced.
\end{enumerate}
In each of the above cases, $X\subset \P^{d-g}$ is non-degenerate and linearly normal, and $\OO_X(1)$ is  non-special.
\end{theoremalpha}

Theorem \ref{T:1crit} follows from Theorem \ref{T:semistable5}. For a description of the Hilbert or Chow
polystable (resp. stable) locus, we refer the reader to Corollary \ref{C:polystable5} (resp. Corollary
\ref{C:stable5}).

The next range where the Hilbert and Chow GIT-semistable loci coincide and stay constant is the interval
$\frac{7}{2}(2g-2)<d<4(2g-2)$, where we have the following description.

\begin{theoremalpha} \label{T:MainThm2}
Consider a point $[X\subset \P^{d-g}]\in \Hilb_d$ with $\frac{7}{2}(2g-2)<d<4(2g-2)$ and $g\geq 3$; assume moreover that $X$ is connected.
Then the following conditions are equivalent:
\begin{enumerate}[(i)]
\item $[X\subset \P^{d-g}] $ is semistable (resp. polystable, resp. stable);
\item $\Ch([X\subset \P^{d-g}])$ is semistable (resp. polystable, resp. stable);
\item $X$ is quasi-wp-stable without tacnodes nor special elliptic tails (with respect to $\OO_X(1)$) and $\OO_X(1)$ is balanced.
\end{enumerate}
In each of the above cases, $X\subset \P^{d-g}$ is non-degenerate and linearly normal, and $\OO_X(1)$ is  non-special.
\end{theoremalpha}

Theorem \ref{T:MainThm2} follows by combining Theorem \ref{T:semistable4}, Corollary \ref{C:polystable4} and Corollary \ref{C:stable4}.

When $d=\frac{7}{2}(2g-2)$, the description of the Hilbert or Chow semistable locus in Theorem \ref{T:MainThm2} breaks down again and we get that the Hilbert and Chow semistable loci admit a different description, similarly to the case $d=4(2g-2)$.

\begin{theoremalpha}\label{T:2crit}
Consider a point $[X\subset \P^{d-g}]\in \Hilb_d$ with $d=\frac{7}{2}(2g-2)$ and $g\geq 3$;
assume moreover that $X$ is connected. Then the following holds:
\begin{enumerate}[(i)]
\item \label{T:2crit1} $[X\subset \P^{d-g}] $ is semistable if and only if $X$ is quasi-p-stable and $\OO_X(1)$ is balanced.
\item \label{T:2crit2} $\Ch([X\subset \P^{d-g}])$ is semistable if and only if $X$ is quasi-wp-stable without special elliptic tails (with respect to $\OO_X(1)$) and $\OO_X(1)$ is balanced.
\end{enumerate}
In each of the above cases, $X\subset \P^{d-g}$ is non-degenerate and linearly normal, and $\OO_X(1)$ is  non-special.
\end{theoremalpha}

Theorem \ref{T:2crit} follows from Theorem \ref{T:semistable3}. For a description of the Hilbert or Chow
polystable (resp. stable) locus, we refer the reader to Corollary \ref{C:polystable3} (resp. Corollary
\ref{C:stable3}).

The next range where the Hilbert and Chow semistable loci coincide and stay constant is the interval
$2(2g-2)<d<\frac{7}{2}(2g-2)$, where we have the following description.

\begin{theoremalpha} \label{T:MainThm3}
Consider a point $[X\subset \P^{d-g}]\in \Hilb_d$ with $2(2g-2)<d<\frac{7}{2}(2g-2)$ and $g\geq 3$; assume moreover that $X$ is connected.
Then the following conditions are equivalent:
\begin{enumerate}[(i)]
\item $[X\subset \P^{d-g}] $ is semistable (resp. polystable, resp. stable);
\item $\Ch([X\subset \P^{d-g}])$ is semistable (resp. polystable, resp. stable);
\item $X$ is quasi-p-stable and $\OO_X(1)$ is balanced (resp. strictly balanced, resp. stably balanced).
\end{enumerate}
In each of the above cases, $X\subset \P^{d-g}$ is non-degenerate and linearly normal, and $\OO_X(1)$ is  non-special.

\end{theoremalpha}

The above Theorem \ref{T:MainThm3} follows by combining Theorem \ref{T:semistable}\eqref{T:semistable2}, Corollary \ref{C:polystable}\eqref{C:polystable2} and Corollary \ref{C:stable}\eqref{C:stable2}.
Note that Theorem \ref{T:MainThm3} breaks down for $d=2(2g-2)$ since, for this value of $d$,  there are stable points $[X\subset \P^{d-g}]\in \Hilb_d$ (hence semistable points $\Ch([X\subset \P^{d-g}])\in \Chow_d$) with $X$ having arbitrary tacnodal singularities and not just tacnodes with a line (see Remark \ref{R:potpseudo-sharp}).

\vspace{0.2cm}

Let us now briefly comment on the assumptions of the above theorems.
First of all, with the exception of Theorem \ref{T:MainThm1}, the other four theorems require that $g\geq 3$.
The reason for this assumption is that the moduli stack of p-stable curves of genus $g$ is not separated for $g=2$
(see \S~\ref{S:wp-stable}) and this causes some extra-difficulties in the GIT analysis. In particular, we use the hypothesis that $g\geq 3$ (whenever p-stable or wp-stable curves are involved) in a crucial way in Theorem \ref{T:auto-grp}, Propositions \ref{P:deg-strata} and \ref{P:completeness}.
Therefore, for simplicity, we restrict in this manuscript to the case $g\geq 3$ whenever dealing with p-stable or
wp-stable curves (i.e. for $d\leq 4(2g-2)$); the GIT analysis for $g=2$ and the missing values of $d$ (i.e. $d=5,6,7,8$) will be dealt with in a future work.
%somewhere else.

%Finally, if $g=3$ then Heyon-Lee proved in \cite{HL} that a $3$-canonical irreducible p-stable curve $X\subset \P^4$ of genus $2$ with one cusp is not GIT polystable (while it is GIT semistable), which shows that the description of  GIT stable and GIT polystable points given in Theorem \ref{T:MainThm2} is false in this case. Probably, the description of GIT semistable points given in Theorem \ref{T:MainThm2} is still true for $g=2$; however, for simplicity, we restrict in this manuscript to the case $g\geq 3$ whenever dealing with quasi-p-stable curves.

Another hypothesis that is present in all the above theorems is the connectivity of the curve $X$.
Indeed, under the assumption that $d>2(2g-2)$, the locus of connected curves in the Hilbert or Chow semistable locus is a connected and irreducible component (see the beginning of Section \ref{S:stratifica} and Corollary \ref{C:irr-quot}), that we call the main component (see Section \ref{S:map-pstable}).
In Section \ref{sec:extra}, we prove that there are no other components in the Hilbert or Chow semistable locus if and only if $\gcd(d, g-1)=1$. More generally, we prove in Theorem \ref{T:conn-comp} that the number of connected components (which are also irreducible) of the Hilbert or Chow semistable locus is equal to the number of partitions of  $\gcd(d, g-1)$.

\vspace{0.2cm}

Now let us make some comments on the strategy of the proof. The approach to the problem of determining the semistable locus is the same as that developed by Mumford, Gieseker and Caporaso: firstly we use Hilbert-Mumford numerical criterion in order to find necessary conditions for a point $[X\subset \P^{d-g}]$ in the Hilbert scheme to be semistable (see Fact 4.20, Corollary 9.4 and Corollary 9.7)
%\ref{F:GM2}\ref{C:quasi-wp-inter}\ref{C:quasi-p-stable}
and finally we characterize the entire semistable locus using combinatorial properties of the multidegree of $\OO_X(1)$ and separateness property of suitable stacks of curves. For $d\geq 4(2g-2)$ and $2(2g-2)<d<\displaystyle\frac{7}{2}(2g-2)$ this strategy does work because the semistable locus consists only of quasi-stable and quasi-pseudo-stable curves respectively, thus in the second step it suffices to work with separated stacks like $\MMg$ and $\MMgp$ respectively (for $\MMgp$ it is necessary to suppose that $g\geq 3$, because $\overline{\mathcal M}_2^{\rm p}$ is not separated).

Unfortunately for $\displaystyle\frac{7}{2}(2g-2)\leq d \leq 4(2g-2)$ it is not very hard to prove the existence of semistable curves admitting cusps and elliptic tails (see Remark 11.4 and Corollary 12.3),
%\ref{rmk:code} \ref{C:esistcode}
so that we have to work with the stack $\MMgwp$ of weakly-pseudo-stable curves, which is not separated. For this reason it is necessary to use other techniques. A very naive idea is to apply again Hilbert-Mumford numerical criterion. We recall that the Hilbert-Mumford criterion states that given a curve $X\subset \P^{d-g}$
$$
[X\subset \P^{d-g}] \text{ is semistable }\Longleftrightarrow \mu([X\subset \P^{d-g}],\rho)\geq 0\text{ for each 1ps }\rho:\Gm\lra \SL_{d-g+1}
$$
(see \cite{Dol} for the definition of $\mu([X\subset \P^{d-g}],\rho)$). ``Unfortunately'' this criterion is easier to apply when we would like to prove the instability of curves rather than the semistability.

One way to solve this difficulty is to apply Tits' results about the parabolic group associated to a fixed one-parameter subgroup (see for more details \cite[Sec. 9.5]{Dol} or \cite[Chap. 2, Sec. 2]{GIT}). These results allowed G. Kempf to prove that if $[X\subset \P^{d-g}]$ is unstable, then there exists a unique one-parameter subgroup which in some sense is responsible for the instability of $[X\subset \P^{d-g}]$. The idea, hence, is to use the properties of the parabolic group to study the behavior of curves having elliptic tails under the action of one parameter subgroups: we prove that, if $[X\subset \P^{d-g}]$ has an elliptic tail, i. e. $X$ is the union of an elliptic curve $F$ and another curve $C$ such that $F$ and $C$ intersect each other in one node, the GIT analysis can be restricted to 1ps $\rho:\Gm\lra \SL_{d-g+1}$ diagonalized by bases of $\P^{d-g}$ that come out from the union of bases of the linear spans $\langle F \rangle$ and $\langle C \rangle$ in $\P^{d-g}$. In other words, we can study the semistability of $X$ by analyzing the subcurves $F$ and $C$ in their linear spans separately. Essentially, this is the content of the Criterion of stability of tails (see Proposition 8.3). %\ref{prop:stab-tail}).

Motivated by this criterion, we study the behavior of polarized elliptic curves $F\subset \P^r$ for some suitable $r$ under the action of one parameter subgroups and we prove that for $\displaystyle\frac{7}{2}(2g-2)<d<4(2g-2)$ there are semistable curves $[X\subset \P^{d-g}]$ that admit non-special elliptic tails (see Remark 11.4)
%\ref{rmk:code})
for all models of non-special elliptic tail (see Corollary 12.3).
%\ref{C:esistcode}).

The final part of the GIT analysis is based on a nice numerical trick. We will explain this trick briefly in the case $\displaystyle\frac{7}{2}(2g-2)< d < 4(2g-2)$. Given a quasi-wp-stable curve $[X\subset \P^{d-g}]\in \Hilb_d$, as above with $F$ non-special, we define a new polarized curve $X'$ by replacing the polarized subcurve $F$ with a polarized smooth curve $Y$ of genus $g$ and degree $d-d_F$ so that $Y$ and $C$ intersect again in one node. If we denote by $d'$ and $g'$ respectively the degree of the new line bundle $L'$ and the genus of $X'$, one can consider the Hilbert point $[X'\subset \P^{d'-g'}]\in \Hilb_{d'}$. It can be easily checked that
$$
\frac{d'}{2g'-2}=\frac{d}{2g-2}
$$
and
$$
\OO_X(1)\text{ is balanced }\Longleftrightarrow \OO_{X'}(1) \text{ is balanced.}
$$
Applying our criterion, one proves that
$$
[X'\subset \P^{d'-g'}] \text{ is semistable } \Longrightarrow [X\subset \P^{d-g}]\text{ is semistable},
$$
so that the GIT analysis can be completed by an induction argument on the number of non-special elliptic tails of $X$. The proof of the base of induction requires the separateness of $\MMgp$, so that we need to suppose again that $g\geq 3$.

\hspace{0.2cm}

Let us now comment on the origin of the two \emph{critical values} $d=4(2g-2)$ and $d=\frac{7}{2}(2g-2)$, at which the Hilbert and Chow semistable loci change. It turns out that the existence of these two critical values is related to the presence in the Chow semistable locus of a point $\Ch([X \subset \P^r])$ whose stabilizer subgroup in $\PGL_{d-g+1}$ contains a copy of the multiplicative subgroup $\Gm$. This resembles very much what happens in the Hassett-Keel program for $\MMg$ where the variations of the log canonical models of $\MMg$ are expected to be accounted for by curves with a $\Gm$-automorphism; see \cite{AFS0}.

The first critical value $d=4(2g-2)$ is due to the presence of Chow semistable points $\Ch([X_0\subset \P^{d-g}])\in \Chow_ d$ such that $X_0$ has a cuspidal elliptic tail which is special with respect to $\OO_{X_0}(1)$. Such a point
has a non-trivial copy of the multiplicative group $\Gm$ in its stabilizer subgroup inside $\PGL_{d-g+1}$ (see Lemma \ref{L:aut-stab} and Theorem \ref{T:auto-grp}). With respect to a suitable one-parameter subgroup $\rho:\Gm\to \GL_{d-g+1}$, whose image in $\PGL_{d-g+1}$ is contained in the stabilizer subgroup of
$[X_0\subset \P^{d-g}]$ (as in the proof of Theorem \ref{T:spec-ell}), we prove in Theorem \ref{T:basin-cusps} that the basins of attraction of $[X_0\subset \P^{d-g}]$ with respect to $\rho$ and $\rho^{-1}$ are the ones depicted in Figure \ref{F:1crit} below.
\vspace{0.5cm}

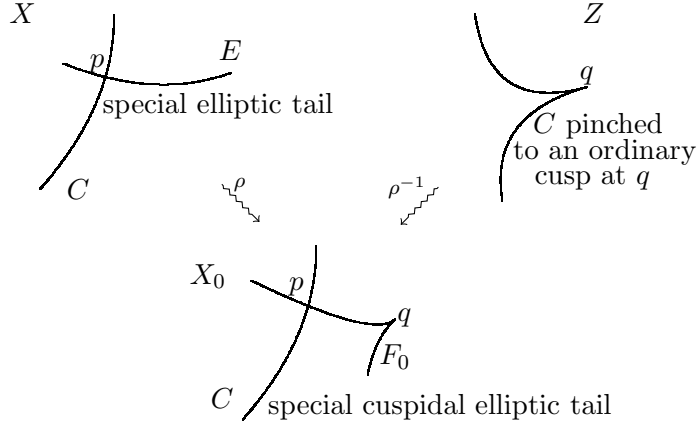
\begin{figure}[!h]
\begin{center}
\unitlength .5mm % = 2.845pt
\linethickness{0.4pt}
\ifx\plotpoint\undefined\newsavebox{\plotpoint}\fi % GNUPLOT compatibility
\begin{picture}(172,100)(20,80)
\qbezier(81.75,119.5)(114.625,103.625)(120,109.25)
\qbezier(120,109.25)(114.625,104.125)(112.75,94.5)
\qbezier(79.5,82.5)(99.25,105)(99,128.5)
\qbezier(25.75,144)(45.5,166.5)(45.25,190)
\qbezier(31.75,177.25)(56.5,167.5)(76.25,174.75)
\qbezier(141.25,190.5)(144.875,163.75)(171,171)
\qbezier(171,171)(144.75,163.625)(148.5,140.75)
\put(21,190.75){\makebox(0,0)[cc]{$X$}}
\put(172.75,190.75){\makebox(0,0)[cc]{$Z$}}
\put(70.25,120.75){\makebox(0,0)[cc]{$X_0$}}
\put(76.25,180.75){\makebox(0,0)[cc]{$E$}}
\put(73,166){\makebox(0,0)[cc]{special elliptic tail}}
\put(40.75,177.25){\makebox(0,0)[cc]{$p$}}
\put(122.5,109.75){\makebox(0,0)[cc]{$q$}}
\put(171,174){\makebox(0,0)[cc]{$q$}}
\put(132,86){\makebox(0,0)[cc]{special cuspidal elliptic tail}}
\put(35.75,143.75){\makebox(0,0)[cc]{$C$}}
\put(174,160.75){\makebox(0,0)[cc]{$C$ pinched}}
\put(177,153.75){\makebox(0,0)[cc]{to an ordinary }}
\put(172.5,146.75){\makebox(0,0)[cc]{cusp at $q$}}
\put(120,99.25){\makebox(0,0)[cc]{$F_0$}}
\put(93.75,117.75){\makebox(0,0)[cc]{$p$}}
\put(74,88.25){\makebox(0,0)[cc]{$C$}}
\put(75,141){\makebox(0,0)[cc]{$\xxrsquigarrow{\phantom{aaa}\rho}$}}
\put(122.5,140.75){\makebox(0,0)[cc]{$\xxlsquigarrow{\phantom{a}\rho^{-1}}$}}
\end{picture}
\caption{The basin of attraction of a curve $X_0$ with a special cuspidal elliptic tail $F_0$.}
 %with respect to $\rho$ and $\rho^{-1}$.}
 \label{F:1crit}
\end{center}
\end{figure}
This implies that, in crossing the critical value $d=4(2g-2)$ (i.e. as $\frac{d}{2g-2}$ passes from $4+\epsilon$ to $4-\epsilon$ for a small $\epsilon$), special elliptic tails become (Hilbert or
Chow) unstable and they get replaced by cusps. Moreover, Hilbert semistability for $d=4(2g-2)$ behaves like
Hilbert (or Chow) semistability for $\frac{7}{2}(2g-2)<d<4(2g-2)$;
hence Hilbert semistability is strictly stronger than Chow semistability for $d=4(2g-2)$.

%Already in the Hilbert semistable locus for $d=4(2g-2)$, such a phenomenon is visible (see Theorem \ref{T:1crit}\eqref{T:1crit1}): quasi-stable curves with special elliptic tails are Hilbert unstable, while certain quasi-wp-stable curves with cusps are Hilbert semistable.

The second critical value $d=\frac{7}{2}(2g-2)$ is due to the presence of Chow semistable points
$\Ch([X_0\subset \P^{d-g}])\in \Chow_ d$ such that $X_0$ has a tacnodal elliptic tail. Such a point
has a non-trivial copy of the multiplicative group $\Gm$ into its stabilizer subgroup with respect to $\PGL_{d-g+1}$ (see Lemma \ref{L:aut-stab} and Theorem \ref{T:auto-grp}). With respect to a suitable one-parameter subgroup $\rho:\Gm\to \GL_{d-g+1}$ whose image in $\PGL_{d-g+1}$ is contained in the stabilizer subgroup of
$[X_0\subset \P^{d-g}]$ (as in the proof of Theorem \ref{T:ell-curves}),
the basins of attraction of $[X_0\subset \P^{d-g}]$ with respect to $\rho$ and $\rho^{-1}$ are depicted in Figure \ref{F:2crit} below (see Theorem \ref{T:basintacn} for the proof).
\begin{figure}[!h]
\begin{center}
\unitlength .5mm % = 2.845pt
\linethickness{0.35pt}
\ifx\plotpoint\undefined\newsavebox{\plotpoint}\fi % GNUPLOT compatibility
\begin{picture}(176,95)(0,87)
\qbezier(77.5,87.75)(97.25,110.25)(97,133.75)
\qbezier(25.75,144)(45.5,166.5)(45.25,190)
\qbezier(31.75,177.25)(56.5,167.5)(76.25,174.75)
\put(21,183.75){\makebox(0,0)[cc]{$X$}}
\put(172.75,183.75){\makebox(0,0)[cc]{$Z$}}
\put(74.25,186){\makebox(0,0)[cc]{$F$}}
\put(80,167){\makebox(0,0)[cc]{non-special elliptic tail}}
\put(40.75,177.25){\makebox(0,0)[cc]{$p$}}
\put(35.75,143.75){\makebox(0,0)[cc]{$Y$}}
\put(91.75,123){\makebox(0,0)[cc]{$p$}}
\put(72,93.5){\makebox(0,0)[cc]{$Y$}}
\put(73,142){\makebox(0,0)[cc]{$\xxrsquigarrow{\phantom{aaa}\rho}$}}
\put(120.5,144){\makebox(0,0)[cc]{$\xxlsquigarrow{\rho^{-1}}$}}
\qbezier(133.25,184.25)(176,165.75)(128.75,149.25)
\put(153.5,186.25){\line(0,-1){39}}
\put(188,169.25){\makebox(0,0)[cc]{tacnode with a line}}
\put(105,94){\makebox(0,0)[cc]{$F_0$}}
\put(117.5,130){\makebox(0,0)[cc]{$E$}}
\put(151,110.5){\makebox(0,0)[cc]{tacnodal elliptic tail}}
\put(161,148.25){\makebox(0,0)[cc]{$E$}}
\put(134.75,90.75){\makebox(0,0)[cc]{$X_0$}}
\qbezier(82.25,125.75)(90.625,117.25)(101.5,120.75)
\qbezier(101.5,120.75)(113.875,124.375)(114.75,116.5)
\qbezier(114.75,116.5)(116.125,109.625)(107,107.25)
\qbezier(107,107.25)(99,106.375)(97,103)
\put(115.25,127.5){\line(0,-1){28.25}}
\end{picture}
\caption{The basin of attraction of a curve $X_0$ with a tacnodal elliptic tail $F_0$.}
 %with respect to the 1ps $\rho$ and $\rho^{-1}$.}
\label{F:2crit}
\end{center}
\end{figure}
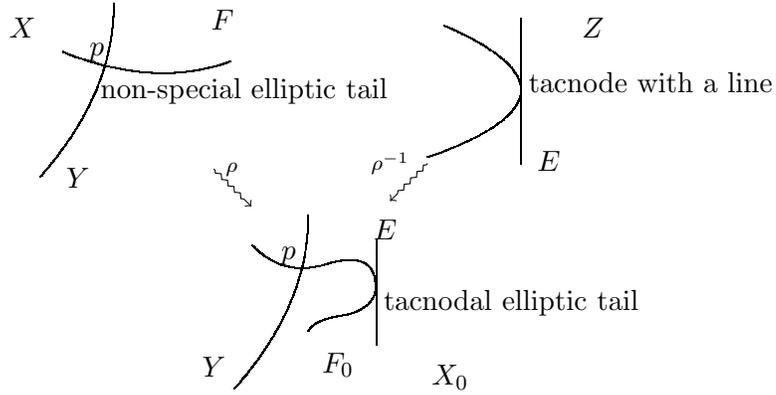
This implies that, in crossing the critical value $d=\frac{7}{2}(2g-2)$ (i.e. as $\frac{d}{2g-2}$ passes from $\frac{7}{2}+\epsilon$ to $\frac{7}{2}-\epsilon$ for a small $\epsilon$), non-special elliptic tails become (Hilbert or
Chow) unstable and they get replaced by tacnodes with a line.
Moreover, Hilbert semistability for $d=\frac{7}{2}(2g-2)$ behaves like Hilbert (or Chow) semistability for $2(2g-2)<d<\frac{7}{2}(2g-2)$; hence Hilbert semistability is strictly stronger than Chow semistability for $d=\frac{7}{2}(2g-2)$.

To conclude, observe that the basins of attraction of Figure \ref{F:1crit} are already visible in the $4$-canonical locus inside ${\rm Hilb}_{4(2g-2)}$ or ${\rm Chow}_{4(2g-2)}$ (because all the elliptic tails are special with respect to the canonical line bundle!) and indeed they were already considered by Hyeon-Morrison in \cite{HMo};
on the other hand, the basins of attraction of Figure \ref{F:2crit} are clearly not visible inside the pluricanonical locus (because they occur for a fractional value of $\frac{d}{2g-2}$!).
%the curve $X_0$ of Figure \ref{F:2crit} cannot be embedded by a power of the canonical line bundle, which restricts to the trivial line bundle on the exceptional component $E$.

\vspace{0.2cm}

Finally, one last comment on the orbit identifications that occur in the GIT quotient. It is well-known the GIT quotient of the (Hilbert or Chow) semistable locus parametrizes polystable orbit (i.e. semistable orbits that are closed inside the semistable locus) and each semistable orbit contains a unique polystable orbit in its closure. If $d>2(2g-2)$ but $d\neq \frac{7}{2}(2g-2)$ or $4(2g-2)$, then Theorems \ref{T:MainThm1}, \ref{T:MainThm2}, \ref{T:MainThm3} imply that
the polystable orbits correspond to the orbits of Hilbert semistable points $[X\subset \P^{d-g}]$  such that moreover $\OO_X(1)$ is strictly balanced (and similarly for Chow semistable points). Indeed, we prove in Section \ref{S:extreme-inequ} that if a Hilbert semistable point $[X\subset \P^{d-g}]$ is such that $\OO_X(1)$ achieves one of the extremes of the basic inequality at a subcurve $Z\subset X$ such that $Z\cap Z^c\subsetneq X_{\rm exc}$, then there is an isotrivial specialization of $[X\subset \P^{d-g}]$ to a Hilbert semistable point $[X'\subset \P^{d-g}]$ such that $X'$ is obtained from $X$ by bubbling  the nodes of $(Z\cap Z^c)\setminus X_{\rm exc}$ (see Theorem \ref{T:spec-clos-orb}); hence the orbit of $[X\subset \P^{d-g}]$ contains the orbit of $[X'\subset \P^{d-g}]$ in its closure. The same thing happens for Chow semistable points. Therefore, Theorems \ref{T:MainThm1}, \ref{T:MainThm2} and \ref{T:MainThm3} say that these are the only orbit identifications that occur in the Hilbert or Chow GIT quotients outside of the critical values $d= \frac{7}{2}(2g-2)$ or $4(2g-2)$. Moreover, an easy
combinatorial argument (see \cite[Lemma 6.3]{Cap}) shows that the extreme of the basic inequalities can be achieved if and only if $\gcd(d+1-g, 2g-2)\neq 1$; therefore if $\gcd(d+1-g, 2g-2)=1$ and $d\neq \frac{7}{2}(2g-2)$ or $4(2g-2)$
then the Hilbert or Chow GIT quotients that we get are geometric, i.e. semistable points are also stable.

On the other hand, if $d$ is equal to one of the two critical values $\frac{7}{2}(2g-2)$ or $4(2g-2)$, then the orbits identifications in the Hilbert and Chow GIT quotient are different. Indeed, while in the Hilbert GIT quotient $\ov{Q}_{d,g}^h$ it is still true that the unique orbits identifications are given by the isotrivial specializations described above, in the Chow GIT quotient $\ov{Q}_{d,g}^c$ there are new isotrivial specializations that correspond to the basins of attraction depicted in Figure \ref{F:1crit} for $d=4(2g-2)$ and Figure \ref{F:2crit} for $d=\frac{7}{2}(2g-2)$. Note that there is a natural morphism $\Xi: \ov{Q}_{d,g}^h\to \ov{Q}_{d,g}^c$ from the Hilbert
GIT quotient to the Chow GIT quotient (because a Hilbert semistable point is also Chow semistable) and we prove in Section \ref{S:map-pstable} that $\Xi$ is an isomorphism if $d=\frac{7}{2}(2g-2)$ (see Proposition \ref{P:fib3.5})
while it is not an isomorphism if $d=4(2g-2)$ (see Proposition \ref{P:fib4}).

\subsection{Application: compactifications of the universal Jacobian}

As an application of Theorems \ref{T:MainThm1}, \ref{T:MainThm2} and \ref{T:MainThm3}, one gets three compactifications of the universal Jacobian stack $\cJ_{d,g}$, i.e. the moduli stack of pairs $(C,L)$ where $C$ is a smooth projective curve of genus $g$ and $L$ is a line bundle of degree $d$ on $C$, and of its coarse moduli space $J_{d,g}$.

To this aim, denote by $\JJst$ (resp. $\JJps$) the category fibered in groupoids over the category of $k$-schemes, whose fiber over a $k$-scheme $S$ is the groupoid of pairs $(f:\cX\to S, \cL)$ where $f:\cX\to S$ is a family of quasi-stable curves (resp. quasi-p-stable curves) of genus $g$ and $\cL$ is a line bundle on $\cX$ of relative degree $d$ over $S$
whose restriction to the geometric fibers of $f$ is properly balanced.
Moreover, denote by $\JJwp$ the category fibered in groupoids over the category of $k$-schemes, whose fiber over a $k$-scheme $S$ is the groupoid of pairs $(f:\cX\to S,\mathcal L)$ where $f$ is a family of quasi-wp-stable curves of genus $g$ and $\cL$ is a line bundle on $\cX$ of relative degree $d$ that is properly balanced on the geometric fibers of $f$ and such that the geometric fibers of $f$ do not contain tacnodes with a line nor special elliptic tails with respect to $\cL$.

In the following theorem, we summarize the properties of $\JJst$, $\JJwp$ and $\JJps$ that will be proved
in Section \ref{S:comp-Jac}.

\begin{theoremalpha}\label{T:New-Comp-Jac}
Let $g\geq 3$ and $d\in \Z$.
%Let $\JJst^{\star}$ be equal to either $\JJst$, $\JJwp$ or $\JJps$.
\begin{enumerate}
\item $\JJst$ (resp. $\JJwp$, $\JJps$) is a smooth,  irreducible and universally closed Artin stack of finite type over $k$ and dimension $4g-4$, containing $\cJ_{d,g}$ as a dense open substack.
%The category fibered in groupoids $\JJps$ representing families of quasi-p-stable curves endowed with a line bundle whose restriction to the geometric fibers is
%properly balanced is a smooth,  irreducible and universally closed Artin stack of dimension $4g-4$.

\item $\JJst$ (resp. $\JJwp$, $\JJps$) admits an adequate moduli space $\Jst$ (resp. $\Jwp$, resp. $\Jps$), which is a normal integral projective variety  of dimension $4g-3$ containing $J_{d,g}$ as a dense open subvariety.\\
Moreover, if ${\rm char}(k)=0$, then $\Jst$ (resp. $\Jwp$, resp. $\Jps$) has rational singularities, hence it is Cohen-Macauly.

\item  Denote by $\wt{H}_d$ the main component of the semistable locus of $\Hilb_d$, i.e. the open subset of $\Hilb_d$ consisting of all the points $[X\subset \P^{d-g}]$ that are semistable and such that $X$ is connected. Then it holds:
\begin{enumerate}[(i)]
\item $\JJst\cong [\wt{H}_d/\GL_{d-g+1}]$ \: and \: $\Jst\cong \wt{H}_d/\!\!/ \GL_{d-g+1}$ \:
      if  \: $d> 4(2g-2)$, \\
\item $\JJwp\cong [\wt{H}_d/\GL_{d-g+1}]$ \: and \: $\Jwp\cong \wt{H}_d/\!\!/ \GL_{d-g+1}$ \:
      if  \: $\frac{7}{2}(2g-2)<d\leq 4(2g-2)$, \\
\item $\JJps\cong [\wt{H}_d/\GL_{d-g+1}]$ \: and \: $\Jps\cong \wt{H}_d/\!\!/ \GL_{d-g+1}$ \:
      if  \: $2(2g-2)<d\leq \frac{7}{2}(2g-2)$.
\end{enumerate}

\item We have the following commutative diagrams
$$\xymatrix{
\JJst \ar[r] \ar[d]_{\Psi^{\rm s}} & \Jst \ar[d]^{\Phi^{\rm s}} && \JJwp \ar[r] \ar[d]_{\Psi^{\rm wp}} & \Jwp \ar[d]^{\Phi^{\rm wp}} && \JJps \ar[r] \ar[d]_{\Psi^{\rm ps}} & \Jps \ar[d]^{\Phi^{\rm ps}}\\
\MMg \ar[r] & \Mg && \MMgwp \ar[r] & \Mgp && \MMgp \ar[r] & \Mgp
}$$
where $\Psi^{\rm s}$ (resp. $\Psi^{\rm wp}$, $\Psi^{\rm ps}$) is universally closed and surjective and
$\Phi^{\rm s}$ (resp. $\Phi^{\rm wp}$, resp. $\Phi^{\rm ps}$) is projective and surjective. Moreover:
\begin{enumerate}[(i)]
\item The morphisms $\Phi^{\rm s}: \Jst\to \Mg$ and $\Phi^{\rm ps}: \Jps\to \Mgp$ have equidimensional fibers of dimension $g$; moreover, if ${\rm char}(k)=0$, $\Phi^{\rm s}$ and $\Phi^{\rm ps}$ are flat over the smooth locus of $\Mg$ and $\Mgp$, respectively.

\item The fiber of the morphism $\Phi^{\rm wp}:\Jwp\to \Mgp$ over a p-stable curve $X\in \Mgp$ has dimension equal to the sum of $g$ with the number of cusps of $X$.
\end{enumerate}

\item Let $\JJst^{\star}$ be equal to either $\JJst$ or $\JJwp$ or $\JJps$. Denote by $\JJst^{\star}\fatslash \Gm$
the rigidification of $\JJst^{\star}$ by $\Gm$ and by  $\wh{\Psi}^{\star}:\JJst^{\star} \to \MMg^{\star}$ the associated morphism, where $\MMg^{\star}$ is equal to either $\MMg$ or $\MMgwp$ or $\MMgp$.
Then the following conditions are equivalent:
\begin{enumerate}[(i)]
\item  $\gcd(d+1-g,2g-2)=1$;
%\item For any  $d'\equiv \pm d \mod 2g-2$ with $2(2g-2)<d'\leq \frac{7}{2}(2g-2)$, the GIT quotient $\wt{H}_{d'}/\PGL_{r+1}$ is geometric, i.e., there are no strictly semistable points;
\item The stack $\JJst^{\star}\fatslash \Gm$  is a DM-stack;
\item The stack $\JJst^{\star}\fatslash \Gm$  is proper;
\item The morphism $\wh{\Psi}^{\star}: \JJst^{\star}\fatslash \Gm\to \MMg^{\star}$  is representable.
\end{enumerate}

\item  If ${\rm char}(k)=0$, then it holds
%or ${\rm char}(k)=p>0$ is bigger than the order of the automorphism group of any p-stable curve of genus $g$, then
\begin{enumerate}[(i)]
\item $(\Phi^{\rm st})^{-1}(X)\cong \ov{\Jac_d}(X)/\Aut(X)$ \: for any $X\in \Mg$,
\item $(\Phi^{\rm ps})^{-1}(X)\cong \ov{\Jac_d}(X)/\Aut(X)$ \: for any $X\in \Mgp$,
\end{enumerate}
where $\ov{\Jac_d}(X)$ is the moduli space of of rank-$1$, torsion-free sheaves on $X$ of degree $d$ that are slope-semistable with respect to $\omega_X$ (and it is called the canonical compactified Jacobian of $X$ in degree $d$).

\end{enumerate}
\end{theoremalpha}

The stack (resp. variety) $\JJst$ (resp. $\Jst$) was introduced by Caporaso in \cite{Cap} and \cite{capneron} and is therefore called the \emph{Caporaso's compactified universal Jacobian stack} (resp. \emph{variety}).
The properties of $\JJst$ and $\Jst$ stated in the above theorem were indeed already known (also for $g=2$),
by the work of Caporaso \cite{Cap}, \cite{capneron} and the third author \cite{melo1}.

%The above Theorem \ref{T:New-Comp-Jac} follows from combining Theorems \ref{geomdesc}, \ref{stackprop} and \ref{T:prop-Pic-ps} and Corollary \ref{C:fiber-ps}.
%A few comments on the above theorem are in order. First of all, the hypothesis on the characteristic of the base field $k$ in part \eqref{T:New-Comp-Jac4}  is needed to guarantee that the automorphism groups of the geometric points of $\JJps$ be linearly reductive. For more details, we refer the reader  to the proof of Corollary \ref{C:fiber-ps} and the discussion following it.

In Section \S\ref{S:sheaves}, we provide also an alternative description of the stack $\JJst$ (resp. $\JJwp$, resp.
$\JJps$) via certain rank-$1$, torsion-free sheaves on stable (resp. wp-stable, resp. p-stable) that are semistable with respect to the canonical line bundle (see Theorem \ref{T:new-descr}).

\subsection{Open problems}

This work leaves unsolved some natural problems for further investigation, that we briefly discuss here.

As we observed above, Theorem \ref{T:MainThm3} does not hold for $d=2(2g-2)$.
The first problem is thus the following.

\begin{problemalpha}\label{ProA}
\noindent
\begin{enumerate}[(i)]
\item Describe the (semi-,poly-)stable points of $\Hilb_d$ and $\Chow_d$ in the case $d=2(2g-2)$.
\item Describe the (semi-,poly-)stable points of $\Hilb_d$ and $\Chow_d$ in the case $d=2(2g-2)-\epsilon$ (for small $\epsilon$).
\item What is the next critical value of $\frac{d}{2g-2}<2$ at which the GIT quotients change?
\end{enumerate}
\end{problemalpha}

As an output of the GIT analysis proposed in Problem \ref{ProA}, one expects to find new compactifications of the universal Jacobian over the Hassett-Hyeon \cite{HH2} moduli spaces $\ov{M}_g^h$ and $\ov{M}_g ^c$ of c-semistable and h-semistable curves, respectively.

\vspace{0.2cm}

In order to understand the relation between the three compactifications $\JJst$, $\JJwp$ and $\JJps$ of the universal
Jacobian stack $\cJ_{d,g}$, the following problem seems natural.

%By analogy with the contraction map $T:\Mgp\to \Mg$ constructed by Hassett-Hyeon in \cite{HH1} (see Remark \ref{R:contr-HH}), the following problem seems very natural.

\begin{problemalpha}\label{ProB}
Describe the birational maps fitting into the following commutative diagram
$$\xymatrix{
\JJst \ar@{-->}[r] \ar[d]_{\Psi^{\rm s}} & \JJwp \ar[d]^{\Psi^{\rm wp}}& \JJps \ar[d]^{\Psi^{\rm ps}}\ar@{-->}[l] \\
\MMg \ar@{^{(}->}[r]& \MMgwp & \MMgp \ar@{_{(}->}[l].
}$$
\end{problemalpha}
More generally, one would like to set up a Hassett-Keel program for
the Caporaso's compactified universal Jacobian stack $\JJst$ and give an interpretation of the alternative compactifications $\JJwp$ and $\JJps$ of $\cJ_{d,g}$ as the first two steps in this program. Moreover, it would be interesting to
study how the new settled Hassett-Keel program for $\JJst$ relates with the classical Hassett-Keel program for $\MMg$.

\subsection{Outline of the manuscript}
We now give a detailed outline of the manuscript.

In Section \ref{S:sing-curves}, we discuss the singular curves that will appear throughout the manuscript: stable, wp-stable and p-stable curves together with their associated stacks in \S\ref{S:wp-stable}; quasi-stable, quasi-wp-stable and quasi-p-stable curves in \S\ref{S:quasi-wp-stable}.
Moreover, we introduce two operations on families of curves: the p-stable reduction that contracts elliptic tails of wp-stable curves to cusps (see Proposition \ref{P:p-stab}) and the wp-stable reduction that contracts exceptional components of quasi-wp-stable curve to either nodes or cusps (see Proposition \ref{P:wp-stab}).

In Section \ref{S:combinato}, we first collect in \S\ref{S:baldeg} several combinatorial results on balanced multidegrees and on the degree class group of Gorenstein curves;
then, we introduce and study in \S\ref{S:stbal} stably and strictly balanced multidegrees on quasi-wp-stable curves.

In Section \ref{S:GIT}, we collect all the general results from GIT that we will need in this work.
In \S\ref{Sec:Hilb-Chow} we set up our GIT problem for $\Hilb_d$ and $\Chow_d$. In \S\ref{Sec:num-crit} we recall the Hilbert-Mumford numerical criterion for $m$-th Hilbert and Chow (semi)stability. Next,  we recall several classical results that will be used in our GIT analysis: basins of attraction (\S\ref{Sec:bas-attra}); flat limits and Gr\"obner basis (\S\ref{S:limGrob});
the parabolic subgroup associated to a one-parameter subgroup (\S\ref{sec:ParabGroup}). We end this section by recalling in \S\ref{Sec:pot-stab} two classical results due to Mumford and Gieseker: the Chow (or Hilbert) stability of smooth curves of genus $g$ embedded by line bundles of degree $d\geq 2g+1$; and the Potential stability Theorem giving necessary conditions for a point of $\Hilb_d$ or of $\Chow_d$ to be semistable, provided that $d>4(2g-2)$.

%Moreover, we recall some well known techniques in GIT (e.g. the Hilbert-Mumford criterion for GIT (semi)stability and the basin of attraction) as well as some classical results in GIT of curves (e.g. the potential stability theorem and the stability for smooth curves of high degree).

In Section \ref{S:pot-pseudo}, we prove the Potential pseudo-stability Theorem \ref{teo-pstab} which gives necessary conditions for a point of $\Hilb_d$ or of $\Chow_d$ to be semistable, provided that $d>2(2g-2)$.

In Section \ref{S:automo}, we compute the stabilizer subgroup of a point of $\Hilb_d$, under the assumption that $d>2(2g-2)$.

In Section \ref{S:extreme-inequ}, we investigate the isotrivial specializations that arise when one of the extremes of the basic inequalities is achieved.

In Section \ref{S:crit-tails}, we give a criterion for the (semi-, poly-)stability of a point of $\Hilb_d$ or $\Chow_d$ whose underlying curve has a tail.

In Section \ref{S:ell-tails}, we deal with the Hilbert or Chow semistability of curves having an elliptic tail (special or not) or having a tacnode with a line. We prove that special elliptic tails become Chow unstable for $d<4(2g-2)$
(see Theorem \ref{T:spec-ell}), ordinary elliptic tails become Chow unstable for $d<\frac{7}{2}(2g-2)$ (see Theorem \ref{T:ell-curves}), tacnodes with a line are Chow unstable for $d>\frac{7}{2}(2g-2)$ (see Theorem \ref{T:tac-line}).
Moreover, we examine the basins of attraction of the curves in Figure \ref{F:1crit} and \ref{F:2crit} (see Theorem \ref{T:basin-cusps} and \ref{T:basintacn}).

In Section \ref{S:stratifica}, we introduce a stratification of the Chow semistable locus by fixing the isomorphism class of a curve and the multidegree of the line bundle that embeds it.
We study the closure of the strata in \S\ref{SS:specia-str} and we prove a completeness
result for these strata in \S\ref{SS:complete}.

In Section \ref{S:semistab1} we characterize (semi, poly)-stable points in $\Hilb_d$ and $\Chow_d$
if either $4(2g-2)<d$ or $2(2g-2)<d\leq \frac{7}{2}(2g-2)$ and $g\geq 3$, thus proving Theorems
\ref{T:MainThm1}, \ref{T:2crit} and \ref{T:MainThm3}.

In Section \ref{S:stab-elltails}, we study the stability of elliptic tails in the range $\frac{7}{2}(2g-2)<d\leq 4(2g-2)$.

In Section \ref{S:semistab2}, we characterize (semi, poly)-stable points in $\Hilb_d$ and $\Chow_d$
in the range $\frac{7}{2}(2g-2)<d\leq 4(2g-2)$, thus proving Theorems
\ref{T:1crit} and \ref{T:MainThm2}.

In Section \ref{S:map-pstable}, we study the geometric properties of the Hilbert and Chow GIT quotients and of their
modular map towards the moduli space of p-stable curves.

In Section \ref{sec:extra}, we determine when the Hilbert or Chow semistable locus admits extra-components made  entirely of non-connected curves.

In Section \ref{S:comp-Jac}, we first recall in \S\ref{S:capcomp} the properties of the Caporaso's compactified universal Jacobian stack $\JJst$ over the moduli stack of stable curves and its moduli space $\Jst$.
Then, in \S\ref{S:2newcomp}, we define and study the two new compactifications $\JJwp$ and $\JJps$ of the universal Jacobian stack $\cJ_{d,g}$ over the moduli stack of wp-stable curves and p-stable curves, respectively. In \S\ref{S:2comp-var}, we prove that $\JJwp$ and $\JJps$
 admit projective moduli spaces $\Jwp$ and $\Jps$, respectively, and we study their properties.
Finally, in \S\ref{S:sheaves}, we provide an alternative description of the stack $\JJst$ (resp. $\JJwp$, resp.
$\JJps$) and of its moduli space via certain rank-$1$, torsion-free sheaves on stable (resp. wp-stable, resp. p-stable) curves that are semistable with respect to the canonical line bundle (see Theorem \ref{T:new-descr}).

The Appendix \ref{S:appendix} contains some positivity results for balanced line bundles on Gorenstein curves, which are used throughout the manuscript and that we find  interesting in their own.

\vspace{0.3cm}

Some of the results of this manuscript (more precisely, Theorems \ref{T:MainThm1} and \ref{T:MainThm3} and Theorem \ref{T:New-Comp-Jac} for $\JJps$) were originally obtained by the first, third and fourth author and then announced in
\cite{BMV}. However, the GIT analysis in the range $\frac{7}{2}(2g-2)\leq d\leq 4(2g-2)$ was left as
an open question (see \cite[Question A]{BMV}). The second author solved this open problem in his PhD thesis \cite{Fel}, by proving Theorems \ref{T:1crit}, \ref{T:MainThm2}, \ref{T:2crit} and Theorem \ref{T:New-Comp-Jac} for $\JJwp$ and then
became a coauthor of this work. Moreover, the presence of extra-components in the Hilbert or Chow semistable locus
made of non-connected curves was also left as an open question in loc. cit. (see \cite[Question C]{BMV}); this was also
solved by the second author and resulted in  Section \ref{sec:extra} of the present work.

\subsection*{Acknowledgements}
The last two authors would like to warmly thank Lucia Caporaso for the  many conversations on topics related to her PhD thesis \cite{Cap}, which were crucial in this work.
The second author would like to thank the referees of his PhD thesis \cite{Fel} for useful comments and suggestions.

%GIT of curves, compactified Jacobians and related topics.Having a deep understanding of her  PhD thesis \cite{Cap}  was crucial in this work.The reading of her PhD thesis \cite{Cap} was the starting point of this work.
%represented a high point in the mathematical educations of the  two mentioned authors.

We would like to thank Silvia Brannetti and Claudio Fontanari, who shared with us some of their ideas during the first phases of this work. We thank Marco Franciosi for discussions on the results of the Appendix.
Finally, we would like to thank Ian Morrison for his interest in this work.

G. Bini has been partially supported by ``FIRST'' Universit\`a di Milano, by MIUR--PRIN \textit{Variet\`a algebriche: geometria, aritmetica e strutture di Hodge} and by the MIUR--FIRB project  \textit{Spazi di moduli e applicazioni}. 
M. Melo has been partially  supported by CMUC (funded by the European Regional Development Fund through the program COMPETE and by FCT under the project PEst-C/MAT/UI0324/2013), and by the FCT-grants PTDC/MAT/111332/2009,  PTDC/MAT-GEO/0675/2012 and EXPL/MAT-GEO/1168/2013.
F. Viviani has been partially  supported by  the MIUR--FIRB project \textit{Spazi di moduli e applicazioni}, by CMUC (funded by the European Regional Development Fund through the program COMPETE and by FCT under the project PEst-C/MAT/UI0324/2013), and by the FCT-grants PTDC/MAT/111332/2009,  PTDC/MAT-GEO/0675/2012 and EXPL/MAT-GEO/1168/2013.

%M. Melo was supported by the FCT project \textit{Espa\c cos de Moduli em Geometria Alg\'ebrica} (PTDC/MAT/111332/2009), by the FCT project \textit{Geometria Alg\'ebrica em Portugal} (PTDC/MAT/099275/2008) and by the Funda\c c\~ao Calouste Gulbenkian program ``Est\'imulo \`a investiga\c c\~ao 2010''. F. Viviani is a member of the research center CMUC (University of Coimbra) and he was supported by  the FCT project \textit{Espa\c cos de Moduli em Geometria Alg\'ebrica} (PTDC/ MAT/111332/2009) and by MIUR--FIRB project \textit{Spazi di moduli e applicazioni}.

 \subsection*{Conventions}
\begin{convention}
 	$k$ will denote an algebraically closed field (of arbitrary characteristic). All \textbf{schemes} are $k$-schemes, and all morphisms are implicitly assumed to respect the $k$-structure.
\end{convention}

\begin{convention}
	A \textbf{curve} is a complete, reduced  and separated scheme (over $k$) of pure dimension $1$ (not necessarily connected). The \textbf{genus} $g(X)$ of a curve $X$ is $g(X) := h^1(X,
	\mathcal O_X)$. The set of \textbf{singular points} of a curve $X$ is denoted by $X_{\rm sing}$.
\end{convention}

\begin{convention}
	A \textbf{subcurve} $Z$ of a curve $X$ is a closed $k$-scheme $Z \subseteq X$ that is reduced  and of pure dimension $1$.  We say that a subcurve $Z\subseteq X$ is proper if
	$Z\neq \emptyset, X$.
	
	Given two subcurves $Z$ and $W$ of $X$ without common irreducible components, we denote by $Z\cap W$ the $0$-dimensional subscheme of $X$ that is obtained as the
	scheme-theoretic intersection of $Z$ and $W$ and we denote by $|Z\cap W|$ its length.
	
	Given a subcurve $Z\subseteq X$, we denote by $Z^c:=\ov{X\setminus Z}$ the \textbf{complementary subcurve} of $Z$ and we set $k_Z=k_{Z^c}:=|Z\cap Z^c|$.
 \end{convention}

\begin{convention}\label{Con:singu}
Let $X$ be a curve. A point $p$ of $X$ is said to be
\begin{itemize}
\item a \textbf{node}  if $\widehat{\OO_{X,p}}\cong k[[x,y]]/(y^2-x^2)$, where $\widehat{\OO_{X,p}}$ is the completion of the local ring $\OO_{X,p}$ of $X$ at $p$;
\item a \textbf{cusp} if $\widehat{\OO_{X,p}}\cong k[[x,y]]/(y^2-x^3)$;
\item a \textbf{tacnode} if $\widehat{\OO_{X,p}}\cong k[[x,y]]/(y^2-x^4)$.
\end{itemize}
A \textbf{tacnode with a line}  of a curve $X$ is a tacnode $p$ of $X$ at which  two irreducible components $D_1$ and $D_2$ of $X$
meet with a simple tangency so that $D_1\cong \P^1$ and $k_{D_1}=2$ (or equivalently $p$ is the set-theoretical intersection of $D_1$ and $D_1^c$).
\end{convention}

\begin{convention}\label{Con:tails}
An \textbf{elliptic tail} of a curve $X$ is a connected subcurve $F$ of genus 1 meeting the rest of the curve in one point; i.e. a connected subcurve $F\subseteq X$ such that
$g(F)=1$ and $k_F=|F\cap F^c|=1$. Moreover, we say that $F$ is
\begin{itemize}
	\item \textbf{nodal} if F is an irreducible rational curve with one node;
	\item \textbf{cuspidal} if F is an irreducible rational curve with one cusp;
	\item \textbf{reducible nodal} if F consists of two smooth rational subcurves meeting in two nodes;
	\item \textbf{tacnodal} if F consists of two smooth rational subcurves meeting in a tacnode.
\end{itemize}
Moreover we define the \textbf{elliptic locus}, which we denote by $X_{\rm{ell}}$, as the union of all the elliptic tails of $X$.
\end{convention}

\begin{convention}
A curve $X$ is called \textbf{Gorenstein} if its dualizing sheaf $\omega_X$ is a line bundle.
\end{convention}

\begin{convention}\label{N:lps}
A curve $X$ has \textbf{locally planar singularities at $p\in X$} if  the completion $\wh{\cO}_{X,p}$ of the local ring of $X$ at $p$ has embedded dimension two, or equivalently if it can be written as
$$\wh{\cO}_{X,p}=k[[x,y]]/(f),$$
for a reduced series $f=f(x,y)\in k[[x,y]]$. A curve $X$ has locally planar singularities if $X$ has locally planar singularities at every $p\in X$. Clearly, a curve with locally planar singularities is Gorenstein. A (reduced) curve has locally planar singularities if and only if it can be embedded in a smooth projective surface (see \cite{AK}).
\end{convention}

\begin{convention}
	A \textbf{family of curves} is a proper, flat morphism $X \to T$ whose geometric fibers are curves. Given a class $\mathcal C$ of curves, a family of curves of $\mathcal C$ is a family of curves
	$X\to T$ whose geometric fibers belong to the class $\mathcal C$. For example: if $\mathcal C$ is the class of nodal curves of genus $g$,
	then a family of nodal curves of genus $g$ is a family of curves whose geometric  fibers are nodal curves of genus $g$.
\end{convention}

\section{Singular curves}\label{S:sing-curves}

The aim of this section is to collect the definitions and basic properties of the curves that we will deal with throughout the manuscript.

\subsection{Stable, p-stable and wp-stable curves}\label{S:wp-stable}

We begin by recalling the definition of stable curves (\cite{DM}), pseudo-stable curves (\cite{Sch}) and weakly-pseudo-stable curves (\cite[Pag. 8]{HMo}) of genus $g\geq 2$.

\begin{defi}\label{D:stable}
A connected  curve $X$ of arithmetic genus $g\geq 2$ is
\begin{enumerate}[(i)]
\item \label{D:DM-stable} {\em stable} if
\begin{enumerate}
\item $X$ has only nodes as singularities;
\item the canonical sheaf $\omega_X$ is ample.
\end{enumerate}
\item \label{D:p-stable} {\em p-stable} (or pseudo-stable) if
\begin{enumerate}
\item $X$ has only nodes and cusps as singularities;
 \item $X$ does not have elliptic tails, i.e. $X_{\rm ell}=\emptyset$;
 \item the canonical sheaf $\omega_X$ is ample.
\end{enumerate}
\item \label{D:wp-stable}
{\em wp-stable} (or weakly-pseudo-stable) if
\begin{enumerate}
\item $X$ has only nodes and cusps as singularities;
\item the canonical sheaf $\omega_X$ is ample.
\end{enumerate}
\end{enumerate}
Note that, in each of the three cases, $\omega_X$ is ample if and only if each connected subcurve $Z$ of $X$ of genus zero is such that $k_Z=|Z\cap Z^c|\geq 3$.
\end{defi}

\begin{rmk}\label{R:st-p-wp}
Note that stable curves and p-stable curves are wp-stable. More precisely:
\begin{enumerate}[(i)]
\item stable curves are exactly those wp-stable curves without cusps.
\item p-stable curves are exactly those wp-stable curves without elliptic tails.
\end{enumerate}
\end{rmk}

We will work throughout the manuscript with the following stacks.

\begin{defi}\label{D:stack-curves}
Let $g\geq 2$. We denote by  $\MMg$ (resp. $\MMgp$, resp. $\MMgwp$)  the stack parametrizing families of stable (resp. p-stable, resp. wp-stable) curves of genus $g$.
\end{defi}

The properties of the above stacks can be summarized in the following

\begin{thm}\label{T:stacks-curves}
Let $g\geq 2$.
\begin{enumerate}[(i)]
\item \label{T:stacks-curves1} $\MMgwp$ is a smooth, irreducible algebraic stack of dimension $3g-3$,
containing $\MMg$ and $\MMgp$ as open  substacks.
\item \label{T:stacks-curves2} $\MMg$ is a proper stack;  $\MMgp$ is a proper stack if $g\geq 3$ and a weakly proper stack if $g=2$; $\MMgwp$ is a weakly proper stack.
 \item \label{T:stacks-curves3} $\MMg$ admits a coarse moduli space $\Mg$; $\MMgp$ admits a coarse moduli space $\Mgp$ for $g\geq 3$ and an adequate moduli
 space $\Mgp$ for $g=2$. $\Mgp$ is also an adequate moduli space for $\MMgwp$.

Moreover,  $\Mg$ and $\Mgp$ are irreducible projective varieties of dimension $3g-3$.
\end{enumerate}
\end{thm}
\begin{proof}
Part \eqref{T:stacks-curves1}: $\MMgwp$  is an algebraic stack since it is an open substack of the stack of all genus $g$ curves, which is well known to be algebraic (see Appendix B of \cite{Smy} by J. Hall, or also \cite{hall}).
By \cite[Prop. 2.4.8]{Ser}, an obstruction space for the deformation functor $\Def_X$ of a wp-stable curve $X$ is the vector space $\Ext^2(\Omega_X^1, \OO_X)$, which is zero according to \cite
[Lemma 1.3]{DM} since $X$ is a reduced curve with locally complete intersection singularities. This implies that $\Def_X$ is formally smooth, hence that $\MMgwp$ is smooth at $X$. Moreover,
from \cite[Thm. 2.4.1]{Ser} and \cite[Cor. 3.1.13]{Ser}, it follows that a reduced curve with locally complete intersection singularities can always be smoothened; therefore the open substack $
\mathcal{M}_g\subset
\MMgwp$ of smooth curves is dense. Since $\mathcal{M}_g$ is irreducible of dimension $3g-3$ (see \cite{DM}), we deduce that $\MMgwp$ is irreducible of dimension $3g-3$ as well. Clearly, $\MMg$ and $\MMgp$ are open substacks of $\MMgwp$ because the
condition of having no cusps or no elliptic tails is an open condition.

Part \eqref{T:stacks-curves2}:
for any $m\geq 2$, denote by $\Chow_{m,\can}^{ss}$ the locally closed sub-locus of the Chow scheme of $1$-cycles of degree
$m(2g-2)$ in $\P^N$ (where $N:=m(2g-2)-g$) consisting of curves which are embedded by the $m$-pluricanonical map and
semistable (see Section \ref{Sec:Hilb-Chow} for more details). It is known that: $\Chow_{m,\can}^{ss}$
consists of stable curves if $m\geq 5$ (see \cite{Mum}); $\Chow_{4,\can}^{ss}$ consists of wp-stable curves
(see \cite{HMo}); $\Chow_{3,\can}^{ss}$ consists of p-stable curves (see \cite{Sch} for $g\geq 3$ and \cite{HL} for
$g=2$). Now, a standard argument (see \cite[Thm. 3.2]{Edi} and \cite[Chap. XII, Thm. 5.6]{ACG}) yields the following isomorphisms of stacks:
\begin{equation}\label{E:stacks-quot}
\begin{aligned}
&  \MMg\cong [\Chow_{m,\can}^{ss}/\PGL_{N+1}] \text{ for any } m\geq 5, \\
&  \MMgwp\cong [\Chow_{4,\can}^{ss}/\PGL_{N+1}], \\
& \MMgp\cong [\Chow_{3,\can}^{ss}/\PGL_{N+1}].
 \end{aligned}
\end{equation}
In particular, it follows that all the above stacks are weakly proper (see \cite[Section 2]{ASvdW}). Moreover, it is well known that there are no
strictly semistable points in $\Chow_{m,\can}^{ss}$ for $m\geq 5$ (see \cite{Mum}) and in $\Chow_{3,\can}^{ss}$ for $g\geq 3$ (see \cite{Gie}).
This yields that $\MMg$ and $\MMgp$ for $g\geq 3$ are proper stacks (see \cite[Section 2]{ASvdW}).

Part \eqref{T:stacks-curves3}: define the GIT quotients
\begin{equation}\label{E:spaces-quot}
\begin{aligned}
&  \Mg:= \Chow_{5,\can}^{ss}/\!\!/ \PGL_{N+1}, \\
&  \Mgp:= \Chow_{3,\can}^{ss}/\!\!/\PGL_{N+1}. \\
 \end{aligned}
\end{equation}
By combining \eqref{E:stacks-quot}, \eqref{E:spaces-quot} and what said above on the strictly semistable points, it follows that $\Mg$ is a coarse moduli
for $\MMg$ and $\Mgp$ is a coarse (resp. adequate) moduli space of $\MMgp$ for $g\geq 3$ (resp. $g=2$), see \cite{alper2}. It was proved in \cite{HMo} that
$$\Mgp\cong \Chow_{4,\can}^{ss}/\!\!/ \PGL_{N+1},$$
which -- combined with \eqref{E:stacks-quot} --  implies that $\Mgp$ is an adequate moduli space for $\MMgwp$.

The fact that $\Mg$ and $\Mgp$ are irreducible projective varieties of dimension $3g-3$ is well-known (see \cite{DM} and \cite{HH1}).

\end{proof}

Note that our stacks $\MMg$, $\MMgp$ and $\MMgwp$ correspond to the stacks $\Mg(A_2^-)$, $\Mg(A_2^+)$ and $\Mg(A_2)$ in \cite{ASvdW}, respectively.

\begin{rmk}\label{R:3stacks}
\noindent
\begin{enumerate}[(i)]
\item \label{R:3stacks1} The stack $\MMgwp$ of wp-stable curves is not proper since in $\Chow_{4,\can}^{ss}$ there are strictly semistable points. Indeed, Hyeon-Morrison proved in \cite{HMo} that the unique orbit specializations occurring in  $\Chow_{4,\can}^{ss}$ (for $g\geq 3$) are the ones depicted in figure \ref{Spec-wpstable} below:

\vspace{0.3cm}

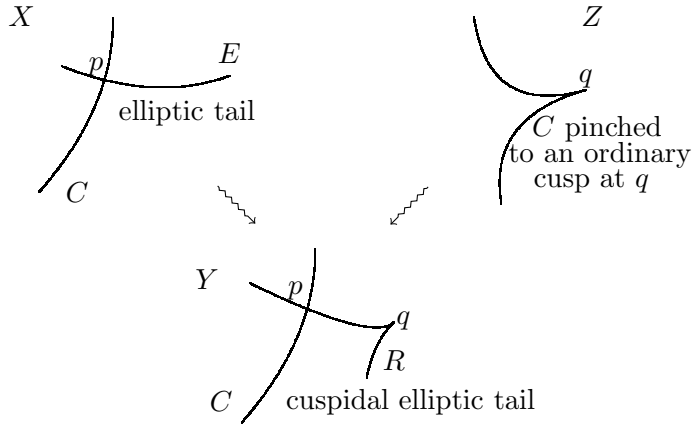
\begin{figure}[h]
\begin{center}
\unitlength .5mm % = 2.845pt
\linethickness{0.4pt}
\ifx\plotpoint\undefined\newsavebox{\plotpoint}\fi % GNUPLOT compatibility
\begin{picture}(172,100)(20,83)
\qbezier(81.75,119.5)(114.625,103.625)(120,109.25)
\qbezier(120,109.25)(114.625,104.125)(112.75,94.5)
\qbezier(79.5,82.5)(99.25,105)(99,128.5)
\qbezier(25.75,144)(45.5,166.5)(45.25,190)
\qbezier(31.75,177.25)(56.5,167.5)(76.25,174.75)
\qbezier(141.25,190.5)(144.875,163.75)(171,171)
\qbezier(171,171)(144.75,163.625)(148.5,140.75)
\put(21,190.75){\makebox(0,0)[cc]{$X$}}
\put(172.75,190.75){\makebox(0,0)[cc]{$Z$}}
\put(70.25,120.75){\makebox(0,0)[cc]{$Y$}}
\put(76.25,180.75){\makebox(0,0)[cc]{$E$}}
\put(65,165){\makebox(0,0)[cc]{elliptic tail}}
\put(40.75,177.25){\makebox(0,0)[cc]{$p$}}
\put(122.5,109.75){\makebox(0,0)[cc]{$q$}}
\put(171,174){\makebox(0,0)[cc]{$q$}}
\put(124.5,88){\makebox(0,0)[cc]{cuspidal elliptic tail}}
\put(35.75,143.75){\makebox(0,0)[cc]{$C$}}
\put(174,160.75){\makebox(0,0)[cc]{$C$ pinched}}
\put(176,153.75){\makebox(0,0)[cc]{to an ordinary }}
\put(172.5,146.75){\makebox(0,0)[cc]{cusp at $q$}}
\put(120,99.25){\makebox(0,0)[cc]{$R$}}
\put(93.75,117.75){\makebox(0,0)[cc]{$p$}}
\put(74,88.25){\makebox(0,0)[cc]{$C$}}
\put(75,141){\makebox(0,0)[cc]{$\xxrsquigarrow{\phantom{aaa}}$}}
\put(122.5,140.75){\makebox(0,0)[cc]{$\xxlsquigarrow{\phantom{a}}$}}
\end{picture}
\caption{Orbit specializations in $\Chow_{4,\can}^{ss}$, i.e. isotrivial specializations in $\MMgwp$.}
\label{Spec-wpstable}
\end{center}
\end{figure}

\noindent The above orbit specializations correspond to isotrivial specializations in the stack $\MMgwp$ (see \cite{ASvdW}). Therefore, the closed points of $\MMgwp$ are the wp-stable curves $X$ such that every elliptic tail of $X$ is cuspidal and every cusp of $X$ is contained in an elliptic tail.

%\item Note that our stacks $\MMg$, $\MMgp$ and $\MMgwp$ correspond to, respectively, the stacks $\Mg(A_2^-)$, $\Mg(A_2^+)$ and $\Mg(A_2)$ in \cite{ASvdW}.
\item \label{R:3stacks2} If ${\rm char}(k)=0$, then the adequate moduli spaces appearing in the above Theorem \ref{T:stacks-curves} are indeed good moduli spaces (see \cite[Prop. 5.1.4]{alper2}).
\end{enumerate}
\end{rmk}

% In this manuscript, however, we will often assume that $g\geq 3$ whenever we will deal with p-stable curves, in order to avoid these technical issues.

Given a wp-stable curve $Y$, it is possible to obtain a p-stable curve, called its p-stable reduction and denoted by $\ps(Y)$, by contracting the elliptic tails of $Y$ to cusps.
The p-stable reduction  works even for families.

\begin{prop}\label{P:p-stab}
Let  $v:\cY\to S$ be a family of wp-stable curves of genus $g\geq 2$.
There exists a commutative diagram
$$\xymatrix{
\cY \ar^{\psi}[rr] \ar_v[dr]& & \ps(\cY) \ar^{{\ps}(v)}[dl]\\
& S & \\
}$$
where ${\ps}(v):\ps(\cY)\to S$ is a family of p-stable curves of genus $g$, called
 the  \emph{p-stable reduction} of $v:\cY\to S$.
For every geometric point $s\in S$, the morphism $\psi_{s}: \cY_s\to \ps(\cY)_s$ contracts
the elliptic tails of $\cY_s$ to cusps of $\ps(\cY)_s$.
Moreover, the formation of the p-stable reduction commutes with base change.

This defines a morphism of stacks $\ps:\MMgwp\to \MMgp$.

\end{prop}
\begin{proof}
If $v:\cY\to S$ is a family of stable curves, the statement was proved by Hassett-Hyeon in \cite[Sec. 3]{HH1} under the assumption that $g\geq 3$ and then
extended to $g=2$ with a similar argument by Heyon-Lee in \cite[Sec. 4]{HL}.
In what follows, we will show how to adapt the argument of loc. cit. in order to work out in our case.

First of all, if $S=k$, then the statement follows from Proposition 3.1 in \cite{HH1}, which asserts that given a stable curve $C$, there is a replacement morphism $\xi_C:C\to \mathcal T(C)$, where
$\mathcal T(C)$ is a pseudo-stable curve of genus $g$, which is an isomorphism away from the loci of elliptic tails and that replaces elliptic tails with cusps. The argumentation is local on the
nodes connecting each genus-one subcurve meeting the rest of the curve in a single node. Since in a wp-stable curve all elliptic tails are connected to the rest of the curve via a single node,  the
same argumentation works also in our case with no further modifications.

Now, we have to prove the statement over an arbitrary base $S$.
%The whole question is now how to make it work over an arbitrary base $S$.
The approach of Hassett-Hyeon is to  consider a faithfully flat atlas $V\to \overline{\mathcal M}_g$
and define the p-stable reduction for the family of stable curves over $V$ induced by the above morphism.
The case of a family over an arbitrary base will follow by base-change from $V\to \overline{\mathcal M}_g$ to $S$.

In our case, we consider a faithful atlas $\rho_\pi:U\to \MMgwp$  of the stack $\MMgwp$ of wp-stable curves and we let $\pi:Z\to U$ be the associated (universal) family of wp-stable curves.
The idea is now to consider an invertible sheaf $L$ on $Z$, which will be a twisted version of the relative dualizing sheaf of $\pi$, such that $L$ is very ample away from the locus of elliptic tails,
and instead has relative degree $0$ over all elliptic tails. Then use $L$ to define an $S$-morphism from $Z$ to a family of p-stable curves which coincides with the previous one over all
geometric fibers of $\pi$.

To be precise, denote by
%$\delta_1\subset \mathcal M_g^{wp}$ the Cartier divisor of elliptic tails, by
$\delta_{1}\subset \overline{\mathcal M}_{g,1}^{\rm wp}$ the boundary divisor of elliptic tails on the universal
 stack $\overline{\mathcal M}_{g,1}^{\rm wp}$ over $\MMgwp$. An argument similar to the proof of Theorem \ref{T:stacks-curves}\eqref{T:stacks-curves1} shows that
 $\overline{\mathcal M}_{g,1}^{\rm wp}$ is smooth; hence $\delta_1$ is a Cartier divisor.
Let $\mu_\pi:Z\to\overline{\mathcal M}_{g,1}^{\rm wp}$ be the classifying morphism corresponding to the family
$\pi:Z\to U$ and  set $L:=\omega_\pi(\mu_\pi^*\delta_{1})$.
The whole point is now to prove that $\pi_*(L^n)$ is locally free and that $L^n$ is relatively globally generated for $n>0$ and that the associated morphism factors through
$$Z\stackrel{\xi_Z}{\to} \mathcal T(Z)\hookrightarrow \mathbb P(\pi_*L^n)$$
where $ \mathcal T(Z)$ is a family of p-stable curves and $\xi_Z$ coincides with the replacement morphism $\xi_C$ for all geometric fibers $C$ of $\pi$.
By browsing carefully through Hassett-Hyeon's argumentation, we easily conclude that everything holds also in our case.

\end{proof}

\begin{rmk}\label{R:contr-HH}
From the above Proposition, we get the existence of a morphism of stacks
\begin{equation}\label{E:mor-stack}
\ps:\MMgwp\to \MMgp,
\end{equation}
which, passing to the adequate moduli spaces, induces the morphism
$T:\Mg\to \Mgp$  studied by Hassett-Hyeon in \cite{HH1} for $g\geq 3$ and by Hyeon-Lee  in \cite{HL} for $g=2$. Indeed, it is proved in loc. cit.  that $T$ is the contraction
of the divisor $\Delta_1\subset \Mg$ of curves having an elliptic tail.
\end{rmk}

\subsection{Quasi-wp-stable curves and wp-stable reduction} \label{S:quasi-wp-stable}

The most general class of singular curves that we will meet throughout this work is the one given in the following:

\begin{defi}\label{D:pre-sing}
\noindent
\begin{enumerate}[(i)]
\item \label{pre-wp-stable} A connected curve $X$  is said to be \emph{pre-wp-stable}
if the only singularities of $X$ are nodes, cusps or tacnodes with a line.
\item \label{pre-p-stable} A connected curve $X$  is said to be \emph{pre-p-stable} if it is pre-wp-stable and it does not have
elliptic tails.
\item \label{pre-stable} A connected curve $X$  is said to be \emph{pre-stable}
if the only singularities of $X$ are nodes.
\end{enumerate}
\end{defi}

Note that wp-stable (resp. p-stable, resp. stable) curves are pre-wp-stable (resp. pre-p-stable, resp. pre-stable) curves.
Moreover, if $p\in X$ is a tacnode with a line lying in $D_1\cong \P^1$ and $D_2$ as in \ref{Con:singu}, then
$(\omega_X)_{|D_1}=\OO_{D_1}$, hence $\omega_X$ is not ample. From this, we get easily that

\begin{rmk}
$X$ is wp-stable (resp. p-stable, resp. stable) if and only if $X$ is pre-wp-stable (resp. pre-p-stable, resp. pre-stable) and $\omega_X$ is ample.
\end{rmk}

The pre-wp-stable curves that we will meet in this manuscript, even when non wp-stable, will satisfy
a very strong condition on connected subcurves where the restriction of the canonical line bundle is not ample, i.e.,
on connected subcurves of genus zero that meet the complementary subcurve in less than three points. This justifies the following

\begin{defi}\label{D:quasi-wp-stable}
A pre-wp-stable curve $X$ is said to be
\begin{enumerate}[(i)]
\item \label{D:quasi1} \emph{quasi-wp-stable} if  every connected subcurve $E\subset X$ such that $g_E=0$ and $k_E\leq 2$
satisfies $E\cong \P^1$ and $k_E=2$ (and therefore it meets the complementary subcurve $E^c$
either in two distinct nodal points of $X$ or in one tacnode of $X$).
\item \label{D:quasi2}\emph{quasi-p-stable} if it is quasi-wp-stable and pre-p-stable.
\item \label{D:quasi3}\emph{quasi-stable} if it is quasi-wp-stable and pre-stable.
\end{enumerate}
The irreducible components $E$ such $E\cong \P^1$ and $k_E=2$ are called
\emph{exceptional} and the subcurve of $X$ given by the union of the exceptional components is denoted by $\exc$. The complementary subcurve
$\exc^c=\ov{X\setminus \exc}$ is called the \emph{non-exceptional} subcurve and is denoted by $\w{X}$.
\end{defi}

Equivalently, a quasi-wp-stable curve is a pre-wp-stable $X$ such that $\omega_X$ is nef (i.e. it has non-negative degree on every subcurve of $X$)
and such that all the connected subcurves $E\subseteq X$ such that $\deg_E \omega_X=0$ (which are called exceptional subcurves) are irreducible.
Note that the term quasi-stable curve was introduced in  \cite[Sec. 3.3]{Cap}.

We summarize the different types of curves that we have introduced so far in Table \ref{Tab:singular}.
%to the following table.

\begin{table}[h]
\begin{center}
\begin{tabular}[5cm]{|p{6,5cm}|p{5cm}|c|}
\hline
SINGULARITIES & $\omega_X$ NEF + IRREDUCIBLE EXCEPTIONAL SUBCURVES & $\omega_X$ AMPLE \\
\hline
\hline
\textbf{pre-wp-stable} = nodes, cusps, tacnodes with a line & \hspace{1cm} \textbf{quasi-wp-stable} & \textbf{wp-stable} \\
\hline
\textbf{pre-p-stable} = pre-wp-stable without elliptic tails &  \hspace{1,1cm} \textbf{quasi-p-stable} & \textbf{p-stable} \\
\hline
\textbf{pre-stable} =  nodes & \hspace{1,2cm} \textbf{quasi-stable} & \textbf{stable} \\
\hline
\end{tabular}
\end{center}
\caption{Singular curves}\label{Tab:singular}
\end{table}

\hspace{0,3cm}

Given a quasi-wp-stable curve $Y$, it is possible to contract all the exceptional components in order to obtain a wp-stable curve, which is called the wp-stable reduction of $Y$ and is denoted by
$\wps(Y)$. This construction indeed works for families.

%To every family $\cX\to S$ of quasi-wp-stable, we will associate a family $\wps(\cX)\to S$ of wp-stable curves, called its weak-pseudo-stable (or simply wp-stable) reduction.

\begin{prop}\label{P:wp-stab}
Let $S$ be a scheme and $u:\cX\to S$ a family of quasi-wp-stable curves.
% i.e. a family of curves whose geometric fibers are quasi-wp-stable curves.
Then there exists a commutative diagram
$$\xymatrix{
\cX \ar^{\phi}[rr] \ar_u[dr]& & \wps(\cX) \ar^{\wps(u)}[dl]\\
& S & \\
}$$
where $\wps(u):\wps(\cX)\to S$ is a family of wp-stable curves, called  the \emph{wp-stable reduction} of $u$.

For every geometric point $s\in S$, the morphism $\phi_{s}: \cX_s\to \wps(\cX)_s$ contracts
the exceptional components $E$ of $\cX_s$ so that
\begin{enumerate}
%\item If $k_Z=1$ then $Z$ is contracted to a smooth point;
\item If  $E\cap E^c$ consists of two distinct nodal points of $X$, then $E$ is contracted to a node;
\item If $E\cap E^c$ consists of one tacnode of $X$, then $E$ is contracted to a cusp.
\end{enumerate}
The formation of the wp-stable reduction commutes with base change.
Furthermore, if $u$ is a family of quasi-p-stable (resp. quasi-stable) curves then $\wps(u)$ is a family of p-stable (resp. stable) curves.
\end{prop}
\begin{proof}

We will follow the same ideas as in the proof of \cite[Prop. 2.1]{Knu} and of \cite[Prop. 6.6]{melo}.
% to construct a wp-stable reduction for $u:\cX\to S$.
Consider the relative dualizing sheaf $\omega_u:=\omega_{\cX/S}$ of the family $u:\cX\to S$.
 It is a line bundle because
the geometric fibers of $u$ are Gorenstein curves by our assumption.  From Corollary \ref{pospropdualizing} in the Appendix we get that for all $i\geq 2$, the restriction of $\omega_u^i$ to a geometric fiber $\cX_s$ is non-special, globally generated and, if $i\geq 3$, normally generated.
Then, we can apply \cite[Cor. 1.5]{Knu} to get the following properties for $\omega_u$:
\begin{enumerate}[(a)]
\item $R^1u_*(\omega_{u}^i)=(0)$ for all $i\geq 2$;
\item $u_*(\omega_{u}^i)$ is $S$-flat for all $i\geq 2$;
\item for any morphism $T\to S$ there are canonical isomorphisms
$$u_*(\omega_{u}^i)\otimes_{\mathcal O_S}\mathcal O_T\to (u\times1)_*(\omega_{u}^i\otimes_{\mathcal O_S}\mathcal O_T)$$
for any $i\geq 2$;
\item the canonical map $u^*u_*(\omega_{u}^i)\to \omega_{u}^i$ is surjective for all $i\geq 3$;
\item the natural maps $(u_*{\omega_{u}^3})^i\otimes u_*\omega_{u}^3\to (u_*{\omega_{u}^3})^{i+1}$ are surjective for $i\geq 1$.
\end{enumerate}
Define now
$$\mathcal S_i:=u_*(\omega_u^i), \text{ for all } i\geq 0.$$
By (a) and (b) above, we know that $\mathcal S_i$ is locally free and flat on $S$ for $i\geq 2$.
Consider $$\mathbb P(\mathcal S_3):=\proj(\Sym(\mathcal S_3)) \to S.$$
Since, by (d) above, the natural map
 $$u^*u_*(\omega_u^3)\to \omega_u^3$$
is surjective, we get that there is a natural $S$-morphism
$$\xymatrix{
\cX \ar^{q}[rr] \ar_u[dr]& & \mathbb P(\mathcal S_3) \ar[dl]\\
& S & \\
}$$
Denote by $\mathcal Y$ the image of $\cX$ via $q$ and by $\phi$ the (surjective) $S$-morphism from $\cX$ to $\mathcal Y$.
%$q:X\to\mathbb P(\mathcal S^3)$.\proj
By (e) above, we get that $$\mathcal Y=\proj(\oplus_{i\geq 0}\mathcal S_i).$$ So, $\phi:\mathcal Y\to S$ is flat since the $\mathcal S_i$'s are flat for $i\geq 2$. To conclude that $\mathcal Y\to S$ is a family of wp-stable curves note that the restriction of $\omega_u^3$ to the geometric fibers of $u$ has positive degree in all irreducible components except the exceptional ones, where it has degree $0$.
Indeed, it is easy to see (see for example \cite[Rmk. 1.20]{Cat}) that, on each geometric fiber $\cX_s$, $\phi$ contracts an exceptional component $E\subseteq \cX_s$ to a node if $E$ meets the complementary curve in two distinct nodal points and
to a cusp if $E$ meets the complementary subcurve in one tacnode.
Moreover, $\Phi$ is an isomorphism outside the non exceptional locus.
We conclude that  $\mathcal Y\to S$ is a family of wp-stable curves, so we set $\wps(\cX):=\mathcal Y$ and $\wps(u):=\mathcal Y\to S$.

Property (c) above implies that forming $\wps$ is compatible with base-change.

The last assertion is clear from the above geometric description of the contraction $\phi_s:\cX_s\to \wps(\cX)_s$ on each geometric point of $u$.

%It is easy to check (PUT A PROOF OR A REFERENCE?)
%that the powers $\omega_{u}$ satisfies the following properties:
%$_{|\cX_s}:=\omega_{\cX_s}$ of every geometric fiber $\cX_s$ of $u$ satisfies:
%\begin{enumerate}
%\item $u_*(\omega_u^3)$ is locally free on $S$;
%\item The restriction of $\omega_u^3$ to a geometric fiber $\cX_s$ of $u$ is very ample if $\cX_s$ is wp-stable
%while it has non-positive degree on the rational subcurves $Z$ of $\cX_s$  such that $k_Z\leq 2$
%(and in particular on the lines that pass through the tacnodes of $\cX_s$).
%\end{enumerate}
%Consider the induced morphism
%$$\xymatrix{
%\cX \ar^{\gamma}[rr] \ar_u[dr]& & \P(u_*(\omega_u^3)) \ar[ld]\\
%& S & \\
%}$$
%From the above properties of $\omega_u^3$, it follows that $\wps(u):\wps(\cX):=\Im(\gamma)\to S$ satisfies the required
%properties and that the formation of $\wps(u)$ commutes with base change.

%{\bf What about flatness? Compare with \cite{Knu}. TO FINISH.}

\end{proof}

\begin{rmk}\label{R:wp-stab}
If $u:\cX\to S$ is a family of quasi-stable curves
%(or more generally semistable curves),
%flat and proper morphism whose geometric fibers are quasi-stable curves (or more generally (DM-)semistable curves)
then the wp-stable reduction $\wps(u):\wps(\cX)\to S$
coincides with the usual stable reduction $\s(u)$ of $u$ (see e.g. \cite{Knu}).
\end{rmk}

The wp-stable reduction allows us to give a more explicit description of the quasi-wp-stable curves.

\begin{coro}\label{C:quasi-vs-wp}
A curve $X$ is quasi-wp-stable (resp. quasi-p-stable, resp. quasi-stable) if and only if it can be obtained from a wp-stable (resp. p-stable, resp. stable) curve $Y$ via an iteration of the
following construction:
\begin{enumerate}[(i)]
\item \label{I:blow1} Normalize $Y$ at a node $p$ and insert a $\P^1$ meeting the rest of the curve in the two  branches of the node.
\item \label{I:blow2} Normalize $Y$ at a cusp and insert a $\P^1$ tangent to the rest of the curve at the branch of the cusp.
\end{enumerate}
In this case, $Y=\wps(X)$.
In particular, given a wp-stable (resp. p-stable, resp. stable) curve $Y$ there exists only a finite number of quasi-wp-stable (resp. quasi-p-stable, resp. quasi-stable) curves $X$
such that $\wps(X)=Y$, which we call  \emph{quasi-wp-stable} (resp. \emph{quasi-p-stable}, resp. \emph{quasi-stable}) \emph{models} of $Y$.
\end{coro}
Note that the above operation \eqref{I:blow2} cannot occur for quasi-stable curves. We call the above operation \eqref{I:blow1} (resp. \eqref{I:blow2}) the
 \textbf{bubbling} of a node (resp. of a cusp).

\begin{proof}
We will prove the Corollary only for quasi-wp-stable curves. The remaining cases are dealt with in the same way.

Let $X$ be a quasi-wp-stable curve and set $Y:=\wps(X)$. By Proposition \ref{P:wp-stab}, the wp-stablization $\phi:X\to Y=\wps(X)$ contracts
each exceptional component $E$ of $X$ to a node or a cusp according to whether $E\cap E^c$ consists of two distinct points or one point with multiplicity two.
Therefore, $X$ is obtained from $Y$ by a sequence of the two operations  \eqref{I:blow1} and \eqref{I:blow2}.

Conversely, if $X$ is obtained from a wp-stable curve $Y$ by a sequence of  operations  \eqref{I:blow1} and \eqref{I:blow2}, then clearly $X$ is quasi-wp-stable and $Y=\wps(X)$.

The last assertion is now clear.
\end{proof}

We end this section with an extension of the p-stable reduction of Proposition \ref{P:p-stab} to families of quasi-wp-stable curves.
%To every family $\cX\to S$ of pre-wp-stable, we will associate a family $\ps(\cX)\to S$ of p-stable curves, called its pseudo-stable (or simply p-stable) reduction.
%This is achieved in two steps: we first associate to a family $\cX\to S$ of quasi-wp-stable curves
%a family $\wps(\cX)\to S$ of wp-stable curves (see definition below) and then to
%$\wps(X)\to S$ a family $\ps(\wps(\cX)):=\ps(\cX)\to S$ of p-stable curves.

\begin{defi}\label{D:p-stab}
Let $S$ be a scheme and $u:\cX\to S$ be a family of quasi-wp-stable curves of genus $g\geq 3$.
 Then there exists a commutative digram
$$\xymatrix{
\phi: \cX\ar^{\phi}[r] \ar_u[dr]& \wps(\cX) \ar^<<<<{\psi}[r] \ar[d]^{{\wps}(u)}& \ps(\wps(\cX)):=\ps(\cX) \ar[dl]^{{\ps}(u):={\ps(\wps}(u))} \\
& S & & \\
}$$
where the family $\wps(u)$ is the wp-stable reduction of the family $u$ (see Proposition \ref{P:wp-stab}) and the family $\ps(\wps(u))$ is the p-stable reduction of the family
$\wps(u)$ (see Proposition \ref{P:p-stab}).

We set $\ps(u):=\ps(\wps(u))$ and we call it the  \emph{p-stable reduction}  of $u$.

\end{defi}

\section{Combinatorial results}\label{S:combinato}

The aim of this section is to collect all the combinatorial results that will be used in the sequel.

\subsection{Balanced multidegree and the degree class group}\label{S:baldeg}

In this subsection, we will study balanced multidegrees and their relationship with the degree class group for a curve $X$ with locally planar singularities (see \ref{N:lps}), generalizing the results of \cite[Sec. 4]{Cap} for nodal curves.

%Let us first recall some combinatorial definitions and results  from \cite[Sec. 4]{Cap}.
% which will be used in the sequel.
%Although the results in loc. cit. are stated for nodal curves, a close inspection of the proofs reveals that the same results are true -- more in general -- for Gorenstein curves.

Fix a connected curve $X$ with locally planar singularities of genus $g\geq 2$ and we denote by $C_1,\ldots, C_{\gamma}$ the irreducible components of  $X$.
A \emph{multidegree} on $X$ is an ordered $\gamma$-tuple of integers
$$\un d=(\un d_{C_1}, \ldots, \un d_{C_{\gamma}})\in \bZ^{\gamma}.$$
We denote by $|\un d|=\sum_{i=1}^{\gamma} \un d_{C_i}$ the total degree of $\un d$. Given a subcurve $Z\subseteq X$, we set $\un d_Z:=\sum_{C_i\subseteq Z} \un d_{C_i}$.
The term multidegree comes from the fact that every line bundle $L$ on $X$ has a multidegree $\un \deg L$ given by $\un \deg L:=(\deg_{C_1}L,\ldots, \deg_{C_{\gamma}}L)$
whose total degree $|\un \deg L|$ is the degree $\deg L$ of $L$.
%However, in this section we will work only with multidegrees  ignoring the fact that they come from line bundles on $X$.

Balanced multidegree are defined by mean of a numerical inequality as it follows.

%We now introduce an inequality condition on the multidegree of a line bundle which will play a key role in the sequel.

\begin{defi}\label{D:bala-multdeg}
Let $\un d$ be a multidegree of total degree $|\un d|=d$. We say that $\un d$ is \emph{balanced} if it satisfies the inequality (called \emph{basic inequality})
\begin{equation}\label{E:basineq-multideg}
\left|\un d_Z-\frac{d}{2g-2}\deg_Z \omega_X\right|\leq \frac{k_Z}{2},
\end{equation}
for every subcurve $Z\subseteq X$,
where $k_Z:=|Z\cap Z^c|$ denotes the length of the scheme theoretic intersection $Z\cap Z^c$ between $Z$ and the complementary subcurve $Z^c:=\ov{X\setminus Z}$.

We denote by $\B_X^d$ the set of all balanced multidegrees on $X$ of total degree $d$.
\end{defi}
For later use, we denote the two extremes of the basic inequality relative to $Z$ by
\begin{equation}\label{E:extr-bas-ineq}
\begin{sis}
& m_Z:=\frac{d}{2g-2}\deg_Z \omega_X-\frac{k_Z}{2},\\
& M_Z:=\frac{d}{2g-2}\deg_Z \omega_X+\frac{k_Z}{2}.\\
\end{sis}
\end{equation}
%Note that $m_Z=M_{Z^c}$ and $M_Z=m_{Z^c}$.

%The terminology ``basic inequality'' was first used in \cite[Sec. 3.1]{Cap}.

Following \cite[Sec. 4.1]{Cap}, we now define an equivalence relation on the set of multidegrees on $X$.
For every irreducible component $C_i$ of $X$, consider the multidegree $\un{C_i}=((\un{C_i})_1,\ldots, (\un{C_i})_{\gamma} )$ of total degree $0$ defined by
$$(\un{C_i})_j=\un{C_i}_{C_j}=\begin{sis}
&|C_i\cap C_j| &  \text{ if } i\neq j,\\
&-\sum_{k\neq i} |C_i\cap C_k| & \text{ if } i=j,
\end{sis}
$$
where, as usual, $|C_i\cap C_j|$ denotes the length of the scheme-theoretic intersection $C_i\cap C_j$ between $C_i$ and $C_j$.
More generally, for any subcurve $Z\subseteq X$, we set $\un Z:=\sum_{C_i\subseteq Z} \un{C_i}$.

\begin{rmk}\label{R:bilin}
From the hypothesis that $X$ has locally planar singularities, it follows that for any two subcurves
$Z,W \subseteq X$ with no common irreducible components we have that
\begin{equation}\label{E:bil-mult1}
(\un Z)_W=\sum_{C_i\subseteq Z}\sum_{C_j\subseteq W} |C_i\cap C_j|=|Z\cap W|.
\end{equation}
\begin{equation}\label{E:bil-mult2}
(\un Z)_Z=-\sum_{C_i\subseteq Z}\sum_{C_j\not\subseteq Z} |C_i\cap C_j|=-|Z\cap Z^c|=-k_Z.
\end{equation}
Indeed, the first equalities in \eqref{E:bil-mult1} and \eqref{E:bil-mult2} follow from the definition of $\un Z$; while the last equality in \eqref{E:bil-mult2} follows from the definition of $k_Z$.
For the second equalities in \eqref{E:bil-mult1} and \eqref{E:bil-mult2}, observe that if $X$ is embedded inside a smooth projective surface $S$ (which is possible by \ref{N:lps}), then $|Z\cap W|$ is equal to the intersection product of the two divisors $Z=\sum_{C_i\subseteq Z} C_i$ and $W=\sum_{C_j\subseteq W}C_j$ of $S$ (and similarly for $|Z\cap Z^c|$). Since the intersection product of divisors on $S$ is bilinear, the second equalities in \eqref{E:bil-mult1} and \eqref{E:bil-mult2} follow.
\end{rmk}

Denote by $\Lambda_X\subseteq \bZ^{\gamma}$ the subgroup of $\bZ^{\gamma}$ generated by the multidegrees $\un{C_i}$ for $i=1,\ldots,\gamma$.
It is easy to see that $\sum_i \un{C_i}=0$ and this is the only relation among the multidegrees $\un{C_i}$. Therefore,
$\Lambda_X$ is a free abelian group of rank $\gamma-1$.

\begin{defi}\label{D:equiv-mult}
Two multidegrees $\un d$ and $\un d'$ are said to be \emph{equivalent}, and we write $\un d\equiv \un d'$,  if $\un d-\un d'\in \Lambda_X$.
In particular, if $\un d\equiv \un d'$ then $|\un d|=|\un d'|$.

For every $d\in \bZ$, we denote by $\Delta_X^d$ the set of equivalence classes of multidegrees of total degree $d=|\un d|$.
Clearly, $\Delta_X^0$ is a finite group under component-wise addition of multidegrees (called the \emph{degree class group }  of $X$) and each
$\Delta_X^d$ is a torsor under $\Delta_X^0$.
\end{defi}

Every element in the degree class group $\Delta_X^d$ admits a (not necessarily unique) balanced representative.

%The connection between importance of balanced multidegrees stems from the following two results. The first result says that every element in $\Delta_X^d$ has a balanced representative.
%The second result investigates the relation between balanced multidegrees that have the same class in $\Delta_X^d$.

\begin{prop}\label{P:bala-rep}
For every multidegree $\un d$ on $X$ of total degree $d=|\un d|$, there exists $\un d'\in \B_X^d$ such that $\un d\equiv \un d'$.
\end{prop}
\begin{proof}
Our proof is a generalization of \cite[Prop. 4.1]{Cap}, where the result is proved for a nodal curve $X$ (see also \cite[Prop. 2.8]{MV}).
%See also \cite[Prop. 3.5]{MV} for a refinement of the above result.

%Let $\un d$ be a multidegree of total degree $d=|\un d|$.
We introduce two rational numbers measuring how far is the multidegree $\un{d}$ from being balanced. For any subcurve $W\subseteq Z$, set
\begin{equation}
\begin{sis}
& \epsilon(\un{d},W):=\un{d}_W-M_W,\\
& \eta(\un{d},W):=-\un{d}_W+m_W.
\end{sis}
\end{equation}
Using the fact that $\un d_W+\un d_{W^c}=d$ and $M_W=d-m_{W^c}$, we get that
\begin{equation}\label{sym-inv}
\epsilon(\un{d},W)=\eta(\un{d},W^c).
\end{equation}
We also set
\begin{equation}
\begin{sis}
& \epsilon(\un{d}):=\max_{W\subseteq X} \epsilon(\un{d},W),\\
& \eta(\un{d}):=\max_{W\subseteq X} \eta(\un{d},W).\\
\end{sis}
\end{equation}
Using \eqref{sym-inv}, we get the relation
\begin{equation}\label{equ-inv}
\epsilon(\un{d})=\eta(\un{d}).
\end{equation}
From (\ref{equ-inv}) and the fact that $\epsilon(\un{d},X)=\eta(\un{d},X)=\epsilon(\un{d},\emptyset)=\eta(\un{d},\emptyset)=0$,
we get  that $\epsilon(\un{d})=\eta(\un{d})\geq 0$. On the other hand, by Definition \ref{D:bala-multdeg}, the multidegree $\un d$ is balanced if and only if $ \epsilon(\un{d}, W),\eta(\un{d}, W)\leq 0$
for any subcurve $W\subseteq X$. Combining these two facts, we get that $\un d$ is balanced if and only if $\epsilon(\un{d})=\eta(\un{d})=0$.

The invariants $\epsilon$ and $\eta$ satisfy the following additive formula: for any two subcurves $W_1, W_2\subseteq X$ with common irreducible components, it holds that
\begin{equation}\label{add-for}
\begin{sis}
& \epsilon(\un{d},W_1\cup W_2)=\epsilon(\un{d},W_1)+\epsilon(\un{d},W_2)+|W_1\cap W_2|, \\
& \eta(\un{d},W_1\cup W_2)=\eta(\un{d},W_1)+\eta(\un{d},W_2)+|W_1\cap W_2|.
\end{sis}
\end{equation}
Let us prove the second additive formula; the proof of the first one is similar and left to the
reader. Using Remark \ref{R:bilin}, we compute:
$$\eta(\un{d},W_1\cup W_2)=-\un{d}_{W_1\cup W_2}+\frac{d}{2g-2}\deg_{W_1\cup W_2}(\omega_X)
-\frac{k_{W_1\cup W_2}}{2}=-\un{d}_{W_1}-\un{d}_{W_2}+$$
$$+\frac{d}{2g-2}(\deg_{W_1}(\omega_X)+\deg_{W_2}(\omega_X))
-\frac{k_{W_1}+k_{W_2}-2|W_1\cap W_2|}{2}=$$
$$= \eta(\un{d},W_1)+\eta(\un{d},W_2)+|W_1\cap W_2|.$$

Consider now the following collections of subcurves of $X$
\begin{equation*}
\begin{sis}
& S^+_{\un{d}}:=\{W\subseteq X\: :\: \epsilon(\un{d},W)=\epsilon(\un{d})\},\\
& S^-_{\un{d}}:=\{W\subseteq X\: :\: \eta(\un{d},W)=\eta(\un{d})\}.
\end{sis}
\end{equation*}
From formula (\ref{sym-inv}) and the equality $\epsilon(\un{d})=\eta(\un{d})$, it follows easily
that
\begin{equation}\label{comp-ext}
W\in S^+_{\un{d}} \Leftrightarrow W^c\in S^-_{\un{d}}.
\end{equation}
The sets $S^{\pm}_{\un{d}}$ are stable under intersection:
\begin{equation}\label{stab-inter}
W_1, W_2 \in S^{\pm}_{\un{d}} \Rightarrow W_1\cap W_2\in S^{\pm}_{\un{d}}.
\end{equation}
We will prove this for $S_{\un{d}}^+$; the proof for $S_{\un{d}}^-$ works
exactly in the same way. Let $\Pi_1:=\ov{W_1\setminus (W_1\cap W_2)}$. Using the additivity
formula (\ref{add-for}) for  the pair $(W_2, \Pi_1)$ and the fact that $W_2\in S^+_{\un{d}}$, we get that
\begin{equation*}
0=\epsilon(\un{d})-\epsilon(\un{d},W_2)\geq \epsilon(\un{d},\Pi_1\cup W_2)-\epsilon(\un{d},W_2)= \epsilon(\un{d}, \Pi_1)+|\Pi_1\cap W_2|.
\end{equation*}
Using this inequality, the additivity formula (\ref{add-for}) for the  pair $(W_1\cap W_2, \Pi_1)$
and the fact that $W_1\in S_{\un{d}}^+$, we get that
$$\epsilon(\un{d})=\epsilon(\un{d}, W_1)=\epsilon(\un{d}, (W_1\cap W_2)\cup \Pi_1)=\epsilon(\un{d},W_1\cap W_2)
+\epsilon(\un{d},\Pi_1)+|\Pi_1\cap(W_1\cap W_2)|$$
$$ \leq  \epsilon(\un{d},W_1\cap W_2)+\epsilon(\un{d},\Pi_1)+|\Pi_1\cap W_2|
\leq \epsilon(\un{d},W_1\cap W_2).$$
 By the maximality of $\epsilon(\un{d})$, we
conclude that $\epsilon(\un{d})=\epsilon(\un{d},W_1\cap W_2)$, i.e. that $W_1\cap W_2 \in S_{\un{d}}^+$.

Since the sets $S_{\un{d}}^{\pm}$ are stable under intersection, they admit minimum elements:
\begin{equation}
\Omega^{\pm}(\un{d}):=\bigcap_{W\in S_{\un{d}}^{\pm}} W\subseteq X.
\end{equation}
Note that (\ref{comp-ext}) implies that $\Omega^+(\un{d})^c\in S_{\un{d}}^-$.  Since $\Omega^-(\un{d})$ is the minimum element of $S_{\un{d}}^-$, we get that $\Omega^-(\un{d})\subseteq \Omega^+(\un{d})^c$,
or in other words that $\Omega^+(\un{d})$ and $\Omega^-(\un{d})$ do not have common irreducible components.  We set
\begin{equation*}
\Omega^0(\un{d}):= (\Omega^+(\un{d})\cup \Omega^-(\un{d}))^c\subseteq X,
~\end{equation*}
so that X is the disjoint union of $\Omega^+(\un{d})$, $\Omega^-(\un{d})$ and $\Omega^0(\un{d})$. Observe that
\begin{equation}\label{proper-Omega}
\un{d}\in \B_X^d  \Leftrightarrow \epsilon(\un{d}) \text{ or } \eta(\un{d})=0 \Leftrightarrow
\Omega^{+}(\un{d}) \text{ or } \Omega^-(\un{d})=\emptyset.
\end{equation}
Now, if $\un d$ is not balanced, then we consider the new multidegree
\begin{equation}\label{new-el}
\un{e}:=\un{d}+\un{\Omega^+(\un{d})}\equiv \un d.
\end{equation}

\un{Claim:} The multidegree $\un{e}$ satisfies one of the two following properties:
\begin{enumerate}[(i)]
\item $\epsilon(\un{e})<\epsilon(\un{d})$,
\item $\epsilon(\un{e})=\epsilon(\un{d})$ and
$\Omega^+(\un{e})\supsetneq \Omega^+(\un{d})$.
\end{enumerate}

Let us show first how, using the Claim, we can  conclude the proof of the Lemma. Indeed, if $\un{e}$ satisfies condition (ii),
we can iterate the substitution (\ref{new-el}) until we reach an element $\un{e'}$
which satisfies condition (i), i.e. $\epsilon(\un{e'})<\epsilon(\un{d})$,
and such that $\un{e'}\equiv \un{d}$.
Now observe that $\epsilon({\un f})\in \frac{\Z}{(2g-2)\Z}$ for any multidegree $\un f$,  because the denominators appearing in $M_W$ and $m_W$ are divisors of $2g-2$.
Therefore, by iterating the substitution (\ref{new-el}),
we will finally reach a multidegree  $\un{e''}$ such that $\epsilon(\un{e''})=0$, i.e.
$\un{e''}\in \B_X^d$, and such that $\un{e''}\equiv \un{d}$, q.e.d.

Let us now prove the Claim. Take any subcurve $W\subseteq X$ and decompose it as
a disjoint union
\begin{equation*}
W=W^+\coprod W^-\coprod W^0,
\end{equation*}
where $W^{\pm}=W\cap \Omega^{\pm}(\un{d})$ and $W^0=W\cap \Omega^0(\un{d})$.
Note that
\begin{equation}\label{ine1}
\epsilon(\un{d},W^+)\leq \epsilon(\un{d}),
\end{equation}
with equality if and only if $W^+=\Omega^+(\un{d})$ because of the minimality property
of $\Omega^+(\un{d})$. Applying (\ref{add-for}) to the pair $(\Omega^+(\un{d}), W^0)$,
we get
\begin{equation}\label{ine2}
\epsilon(\un{d}, W^0)=\epsilon(\un{d},W^0\cup \Omega^+(\un{d}))-\epsilon(\un{d}, \Omega^+(\un{d}))
-|W^0\cap \Omega^+(\un{d})| \leq -|W^0\cap \Omega^+(\un{d})|,
\end{equation}
where we used that $\epsilon(\un{d}, W^0\cup \Omega^+(\un{d}))\leq
\epsilon(\un{d})=\epsilon(\un{d}, \Omega^+(\un{d}))$. Applying once more formula
(\ref{add-for}) to the pair $(W^-, \Omega^+(\un{d})\cup\Omega^0(\un{d}))$, we get
$$\epsilon(\un{d}, W^-)=\epsilon(\un{d},W^-\cup \Omega^+(\un{d})\cup\Omega^0(\un{d}))-
\epsilon(\un{d}, \Omega^+(\un{d})\cup\Omega^0(\un{d}))-$$
\begin{equation}\label{ine3}
-|W^-\cap (\Omega^+(\un{d})\cup\Omega^0(\un{d}))|
\leq -|W^-\cap (\Omega^+(\un{d})\cup\Omega^0(\un{d}))|,
\end{equation}
where we used that (see (\ref{equ-inv}) and (\ref{sym-inv}))
$$\epsilon(\un{d},W^-\cup \Omega^+(\un{d})\cup \Omega^0(\un{d}))\leq \epsilon(\un{d})= \eta(\un{d})=
\eta(\un{d}, \Omega^-(\un{d}))= \epsilon(\un{d}, \Omega^-(\un{d})^c)=
\epsilon(\un{d}, \Omega^+(\un{d})\cup\Omega^0(\un{d})).$$
Moreover, if the equality holds in (\ref{ine3}), then by (\ref{sym-inv})
$$\eta(\un{d})=\epsilon(\un{d},W^-\cup \Omega^+(\un{d})\cup \Omega^0(\un{d}))=
\eta(\un{d}, \Omega^-(\un{d})\setminus W^-),$$
which implies that $\Omega^-(\un{d})\setminus W^-\in S^-_{\un{d}}$ and hence that $W^-=\emptyset$
because of the minimality property of $\Omega^-(\un{d})$. Using the formula
$$\epsilon(\un{e},W)=\epsilon(\un{d},W)+ \un{\Omega^+(\un{d})}_W$$
and Remark \ref{R:bilin}, the above inequalities (\ref{ine1}), (\ref{ine2}), (\ref{ine3})
give:
\begin{equation}\label{ine4}
\begin{sis}
& \epsilon(\un{e},W^+)=\epsilon(\un{d},W^+)-|W^+\cap \Omega^+(\un{d})^c|
\leq  \epsilon(\un{d})-|W^+\cap \Omega^+(\un{d})^c|,\\
& \epsilon(\un{e},W^0)=\epsilon(\un{d},W^0)+ |W^0\cap \Omega^+(\un{d})|\leq 0 ,\\
& \epsilon(\un{e},W^-)=\epsilon(\un{d},W^-)+ |W^-\cap \Omega^+(\un{d})|
\leq -|W^-\cap \Omega^0(\un{d})|.\\
\end{sis}
\end{equation}
Using twice the additive formula (\ref{add-for}) for the disjoint union
$W=W^+\coprod W^0\coprod W^-$ and the above inequalities (\ref{ine4}),  we compute
$$\epsilon(\un{e},W)=\epsilon(\un{e},W^+)+\epsilon(\un{e},W^0)+\epsilon(\un{e},W^-)
+|W^+\cap W^0| +|W^+\cap  W^-|+|W^0\cap  W^-|\leq$$
\begin{equation}\label{ine5}
\leq \epsilon(\un{d})-
|W^+\cap (\Omega^0(\un{d})\setminus W^0)| -
|W^+\cap  (\Omega^-(\un{d})\setminus W^-)|-|W^-\cap (\Omega^0(\un{d})\setminus W^0)|\leq
\epsilon(\un{d}).
\end{equation}
In particular, we have that $\epsilon(\un{e})\leq \epsilon(\un{d})$. If
the inequality in (\ref{ine5}) is attained for some subcurve $W\subseteq X$,
i.e. if $\epsilon(\un{e})=\epsilon(\un{d})$,
% and if $\Omega^+(\un e)\subseteq W$)
 then also the inequalities
in (\ref{ine1}) and (\ref{ine3}) are attained for $W$, and we observed before that this
implies that
\begin{equation}\label{sharp1}
\begin{sis}
& W^+=\Omega^+(\un{d}), \\
& W^-=\emptyset.
\end{sis}
\end{equation}
Moreover, all the inequalities in (\ref{ine5}) are attained for $W$ and, substituting
(\ref{sharp1}), this implies that
\begin{equation}\label{sharp2}
\begin{sis}
&|\Omega^+(\un{d})\cap (\Omega^0(\un{d})\setminus W^0)|=0, \\
&|\Omega^+(\un{d})\cap \Omega^-(\un{d})|=0.
\end{sis}
\end{equation}
Since $X$ is connected by hypothesis and $\Omega^+(\un{d})$ is a proper subcurve of $X$
 because we assumed $\un{d}\not \in \B_X^d $ (see \eqref{proper-Omega}), we deduce that  (using (\ref{sharp2})):
$$0<k_{\Omega^+(\un{d})}=|\Omega^+(\un{d})\cap (\Omega^-(\un{d})\cup \Omega^0(\un{d}))|=|(\Omega^+(\un{d})\cap  W^0|.$$
This gives that $W^0\neq \emptyset$, which implies that $W=W^+\cup W^0\supsetneq W^+=\Omega^+(\un{d})$
by (\ref{sharp1}). Since this holds for all subcurves $W\subseteq X$ such that $\epsilon(\un e,W)=\epsilon (\un d) (=\epsilon (\un e))$, it holds in particular for $\Omega^+(\un e)$. Therefore, we get that $\Omega^+(\un{e})\supsetneq \Omega^+(\un{d})$ and the claim is proved.

\end{proof}

The next result describes the relation between two balanced multidegrees that have the same class in the degree class group.

\begin{prop}\label{P:diff-bal}
Let $\un d, \un d'\in \B^d_X$. Then $\un d\equiv \un d'$ if and only if there exist subcurves $Z_1\subseteq\ldots\subseteq Z_m$ of $X$ such that
$$\begin{sis}
& \un d_{Z_k}=M_{Z_k} \text{ and } \: \: \un d'_{Z_k}=m_{Z_k}  \text{ for } 1\leq k\leq m,\\
& \un d'=\un d+\sum_{k=1}^m \un{Z_k}.\\
\end{sis}$$
%This implies that
%$$ \un d'_{Z_i}=\frac{d}{2g-2}\deg_{Z_i}\omega_X-\frac{k_{Z_i}}{2} \text{ for } 1\leq i\leq m.$$
Moreover, the subcurves $Z_k$ can be chosen so that $Z_k^c\cap Z_h=\emptyset$ for $k>h$.
\end{prop}
\begin{proof}
The proof is a generalization of \cite[Lemma 4.1 and p. 625]{Cap}, where the result is stated for DM-semistable curves.

The if implication is clear; let us prove the only if implication.
By the hypothesis $\un d\equiv \un d'$  together with Definition \ref{D:equiv-mult}, we can write
\begin{equation}\label{E:exp1}
\un d-\un d'=\sum_{i=1}^{\gamma} \alpha_i \un{C_i},
\end{equation}
for some $\alpha_i\in \Z$. Up to adding a suitable multiple of $\sum_{i=1}^{\gamma} \un{C_i}=0$ on the right hand side, we can normalize \eqref{E:exp1} in such a way that $\min_i \{\alpha_i\}=0$.
Set $m:=\max_i\{\alpha_i\}$ and consider the following subcurves of $X$
$$W_l=\bigcup_{\alpha_i=l} C_l\subseteq X \text{ for any } 0\leq l\leq q.$$
Note that $X=\bigcup_l W_l$ and that $W_l$ and $W_k$ do not have common irreducible components
if $k\neq l$.

We will prove that the subcurves $Z_k:=\bigcup_{0\leq l\leq k} W_l\subseteq X$ (for $1\leq k\leq m$) satisfy the desired properties. We also set $Z_0=W_0$ for convenience.
 %and $Z_k=Z_m=X$ for any $k\geq m+1$.
 Note that, by construction, we have that $Z_0\subseteq Z_1\subseteq Z_2\subseteq \ldots \subseteq Z_m=X$.
In terms of these subcurves, the expression \eqref{E:exp1} is equivalent to
\begin{equation}\label{E:exp2}
\un d-\un d'=\sum_{l=0}^m l\cdot \un{W_l}=\sum_{k=1}^m \un{Z_k}.
\end{equation}
We will now prove, by induction on $k$,  the following

\un{Claim:} For every $0\leq k\leq m$, we have that
\begin{equation}\label{E:indprop}
\begin{sis}
& \un d_{Z_k}=M_{Z_k}, \\
& \un d'_{Z_k}=m_{Z_k}, \\
& |W_k\cap W_l|=0 \text{ for any } k+2\leq l\leq m.
\end{sis}
\end{equation}
Let us first prove the base case $k=0$. Applying the basic inequality for $\un d$ and $\un d'$ relative to the subcurve $Z_0=W_0$ and using Remark \ref{R:bilin} and the expression \eqref{E:exp2}, we compute
\begin{equation*}
\sum_{l\geq 1} |W_0\cap W_l|=k_{W_0}\geq (\un d-\un d')_{W_0}=\sum_{l\geq 1} l\cdot \un{W_l}_{W_0}=\sum_{l\geq 1} l\cdot |W_0\cap W_l|.
\end{equation*}
Therefore, we must have that $|W_0\cap W_l|=0$ for $l\geq 2$ and the first inequality must be achieved, which happens if and only if $\un d_{Z_0}=M_{Z_0}$ and $\un d'_{Z_0}=m_{Z_0}$.

Assume now that the claim is proved for $0,\ldots, k-1$ and let us prove it for $k$.
Applying the basic inequality for $\un d$ and $\un d'$ relative to the subcurve $Z_{k}=Z_{k-1}\cup W_{k}$ and using Remark \ref{R:bilin} and the expression \eqref{E:exp2}, we get
\begin{equation*}
k_{Z_{k}}\geq (\un d-\un d')_{Z_{k-1}\cup W_{k}}=k_{Z_{k-1}}+ (k-1)|W_{k-1}\cap W_k|-k |W_k\cap W_k^c| +\sum_{h\geq k+1} h\cdot |W_k\cap W_h|=
\end{equation*}
$$=k_{Z_{k-1}}- |W_{k-1}\cap W_k| +\sum_{h\geq k+1} (h-k) \cdot |W_k\cap W_h|=k_{Z_k}+\sum_{h\geq k+1} (h-k-1) \cdot |W_k\cap W_h|.
$$
Therefore, we must have that $|W_k\cap W_l|=0$ for $k+2\leq l\leq m$ and the first inequality must be achieved, which happens if and only if $\un d_{Z_k}=M_{Z_k}$ and $\un d'_{Z_k}=m_{Z_k}$, q.e.d.

\vspace{0.2cm}

In order to conclude the proof, it remains to observe that the third condition in \eqref{E:indprop} is equivalent to the fact that $Z_k^c\cap Z_h=\emptyset$ for $k>h$.
\end{proof}

\subsection{Stably and strictly balanced multidegrees on quasi-wp-stable curves}\label{S:stbal}

We now specialize to the case where $X$ is a quasi-wp-stable curve of genus $g\geq 2$ (see Definition \ref{D:quasi-wp-stable})
\footnote{Actually, the reader can easily check that all the results of this  subsection are valid more in general if $X$ is a G-quasistable curve of genus $g\geq 2$ (in the sense of
Definition \ref{D:G-sing}) with locally planar singularities.}.

Given a balanced multidegree $\un d$ on $X$, the basic inequality \eqref{E:basineq-multideg} gives  that
$\un d_E=-1, 0, 1$ for every exceptional component $E\subset X$. The multidegrees such that $\un d_E=1$ on each exceptional component $E\subset X$ will play a special role in the sequel;
hence they deserve a special name.

\begin{defi}\label{D:prop-bal}
%Let $X$ be a quasi-wp-stable curve of genus $g\geq 2$.
We say that a multidegree $\un d$ on $X$ is \emph{properly balanced}
 if $\un d$ is balanced and $\un d_E=1$ for every exceptional component $E$ of $X$.

We denote by $B_X^d$ the set of all properly balanced multidegrees on $X$ of total degree $d$.
\end{defi}

%By Proposition \ref{nefample}\eqref{ample} of the Appendix, we have the following

%\begin{rmk}\label{R:propbal-ample}
%A balanced line bundle of degree $d>\frac{3}{2}(2g-2)$ on a quasi-wp-stable curve $X$ is properly balanced if and only if it is ample.
%Therefore, for $d>\frac{3}{2}(2g-2)$, the set $B_X^d$ is the set of all the multidegrees of ample balanced line bundles on $X$.
%\end{rmk}

%Among the properly balanced multidegrees on $X$, we will distinguish two subclasses

The aim of this subsection is to investigate the behavior of properly balanced multidegrees  on a quasi-wp-stable curve $X$, which
attain the equality in the basic inequality \eqref{E:basineq-multideg} relative to some subcurve $Z\subseteq X$. With this in mind, we introduce the following definitions.

\begin{defi} \label{D:balanced}
A properly balanced multidegree $\un d$ on $X$ is said to be
\begin{enumerate}[(i)]
\item \label{D:bala1}
%(or its multidegree)
\emph{strictly balanced} if any proper subcurve $Z\subset X$ such that $\un d_Z=M_Z$ satisfies  $Z\cap Z^c\subset \exc$.
%it holds $$\un d_Z=M_Z\Longrightarrow Z\cap Z^c\subset \exc.$$
%for each proper subcurve
%$Y$ of $X$ for which one of the two inequalities in (\eqref{E:basineq-multideg}) is not strict, then
%the intersection $Y\cap Y^c$ is contained $\exc$.
\item \label{D:bala2}
%(or its multidegree)
\emph{stably balanced} if any proper subcurve $Z\subset X$ such that $\un d_Z=M_Z$ satisfies $Z\subseteq \exc$.
%it is properly balanced and if for each proper subcurve $Y\subset X$ it holds $$\deg_Y L=M_Y\Longrightarrow Y\subset \exc.$$
%for which one of the two inequalities in (\eqref{E:basineq-multideg}) is not strict, then
%either $Y$ or $Y^c$ is entirely contained in $\exc$.
\end{enumerate}
\end{defi}

When $X$ is a quasi-stable curve, the above Definition \ref{D:balanced}\eqref{D:bala1}
coincides with the definition of {\rm extremal} in \cite[Sec. 5.2]{Cap}, while Definition
\ref{D:balanced}\eqref{D:bala2} coincides with the definition of G-stable in \cite[Sec. 6.2]{Cap}.
Here we adopt the terminology of \cite[Def. 2.3]{BFV}.

\begin{defi}\label{D:bal-lb}
We will say that a line bundle $L$ on $X$ is balanced if and only if its multidegree $\un{\deg} L$ is balanced, and similarly for properly balanced, strictly balanced, stably balanced.
\end{defi}

\begin{rmk}\label{R:bala}
In order to check that a multidegree $\un d$ on $X$ is balanced (resp. strictly balanced, resp. stably balanced), it is enough to check the conditions of Definitions \ref{D:bala-multdeg} and
\ref{D:balanced} only for the subcurves $Z\subset X$ such that $Z$ and $Z^c$ are \emph{connected}. This follows easily from the following facts. If $Z$ is a subcurve of $X$ and we denote by
$\{Z_1,\ldots, Z_c\}$ the connected components of $Z$, then the following hold:
\begin{enumerate}[(i)]
\item The upper (resp. lower) inequality in \eqref{E:basineq-multideg} is achieved for $Z$ if and only if the upper (resp. lower) inequality in \eqref{E:basineq-multideg} is achieved for every $Z_i$.
This follows from the (easily checked) additivity relations
$$\begin{sis}
& \deg_Z L=\sum_i \deg_{Z_i} L, \\
& \deg_Z \omega_X =\sum_i \deg_{Z_i} \omega_X, \\
& k_Z=\sum_i k_{Z_i}.
\end{sis}
$$
\item $Z\cap Z^c \subseteq \exc$ if and only if $Z_i\cap Z_i^c\subseteq \exc$ for every $i$. Similarly, $Z\subseteq \exc $ if and only if $Z_i\subseteq \exc$ for every $Z_i$.
\item If $Z^c$ is connected, then $\displaystyle Z_i^c=\cup_{j\neq i} Z_j \cup Z^c$ is connected for every $Z_i$.
\end{enumerate}

\end{rmk}

%There is an equivalence relation of the set of balanced line bundles on a quasi-stable curve $C$.

%\begin{defi}\label{equiv-rel}
%Given two balanced line bundles $L$ and $L'$ on a quasi-stable curve $C$, we say that $L$ and $L'$ are {\rm %equivalent}, and we write $(C,L)\equiv (C,L')$, if $L_{|\w{C}}\cong L'_{|\w{C}}$. The equivalence class of a pair %$(C,L)$ is denoted by $[(C,L)]$.
%\end{defi}

%Note that the above equivalence relation $\equiv$ clearly preserves the multidegree of the line bundles, hence it %preserves the condition of being strictly balanced or
%stably balanced.

The next result explains the relation between stably balanced and strictly balanced line bundles.

\begin{lemma}\label{compa-bal}
A multidegree $\un d$ on a quasi-wp-stable curve $X$ of genus $g\geq 2$
is stably balanced  if and only if $\un d$ is strictly balanced and $\w{X}=\ov{X\setminus \exc}$ is connected.
\end{lemma}
\begin{proof}
The proof is an easy adaptation of  \cite[Lemma 2.6]{BFV}  from  quasi-stable curves  to quasi-wp-stable curves. However, we include it here for the reader's convenience.

Assume first that $\un d$ is strictly balanced and that $\w{X}$ is connected. Let $Z$ be a proper subcurve of $X$ such that
$\un d_Z=M_Z$. Then $Z\cap Z^c\subset X_{\rm exc}$ because $\un d$ is
strictly balanced by hypothesis. Therefore the non-exceptional subcurve $\w{X}$ can be written as a disjoint union of the two subcurves
$Z\cap \w{X}$ and $Z^c\cap \w{X}$. Since $\w{X}$ is connected by hypothesis, we must have that either $Z\cap \w{X}=\emptyset$
or $Z^c\cap \w{X}=\emptyset$, which implies that either $Z\subseteq X_{\rm exc}$ or $Z^c\subseteq X_{\rm exc}$, respectively.
However, only the first case can occur because $\un d_Z=M_Z$ (in the second case, we would have $\un d_Z=m_Z$).
 This shows that $\un d$ is stably balanced.

Conversely, assume that $\un d$ is stably balanced. Clearly, this implies that $\un d$ is strictly balanced. Assume, by contradiction, that
$\w{X}$ is not connected. Then we can find two proper disjoint subcurves $D_1$ and  $D_2$ of $X$ that are not contained in
$X_{\rm exc}$ and such that  $E:=(D_1\cup D_2)^c$ is the union of $r\geq 1$ exceptional components of $X$.
It is easily checked that
\begin{equation*}
\left\{\begin{aligned}
 & \deg_{D_1\cup E}(\omega_X)=\deg_{D_1}(\omega_X), \\
 & k_{D_1\cup E}=k_{D_1}=r, \\
 & \un d_{D_1\cup E} =\un d_{D_1}  +r.
\end{aligned}\tag{*}
\right.
\end{equation*}
Applying the inequality \eqref{E:basineq-multideg}  to the subcurves $D_1$ and $D_1\cup E$ and using (*), we get
\begin{equation*}
%\left\{\begin{aligned}&
\frac{r}{2}=\frac{k_{D_1}}{2}\geq \un d_{D_1\cup E}-\frac{d}{2g-2}\deg_{D_1\cup E}(\omega_X)=r+\un d_{D_1}-\frac{d}{2g-2}\deg_{D_1}(\omega_X)\geq r-\frac{k_{D_1}}{2}=\frac{r}{2}.
\end{equation*}
Therefore, we have that $\un d_{D_1\cup E}=M_{D_1\cup E}$ and this contradicts the fact that
$\un d$ is strictly balanced, since $\emptyset \not\neq D_1\cup E\not\subseteq X_{\rm exc}$ by construction.

\end{proof}

The next result addresses the problem of whether all properly balanced line bundles of degree $d$ are stably balanced for every quasi-wp-stable curve of genus $g$.

\begin{lemma}\label{L:coprime}
Fix two integers $d$ and $g\geq 2$. The following conditions are equivalent:
\begin{enumerate}[(i)]
\item \label{L:coprime1} $\gcd(d+1-g,2g-2)=1$.
\item \label{L:coprime2} For every quasi-wp-stable (resp. quasi-stable, resp. quasi-p-stable) curve $X$ of genus $g$, every properly balanced line bundle on $X$ of degree $d$ is stably balanced.
\item \label{L:coprime3} For every quasi-wp-stable curve (resp. quasi-stable, resp. quasi-p-stable) $X$ of genus $g$, every properly balanced line bundle on $X$ of degree $d$ is strictly balanced.
\item \label{L:coprime4} For every quasi-wp-stable curve (resp. quasi-stable, resp. quasi-p-stable) $X$ of genus $g$, every strictly balanced line bundle on $X$ of degree $d$ is stably balanced.
\end{enumerate}
\end{lemma}
\begin{proof}
The proof is a generalization of \cite[Prop. 6.2, Lemma 6.3]{Cap}.

Let us first prove the implication \eqref{L:coprime1}$\Rightarrow$\eqref{L:coprime2}. First of all, notice that the numerical condition in \eqref{L:coprime1} is equivalent to the two numerical conditions:
\begin{equation*}
\gcd(d, g-1)=1, \tag{*}
\end{equation*}
\begin{equation*}
d\not\equiv g-1 \mod 2. \tag{**}
\end{equation*}
Let $X$ be a quasi-wp-stable curve of genus $g$ and let $L$ be any properly balanced line bundle on $X$ of degree $d$. Call $\un d$ the multidegree of $L$. In order to show that $L$ is stably balanced we have to show, using Remark \ref{R:bala}, that any connected proper subcurve $Z\subset X$ with connected complementary subcurve $Z^c$ and such that
\begin{equation}\label{E:eq1}
\un d_Z=\frac{d}{2g-2}\deg_Z \omega_X+\frac{k_Z}{2} \text{ or, equivalently, } (2g-2)\un d_Z= d\cdot \deg_Z(\omega_X)+(g-1)k_Z
\end{equation}
is an exceptional component of $X$.
Equation \eqref{E:eq1} implies, using our assumption (*), that $g-1$ divides $\deg_Z(\omega_X)$. Since $\omega_X$ is nef by assumption, we must have that $0\leq \deg_Z(\omega_X)\leq 2g-2$; hence, the only possibilities are
$\deg_Z(\omega_X)=0, g-1, 2g-2$.
\begin{itemize}
\item  If $\deg_Z(\omega_X)=0$ then $Z$ is an exceptional component of $X$
(see Definition \ref{D:quasi-wp-stable}) and we are done.
\item If $\deg_Z(\omega_X)=2g-2$ then $\deg_{Z^c}(\omega_X)=2g-2$ which implies that $Z^c$ is an exceptional component. This yields that $k_Z=2$ and $\un d_Z=d-\un d_{Z^c}=d-1$ (since $\un d$ is properly balanced), which contradicts \eqref{E:eq1}.
\item If $\deg_Z(\omega_X)=g-1$ then, dividing \eqref{E:eq1} by $g-1$ and taking congruence modulo $2$, we obtain that $d\equiv k_Z \mod 2$. On the other hand, using the formula $\deg_Z(\omega_X)=2g(Z)-2+k_Z$, we have that $g-1=\deg_Z(\omega_X)\equiv k_Z \mod 2$. By putting these two congruences together, we obtain that $d\equiv k_Z \equiv g-1 \mod 2$, which contradicts our assumption (**).
\end{itemize}

The implications \eqref{L:coprime2}$\Rightarrow$\eqref{L:coprime3} and \eqref{L:coprime2}$\Rightarrow$\eqref{L:coprime4} are clear.

Let us now prove the implication \eqref{L:coprime3}$\Rightarrow$\eqref{L:coprime1}. By contradiction, we will assume that the numerical condition $\gcd(2g-2,d+1-g)\neq 1$ is not satisfied  and we will construct a curve $X$ of genus $g$ which is both stable and p-stable (hence in particular wp-stable) together with a line bundle $L$ on $X$ of degree $d$ which is properly balanced but not strictly balanced. We will distinguish two (overlapping) cases, according to whether condition (*) or condition (**) is not satisfied.

\un{Case 1}: $d\equiv g-1\mod 2$.

Let $X$ be a curve made of two smooth irreducible components $Y_1$ and $Y_2$ of genera $g(Y_1)=g(Y_2)=0$ meeting in $k:=g+1\geq 3$ nodal points. Clearly, $X$ is a stable and p-stable curve of genus $g$, so that $\exc=\emptyset$.
 Let $L$ be a line bundle on $X$ of multidegree
$$(\deg_{Y_1}(L),\deg_{Y_2}(L))=\left(\frac{d+g+1}{2}, \frac{d-g-1}{2}\right).$$
It is easily checked that
$$\begin{aligned}
& M_{Y_1}=\frac{d}{2g-2}\deg_{Y_1}(\omega_X)+\frac{k_{Y_1}}{2}=\frac{d}{2g-2}(g-1)+\frac{g+1}{2}= \frac{d+g+1}{2}=\deg_{Y_1}(L),\\
& m_{Y_2}=\frac{d}{2g-2}\deg_{Y_2}(\omega_X)-\frac{k_{Y_2}}{2}=\frac{d}{2g-2}(g-1)-\frac{g+1}{2}=\frac{d-g-1}{2}=\deg_{Y_2}(L).\\
\end{aligned}$$
Therefore, $L$ is properly balanced but not strictly balanced.

\un{Case 2}: $\gcd(d, g-1)\neq 1$.

Let $X$ be a curve made of two smooth irreducible components $Y_1$ and $Y_2$ of genera, respectively, $g(Y_1)=0$ and $g(Y_2)=(g-1)-\frac{2g-2}{\gcd(d,g-1)}\geq 0$ meeting in $k:=\frac{2g-2}{\gcd(d,g-1)}+2\geq 3$ nodal points. Clearly, $X$ is a stable and p-stable curve of genus $g$, so that $\exc=\emptyset$.
 Let $L$ be a line bundle on $X$ of multidegree
$$(\deg_{Y_1}(L),\deg_{Y_2}(L))=\left(\frac{d+g+1}{\gcd(d,g-1)}+1,d-\frac{d+g+1}{\gcd(d,g-1)}-1\right).$$
It is easily checked that
$$M_{Y_1}=\frac{d}{2g-2}\deg_{Y_1}(\omega_X)+\frac{k_{Y_1}}{2}=\frac{d}{2g-2}\frac{2g-2}{\gcd(d,g-1)}+\frac{2g-2}{2\gcd(d,g-1)}+1=\deg_{Y_1}(L),$$
which also implies that $m_{Y_2}=\deg_{Y_2}(L)$.
Therefore, $L$ is properly balanced but not strictly balanced.

Finally, let us now prove the implication \eqref{L:coprime4}$\Rightarrow$\eqref{L:coprime1}. By contradiction, we will assume that the numerical condition $\gcd(2g-2,d+1-g)\neq 1$
 is not satisfied,  i.e. that either $d\equiv g-1\mod 2$ or $\gcd(d, g-1)\neq 1$,  and we will construct a curve $\wt X$ of genus $g$ which is both quasi-stable and quasi-p-stable (hence in particular quasi-wp-stable) together with a line bundle $\wt L$ on $\wt X$ of degree $d$ which is strictly balanced but not stably balanced.
Indeed, let  $\wt X$ be the curve obtained from the curve $X$ constructed above (in Case 1 and in Case 2) by bubbling all the nodes. In other words, $\wt X$ is made by two smooth irreducible components $Y_1$ and $Y_2$ of genera $g(Y_1)$ and $g(Y_2)$ (as specified above) joined by $k$ exceptional components $\{E_1, \ldots, E_k\}$. In particular, $\wt X$ is a quasi-stable and quasi-p-stable curve of genus $g$ with the property that
$\wt X_{\rm sing}\subset \wt X_{\rm exc}=\bigcup_{i=1}^k E_i$.
By the above computation, it follows easily that $m_{Y_1}, m_{Y_2}\in \Z$ and that $m_{Y_1}+m_{Y_2}+k=d$.  Let $\wt L$ be any line bundle on $X$ having degree $1$ on each exceptional component $E_j$ of $\wt{X}$ and such that
$\deg_{Y_i}(\wt L)=m_{Y_i}$ for $i=1,2$. Then, it follows easily that $\wt L$ is a properly balanced line bundle of degree $d$, which is moreover strictly balanced since  $\wt X_{\rm sing}\subset \wt X_{\rm exc}$.
However, $\wt L$ is not stably balanced since if we set $Z:=\bigcup_{i=1}^k E_i \cup Y_1$ then $\deg_{Z}(L)=m_{Y_1}+k=M_{Z}$ and $Z\not \subset \wt X_{\rm exc}$.

\end{proof}

The importance of strictly balanced multidegrees is that they are unique in their equivalence class in $\Delta_X^d$, at least among the properly balanced multidegrees.

\begin{lemma}\label{L:equiv-strict}
Let $\un d, \un d'\in B_X^{\un d}$ be two properly balanced multidegrees of total degree $d$ on a quasi-wp-stable curve $X$ of genus $g\geq 2$.
If $\un d \equiv \un d'$ and $\un d$ is strictly balanced, then $\un d=\un d'$.
\end{lemma}
\begin{proof}
According to Proposition \ref{P:diff-bal}, there exist subcurves $Z_1\subseteq\ldots\subseteq Z_m$ of $X$ such that
\begin{equation}\label{E:hypo1}
\un d'=\un d+\sum_{i=1}^m \un{Z_i},
\end{equation}
\begin{equation}\label{E:hypo2}
\un d_{Z_i}=M_{Z_i}=\frac{d}{2g-2}\deg_{Z_i}\omega_X+\frac{k_{Z_i}}{2} \text{ for } 1\leq i\leq m,
\end{equation}
\begin{equation}\label{E:hypo3}
Z_i^c\cap Z_j=\emptyset \text{ for } i>j.
\end{equation}
Assume, by contradiction,  that $\un d\not\equiv \un d'$; hence, using \eqref{E:hypo1}, we can assume that  $Z:=Z_1$ is a proper subcurve of $X$.
 From \eqref{E:hypo2} and the fact that $\un d$ is strictly balanced, we deduce that
$Z\cap Z^c\subset \exc$. Therefore, there exists an exceptional component $E\subseteq \exc$ such that one of the following four possibilities occurs:
$$\begin{aligned}
& \text{Case (I): } E\subseteq Z  \: \text{Êand }\:Ê|E\cap Z^c|=1,\\
& \text{Case (II): } E\subseteq Z \: \text{Êand }Ê\: |E\cap Z^c|=2,\\
& \text{Case (III): } E\subseteq Z^c \: \text{Êand }\:Ê|E\cap Z|=1,\\
& \text{Case (IV): } E\subseteq Z^c \: \text{Êand }\: Ê|E\cap Z|=2.\\
\end{aligned}
$$
Note that in Cases (II) or (IV), the intersection of $E$ with $Z$ or $Z^c$ consists either of two distinct points or of one point of multiplicity two.

\un{Claim}: Cases (III) and (IV) cannot occur.

By contradiction, assume first that case (III) occurs. Consider the subcurve $Z\cup E$ of $X$. We have clearly that
$$\begin{sis}
& \un d_{Z\cup E}=\un d_Z +1,\\
& \deg_{Z\cup E}\omega_X=\deg_Z\omega_X, \\
& k_{Z\cup E}=k_Z.
\end{sis}
$$
Therefore, using \eqref{E:hypo2}, we have that
$$\un d_{Z\cup E}=\un d_Z+1=\frac{d}{2g-2}\deg_{Z}\omega_X+\frac{k_{Z}}{2}+1=\frac{d}{2g-2}\deg_{Z\cup E}\omega_X+\frac{k_{Z\cup E}}{2}+1,$$
which contradicts the basic inequality \eqref{E:basineq-multideg} for $\un d$ with respect to the subcurve $Z\cup E\subseteq X$.

Assume now that case (IV) occurs. For the subcurve $Z\cup E\subseteq X$, we have that
$$\begin{sis}
& \un d_{Z\cup E}=\un d_Z +1,\\
& \deg_{Z\cup E}\omega_X=\deg_Z\omega_X, \\
& k_{Z\cup E}=k_Z-2.
\end{sis}
$$
Therefore, using \eqref{E:hypo2}, it follows that
$$\un d_{Z\cup E}=\un d_Z+1=\frac{d}{2g-2}\deg_{Z}\omega_X+\frac{k_{Z}}{2}+1=\frac{d}{2g-2}\deg_{Z\cup E}\omega_X+\frac{k_{Z\cup E}}{2}+2,$$
which contradicts the basic inequality \eqref{E:basineq-multideg} for $\un d$ with respect to the subcurve $Z\cup E\subseteq X$. The claim is now proved.

Therefore, only cases (I) or (II) can occur. Note that
\begin{equation}\label{E:equat1}
\un{Z}_{E}=-|E\cap Z^c|=\begin{cases}
-1 & \text{ if case (I) occurs,} \\
-2 & \text{ if case (II) occurs.}\\
\end{cases}
\end{equation}
Note also that, in any case, we must have that $E\subseteq Z=Z_1$. Using \eqref{E:hypo3}, we get that $E\cap Z_i^c=\emptyset$ for any $i>1$, which implies that
\begin{equation}\label{E:equat2}
\un{Z_i}_{E}=0 \: \text{   for any }   i>1.
\end{equation}
We now evaluate \eqref{E:hypo1} at the subcurve $E$: using that $\un d_E=1$ because $\un d$ is strictly balanced and equations \eqref{E:equat1} and \eqref{E:equat2}, we
conclude that
$$\un d'_E=\begin{cases}
0 & \text{ if case (I) occurs,}\\
-1 & \text{   if case (II) occurs.}
\end{cases}
$$
In both cases, this contradicts the assumption that $\un d'$ is properly balanced.
\end{proof}

We conclude this subsection with the following Lemma, which will be used several times in what follows.

\begin{lemma}\label{L:equiv-lb}
Let $X$, $Y$ and $Z$ be quasi-wp-stable curves of genus $g\geq 2$.
Let $\sigma:Z\to X$ and $\sigma':Z\to Y$ be two surjective maps given by contracting some of the exceptional components of $Z$.
Let $\un d$ (resp. $\un d'$) be a properly balanced multidegree on $X$ (resp. on $Y$). Denote by $\w{\un d}$ the pull-back of $\un d$ on $Z$ via $\sigma$, i.e., the multidegree on $Z$ given
on a subcurve $W\subseteq Z$ by
$$\w{\un d}_W=
\begin{cases}
\un d_{\sigma(W)} & \text{ if } \sigma(W) \text{ is a subcurve of }  X,\\
0 & \text{   if }   W \: \text{ is contracted by } \sigma \: \text{ to a point.}
\end{cases}$$
In a similar way, we define the pull-back $\w{\un d'}$ of $\un d'$ on $Z$ via $\sigma'$. The following is true:
\begin{enumerate}[(i)]
\item \label{L:equiv-lb1} $\w{\un d}$ and $\w{\un d'}$ are balanced multidegrees.
\item \label{L:equiv-lb2} If $\un d$ is strictly balanced and $\w{\un d}\equiv \w{\un d'}$ then there exists a map $\tau: X\to Y$ such that the following diagram commutes
$$
\xymatrix{
& Z \ar[ld]_{\sigma} \ar[rd]^{\sigma'}& \\
X \ar[rr]^{\tau} & & Y
}
$$
\end{enumerate}
\end{lemma}
\begin{proof}
Part \eqref{L:equiv-lb1}: let us prove that $\w{\un d}$ is balanced; the proof for $\w{\un d'}$ being analogous.
Consider a connected subcurve $W\subseteq Z$ and let us show that $\w{\un d}$ satisfies the basic inequality \eqref{E:basineq-multideg} with respect to the subcurve $W\subseteq Z$.
If $W$ is contracted by $\sigma$ to a point, then $W$ must be an exceptional component of $Z$. In this case, we have that $\w{\un d}_Z=0$, $k_{W}=2$ and $\deg_W(\omega_Z)=0$
so that \eqref{E:basineq-multideg} is satisfied. If $\sigma(W)$ is a subcurve of $X$, then $\w{\un d}_W=\un d_{\sigma(W)}$ and, since $\sigma$ contracts only exceptional components of $Z$, it is easy to see that $\deg_W(\omega_Z)=\deg_{\sigma(W)}(\omega_X)$   and
that $|W\cap W^c|=|\sigma(W)\cap \sigma(W)^c|$.
%as it is easily seen from the fact that .
 Therefore, in this case, the basic inequality
for $\w{\un d}$ with respect to $W$ follows from the basic inequality for $\un d$ with respect to $\sigma(W)$.

Part \eqref{L:equiv-lb2}:
start by noticing that if every exceptional component $E\subset Z$ which is contracted
by $\sigma$ is also contracted by $\sigma'$, then $\sigma'$ factors through $\sigma$, so the map $\tau$ exists. Let us now prove that in order for the map $\tau$ to exist, it is also necessary  that every exceptional component $E\subset Z$ which is contracted
by $\sigma$ is also contracted by $\sigma'$. By contradiction, assume that $\tau$ exists and that there exists an exceptional component $E\subset Z$ which is contracted by $\sigma$ but not by $\sigma'$.
Then we have that
\begin{equation}\label{E:twist0}
\begin{sis}
& \w{\un d}_E=0, \\
& \w{\un d'}_E=\un d'_{\sigma(E)}=1,
\end{sis}
\end{equation}
where in the last equation we have used that $\sigma(E')$ is an exceptional component of $Y$ and that $\un d'$ is properly balanced.

%This implies that the multidegrees $\un d'$ and $\w{\un d}$ on $Y$ are equivalent in the sense of Definition \ref{D:equiv-mult}.
Since $\w{\un d}$ is equivalent to $\w{\un d'}$ by assumption, Proposition \ref{P:diff-bal} implies that we can find subcurves $W_1\subseteq\ldots\subseteq W_m\subseteq Z$ such that
\begin{equation}\label{E:twist1}
\w{\un d}=\w{\un d'}+\sum_{i=1}^m \un{W_i},
\end{equation}
\begin{equation}\label{E:twist2}
\w{\un d'}_{W_i}=\frac{d}{2g-2}\deg_{W_i}\omega_Z+\frac{k_{W_i}}{2} \text{ for } 1\leq i\leq m,
\end{equation}
\begin{equation}\label{E:twist3}
W_i^c\cap W_j=\emptyset \text{ for } i>j.
\end{equation}
From \eqref{E:twist0} and \eqref{E:twist1}, we get that
\begin{equation}\label{E:twist4}
\sum_{i=1}^m \un{W_i}_E=-1.
\end{equation}
Denote by $C_1$ and $C_2$ the irreducible components of $Y$ that intersect $E$, with the convention
that $C_1=C_2$ if there is only one such irreducible component of $Y$ that meets $E$ in two distinct points or
in one point with multiplicity $2$.
It follows  from Remark \ref{R:bilin}  that
for any subcurve $W\subseteq Z$ with complementary subcurve $W^c$ we have that
$$\un W_E=
\begin{sis}
2 & \text{ if } E\subseteq W^c \text{ and }  C_1\cup C_2 \subseteq W,\\
1 & \text{ if } E\subseteq W^c \text{ and exactly one among }  C_1 \text{ and } C_2 \text{ is a subcurve of } W,\\
0 & \text{ if } E\cup C_1\cup C_2 \subseteq W^c \text{ or }  E\cup C_1\cup C_2 \subseteq W,\\
-1& \text{ if } E\subseteq W \text{ and exactly one among }  C_1 \text{ and } C_2 \text{ is a subcurve of } W, \\
-2& \text{ if } E\subseteq W \text{ and }  C_1\cup C_2 \subseteq W^c. \\
\end{sis}
$$
Using this formula, together with \eqref{E:twist4} and \eqref{E:twist3}, it is easy to see that $C_1$ must be different from $C_2$ and that, up to exchanging $C_1$ with $C_2$, there exists an integer $1\leq q\leq m$ such that
\begin{equation}\label{E:twist5}
\begin{sis}
& E\cup C_1\cup C_2 \subseteq  W_i^c & \text{ if } i<q,\\
& E\cup C_1\subset W_q \text{ and } C_2\subseteq W_q^c,  \\
& E\cup C_1\cup C_2 \subseteq  W_i & \text{ if } i>q.\\
\end{sis}
\end{equation}
Let us now compute $\w{\un d}_{W_q}$. From \eqref{E:twist3}, we get that
$$\un{W_i}_{W_q}=\begin{cases}
- k_{W_q} & \text{ if } i=q, \\
0 & \text{ if } i\neq q.
\end{cases}$$
Combining this with \eqref{E:twist1} and \eqref{E:twist2}, we get that
\begin{equation}\label{E:twist6}
\w{\un d}_{W_q}=\w{\un d'}_{W_q}-k_{W_q}=\frac{d}{2g-2}\deg_{W_q}\omega_Z-\frac{k_{W_q}}{2}.
\end{equation}
Consider now the subcurve $\sigma(W_q)$ of $X$. By \eqref{E:twist6}, we have that
$$\un d_{\sigma(W_q)}=\w{\un d}_{W_q}=\frac{d}{2g-2}\deg_{W_q}\omega_Z-\frac{k_{W_q}}{2}=\frac{d}{2g-2}\deg_{\sigma(W_q)}\omega_X-\frac{k_{\sigma(W_q)}}{2},$$
and by \eqref{E:twist5} we have that
$$\sigma(W_q)\cap \sigma(W_q)^c\not\subseteq X_{\rm exc}.$$
This contradicts the fact that $\un d$ is strictly balanced.

\end{proof}

\section{Preliminaries on GIT}\label{S:GIT}

\subsection{Hilbert and Chow schemes of curves}\label{Sec:Hilb-Chow}

Fix, throughout   this manuscript, two integers $d$ and $g\geq 2$ and write $d:=v(2g-2)=2v(g-1)$ for some (uniquely determined) rational number $v$. Set $r+1:=d-g+1=(2v-1)(g-1)$.

Let $\Hilb_{d,g}$ (or $\Hilb_d$ when $g$ is clear from the context) be the Hilbert scheme parametrizing subschemes of $\P^r=\P(V)$ having Hilbert polynomial
 $P(m):=md+1-g$, i.e., subschemes of $\P^r$ of dimension $1$, degree $d$ and arithmetic genus $g$.
An element $[X\subset \P^r]$ of $\Hilb_d$ is thus a $1$-dimensional scheme $X$ of arithmetic genus $g$ together with an embedding $X\stackrel{i}{\hookrightarrow} \P^r$ of degree $d$.  We let $\OO_X(1):=i^*\OO_{\P^r}(1)\in \Pic^d(X)$. The group $\GL(V)\cong \GL_{r+1}$ (hence its subgroup $\SL(V)\cong \SL_{r+1}$) acts on $\Hilb_d$ via its natural action on $\P^r=\P(V)$.
Given an element $[X\subset \P^r]\in \Hilb_d$, we will denote by $\Orb([X\subset \P^r])$ its \emph{orbit} with respect to the above action of $\GL(V)$ (or equivalently of $\SL(V)$).
%To easy the notation, we will sometimes write $X\in \Hilb_d$ to mean a point of $\Hilb_d$, i.e., we omit to indicate %the embedding  $X\stackrel{i}{\hookrightarrow} \P^r$.

It is well-known (see \cite[Lemma 2.1]{MS})  that for any  $m\geq M:=\binom{d}{2}+1-g$ and
any $[X\subset \P^r]\in \Hilb_d$ it holds that:
\begin{itemize}
\item $\OO_X(m)$ has no higher cohomology;
\item The natural map
$$
\Sym^m V^{\vee} \rightarrow \Gamma(\OO_X(m))=H^0(X, \OO_X(m))
$$
is surjective.
\end{itemize}
Under these hypotheses, the $m$-th Hilbert point of $[X\subset \P^r]\in \Hilb_d$
is defined to be
$$[X\subset \P^r]_m := \left[ \Sym^mV^{\vee} \surj \Gamma(\cO_X(m)) \right] \in
\Gr(P(m),\Sym^mV^{\vee}) \inj \P\left(\bigwedge^{P(m)} \Sym^m V^{\vee}\right),
$$
where $\Gr(P(m),\Sym^m V^{\vee})$ is the Grassmannian variety parametrizing $P(m)$-dimensional quotients of
$\Sym^m V^{\vee}$, which lies naturally in $\P\left(\bigwedge^{P(m)} \Sym^m V^{\vee}\right)$ via the Pl\"ucker embedding.

%Note that $X$ is determined by $[X]_m$ provided $X$ is
%cut out by forms of degree $m$.
\noindent For any $m\geq M$, we get a closed $\SL(V)$-equivariant embedding (see \cite[Lect. 15]{MumCAS}):
$$
\begin{array}{rcl}
j_m: \Hilb_d & \inj & \Gr(P(m),\Sym^m V^{\vee})  \inj \P(\bigwedge^{P(m)} \Sym^m V^{\vee}):=\P \\
 \ [X\subset \P^r] & \mapsto & [X\subset \P^r]_m.
\end{array}
$$
Therefore, for any $m\geq M$, we get an ample $\SL(V)$-linearized line bundle
$\Lambda_m:=j_m^*\OO_{\P}(1)$ and we denote by
$$\Hilb_d^{s,m}\subseteq \Hilb_d^{ss,m} \subseteq \Hilb_d$$
the locus of points that are stable and semistable   with respect to
%under the $\SL(V)$-action linearized by
$\Lambda_m$, respectively.
 If $[X\subset \P^r]\in \Hilb_d^{s,m}$ (resp. $[X\subset \P^r]\in \Hilb_d^{ss,m}$), we say that
$[X\subset \P^r]$ is {\em $m$-Hilbert stable} (resp. {\em semistable}).

The ample cone of $\Hilb_d$ admits a
finite decomposition into locally-closed cells, such that the
stable and the semistable locus are constant for linearizations taken from a given
cell \cite[Theorem 0.2.3(i)]{DolHu}.  In particular, $\Hilb_d^{s,m}$
and $\Hilb_d^{ss,m}$ are constant for $m\gg 0$. We set
$$\begin{sis}
& \Hilb_d^s:=\Hilb_d^{s,m} \text{ for } m\gg 0,\\
& \Hilb_d^{ss}:=\Hilb_d^{ss,m} \text{ for } m\gg 0.
\end{sis}$$
If $[X\subset \P^r]\in \Hilb_d^{s}$ (resp. $[X\subset \P^r]\in \Hilb_d^{ss}$, $[X\subset \P^r]\in \Hilb_d^{ss}\setminus \Hilb_d^{s}$), we say that  $[X\subset \P^r]$ is {\em  Hilbert stable} (resp. {\em  semistable}, {\em  strictly semistable}).
If $[X\subset \P^r]\in \Hilb_d^{ss}$ is such that the $\SL(V)$-orbit $\Orb([X\subset \P^r])$ of $[X\subset \P^r]$ is closed inside $\Hilb_d^{ss}$ then we say that $[X\subset \P^r]$ is
{\em Hilbert polystable}.

Let $\Chow_d \stackrel{j}{\inj}\P(\otimes^{2} \Sym^d V^{\vee}):=\P'$ the Chow scheme parametrizing $1$-cycles of
$\P^r$ of degree $d$ together with its natural $\SL(V)$-equivariant embedding $j$ into the projective space $\P(\otimes^2 \Sym^d V^{\vee})$ (see \cite[Lect. 16]{MumCAS}). Therefore, we have an ample $\SL(V)$-linearized line bundle $\Lambda:=j^*\OO_{\P'}(1)$
 and we denote by
$$\Chow_d^s\subseteq \Chow_d^{ss} \subseteq \Chow_d$$
the locus of points of $\Chow_d$ that are, respectively, stable and semistable with respect to $\Lambda$.
%under the $\SL(V)$-action.
%We call the points of these loci, respectively, {\em Chow stable and semistable}.

There is an $\SL(V)$-equivariant Hilbert-Chow morphism (see \cite[\S 5.4]{GIT}):
$$
\begin{array}{rcl}
\Ch: \Hilb_d & \ra & \Chow_d  \\
      \   [X\subset \P^r] & \mapsto & \Ch([X\subset \P^r]).
\end{array}
$$
We say that $[X\subset \P^r]\in \Hilb_d$ is {\em Chow stable} (resp. \emph{semistable}, {\em  strictly semistable}) if $\Ch([X\subset \P^r])\in \Chow_d^s$ (resp. $\Chow_d^{ss}$, $\Chow_d^{ss}\setminus\Chow_d^s$).
We say that $[X\subset \P^r]\in \Hilb_d$ is {\em Chow polystable} if $\Ch([X\subset \P^r])\in  \Chow_d^{ss}$ and its $\SL(V)$-orbit  is closed inside $\Chow_d^{ss}$.
Clearly, this is equivalent to asking that $[X\subset \P^r]\in \Ch^{-1}(\Chow_d^{ss})$ and that the $\SL(V)$-orbit $\Orb([X\subset \P^r])$ of $[X\subset \P^r]$ is closed inside
 $\Ch^{-1}(\Chow_d^{ss})$.

The relation between asymptotically Hilbert (semi)stability and Chow (semi)stability is given by the following (see \cite[Prop. 3.13]{HH2})
%By applying functoriality of stability \cite[Theorem 2.1]{Rei}, we get the following
\begin{fact}  \label{HilbtoChow}
There are inclusions
$$\Ch^{-1}(\Chow_d^s)\subseteq \Hilb_d^s\subseteq \Hilb_d^{ss}\subseteq \Ch^{-1}(\Chow_d^{ss}).$$
%Let $[X]\in \Hilb_d$.
%\begin{enumerate}
%\item If $\Ch(X)\in \Chow_d^s$ then $[X]\in \Hilb_d^s$;
%\item If $[X]\in \Hilb_d^{ss}$ then $\Ch(X)\in \Chow_d^{ss}$.
%\end{enumerate}
In particular, there is a natural morphism of GIT-quotients
$$\Hilb_d^{ss}/SL(V)\to \Ch^{-1}(\Chow_d^{ss})/SL(V).$$
\end{fact}
Note also that in general there is no obvious relation  between Hilbert and Chow polystability.

\subsection{Hilbert-Mumford numerical criterion for $m$-Hilbert and Chow (semi)stability} \label{Sec:num-crit}

Let us now recall the Hilbert-Mumford numerical criterion for the $m$-Hilbert (semi)stability and Chow
(semi)stability of a point $[X\subset \P^r]\in \Hilb_d$, following \cite[Sec. 0.B]{Gie} and \cite[Sec. 2]{Mum}
(see also \cite[Chap. 4.B]{HM}).
% (see also \cite[Sec. 1.4]{Cap}).
Although the criterion in its original form involves one-parameter subgroups (in short 1ps) of $\SL(V)$, it is technically convenient to
work with 1ps of $\GL(V)$ (see \cite[pp. 9-10]{Gie} for an explanation on how to pass from 1ps of $\SL(V)$ to 1ps of
$\GL(V)$, and conversely).

Let $\rho: \Gm \to \GL(V)$ be a 1ps  and let $x_0,\ldots, x_{r}$ be
coordinates of $V$ that diagonalize the action of $\rho$, so that for $i=0,\ldots, r$ we have
$$\rho(t)\cdot x_i = t^{w_i}x_i \:  \text{ with } w_i\in \Z.$$
 The total weight of $\rho$ is by definition
$$w(\rho):=\sum_{i=0}^r w_i.$$
Given a monomial $B=x_0^{\beta_0}\ldots x_r^{\beta_r}$, we define the weight of $B$ with respect to $\rho$ to be
$$w_{\rho}(B)=\sum_{i=0}^r \beta_i w_i.$$
%In order to state the numerical criterion for Chow (semi)stability, we need to introduce some further notation.
For any $m\geq M$ as in Section \ref{Sec:Hilb-Chow} and
any 1ps $\rho$ of $\GL(V)$, we introduce the following function
\begin{equation}\label{func-w}
W_{X,\rho}(m):=\min \left\{\sum_{i=1}^{P(m)} w_{\rho}(B_i)
%\:  : \: \{ M_1, \ldots, M_{P(m)}\}\in \Sym^mV^{\vee} \\ \text{ which restrict to a basis of } H^0(X, \OO_X(m))
\right\},
\end{equation}
where the minimum runs over all the collections of $P(m)$ monomials $\{B_1,\ldots,B_{P(m)} \} \in \Sym^mV^{\vee}$
which restrict to a basis of $H^0(X, \OO_X(m))$. It is easy to check that $W_{X,\rho}(m)$ coincide with the filtered Hilbert function of \cite[Def. 3.15]{HH2}. In the sequel, we will often write $W_{\rho}(m)$ instead of $W_{X,\rho}(m)$ when there is no danger of confusion.

The Hilbert-Mumford numerical criterion for $m$-th Hilbert (semi)stability translates into the following (see \cite[p. 10]{Gie} and also \cite[Prop. 4.23]{HM}).
%(see \cite[Sec. 0.B]{Gie} for a proof)

\begin{fact}[{\bf Numerical criterion for $m$-Hilbert (semi)stability}]\label{Hilb-crit}
Let $m\geq M$ as before. A point $[X\subset \P^r]\in \Hilb_d$ is $m$-Hilbert stable (resp. semistable) if and only if for every 1ps $\rho:\Gm\to \GL(V)$ of total weight $w(\rho)$
we have that
\begin{equation*}
\mu([X\subset \P^r]_m,\rho):=\frac{w(\rho)}{r+1}m P(m)-W_{X,\rho}(m) > 0
\end{equation*}
(resp. $\geq$).
\end{fact}

Indeed, the function $\mu([X\subset \P^r]_m,\rho)$ introduced above coincides with the Hilbert-Mumford index of $[X\subset \P^r]_m\in \P\left(\bigwedge^{P(m)} \Sym^m V^{\vee}\right)$
relative to the 1ps $\rho$ (see \cite[2.1]{GIT}).

The function $W_{X,\rho}(m)$ also allows one to state the numerical criterion for Chow (semi)stability.
According to \cite[Prop. 2.11]{Mum} (see also \cite[Prop. 3.16]{HH2}), the function $W_{X,\rho}(m)$ is an integer valued polynomial of degree $2$  for $m\gg 0$.  We define
$e_{X,\rho}$ (or $e_{\rho}$ when there is no danger of confusion) to be the normalized leading coefficient of $W_{X,\rho}(m)$, i.e.,
\begin{equation}
\left|W_{X,\rho}(m)-e_{X,\rho}\frac{m^2}{2}\right|< C m,
\end{equation}
for $m\gg 0$ and for some constant $C$.

The Hilbert-Mumford numerical criterion for Chow (semi)stability translates into the following (see \cite[Thm. 2.9]{Mum}).
%and also \cite[Cor. 3]{Mo})

\begin{fact}[{\bf Numerical criterion for Chow (semi)stability}]\label{Chow-crit}
A point $[X\subset \P^r]\in \Hilb_d$ is Chow stable (resp. semistable)
%belongs to $\Ch^{-1}(\Chow_d^s)$ (resp. to $\Ch^{-1}(\Chow_d^{ss})$)
if and only if for every 1ps $\rho:\Gm\to \GL(V)$ of total weight
$w(\rho)$ we have that
\begin{equation*}
e_{X,\rho} < 2 d \cdot \frac{w(\rho)}{r+1}
\end{equation*}
(resp. $\leq$).
\end{fact}

\begin{rmk}
Observe that $2 d \cdot \frac{w(\rho)}{r+1} $ is the normalized leading coefficient of the polynomial $\frac{w(\rho)}{r+1}m P(m)=
\frac{w(\rho)}{r+1}m (dm+1-g)$. Therefore, combining Fact \ref{Chow-crit} and Fact \ref{Hilb-crit} for $m\gg 0$, one gets a proof
of Fact \ref{HilbtoChow}.
\end{rmk}

The following definition is very natural.

\begin{defi}\label{D:stab-rho}
Let $[X\subset \P^r]\in \Hilb_d$ and let $\rho$ be a one-parameter subgroup of $\GL_{r+1}$. We say that
\begin{enumerate}[(i)]
\item $[X\subset \P^r]$ is {\em Hilbert semistable} (resp. {\em Chow semistable}) {\em with respect to $\rho$} if
%it is Hilbert semistable (resp. Chow semistable) with respect to the action of $\Gm$ induced by $\rho$, i.e., by Fact \ref{Hilb-crit} and Fact \ref{Chow-crit}, if
$$
W_{X,\rho}(m)\leq \frac{w(\rho)}{r+1}\,mP(m)\quad\text{for }m\gg 0\quad\bigg(\text{resp. }e_{X,\rho}\leq\frac{2d}{r+1}\,w(\rho)\bigg);
$$
Moreover, we say that $[X\subset \P^r]$ is {\em Hilbert strictly semistable} (resp. {\em Chow strictly semistable}) {\em with respect to $\rho$} if
$$
W_{X,\rho}(m)=\frac{w(\rho)}{r+1}\,mP(m)\quad\text{for }m\gg 0\quad\bigg(\text{resp. }e_{X,\rho}=\frac{2d}{r+1}\,w(\rho)\bigg);
$$
\item $[X\subset \P^r]$ is {\em Hilbert stable} (resp. {\em Chow stable}) {\em with respect to $\rho$} if
%it is Hilbert stable (resp. Chow stable) with respect to the action of $\Gm$ induced by $\rho$, i.e., by Fact \ref{Hilb-crit} and Fact \ref{Chow-crit}, if
$$
W_{X,\rho}(m)<\frac{w(\rho)}{r+1}\,mP(m)\quad\text{for }m\gg 0\quad\bigg(\text{resp. }e_{X,\rho}<\frac{2d}{r+1}\,w(\rho)\bigg);
$$
\item $[X\subset \P^r]$ is {\em Hilbert polystable} (resp. {\em Chow polystable}) {\em with respect to $\rho$} if one of the following conditions is satisfied:
\begin{enumerate}
	\item $[X\subset \P^r]$ is Hilbert stable (resp. Chow stable) with respect to $\rho$;
	\item $[X\subset \P^r]$ is Hilbert strictly semistable (resp. Chow strictly semistable) with respect to $\rho$ and
$$
\lim_{t\ra 0}\rho(t)[X\subset\P^r]\in \Orb([X\subset\P^r]).
$$
\end{enumerate}
\end{enumerate}
\end{defi}

\begin{rmk}\label{R:poly-1ps}
Let $[X\subset \P^r]\in \Hilb_d$ and $\rho$ be a one-parameter subgroup of $\GL_{r+1}$. Applying Definition \ref{D:stab-rho}, Fact \ref{Hilb-crit} and Fact \ref{Chow-crit}, we have that $[X\subset \P^r]$ is Hilbert semistable (resp. polystable, stable) if and only if $[X\subset \P^r]$ is Hilbert semistable (resp. polystable, stable) with respect to any one-parameter subgroup of $\GL_{r+1}$. The same holds for the Chow semistability (resp. polystability, stability).
\end{rmk}
%We conclude this subsection by mentioning the following result of Hassett-Hyeon (\cite[Prop. 3.17]{HH2}) which allows, under
%suitable hypothesis, to compute the Hilbert-Mumford index $\mu([X]_m,\rho)$ for any $m$ in terms of the Hilbert-Mumford index at low values of $m$.

%\begin{fact}[Hassett-Hyeon]\label{F:comp-poly-easy}
%Let $X\in \Hilb_d$ and assume that $X$ is connected of pure dimension one, $X\subset \P^r$ is linearly normal (i.e., embedded by the complete linear system $\OO_X(1)$) and $\OO_X(1)$ is
%non-special (i.e., $H^1(X,\OO_X(1))=0$). Then the Hilbert-Mumford index $\mu([X]_m, \rho)$ is defined for any $m\geq 2$ and it satisfies the formula
%$$\mu([X]_m,\rho)= (m-1)\left[(3-m)\mu([X]_2,\rho)+\left(\frac{m}{2}-1\right)\mu([X]_3,\rho) \right].$$
%Equivalently, the function $W_{X,\rho}(m)$ is a polynomial for $m\geq 2$ and it satisfies the formula
%$$W_{X,\rho}(m)=m^2\left[-W_{X,\rho}(2)+\frac{W_{X,\rho}(3)}{2}+\frac{w(\rho)}{2} \right]+$$
%$$+m\left[4W_{X,\rho}(2)-\frac{3W_{X,\rho}(3)}{2}-\frac{5w(\rho)}{2} \right]+ \left[-3W_{X,\rho}(2)+W_{X,\rho}(3)+3w(\rho) \right].$$
%In particular, the normalized leading coefficient $e_{X,\rho}$ of $W_{X, \rho}(m)$ is equal to
%$$e_{X, \rho}=\frac{w(\rho)}{r+1}2d+2\mu([X]_2,\rho)-\mu([X]_3,\rho)= w(\rho)-2W_{X,\rho}(2)+W_{X,\rho}(3).$$
%\end{fact}

Let $[X\subset \P^r]\in \Hilb_d$. If $C$ is a subscheme of $X$ of arithmetic genus $g_C$, we can consider the new point $[C\subset \P^r]\in \Hilb_{\deg\OO_{C}(1),g_C}$ and also $W_{C,\rho}(m)$ and $e_{C,\rho}$ with respect to a one-parameter subgroup $\rho:\Gm\ra \GL_{r+1}$. The next result says that we can estimate or compute $e_{X,\rho}$ in terms of the weights of the subschemes of $X$.

\begin{prop}\label{prop:sumsub}
Let $[X\subset \P^r]\in \Hilb_d$ and let $\rho$ be a one-parameter subgroup of $\GL_{r+1}$.
\begin{enumerate}[(i)]
	\item\label{prop:sumsub1} If $Y$ is a subscheme of $X$ and the weights of $\rho$ are non-negative, then $W_{X,\rho}(m)\geq W_{Y,\rho}(m)$ (in particular $e_{X,\rho}\geq e_{Y,\rho}$).
	\item\label{prop:sumsub2} If $X$ is reduced (possibly non connected), has pure dimension 1 and $\{X_i\}_{i=1,\ldots,n}$ is a collection of subcurves of $X$ such that
$$
(X_i)^c=\bigcup_{k\neq i} X_{k}
$$
for each $i=1,\ldots,n$, then
$$
e_{X,\rho}=\sum_{i=1}^n e_{X_i,\rho}.
$$
\end{enumerate}
\end{prop}
\begin{proof}
Let us prove \eqref{prop:sumsub1}. Denote by $P_X$ and $P_Y$ the Hilbert polynomials  of $X$ and $Y$, respectively, and consider a monomials basis $\{B_1,\ldots,B_{P_X(m)}\}$ of $H^0(X,\OO_X(m))$ such that
$$
W_{X,\rho}(m)=\sum_{i=1}^{P_X(m)}w_{\rho}(B_i).
$$
Since the restriction map $H^0(X,\OO_X(m))\lra H^0(Y,\OO_Y(m))$ is onto for $m\gg 0$, up to reordering the monomials, we can assume that $\{B_1,\ldots,B_{P_Y(m)}\}$ is a monomial basis of $H^0(Y,\OO_Y(m))$. Hence
$$
W_{Y,\rho}(m)\leq \sum_{i=1}^{P_Y(m)}w_{\rho}(B_i)\leq \sum_{i=1}^{P_X(m)}w_{\rho}(B_i)=W_{X,\rho}(m)
$$
and \eqref{prop:sumsub1} is proved.

Now we will prove \eqref{prop:sumsub2}. We can assume that $n=2$. Let $x_1,\ldots,x_{r+1}$ be the coordinates of $V$ that diagonalize $\rho$ and denote by $w_1,\ldots,w_{r+1}\in \Z$ the weights of $\rho$. Consider the exact sequence of sheaves
\begin{equation}\label{eq:XX1X2}
0 \lra \OO_{X} \lra \OO_{X_1}\oplus \OO_{X_2}\lra \OO_{X_1\cap X_2}\lra 0
\end{equation}
and the other ones obtained by tensoring \eqref{eq:XX1X2} by $\OO_{X}(m)$ with $m\in \Z$. For $m\gg 0$ we get the exact sequence
$$
H^0(X,\OO_{X}(m)) \hookrightarrow H^0(X_1,\OO_{X_1}(m))\oplus H^0(X_2,\OO_{X_2}(m))\twoheadrightarrow H^0(X_1\cap X_2,\OO_{X_1\cap X_2}(m)).
$$
Since $X_1\cap X_2$ is a  $0$-dimensional scheme of length $k:=k_{X_1}=k_{X_2}$,  we have
$h^0(X_1\cap X_2,\OO_{X_1\cap X_2}(m))=k$ for each $m\in \Z$. Denote by $P(m),P_1(m),P_2(m)$ the Hilbert polynomials of $X$, $X_1$, and $X_2$ respectively
(satisfying $P_1(m)+P_2(m)=P(m)+k$ by the last exact sequence) and let $\{B_1,\ldots,B_{P(m)}\}$ be a monomial basis of $H^0(X,\OO_X(m))$ such that
$$
W_{X,\rho}(m)=\sum_{i=1}^{P(m)}w_{\rho}(B_i).
$$
Now,  consider the linear independent vectors obtained by restricting the above basis to $X_1$ and to $X_2$:
$$
C_1=({B_1}_{|X_1},{B_1}_{|X_2}),\ldots, C_{P(m)}=({B_{P(m)}}_{|X_1},{B_{P(m)}}_{|X_2})\in H^0(X_1,\OO_{X_1}(m))\oplus H^0(X_2,\OO_{X_2}(m))
$$
Adding other vectors $C_{P(m)+j}=(B_{1j},B_{2j})$ for $j=1,\ldots,k$, we can complete the linear independent set $\{C_1,\ldots,C_{P(m)}\}$ to a basis of
$H^0(X_1,\OO_{X_1}(m))\oplus H^0(X_2,\OO_{X_2}(m))$. Now,  it is easy to check that, up to reordering the vectors, $\pi_1(C_1),\ldots,\pi_1(C_{P_1(m)})$ are linear independent in
$H^0(X_1,\OO_{X_1}(m))$ and $\pi_2(C_{P_1(m)+1}),\ldots,\pi_2(C_{P(m)+k})$ are linear independent in $H^0(X_2,\OO_{X_2}(m))$, where we denote by $\pi_i$ the projection of
$H^0(X_1,\OO_{X_1}(m))\oplus H^0(X_2,\OO_{X_2}(m))$ onto the $i$-th factor. This implies that, up to reordering the vectors again, there exists $k_1\in \Z$ with $k_1\leq k$ such that ${B_1}_{|X_1},\ldots,{B_{P_1(m)-k_1}}_{|X_1}$ are linear independent in
 $H^0(X_1,\OO_{X_1}(m))$ and  ${B_{P_1(m)-k_1+1}}_{|X_2},\ldots,{B_{P(m)}}_{|X_2}$ are linear independent in $H^0(X_2,\OO_{X_2}(m))$. Finally, setting $k_2:=k-k_1$, we can consider other monomials $B'_1,\ldots,B'_{k_1},B''_{1},\ldots, B''_{k_2}$ so that
$$
\{B_1,\ldots,B_{P_1(m)-k_1},B'_1,\ldots,B'_{k_1}\}\quad\text{is a monomial basis for }H^0(X_1,\OO_{X_1}(m)),
$$
$$
\{B_{P_1(m)-k_1+1},\ldots,B_{P(m)},B''_{1},\ldots, B''_{k_2}\}\quad\text{is a monomial basis for }H^0(X_2,\OO_{X_2}(m)).
$$
Denoting by $\widetilde{w}=\max_i\{w_i\}$, we have
\begin{eqnarray}
W_{X,\rho}(m) & = & \sum_{i=1}^{P(m)}w_{\rho}(B_i)
= \sum_{i=1}^{P_1(m)-k_1}w_{\rho}(B_i)+\sum_{i=P_1(m)-k_1+1}^{P(m)}w_{\rho}(B_i)\nonumber \\
&\geq & W_{X_1,\rho}(m)-\sum_{i=1}^{k_1}w_{\rho}(B'_i)+
W_{X_2,\rho}(m)-\sum_{i=1}^{k_2}w_{\rho}(B''_i)\nonumber\\
& \geq & W_{X_1,\rho}(m)+W_{X_2,\rho}(m)-k\widetilde{w}\,m=\bigg(\frac{e_{X_1,\rho}+e_{X_2,\rho}}{2} \bigg) m^2+O(m)\nonumber,
\end{eqnarray}
which implies that
$$
e_{X_\rho}\geq e_{X_1,\rho}+e_{X_2,\rho}.
$$
Now,  we will prove the reverse inequality. Let $F$ be a homogeneous polynomial of degree $h\geq 1$ vanishing identically on $X_1$ and regular on $X_2$.
Let $\{B_1,\ldots,B_{P_1(m)}\}$ be a monomial basis of $H^0(X_1,\OO_{X_1}(m))$ and $\{B'_1,\ldots,B'_{P_2(m-h)}\}$ a monomial basis of $H^0(X_2,\OO_{X_2}(m-h))$ such that
$$
W_{X_1,\rho}(m)=\sum_{i=1}^{P_1(m)}w_{\rho}(B_i)\qquad\text{and}\qquad W_{X_2,\rho}(m-h)=\sum_{i=1}^{P_2(m-h)}w_{\rho}(B'_i).
$$
It is easy to check that $B_1,\ldots,B_{P_1(m)},FB'_1,\ldots,FB'_{P_2(m-h)}$ are linearly independent in $H^0(X,\OO_X(m))$, so that, setting $d_2=\deg X_2$, we have
\begin{equation*}
\begin{aligned}
\dim \left\langle B_1,\ldots,B_{P_1(m)},FB'_1,\ldots,FB'_{P_2(m-h)}\right\rangle  & =  P_1(m)+P_2(m-h)= \\
=P_1(m)+P_2(m)-d_2h= &P(m)+k-d_2h\leq P(m).
\end{aligned}
\end{equation*}
Adding possibly other monomials $B''_{1},\ldots,B''_{d_2h-k}$, we get a basis of $H^0(X,\OO_X(m))$.
Actually we would like to work with a monomial basis in order to apply the Hilbert-Mumford numerical criterion (Fact \ref{Chow-crit}), so suppose that $F=M_1+\ldots+M_p$, where $M_1, \ldots,M_p$ are monomials of degree $h$. It is an easy exercise to prove that for $j=1,\ldots, P_2(m-h)$ we can choose monomials $M_{i_j}$ such that
$$
B_1,\ldots,B_{P_1(m)},M_{i_1}B'_1,\ldots,M_{i_{P_2(m-h)}}B'_{P_2(m-h)},B''_{1},\ldots,B''_{d_2h-l}
$$
are linearly independent. For each $m\gg 0$ we get
\begin{eqnarray}
W_{X,\rho}(m)& \leq & \sum_{j=1}^{P_1(m)}w_{\rho}(B_j)+\sum_{j=1}^{P_2(m-h)}w_{\rho}(M_{i_j}B'_j)+\sum_{j=1}^{d_2h-k}w_{\rho}(B''_j)\nonumber\\
& \leq & W_{X_1,\rho}(m)+W_{X_2,\rho}(m-h)+h\widetilde{w}P_2(m-h)+(d_2h-k)\widetilde{w}m \nonumber \\
& = & \bigg(\frac{e_{X_1,\rho}+e_{X_2,\rho}}{2} \bigg) m^2+O(m).\nonumber
\end{eqnarray}
This implies that
$$
e_{X_\rho}\leq e_{X_1,\rho}+e_{X_2,\rho}
$$
and we are done.
\end{proof}
\begin{rmk}
Proposition \ref{prop:sumsub}\eqref{prop:sumsub2} improves the estimate of \cite[Chap. 4, Ex. 4.49]{HM}, which however holds even for non-reduced
1-dimensional complete subschemes of $\P^r$.
\end{rmk}

Proposition \ref{prop:sumsub}\eqref{prop:sumsub2} holds only for the Chow weight. Later on, we will see a class of examples with $n=2$ (see Lemma \ref{lem:sumhilb}), which in general do not satisfy the equality $W_{X,\rho}(m)= W_{X_1,\rho}(m)+W_{X_2,\rho}(m)$.

We conclude this subsection by recalling two technical lemmas which are very useful to estimate $e_{X,\rho}$. 

\begin{lemma}\label{lem:2wdeg}
Let $[X\subset \P^r]\in \Hilb_d$ and consider a 1ps $\rho$ of $\GL_{r+1}$ diagonalized by a system of coordinates $\{x_1,\ldots,x_{r+1}\}$ with weights $w_1,\ldots,w_{r+1}$. Suppose that for some $1\leq n\leq r$
\begin{enumerate}
\item $x_1,\ldots,x_{n}$ vanish on $X_{\rm red}$;
\item $w_1=\ldots=w_{n}=0$ and $w:=w_{n+1}=\ldots=w_{r+1}$.
\end{enumerate}
Then
$$e_{X,\rho}=2\,w\,\deg\OO_X(1).$$
\end{lemma}
\begin{proof}
We use some ideas from [Sch91, Lemma 1.2]. Since $x_1,\ldots,x_{n}$ vanish on $X_{\rm red}$, for each $j=1,\ldots,n$ we can define
$$
m_j=\max_{m}\{x_j^m\text{ does not vanish identically on }X\}.
$$
If $\{B_1,\ldots,B_{P(m)}\}$ is a monomial basis of $H^0(X,\OO_{X}(m))$, then
$$
w\bigg(m-\sum_{j=1}^n m_j\bigg)\leq w(B_i)\leq wm \hspace{0.3cm} \text{ for every } 1\leq i\leq P(m), �
$$
which implies that
$$
w\bigg(m-\sum_{j=1}^n m_j\bigg)(dm-g+1)\leq \sum_{i=1}^{P(m)}w(B_i)\leq wm(dm-g+1).
$$
We deduce that
$$
e_{X,\rho}=2\,w\,\deg\OO_X(1).
$$
\end{proof}

For a proof of the following , see [Sch91, Lemma 1.4].

\begin{lemma}\label{L:schub} Let $[X\subset \P^r]\in \Hilb_d$ be a reduced curve and let $\nu: X^{\nu}\lra X$ be its normalization. Consider a 1ps $\rho$ of $\GL_{r+1}$ diagonalized by a system of coordinates $\{x_1,\ldots,x_{r+1}\}$ with weights $w_1,\ldots,w_{r+1}$. Given a set of points $\{p_1,\ldots,p_k\}\subset X$ and its inverse image $\nu^{-1}({p_1,\ldots,p_k})=\{q_1,\ldots,q_n\}$, suppose that
\begin{enumerate}
	\item for each $j=1,\ldots,n$, there exists $i$ such that $\ord_{q_j}(\nu^*(x_i))=0$ 
	%does not vanish identically on the irreducible component of $X^{\nu}$ containing $p_j$
	and $w_i=0$;
	\item there exist positive integers $a_1,\ldots,a_n$ such that $\ord_{q_j}(\nu^*(x_i))+w_i\geq a_j$ for each $i=1,\ldots,r+1$ and $j=1,\ldots,n$.
\end{enumerate}
Then
$$
e_{X,\rho}\geq \sum_{i=1}^n a_n^2
$$
\end{lemma}

\subsection{Basins of attraction}\label{Sec:bas-attra}

Basins of attraction represent a useful tool in the  study of the orbits which are identified in a GIT quotient.
%In order to study the orbits which are identified in a GIT quotient, a useful tool is represented by the basin of attraction.
We review the basic definitions,
following the presentation in \cite[Sec. 4]{HH2}.

\begin{defi}\label{D:bas-attra}
Let $[X_0\subset \P^r]\in \Hilb_d$ and $\rho:\Gm\to \GL_{r+1}$ a 1ps of $\GL_{r+1}$ that stabilizes $[X_0\subset \P^r]$.
The $\rho$-\emph{basin of attraction} of $[X_0\subset \P^r]$ is the subset
$$A_{\rho}([X_0\subset \P^r]):=\{[X\subset \P^r]\in \Hilb_d\: :\: \lim_{t\to 0} \rho(t)\cdot [X\subset \P^r]=[X_0\subset \P^r] \}.$$
\end{defi}
Clearly, if $[X\subset \P^r]\in A_{\rho}([X_0\subset \P^r])$ then $[X_0\subset \P^r]$ belongs to the closure of the $\SL_{r+1}$-orbit $O([X\subset \P^r])$ of $[X\subset \P^r]$.
Therefore, if $[X_0\subset \P^r]$ is Hilbert semistable
(resp. Chow semistable) then every $[X\subset \P^r]\in A_{\rho}([X_0\subset \P^r])$
is Hilbert semistable (resp. Chow semistable) and is identified with $[X_0\subset \P^r]$ in the GIT quotient $\Hilb_d^{ss}/SL_{r+1}$ (resp. $\Ch^{-1}(\Chow_d^{ss})/SL_{r+1}$).

The following well-known properties of the basins of attraction (see e.g. \cite[p. 24-25]{HH2}) will be used in the sequel.
%  (see for example \cite[p. 24-25]{HH2}).

\begin{fact}\label{F:weight-basin}
Same notation as in Definition \ref{D:bas-attra} and let $m\geq M$ as in Section \ref{Sec:Hilb-Chow}.
\begin{enumerate}[(i)]
\item If $\mu([X_0\subset \P^r]_m,\rho)<0$ (resp. $e_{X_0,\rho} > 2 d \cdot \frac{w(\rho)}{r+1}$) then every $[X\subset \P^r]\in A_{\rho}([X_0\subset \P^r])$ is not $m$-Hilbert semistable
(resp. not Chow semistable).
\item If $\mu([X_0\subset \P^r]_m,\rho)=0$ (resp. $e_{X_0,\rho} = 2 d \cdot \frac{w(\rho)}{r+1}$)  then $[X_0\subset \P^r]$ is $m$-Hilbert semistable (resp. Chow semistable)
if and only if every $[X\subset \P^r]\in A_{\rho}([X_0\subset \P^r])$ is $m$-Hilbert semistable (resp. Chow semistable).
\end{enumerate}
\end{fact}
\subsection{Flat limits and Gr\"obner bases}
\label{S:limGrob}
A useful technique for computing the limit $\lim_{t\ra 0}\rho(t)[X\subset \P^r]$ is based on the theory of Gr\"obner bases (see \cite{HeHi} for the general theory and \cite{HHL}
and for its applications to GIT). Let $\rho: \Gm \to \GL(V)$ be a 1ps  and let $\{x_1,\ldots, x_{r+1}\}$ be
coordinates of $V$ that diagonalize the action of $\rho$, so that for $i=1,\ldots, r+1$ we have
$$\rho(t)\cdot x_i = t^{w_i}x_i \:  \text{ for  some } w_i\in \Z.$$
If $a=(a_1,\ldots,a_{r+1})\in \N^{r+1}$, we define the monomial
$$
x^{a}:=x_1^{a_1}x_{2}^{a_2}\ldots x_{r+1}^{a_{r+1}}\in S:=k[x_1,\ldots,x_{r+1}].
$$
Let us define the following order $\prec_{\rho}$ (called the \emph{$\rho$-weighted graded order}) on the set of monomials of $S$. If $x^{a}$ and $x^b$ are monomials, we say that $x^{a} \prec_{\rho} x^{b}$ if
\begin{enumerate}
	\item $\deg\, x^a<\deg\, x^b$ or
	\item $\deg\, x^a=\deg\, x^b$ and $w_{\rho}(x^a)<w_{\rho}(x^b)$.
\end{enumerate}
It is easy to notice that the order $\prec_{\rho}$ is not total, in general. In order to have a total order  $\prec$ (also called monomial order) that refines $\prec_{\rho}$, it suffices to fix a lexicographical order $<$ on the set of monomials of
$S$, for example the one induced by declaring that $x_1<x_2<\ldots < x_{r+1}$, and to say that $x^{a} \prec x^{b}$ if
\begin{enumerate}
	\item $x^{a} \prec_{\rho} x^{b}$ or
	\item $\deg\, x^a=\deg\, x^b$, $w_{\rho}(x^a)=w_{\rho}(x_b)$ and $x^a<x^b$.
\end{enumerate}
We call the above monomial order $\prec$ a \emph{$\rho$-weighted lexicographic order}. Moreover,  if $f=\sum c_a x^a\in S$ and $I$ is an ideal of $S$, we denote by
\begin{enumerate}
	\item $\In_{\prec_{\rho}}(f)$ the sum of the terms of $f$ of maximal order with respect to $\prec_{\rho}$;
	\item $\In_{\prec_{\rho}}(I)=\langle \In_{\prec_{\rho}}(f)\,|\,f\in I\rangle$;
	\item $\In_{\prec}(f)$ the monomial (hence without coefficient) of maximal order with respect to $\prec$;
	\item $c_{\prec}(f)$ the coefficient of $\In_{\prec}(f)$ in $f$;
	\item $\In_{\prec}(I)=\langle \In_{\prec}(f)\,|\,f\in I\rangle$;
	\item $w(f)=\max\{w_{\rho}(x^a)\,|\,c_a\neq 0\}$ and $\widetilde{f}(x_1,\ldots,x_{r+1},t)=t^{w(f)}f(t^{-w_1}x_1,\ldots,t^{-w_{r+1}}x_{r+1})$;
	\item $\widetilde{I}=\langle \widetilde{f},f\,|\,f\in I\rangle\subset S[t]$.
\end{enumerate}
Now, we recall the definition of Gr\"obner basis with respect to a monomial order (see \cite[Definition 2.1.5]{HeHi}).
\begin{defi}
Let $I$ be an ideal of $S$ and $\prec$ a monomial order. A system of generators $\{f_1,\ldots,f_{n}\}$ of $I$ is said to be a \textbf{Gr\"obner basis} for $I$ with respect to $\prec$ if $\In_{\prec}(I)=\langle \In_{\prec}(f_1),\ldots,\In_{\prec}(f_n)\rangle$.
\end{defi}
In the sequel, we will use some facts about Gr\"obner bases. First of all, we recall a famous criterion to determine whether a system of generators of an ideal is a Gr\"obner basis or not (cf. \cite[Theorem 2.3.2]{HeHi}). Let $f_1,f_2\in S$ be two homogeneous polynomial and define
$$
S(f_1,f_2)=\frac{\text{l.c.m.}(\In_{\prec}(f_1),\In_{\prec}(f_2))}{c_{\prec}(f_1)\,\In_{\prec}(f_1)}\, f_1-\frac{\text{l.c.m.}(\In_{\prec}(f_1),\In_{\prec}(f_2))}{c_{\prec}(f_2)\,\In_{\prec}(f_2)}\, f_2.
$$
where $\text{l.c.m.}(\In_{\prec}(f_1),\In_{\prec}(f_2))$ is the least common multiple of $\In_{\prec}(f_1)$ and $\In_{\prec}(f_2)$.
\begin{fact}\label{F:buch}\rm{(Buchberger's criterion)}
Let $I=\langle f_1,\ldots, f_{n}\rangle$ be an ideal in $S$ and $\prec$ a monomial order. The system of generators $\{f_1,\ldots, f_{n}\}$ is a Gr\"obner basis with respect to $\prec$ if and only if
$$
\In_{\prec}(S(f_i,f_j))\in \langle \In_{\prec}(f_1),\ldots, \In_{\prec}(f_n)\rangle
$$
for each $i,j\in \{1,\ldots,n\}$.
\end{fact}
Now,  we recall a basic fact about the relation between Gr\"obner bases and flat limits (see \cite[Theorem 3]{HHL} or for more details \cite[Sec. 3.2]{HeHi}).
\begin{fact}\label{T:flatlimit}
If $I\subset S$ is an ideal, then the $k[t]$-algebra $S[t]/\widetilde{I}$ is free as a $k[t]$-algebra. Moreover, the following hold:
\begin{equation}\label{eq:flatlimit}
S[t]/\widetilde{I}\otimes_{k[t]}k[t,t^{-1}]\cong (S/I)[t,t^{-1}]\quad \text{and}\quad S[t]/\widetilde{I}\otimes_{k[t]}k[t]/(t)\cong S/\In_{\prec_{\rho}}(I).
\end{equation}
\end{fact}
We obtain a useful corollary.
\begin{coro}\label{C:flatlimit}
Let $[X\subset\P^r]\in \Hilb_d$ and let $\rho$ be a one-parameter subgroup of $\GL_{r+1}$.
Denote by $I$ the homogeneous ideal of $X$. Then $[V(\In_{\prec_{\rho}}(I))\subset \P^r]=\lim_{t \ra 0}\rho(t)[X\subset \P^r]$.
\end{coro}
\begin{proof}
By Fact \ref{T:flatlimit} we have a family of curves $\cX\ra \A^1$ whose central fiber is $V(\In_{\prec_{\rho}}(I))\subset\P^r$. This yields a map $\beta:\A^1\ra\Hilb_d$ which coincide away from $0\in \A^1$ with the map $\alpha:\A^1\ra \Hilb_d$ induced by $\rho$. Since $\Hilb_d$ is projective, the maps $\alpha$ and $\beta$ coincides everywhere, and we are done.
\end{proof}
Finally, the following fact allows us to compute explicitly the ideal $\In_{\prec_{\rho}}(I)$ (see \cite[Theorem 3]{HHL} or, for more details, \cite[Sec. 3.2]{HeHi}).
\begin{fact}\label{T:groebflatlimit}
Let $\{f_1,\ldots,f_{n}\}$ be a Gr\"obner basis for $I$ with respect to a $\rho$-weighted lexicographical order $\prec$ that refines $\prec_{\rho}$. Then
\begin{enumerate}[(i)]
	\item $\widetilde{f_1},\ldots,\widetilde{f_n}$ generate $\widetilde{I}$;
	\item $\In_{\prec_{\rho}}(f_1),\ldots,\In_{\prec_{\rho}}(f_n)$ generate $\In_{\prec_{\rho}}(I)$.
\end{enumerate}
\end{fact}

\subsection{The parabolic group}
\label{sec:ParabGroup}

Here we recall a classical result due to J. Tits (see for more details \cite[Sec. 9.5]{Dol} or \cite[Chap. 2, Sec. 2]{GIT}), which is very useful to study the semistable locus of the action of a reductive group $G$ on an algebraic variety. Let $X\subset \P(V)$ be a projective variety and $G$ a reductive group that acts on $X$ via a linear representation in $V$. For the sake of simplicity, we assume that $G=\text{GL}(W)$ for some vector space $W$. By the Hilbert-Mumford criterion, $x\in X$ is semistable if and only if for every one-parameter subgroup $\rho:\Gm \ra \text{GL}(W)$ we have that $\mu(x,\rho)\geq 0$. We know that every one-parameter subgroup is diagonalized by some basis of $V$. A priori, it does not suffice to check the condition of the Hilbert-Mumford criterion for all one-parameter subgroups, which are diagonalized by a fixed basis of $V$: this represents the main difficulty in characterizing the semistable locus. Tit's result allows one to identify the one-parameter subgroups, which give the ``worst'' weights so that the research of a destabilizing one-parameter subgroup is less intricate.

\begin{defi}
We define the \textbf{parabolic group} with respect to a one-parameter subgroup $\rho$ by setting
$$
P(\rho)=\big\{\,g\in \text{GL}(W)\,|\,\text{ there exists }\lim_{t\rightarrow 0}\rho(t) g \rho(t)^{-1}\big\}\subset \text{GL}(W).
$$
\end{defi}

\begin{fact}\label{F:parabolic}
The group $P(\rho)$ is a parabolic subgroup of $\GL(W)$, i.e. it contains a Borel subgroup. Moreover, if $x\in X$, then $\mu(x,\rho)=\mu(x,A^{-1}\rho A)$ for each $A\in P(\rho)$.
\end{fact}
For a proof see \cite[Lemma 9.2, Lemma 9.3]{Dol} or \cite[Def. 2.3/Prop. 2.6]{GIT}. It is not difficult to show that when we consider the action of $\GL_{r+1}$ on $\Hilb_{d}$, if the weights of the 1ps $\rho$ with respect to a diagonalizing basis $\{x_1,\ldots,x_{r+1}\}$ of $V$ satisfy the inequalities $w_1\geq \ldots\geq w_{r+1}$, then $P(\rho)$ contains the group of the upper triangular matrices with respect to the coordinates $\{x_1,\ldots,x_{r+1}\}$. This fact has a useful consequence.

\begin{coro}\label{cor:base}
Let $[X\subset \P^r]\in \Hilb_d$ and  let $Y:=(y_1,\ldots, y_{r+1})^t$ be an arbitrary basis of $V$.
\begin{enumerate}[(i)]
\item \label{C:base1} Let $\rho:\Gm\ra \GL_{r+1}$ be a 1ps diagonalized by the basis coordinates
$X=(x_1,\ldots,x_{r+1})^t$ with weights $w_1, \ldots, w_{r+1}$, respectively.
Then there exist a lower unitriangular matrix $A=(a_{ij})$ and a one-parameter subgroup $\rho':\Gm\ra \GL_{r+1}$ diagonalized by the new coordinates $(z_1,\ldots,z_{r+1})^t=:Z=AY$
such that
$$
\rho'(t)z_i=t^{w_{\sigma(i)}}z_i\quad \text{for some } \sigma\in S_{r+1}\quad \text{ and } \quad W_{X,\rho}(m)=W_{X,\rho'}(m)\: \text{ for } \: m\gg 0.
$$
\item \label{C:base2} $[X\subset \P^r]$ is Hilbert semistable (resp. polystable, stable) if and only if it is Hilbert semistable (resp. polystable, stable) with respect to all the one-parameter subgroups which are diagonalized by $Z=AY$ for every lower unitriangular matrix $A$. The same holds for the Chow semistability (resp. polystability, stability).
\end{enumerate}
\end{coro}
\begin{proof}
In order to prove \eqref{C:base1}, it suffices to assume that $w_1\geq \ldots\geq w_{r+1}$ and that
$$
y_1= x_1,\ldots,\text{ }y_{l-1}= x_{l-1},\text{ }y_{l}= y=\sum_{i=1}^{r+1}\lambda_i x_i,\text{ }y_{l+1}= x_{l+1},\ldots,y_{r+1}= x_{r+1},
$$
where $\lambda_1,\ldots,\lambda_{r+1}\in k$.
% such that
%$$y=\sum_{i=1}^{r+1}\lambda_i x_i.$$
Now,  we define the following basis of coordinates:
$$
z_i=x_i \: \text{ if } \: i\neq l \: \text{ and }\:
z_l=y-\sum_{i=1}^{l-1}\lambda_i x_i=\sum_{i=l}^{r+1}\lambda_i x_i.
$$
Let $A$ and $B$ be the matrices such that $Z=AY=BX$. By construction, $A$ is lower unitriangular and $B$ is upper unitriangular, hence $B\in P(\rho)$ by Fact \ref{F:parabolic} and
\begin{equation}\label{eq:WrhoWrhop}
W_{X,\rho}(m)=W_{X,B^{-1}\rho B}(m)
\end{equation}
for $m\gg 0$. Now,  if we define $\rho'=B^{-1}\rho B$, then $\rho'$ is diagonalized by the coordinates $Z$.

It remains to prove \eqref{C:base2}.
The ``only if'' implication follows from Remark \ref{R:poly-1ps}. In order to prove the ``if'' direction, consider a 1ps $\rho$ of $\GL_{r+1}$ diagonalized by a basis $X=(x_1,\ldots,x_{r+1})$.
%Without loss of generality, we can assume that the weights of $\rho$ with respect to the basis $X$ are such that
%$w_1\geq \ldots\geq w_{r+1}$.
Using \eqref{C:base1}, we can find a lower unitriangular matrix $A$ such that the 1ps $\rho':=A^{-1}\rho A$ is diagonalized by the basis $Z=AY$ and is such that
\begin{equation}\label{E:rho1}
w(\rho')=w(\rho) \: \: \text{ and } \: W_{X,\rho'}(m)=W_{X;\rho}(m)\: \: \text{ for } \gg 0.
\end{equation}
The equalities \eqref{E:rho1} imply \eqref{C:base2} for the (semi)stability. Now,  let us prove \eqref{C:base2} for the polystability. Since $A\in P(\rho)$, there exists $\lim_{t\ra 0} (\rho(t)A\rho(t)^{-1})$: call it $B$. We have
\begin{eqnarray}\label{E:rho2}
\lim_{t\ra 0}\rho'(t)[X\subset \P^r] &=& \lim_{t\ra 0}(A^{-1}\rho(t) A[X\subset \P^r])= \lim_{t\ra 0}A^{-1}(\rho(t)A\rho(t)^{-1})(\rho(t)[X\subset \P^r])\nonumber \\
&=& A^{-1} \cdot \lim_{t\ra 0} (\rho(t)A\rho(t)^{-1})\cdot \lim_{t\ra 0}\rho(t)[X\subset \P^r]= A^{-1}B\cdot \lim_{t\ra 0}\rho(t)[X\subset \P^r].
\end{eqnarray}
Combining \eqref{E:rho1} and \eqref{E:rho2}, we see that $[X\subset \P^r]$ is Hilbert (resp. Chow) polystable with respect to $\rho'$ if and only if it Hilbert (resp. Chow)
polystable with respect to $\rho$; combined with Remark \ref{R:poly-1ps}, this concludes the proof of \eqref{C:base2}.
\end{proof}

\subsection{Stability of smooth curves and Potential stability}\label{Sec:pot-stab}

Here we recall two basic results due to Mumford and Gieseker: the stability of smooth curves of high degree and the
(so-called) potential stability theorem.

\begin{fact}\label{F:stab-smooth}
If $[X\subset \P^r]\in \Hilb_d$ is connected and smooth and $d\geq 2g+1$, then $[X\subset \P^r]$ is Chow stable.
% \in \Ch^{-1}(\Chow_d^s)$.
\end{fact}

For a proof, see \cite[Thm. 4.15]{Mum}.
In \cite[Thm. 1.0.0]{Gie}, a weaker form of the above Fact is proved:
if $[X\subset \P^r]\in \Hilb_d$ is connected and smooth and $d\geq 10(2g-2)$ then $[X\subset \P^r]$ is Hilbert stable.
See also \cite[Chap. 4.B]{HM} and \cite[Sec. 2.4]{Mo} for an overview of the proof.

%The so-called potential stability theorem takes the following form:

\begin{fact}[Potential stability]\label{F:GM2}
If $d>4(2g-2)$ and $[X\subset \P^r] \in \Ch^{-1}(\Chow_d^{ss})\subset \Hilb_d$ (with $X$ possibly non connected) then:
\begin{enumerate}[(i)]
\item $X$ is reduced of pure dimension one
and has at most nodes as singularities. In particular, $X$ is a pre-stable curve whenever it is connected.
\item $X\subset \P^r$ is non-degenerate, linearly normal (i.e., $X$ is
embedded by the complete linear system $|\OO_X(1)|$) and $\OO_X(1)$ is non-special (i.e., $H^1(X, \OO_X(1))=0$).
\item The line bundle $\OO_X(1)$ on $X$ is balanced (see Definition \ref{D:bal-lb}).
% i.e., for any subcurve $Y\subset X$ it holds
%\begin{equation}\label{bas-ineq2}
%\left|\deg_Y(\OO_X(1))-\frac{d}{2g-2}\deg_Y(\omega_X)\right|\leq \frac{k_Y}{2},
%\end{equation}
%where $k_Y:=|Y\cap Y^c|$ is the cardinality of the intersection of $Y$ with
%the complementary subcurve $Y^c:=\ov{X\setminus Y}$.
\end{enumerate}
\end{fact}
\begin{proof}
For the connected case, see \cite[Prop. 4.5]{Mum}. In \cite[Thm. 1.0.1, Prop. 1.0.11]{Gie}, the same conclusions are shown to hold under the stronger hypothesis that $[X\subset \P^r]\in \Hilb_d^{ss}$
 and $d\geq 10(2g-2)$.
See also \cite[Chap. 4.C]{HM} and \cite[Sec. 3.2]{Mo} for an overview of the proof. If $X$ is not connected, the argument is analogous to Theorem \ref{teo-pstab} below.
\end{proof}
%The terminology "basic inequality" was first used in \cite[Sec. 3.1]{Cap}.

\begin{rmk}\label{R:potstab-sharp}
%\noindent
%\begin{enumerate}[(i)]
%\item
The hypothesis that $d> 4(2g-2)$ in Fact \ref{F:GM2} is sharp: in \cite{HMo} it is proved that all
the $4$-canonical p-stable curves (which in particular may have cusps) belong to $\Hilb_{4(2g-2)}^s$.

%\item It is easy to check that a connected nodal curve $X$ with an ample line bundle of degree $d$ which
%satisfies the subcurve inequality (\ref{bas-ineq2})
%(for example, the line bundle $\OO_X(1)$ of Fact \ref{F:GM2}) is \emph{quasi-stable}
%(see \cite[Sec. 3.3]{Cap}) i.e., every connected subcurve $Y$ of $X$ of arithmetic genus $0$
%and having $k_Y\leq 3$ is such that $Y\cong \P^1$ and $k_Y=2$.
%\end{enumerate}
\end{rmk}

\section{Potential pseudo-stability theorem}\label{S:pot-pseudo}

The aim of this section is to generalize the Potential stability theorem (see Fact \ref{F:GM2}) for
smaller values of $d$. The main result is the following theorem, which we call Potential pseudo-stability
Theorem for the relations with the pseudo-stable curves (see Definition \ref{D:stable}\eqref{D:p-stable}).

\begin{thm}\label{teo-pstab} {\rm (Potential pseudo-stability theorem)}
If $d>2(2g-2)$ and $[X\subset \P^r] \in  \Ch^{-1}(\Chow_d^{ss})$ $\subset \Hilb_d$ (with $X$ possibly not connected), then
\begin{enumerate}[(i)]
\item \label{teo-pstab1}$X$ is a pre-wp-stable curve, i.e. it is reduced and its singularities are at most nodes, cusps and tacnodes with a line.
% i.e. $X$ is pre-wp-stable (see Definition \ref{D:pre-sing}\eqref{pre-wp-stable});
% reduced and has at most nodes, cusps or tacnodes with a line as singularities;
\item \label{teo-pstab2} $X\subset \P^r$ is non-degenerate, linearly normal (i.e., $X$ is
embedded by the complete linear system $|\OO_X(1)|$) and $\OO_X(1)$ is non-special (i.e., $H^1(X, \OO_X(1))=0$);
\item \label{teo-pstab3}
The line bundle $\OO_X(1)$ on $X$ is balanced (see Definition \ref{D:bal-lb}).
% i.e., for any subcurve $Y\subset X$ it holds
%\begin{equation}\label{bas-ineq}
%\left|\deg_Y(\OO_X(1))-\frac{d}{2g-2}\deg_Y(\omega_X)\right|\leq \frac{k_Y}{2},
%\end{equation}
%where $k_Y:=l(\OO_{Y\cap Y^c})$ is the length of the finite scheme given by the schematic intersection
% of $Y$ with the complementary subcurve $Y^c:=\ov{X\setminus Y}$.
\end{enumerate}
%\item \label{teo-pstab-b}
%If moreover $2(2g-2)<d <4(2g-2)$ or $d=4(2g-2)$ and $X\in \Hilb_d^{ss}$ then $X$ does not have any elliptic tails.

\end{thm}

\begin{proof}
To prove the claim, we adapt various results in \cite{Mum}, \cite{Gie}, \cite{Sch}, \cite[Chap. 4]{HM}
and \cite[Sec. 7]{HH2}. Let us indicate the different steps of the proof. Suppose that $[X\subset \P^r]\in  \Ch^{-1}(\Chow_d^{ss})\subset \Hilb_d$ (with $X$ possibly non connected).
We will denote by $X'\subset X$ the union of the connected components of $X$ of dimension $1$.

$\bullet$ $X'_{\rm red}$ is non-degenerate. 

Under the assumption that $X=X'$, this follows from \cite[Prop. 1.0.2]{Gie} (see also the step 1 of the proof of \cite[Chap. 4, Thm. 4.45]{HM}). We will include a  proof that works also in our setting. 

Suppose, by contradiction, that $X'_{\rm red}$ is degenerate. Then, there exists a section $s\in H^0(\P^1,\OO_{\P^1}(1))$ that vanishes identically on $X'_{\rm red}$. Let $\{x_1,\ldots,x_{r+1}\}$ be a system of coordinates with $x_1=s$ and consider a 1ps diagonalized by $\{x_1,\ldots,x_{r+1}\}$ with weights $w_1 = 0$ and $w_2 =\ldots = w_{r+1} = 1$. By Lemma \ref{lem:2wdeg}, we get $e_{X',\rho}=2d$.
% and a positive integer $q$ such that $s^q$ vanishes identically on $X'$.  Denoting by $Q$ the Hilbert polynomial of $X'$, there exists an integer $k\leq 1-g$ such that $Q(m)=dm+k$ for $m\gg 0$. If $\{B_1,\ldots,B_{Q(m)}\}$ is a monomial basis of $H^0(X',\OO_{X'}(m))$, then
%$$
%m-q+1\leq w(B_i)\leq m \hspace{0.3cm} \text{ for every } 1\leq i\leq P(m), �
%$$
%which implies that
%$$
%(m-q+1)(dm+k)=(m-q+1)Q(m)\leq \sum_{i=1}^{Q(m)}w(B_i)\leq mQ(m)= m(dm+k).
%$$
%We deduce that
Now we apply Proposition \ref{prop:sumsub}\eqref{prop:sumsub1} and we have
$$
e_{X,\rho}\geq e_{X',\rho}=2d>\frac{2d}{r+1}\,r.
$$
This implies that $[X\subset \P^r]$ is Chow unstable.

$\bullet$ For each subcurve $Z$ of $X'$ it holds that
\begin{equation}\label{eq:basineqgrezzak0}
2\,\deg_{Z}\OO(1)\leq \frac{2d}{r+1}\,h^0(Z_{\rm red},\OO_{Z_{\rm red}}(1)).
\end{equation}

Indeed, consider the restriction map $\pi:H^0(\P^r,\OO_{\P^r}(1))\lra H^0(Z_{\rm red},\OO_{Z_{\rm red}}(1))$ and choose a system of coordinates $\{x_1,\ldots,x_{r+1}\}$ such that $(x_1,\ldots,x_{r_1})$ is a basis of the kernel $K$ of $\pi$. Now let $\rho$ be the
1ps that, in the above coordinates, has the diagonal form $\rho(t)\cdot x_i = t^{w_i}x_i$ where
\begin{equation}\label{eq:wbasinegrezzak0}
w_i=
\begin{cases}
0 \quad \text{ if }\,\, 1\leq i\leq r_1,\\
1 \quad \text{ if }\,\, r_1+1\leq i\leq r+1.
\end{cases}
\end{equation}
By Lemma \ref{lem:2wdeg} we have $e_{Z,\rho}=2\deg_Z\OO(1)$, so that applying Proposition \ref{prop:sumsub}\eqref{prop:sumsub1} and \ref{prop:sumsub}\eqref{prop:sumsub2} we obtain
$$
e_{X,\rho}\geq e_{Z,\rho}=2\,\deg_Z\OO(1).
$$
Since $[X\subset \P^r]$ is Chow semistable and $w(\rho)\leq h^0(Z_{\rm red},\OO_{Z_{\rm red}}(1))$, we deduce the inequality \eqref{eq:basineqgrezzak0}.

$\bullet$ $X'=Y\sqcup Z$, where $Y$ is generically reduced and $Z$ is a disjoint union of lines of multiplicity 2.

%We use some ideas from the proof of \cite[Lemma 2.4]{Sch}, which works under the assumption that $X$ is connected and $d>2(2g-2)$, and from the proof of \cite[Prop. 1.0.2]{Gie} (see also \cite[Lemma 7.4]{HH2} for another proof of the connected case under the assumption that $d>\frac{3}{2}(2g-2)$).

Indeed, let $C$ be an irreducible component of $X'$  that is not generically reduced and denote by $n$ its multiplicity. There are two cases:
\begin{enumerate}
	\item \label{nonred1} $k_C\neq 0$;
	\item \label{nonred2} $k_C=0$.
\end{enumerate}

Suppose that case \eqref{nonred1} occurs and set $D=\overline{X'\setminus C}$. Choose $p\in C_{\rm red}\cap D_{\rm red}$ and a system of coordinates $\{x_1\,\ldots,x_{r+1}\}$ such that $x_2,\ldots,x_{r+1}$ vanish at $p$. If $\rho$ is a 1ps diagonalized by $\{x_1\,\ldots,x_{r+1}\}$ with weights $w_1=1$ and $w_2=\ldots=w_{r+1}=0$, then by Lemma \ref{L:schub} we get $e_{C_{\rm red}, \rho}\geq 1$ and $e_{D_{\rm red},\rho}\geq 1$, hence
$$
e_{X',\rho}\geq n\,e_{C_{\rm red},\rho}+e_{D_{\rm red},\rho}\geq 2+1=3.
$$
By Proposition \ref{prop:sumsub}\eqref{prop:sumsub1}, we have
$$
e_{X,\rho}\geq e_{X',\rho}\geq 3>\frac{2d}{r+1} \quad\text{if}\quad \frac{d}{r+1}<\frac{3}{2} \quad\bigg(\Longleftrightarrow d>\frac{3}{2}(2g-2)\bigg)
$$
hence, under our assumption on $d$, $[X\subset \P^r]$ is Chow unstable.

It remains to analyze case \eqref{nonred2}. Consider the exact sequence
\begin{equation}\label{eq:Cred}
0\lra \OO_{C_{\rm red}}\lra \OO(1)_{|C_{\rm red}}\lra \OO_{D}\lra 0,
\end{equation}
where $D$ is a divisor associated to $\OO(1)_{|C_{\rm red}}$ with support on the smooth locus of $C_{\rm red}$. From the cohomology exact sequence associated to (\ref{eq:Cred}) it follows that
\begin{equation}\label{eq:Cred2}
h^0(C_{\rm red},\OO_{C_{\rm red}}(1))\leq h^0(C_{\rm red},\OO_{C_{\rm red}})+h^0(C_{\rm red},\OO_{D})=\deg (C_{\rm red})+1.
\end{equation}
By \eqref{eq:Cred2} and \eqref{eq:basineqgrezzak0} we have
$$
2n\,\deg\, \OO_{C_{\rm red}}(1)\leq \frac{2d}{r+1}h^0(C_{\rm red},\OO_{C_{\rm red}}(1))\leq\frac{2d}{r+1}(\deg\, \OO_{C_{\rm red}}(1)+1). 
$$
Since $d>2(2g-2)$ if and only if $\frac{d}{r+1}<\frac{4}{3}$, we obtain the inequality
$$
n< \frac{4(\deg\, \OO_{C_{\rm red}}(1)+1)}{3\,\deg\, \OO_{C_{\rm red}}(1)}.
$$
Suppose, by contradiction, that $\deg\, \OO_{C_{\text{red}}}(1)\geq 2$. This implies that
$$
n< \frac{4(\deg\, \OO_{C_{\rm red}}(1)+1)}{3\,\deg\, \OO_{C_{\rm red}}(1)}\leq\frac{4}{3}\cdot\frac{3}{2}=2,
$$
so that $n=1$, which is absurd. If $\deg\, \OO_{C_{\rm red}}(1)=1$ (i. e. $C_{\text{red}}$ is a line), we obtain
$$
n<\frac{8}{3},
$$
hence $n\leq 2$. We deduce that if $C$ is a non-reduced connected component of $X'$, then $C$ is a $\P^1$ with multiplicity 2.

$\bullet$ $Y$ does not have triple points. 

This follows from \cite[Prop. 3.1, p. 69]{Mum} or \cite[Prop. 1.0.4]{Gie}, both of which are easily seen, by direct inspection, to work under the assumption that $d>\frac{3}{2}(2g-2)$. We will give a sketch of the proof.

Given a triple point $p$, we can choose $\{x_1,\ldots,x_{r+1}\}$ such that $p=[1,0,\ldots,0]$. Consider the 1ps $\rho$ diagonalized by $\{x_1,\ldots,x_{r+1}\}$ with weights $w_1=1$ and $w_2=\ldots=w_{r+1}=0$. 
By \cite[Prop. 3.1, p. 69]{Mum} or \cite[Prop. 1.0.4]{Gie}, we get $e_{Y,\rho}\geq 3$. By Proposition \ref{prop:sumsub}\eqref{prop:sumsub1} we have
$$
e_{X,\rho}\geq e_{Y,\rho}\geq 3>\frac{2d}{r+1}\quad\text{if}\quad \frac{d}{r+1}<\frac{3}{2} \quad\bigg(\Longleftrightarrow d>\frac{3}{2}(2g-2)\bigg),
$$
hence $[X\subset \P^r]$ is Chow unstable.

$\bullet$ $Y$ does not have non-ordinary cusps. 

This follows from \cite[Lemma 2.3]{Sch} or \cite[Lemma 7.2]{HH2}. We will include a proof for completeness.

Let $p$ be a non-ordinary cusp. Consider the normalization map $\nu:Y^{\nu}\lra Y$ and set $q=\nu^{-1}(p)$. Since $p$ is a non-ordinary cusp, there exists a system of coordinates $\{x_1\ldots,x_{r+1}\}$ such that
$$
\begin{cases}
\ord_{q}(\nu^*(x_1))=0 \\
\ord_{q}(\nu^*(x_2))=2 \\
\ord_{q}(\nu^*(x_3))=4\\
\ord_{q}(\nu^*(x_i))\geq\, 5 \quad \text{ for any }\,\, 4\leq i\leq r+1. \\
\end{cases}
$$
If $\rho$ is a 1ps diagonalized by $\{x_1\ldots,x_{r+1}\}$ with weights $w_1=5$, $w_2=3$, $w_3=1$ and $w_4=\ldots=w_{r+1}=0$, then 
$\ord_{q}(\nu^*(x_i))+w_i\geq 5$ for each $i=1,\ldots,r+1$, so that applying Lemma \ref{L:schub} we get $e_{Y,\rho}\geq 5^2=25$.
By Proposition \ref{prop:sumsub}\eqref{prop:sumsub1} we have
$$
e_{X,\rho}\geq e_{Y,\rho}\geq 25>\frac{2d}{r+1}\,9 \quad\text{if}\quad \frac{d}{r+1}<\frac{25}{18} \quad\bigg(\Longleftrightarrow d>\frac{25}{14}(2g-2)\bigg).
$$
Since $\frac{25}{14}(2g-2)<2(2g-2)$, $[X\subset \P^r]$ is Chow unstable.

$\bullet$ $Y$ does not have higher order tacnodes or tacnodes in which one of the two branches does not belong to a line. 

By contradiction, let $p\in Y$ be a tacnode which contradicts our claim and suppose that $C$ and $D$ are the two branches. There are two cases:
\begin{enumerate}
	\item \label{tacn1} Neither $C_{\rm red}$ nor $D_{\rm red}$ are lines in $\P^r$;
	\item \label{tacn2} $C_{\rm red}$ or $D_{\rm red}$ is a line in $\P^r$ and $p$ is a non-ordinary tacnode.
\end{enumerate}

Case \eqref{tacn1} follows from \cite[Lemma 2.2]{Sch} and \cite[Lemma 7.3]{HH2}. We will include a proof for completeness.

As above, consider the normalization map $\nu$ and set $\nu^{-1}(p)=\{q_1,q_2\}$. We can choose a system of coordinates $\{x_1\ldots,x_{r+1}\}$ such that
$$
\begin{cases}
\ord_{q_j}(\nu^*(x_1))=0 \\
\ord_{q_j}(\nu^*(x_2))=1 \\
\ord_{q_j}(\nu^*(x_i))\geq 2\quad \text{ for any }\,\, 3\leq i\leq r+1, \\
\end{cases}
$$
for any $j=1,2$. Let $\rho$ the 1ps diagonalized by $\{x_1\ldots,x_{r+1}\}$ with weights $w_1=2$, $w_2=1$ and $w_3=\ldots=w_{r+1}=0$.
Since $\ord_{q_j}(\nu^*(x_i))+w_i\geq 2$ for each $j=1,2$ and $i=1,\ldots,r+1$, by Lemma \ref{L:schub} we have
$$
e_{Y,\rho}\geq 2\cdot 2^2=8.
$$
Applying Proposition \ref{prop:sumsub}\eqref{prop:sumsub1} we get
$$
e_{X,\rho}\geq e_{Y,\rho}\geq 8>\frac{2d}{r+1}\,3 \quad\text{if}\quad \frac{d}{r+1}<\frac{4}{3} \quad\bigg(\Longleftrightarrow d>2(2g-2)\bigg)
$$
and $[X\subset \P^r]$ is Chow unstable.

Suppose that case \eqref{tacn2} occurs. We can assume that $D_{\rm red}$ is a line in $\P^r$. Since the intersection multiplicity of $C_{\rm red}$ and $D_{\rm red}$ at $p$ is greater or equal than 3, if a section $s\in H^0(\P^r,\OO_{\P^r}(1))$ vanishes identically on $D_{\rm red}$, then $\ord_{p}(s_{|C_{\rm red}})\geq 3$. Therefore, there exists a system of coordinates $x_1,\ldots,x_{r+1}$ such that
$$
\begin{cases}
\ord_p({x_1}_{|C_{\rm red}}) = 0 \\
\ord_p({x_2}_{|C_{\rm red}}) = 1 \\
\ord_p({x_i}_{|C_{\rm red}}) \geq 3 \quad \text{ for any }\,\, 3\leq i\leq r+1. \\
\end{cases}
$$
Let $\rho$ be a 1ps diagonalized by $x_1,\ldots,x_{r+1}$ with weights $w_1=3$, $w_2=2$ and $w_3=\ldots=w_{r+1}=0$. 
We notice that $\ord_{p}({x_i}_{C_{\rm red}})+w_i\geq 3$ for each $i=1,\ldots,r+1$, hence by Lemma \ref{L:schub} we have
$$e_{C_{\rm red},\rho}\geq 3^2=9.$$
It is very easy to compute $e_{D_{\rm red},\rho}$ because $\{x_1^m,x_1^{m-1}x_2,\ldots,x_1x_2^{m-1},x_2^{m}\}$ is the unique monomial basis of $H^0(D_{\rm red},\OO_{D_{\rm red}}(m))$ with respect to $x_1,\ldots,x_{r+1}$: we have 
$$
W_{D_{\rm red}, \rho}(m)=\sum_{j=0}^m w_{\rho}(x_1^j x_2^{m-j})=\sum_{j=0}^m (2m+j)=\frac{5}{2}m^2+\frac{5}{2}m,
$$
hence $e_{D_{\rm red}, \rho}=5$. 
Applying Proposition \ref{prop:sumsub}\eqref{prop:sumsub1} and Proposition \ref{prop:sumsub}\eqref{prop:sumsub2} we get
$$
e_{X,\rho}\geq e_{Y_{\rm red},\rho}=e_{C_{\rm red},\rho}+e_{D_{\rm red},\rho}=9+5=14.
$$
Since $w(\rho) = 5$, we deduce
$$
e_{X}\geq 14 > \frac{2d}{r+1}\,5 \quad\text{if}\quad \frac{d}{r+1}<\frac{7}{5} \quad\bigg(\Longleftrightarrow d>\frac{7}{4}(2g-2)\bigg)
$$
and again $[X\subset \P^r]$ is Chow unstable.

$\bullet$ $H^1(X_{\text{red}},\OO_{X_{\text{red}}}(1))=0$.

   A crucial ingredient in the proof is the Clifford's theorem \cite[Thm. 7.7]{HH2} for reduced curves with nodes, cusps and tacnodes
 (generalizing the Clifford's theorem of Gieseker-Morrison for nodal curves in \cite[Thm. 0.2.3]{Gie}). 
 %From the proof of \cite[Thm. 7.7]{HH2}, we deduce the following
\begin{fact}[{\bf Clifford's theorem}]\label{F:Cliff-thm}
Let $X$ be a reduced connected curve with nodes, cusps and tacnodes and let $L$ be a line bundle on $X$ generated by global sections. Assume that $H^1(X,L)\neq 0$ and consider a non-zero
section $s\in H^0(X, \omega_X\otimes L^{-1})\cong H^1(X,L)^{\vee}$. Let $C$ be the subcurve of $X$ which is the union of all the irreducible components of $X$, where $s$ is not identically zero.
Then
\begin{equation}\label{Cliff}
h^0(C,L_{|C})\leq \frac{\deg_{C}L}{2}+1.
\end{equation}
\end{fact}
Now we proceed following an argument similar to the one used by Gieseker in the Claim of \cite[Prop. 1.0.8]{Gie}, with some modifications. Observe that, since it is obvious that $H^1(X_{\text{red}},\OO_{X_{\text{red}}}(1))=H^1(X'_{\text{red}},$ $ \OO_{X'_{\text{red}}}(1))$ and
    $H^1(Z_{\text{red}},\OO_{Z_{\text{red}}}(1))=0$, it suffices to prove that $H^1(Y_{\text{red}},$ $ \OO_{Y_{\text{red}}}(1))=0$. 
 Suppose then, by contradiction, that $H^1(Y_{\rm red},$ $\OO_{Y_{\rm red}}(1))\neq 0$: there exists a connected component $W\subset Y_{\rm red}$ such that $H^1(W,\OO_{W}(1))\neq 0$. Choose a non-zero section
$$
0\neq s\in H^0(W,\omega_{W} \otimes \OO_{W}(-1))\cong H^1(W,\OO_{W}(1))^{\vee}.
$$
Let $C$ be the subcurve of $W$ which is the union of all the
irreducible components of $W$ where $s$ is not identically zero. Fact \ref{F:Cliff-thm} implies that
$$
h^0(C,\OO_C(1))\leq \frac{\deg_C\OO(1)}{2}+1.
$$
%We notice that $s$ vanishes at the points of intersection of $C$ with the complementary subcurve $C^c$.
By the inequality \eqref{eq:basineqgrezzak0}, we obtain
$$
2\,\deg_{C}\OO(1)\leq \frac{2d}{r+1}h^0(C,\OO_{C}(1))\leq\frac{2d}{r+1}\left(\frac{\deg_C\OO(1)}{2}+1\right).
$$
Using our assumption $d>2(2g-2)$, which is equivalent to the inequality
$\displaystyle \frac{d}{r+1}<\frac{4}{3},$
we get that
$$
2\,\deg_{C}\OO(1)<\frac{4}{3}\left(\deg_C\OO(1)+2\right) \Longleftrightarrow \deg_C\OO(1)<4
\Longleftrightarrow \deg_C\OO(1)=1, \text{ } 2\text{ or }3.
$$
First, suppose that $\deg_C\OO_C(1)=1$ or 2. If $C$ is irreducible, then $C\cong \P^1$ and we get a contradiction since  $H^1(\P^1,\OO_{\P^1}(1))=0$ and $H^1(\P^1,\OO_{\P^1}(2))=0$. If $C$ is reducible, then $\deg_C\OO(1)=2$ and we can write $C=C_1\cup C_2$ where $C_1\cong C_2\cong \P^1$, $\deg_{C_1}\OO(1)=\deg_{C_2}\OO(1)=1$ (i.e. $C_1$ and $C_2$ are lines) and $|C_1\cap C_2|=1$. This gives the exact sequence
$$
0 \longrightarrow \OO_{C_1} \oplus \OO_{C_2} \longrightarrow \OO_C (1) \longrightarrow \OO_{C_1\cap C_2} \longrightarrow 0.
$$
From the exact sequence of cohomology we get that $H^1(C,\OO_C (1))=0$ and again we have a contradiction.

Now suppose that $\deg_C\OO_C(1)=3$. If $C$ is irreducible, then either $C\cong \P^1$
%(and again we get a contradiction since  $H^1(\P^1,\OO_{\P^1}(3))=0$)
or $C$ is an elliptic curve (smooth, nodal or cuspidal) in $\langle C\rangle\cong\P^2$, so we obtain that $H^1(C,\OO_{C}(1))=0$, which is absurd.
%so that again $H^1(C,\OO_{C}(1))=0$, absurd.
%If $C$ is singular and irreducible, $C$ is a an elliptic tail in $\langle C\rangle\cong\P^2$ with one node or one ordinary cusp. Consider the normalization map $\nu:\widetilde{C}\ra C$ of $C$ and let $p\in C$ be the singular point of $C$. We have the following exact sequence
%$$
%0\lra \OO_C(1)\lra v_*\OO_{\widetilde{C}} \otimes \OO_C(1) \lra \widetilde{\OO}_{C,p}/\OO_{C,p}\lra 0
%$$
%If we consider the exact sequence of cohomology, we get that $H^1(C,\OO_C (1))=0$, absurd.
Finally assume that $C$ is reducible. If $C$ has 2 irreducible components $C_1$ and $C_2$, then $C_1\cong C_2\cong \P^1$ and, up to reordering, we can assume that $\deg_{C_1}\OO(1)=1$ (i.e. $C_1$ is a line) and $\deg_{C_2}\OO(1)=2$ (i.e. $C_2$ is a conic). There are two cases: either $|C_1\cap C_2|=2$ (which happens if and only if $C_1$ and $C_2$ lie in the same plane) or $|C_1\cap C_2|=1$.
In the former case, we have the following exact sequence
$$
0 \longrightarrow \OO_{C_1}(-1) \oplus \OO_{C_2} \longrightarrow \OO_C (1) \longrightarrow \OO_{C_1\cap C_2} \longrightarrow 0.
$$
Again, using the exact sequence of cohomology, we obtain that $H^1(C,\OO_C (1))=0$, absurd.
The latter case is dealt with similarly and it is left to the reader.
If $C$ has 3 irreducible components $C_1$, $C_2$ and $C_3$, then $C_1\cong C_2\cong C_3\cong\P^1$ and $\deg_{C_1}\OO(1)=\deg_{C_2}\OO(1)=\deg_{C_3}\OO(1)=1$ (i.e $C_1$, $C_2$ and $C_3$ are lines).
There are two cases: either each of the $C_i$'s intersects all the others (which happens if and only if the $C_i$'s
lie on the same plane) or the $C_i$'s form a chain.
%Set $\{p\}=C_1\cap C_2$, $\{q\}=C_1\cap C_3$ and $\{r\}=C_2\cap C_3$.
In the former case, we have the exact sequence
$$
0 \longrightarrow \OO_{C_1}(-1) \oplus \OO_{C_2}(-1) \oplus \OO_{C_3}(-1)\longrightarrow \OO_C (1) \longrightarrow \OO_{p}\oplus \OO_{q}\oplus \OO_{r} \longrightarrow 0
$$
from which we obtain again $H^1(C,\OO_{C}(1))=0$, which is absurd. The latter case is similar and left to the reader.

$\bullet$ $X'$ is generically reduced, i.e. $Z=\emptyset$. 
%The proof uses some ideas from \cite[Prop. 1.0.2]{Gie}. 

Indeed, suppose, by contradiction, that $Z\neq\emptyset$ and let $E\subset Z$ be a connected component of $Z$. By hypothesis, $E$ is a double line. Setting $I=I(E)$, consider a primary decomposition
$$
I=J_1\cap \ldots\cap J_k,
$$
where $J_1$ is $I(E_{\rm red})$-primary. We notice that $J_1$ is uniquely determined by \cite[Thm. 6.8(iii)]{Mat} and there exists a system of coordinates $\{x_1\ldots,x_{r+1}\}$ in $\P^r$ such that
$$
J_1\subset \langle x_3,\ldots,x_r,x_{r+1}^2\rangle:=J.
$$
Denote by $E_0$ the subscheme of $E$ defined by $J$ and set $W:=E^c$, $n:=h^0(Z,\OO_Z)$ and $m:=h^0(Y,\OO_Y)$. Consider the exact sequence
\begin{equation}\label{eq:Wred}
0\lra \OO_{W_{\rm red}}\lra \OO(1)_{|W_{\rm red}}\lra \OO_{D}\lra 0,
\end{equation}
where $D$ is a divisor associated with $\OO(1)_{|W_{\rm red}}$ and having support on the smooth locus of $W_{\rm red}$. Observing that $g(E)\leq g(E_0)= 0$, it is easy to check that
$$
h^0(\OO_{W_{\rm red}})=m+n-1,\text{ } h^1(\OO_{W_{\rm red}})\geq g+1+m+n-2 \text{ }\text{ and } h^0(\OO_{D})=d-1-n,
$$
so that from the exact sequence of cohomology associated to \eqref{eq:Wred} we obtain that
\begin{eqnarray}
h^0(W_{\rm red},\OO(1)_{|W_{\rm red}})&=&h^0(W_{\rm red},\OO_{W_{\rm red}})+h^0(W_{\rm red},\OO_{D})-h^1(W_{\rm red},\OO_{W_{\rm red}})\nonumber\\
&=&d-g-n-1<d-g+1=h^0(\P^r,\OO_{\P^r}(1)).\nonumber
\end{eqnarray}
This implies that the restriction map $\pi:H^0(\P^r,\OO_{\P^r}(1))\lra H^0(W_{\rm red},\OO_{\P^r}(1)_{|W_{\rm red}})$ has kernel $K\neq 0$. Since $[X\subset \P^r]$ is Chow semistable, $X$ is non-degenerate in $\P^r$, hence there exists a non-zero section $s\in K$ which is regular on $E$. Let $\{x_1\ldots,x_{r+1}\}$ be a system of coordinates such that $x_1=s$ and
$$
E_{\text{red}}=\bigcap_{i=3}^{r+1}\{x_i=0\}.
$$
Consider a 1ps $\rho$ diagonalized by $\{x_1\ldots,x_{r+1}\}$ with weights $w_1=0$ and $w_2=\ldots=w_{r+1}=1$. It is not difficult to check that $e_{E_0,\rho}=2$ and $e_{W,\rho}=2(d-2)$. By \cite[Chap. 4, Ex. 4.49]{HM} and Proposition \ref{prop:sumsub}\eqref{prop:sumsub1} we get that
$$
e_{X,\rho}\geq e_{E_0}+e_{W,\rho}=2d-2>2d-\frac{2d}{r+1}=\frac{2d}{r+1}\,r \quad \text{if}\quad \frac{d}{r+1}>1\quad (\Longleftrightarrow g\geq 2),
$$
hence $[X\subset \P^r]$ is Chow unstable.

$\bullet$ $X$ is reduced of pure dimension 1 and $\ref{teo-pstab}(\ref{teo-pstab2})$ holds. 
%The proof uses some ideas from \cite[Prop. 1.0.8]{Gie}.

Indeed, denote by $\mathcal{I}$ the ideal sheaf of nilpotents in $\OO_{X'}$ and consider the exact sequence
\begin{equation}\label{eq:ExSeqNilp}
0 \lra \mathcal{I}\otimes \OO_{X'}(1) \lra \OO_{X'}(1) \lra \OO_{X_{\rm red}'}(1) \lra 0.
\end{equation}
From the previous steps we know that $h^1(X_{\rm red}',\OO_{X_{\rm red}'}(1)) = 0$. Moreover, since $X'$ is generically reduced, $\mathcal{I}$ has finite support, hence $h^1(X',\mathcal{I}\otimes \OO_{X'}(1))=0$. From the exact sequence of cohomology associated
to \eqref{eq:ExSeqNilp}, we deduce that $h^1(X',\OO_{X'}(1))=0$, i. e. $\OO_{X'}(1)$ in non-special, and $h^0(X',\OO_{X'}(1))=d-g_{X'}+1$, where $g_{X'}$ is the arithmetic genus of $X'$.

Now we are ready to prove that $X'$ is reduced of pure dimension 1. Firstly, notice that, by definition of $X'$, we have that $g_{X'}\geq g$ with equality if and only if $X=X'$. Since the restriction map $\pi:H^0(\P^r,\OO_{\P^r}(1))\lra H^0(X_{\rm red}',\OO_{X_{\rm red}'}(1))$ is injective, we have
\begin{eqnarray}
d-g+1 & = & h^0(\P^{r},\OO_{\P^r}(1))\leq h^0(X_{\rm red}',\OO_{X_{\rm red}'}(1)) \nonumber\\& = & h^0(X',\OO_{X'}(1))-h^0(X',\mathcal{I}\otimes \OO_{X'}(1))\nonumber\\ & = & d-g_{X'}+1-h^0(X',\mathcal{I}\otimes \OO_{X'}(1)),\nonumber
\end{eqnarray}
hence $h^0(X',\mathcal{I}\otimes \OO_{X'}(1))=0$ and $g=g_{X'}$. Since $\mathcal{I}$ has finite support, then $\mathcal{I}$ is the zero sheaf, i. e. $X'$ is reduced. Moreover $g=g_{X'}$ implies that $X=X'$ has pure dimension 1. Finally, $X\subset \P^r$ is linearly normal because the restriction map $\pi:H^0(\P^r,\OO_{\P^r}(1))\lra H^0(X,\OO_{X}(1))$ is injective between vector spaces of the same dimension, hence it is also surjective. 

$\bullet$ $\OO(1)$ is balanced, i.e. $\ref{teo-pstab}(\ref{teo-pstab3})$ holds.

First of all, observe that is enough to prove that  
\begin{equation}\label{E:half-basic}
\deg_{Z}\OO(1)\geq \frac{d}{2g-2}\deg_Z\omega_X-\frac{k_Z}{2}.
\end{equation}
 for any subcurve $Z\subset X$. Indeed, applying \eqref{E:half-basic} to $Z^c$ and using that  $d=\deg_Z\OO(1)+\deg_{Z^c}\OO(1)$ and $2g-2=\deg_Z\omega_X+\deg_{Z^c}\omega_X$, we get 
\begin{equation*}
\deg_Z \OO(1)=d-\deg_{Z^c}\OO(1)\leq d-\frac{d}{2g-2}\deg_{Z^c}\omega_{X}+\frac{k_{Z^c}}{2}= \frac{d}{2g-2}\deg_Z\omega_X+\frac{k_Z}{2},
\end{equation*}
which combined with \eqref{E:half-basic} gives that $\OO(1)$ satisfies the basic inequality \eqref{E:basineq-multideg}.

Inspired by \cite[Prop. 1.0.7 and Prop. 1.0.10]{Gie}, we will divide the proof of \eqref{E:half-basic} in two steps.

\un{Step I:} \eqref{E:half-basic} holds for the subcurves $Z\subseteq X$ for which $2\,\deg_{Y_j}\OO(1) \geq |Z\cap Y_j|$, where $\{Y_1,\ldots, Y_n\}$ are the irreducible components of $Z^c$. 

Indeed, set $k_j=|Z\cap Y_j|$ and let $\rho$ be the 1ps defined as in the proof of the inequality \eqref{eq:basineqgrezzak0}. By Lemma \ref{lem:2wdeg} we have $e_{Z,\rho}=2\deg_Z\OO(1)$. Now we will estimate $e_{Y_j,\rho}$ for each $j$. With the notation of the proof of \eqref{eq:basineqgrezzak0}, there are two cases:
\begin{enumerate}
  \item \label{basineqgrezza1} Each section in $K$ vanishes identically on $Y_j$.
	\item \label{basineqgrezza2} There exists a section $x_i\in K$ that does not vanish identically on $Y_j$.
\end{enumerate}
The case \eqref{basineqgrezza1} is an easy application of Lemma \ref{lem:2wdeg}: we get $e_{Y_j,\rho}=2\,\deg_{Y_j}\OO(1)$.

Now suppose that case \eqref{basineqgrezza2} occurs. Denote by $\{q_1,\ldots,q_{k_j}\}$ the inverse image via the normalization map $\nu$ of the set $Z\cap Y_j$. We notice that $\ord_{q}(x_i)+w_i\geq 1$ for each $q\in\{q_1,\ldots,q_{k_j}\}$ and $i=1,\ldots,r+1$, hence applying Lemma \ref{L:schub} we get $e_{Y_j,\rho}\geq k_j$.

Observe that in both cases $e_{Y_j,\rho}\geq k_j$, by our additional hypothesis. Applying Proposition \ref{prop:sumsub}\eqref{prop:sumsub2} we obtain
$$
e_{X,\rho}=e_{Z,\rho}+\sum_{j=1}^{n}e_{Y_j,\rho}\geq 2\,\deg_Z\OO(1)+\sum_{j=1}^{n}k_j=2\,\deg_Z\OO(1)+k_Z.
$$
Since $[X\subset \P^r]$ is Chow semistable and $w(\rho)\leq h^0(Z,\OO_{Z}(1))$, we deduce that 
\begin{equation}\label{eq:basineqgrezza}
2\,\deg_{Z}\OO(1)+k_Z\leq \frac{2d}{r+1}\,h^0(Z,\OO_{Z}(1)).
\end{equation}
Now, the line bundle $\OO(1)$ is non-special on $X$ by what proved before; this implies that also the restriction $\OO_Z(1)$ is non special on $Z$, so that 
\begin{equation}\label{E:h0-deg}
h^0(Z,\OO_{Z}(1))= \deg_Z \OO(1)+1-g_Z.
\end{equation}
Substituting \eqref{E:h0-deg} into \eqref{eq:basineqgrezza} and using that $r+1=d-g+1$, we get 
\begin{equation*}
\deg_{Z}\OO(1)\geq \frac{d}{2g-2}[2g_Z-2+k_Z]-\frac{k_Z}{2}=\frac{d}{2g-2}\deg_Z\omega_X-\frac{k_Z}{2},
\end{equation*}
which concludes the proof of Step I.

\un{Step II:} \eqref{E:half-basic} holds for any subcurve $Z\subseteq X$.

By contradiction, suppose that \eqref{E:half-basic} does not hold for some (proper) subcurve of $X$. Take a proper subcurve $W\subset X$  for which \eqref{E:half-basic} does not hold, i.e. such that 
\begin{equation}\label{E:formu1}
\deg_{W}\OO(1)< \frac{d}{2g-2}\deg_W\omega_X-\frac{k_W}{2},
\end{equation}
and which is maximal with this property, i.e. \eqref{E:half-basic} holds  for any subcurve $Z\supsetneq W$. By Step I, there should exist an irreducible component $Y$ of $W^c$ such that 
\begin{equation}\label{E:formu2}
2 \deg_Y\OO(1)<|W\cap Y|. 
\end{equation}
Since $\deg_Y \OO(1)\geq 1$ and $|Y\cap W|\leq k_Y$, the above inequality implies that 
\begin{equation}\label{E:formu3}
k_Y\geq 3.
\end{equation}
The subcurve $W\cup Y$ contains $W$ and it is not equal to $W$; therefore, by maximality of $W$, \eqref{E:half-basic} holds for $W\cup Y$, i.e. 
 \begin{equation}\label{E:formu4}
\deg_{W}\OO(1)+\deg_{Y}\OO(1)= \deg_{W\cup Y}\OO(1)\geq \frac{d}{2g-2}\deg_{W\cup Y}\omega_X-\frac{k_{W\cup Y}}{2}=
\end{equation}
$$= \frac{d}{2g-2}[\deg_{W}\omega_X+\deg_{Y}\omega_X]-\frac{k_{W}+k_Y}{2}+|W\cap Y|,
$$
where we used the formula $k_{W\cup Y}=k_W+k_Y-2|W\cap Y|$. Substituting \eqref{E:formu1}� and \eqref{E:formu2} into inequality \eqref{E:formu4}, we get 
\begin{equation}\label{E:formu5}
\frac{|W\cap Y|}{2}>\frac{d}{2g-2}\deg_{Y}\omega_X-\frac{k_Y}{2}+|W\cap Y|=\frac{d}{2g-2}[2g_Y-2+k_Y] -\frac{k_Y}{2}+|W\cap Y|.
\end{equation}�
Using that $g_Y\geq 0$ since $Y$ is connected and that $d>2(2g-2)$ by assumption, we get 
$$\frac{|W\cap Y|}{2}>2(k_Y-2)-\frac{k_Y}{2}+|W\cap Y|=\frac{3}{2}k_Y-4+|W\cap Y|.$$
Using the obvious fact that $|W\cap Y|\geq 0$, we deduce that $8>3k_Y$ which contradicts \eqref{E:formu3}.� 

%In general it can be proved that the additional hypothesis is always satisfied for any subcurve $Z$ (for a proof see \cite[Appendix]{Mum}).
%For the general case see \cite[Prop. 1.0.10]{Gie}.

\end{proof}

\begin{rmk}\label{R:potpseudo-sharp}
\noindent
%\begin{enumerate}[(i)]
%\item
The hypothesis that $d> 2(2g-2)$ in the above Theorem (\ref{teo-pstab}) is sharp: in \cite[Thm. 2.14]{HH2} it is proved that all the $2$-canonical h-stable curves in the sense of \cite[Def. 2.5, Def. 2.6]{HH2} (which in particular can have
arbitrary tacnodes and not only tacnodes with a line) belong to $\Hilb_{2(2g-2)}^s$.
%\item The conclusion of the above Theorem are of course true under the stronger condition that $X\in \Hilb_d^{ss}
%\subset \Ch^{-1}(\Chow_d^{ss})$ and $d>2(2g-2)$.
%since $\Hilb_d^{ss}\subset \Ch^{-1}(\Chow_d^{ss})$ (see Fact \ref{prop:HilbtoChow}).
%Under these stronger hypothesis on $X$, the above Theorem can be proved following the same argument as in  the %proof of \cite[Thm. 1.0.1]{Gie}.
%\end{enumerate}
\end{rmk}

\subsection{Balanced line bundles and quasi-wp-stable curves}

The aim of this subsection is to study the following

\begin{question}\label{Q:bal-quasi}
Given a pre-wp-stale curve $X$, what kind of restrictions does the existence of an \emph{ample} balanced line bundle $L$
%, i.e., an ample line bundle $L$   satisfying the basic inequality (\ref{E:basineq-multideg}),
impose on $X$?
%kind of restrictions imposes on $X$ the existence of an ample balanced line bundle
\end{question}

The following result gives an answer to the above question.

\begin{prop}\label{P:bal-quasi}
Let $X$ be a pre-wp-stable curve of genus $g\geq 2$.
%reduced connected curve with nodes, cusps and tacnodes with a line as singularities.
If there exists an ample balanced line bundle $L$ on $X$ of degree $d\geq g-1$, then
$X$ is quasi-wp-stable and $L$ is properly balanced.
%$\deg_E L=1$ for every exceptional component $E$ of $X$.
\end{prop}
\begin{proof}
Let $Z$ be a connected rational subcurve of $X$ (equivalently $Z$ is a chain of $\P^1$'s) such that $k_Z\leq 2$.  Clearly, $k_Z\geq 1$ since
$X$ is connected and $Z\neq X$ because $g\geq 2$.

If $k_Z=1$ then $\deg_Z(\omega_X)=-1$ and the basic inequality (\ref{E:basineq-multideg})
together with the hypothesis that $d\geq g-1$ gives that
$$\deg_Z(L)\leq \frac{d}{2g-2}\deg_Z(\omega_X)+\frac{k_Z}{2}=-\frac{d}{2g-2}+\frac{1}{2} \leq 0.$$
This contradicts the fact that $L$ is ample.

If $k_Z=2$ then $\deg_Z(\omega_X)=0$ and the basic inequality (\ref{E:basineq-multideg}) gives that
$$\deg_Z(L)\leq \frac{d}{2g-2}\deg_Z(\omega_X)+\frac{k_Z}{2}=1.$$
Since $L$ is ample, it has positive degree on each irreducible component of $Z$; therefore, $Z$
must be irreducible  which implies that $Z\cong \P^1$ and $\deg_Z L=1$.
\end{proof}

Combining the previous Proposition \ref{P:bal-quasi} with the potential stability Theorem (see Fact \ref{F:GM2})
and the Potential pseudo-stability Theorem \ref{teo-pstab}, we get the following

\begin{coro}\label{C:quasi-wp-stable}
\noindent
\begin{enumerate}[(i)]
\item \label{C:quasi-wp-stable1} If $d>2(2g-2)$ and $[X\subset \P^r]\in \Ch^{-1}(\Chow_d^{ss})\subset \Hilb_d$ with $X$ connected then $X$ is a quasi-wp-stable curve and $\OO_X(1)$
is properly balanced.
\item \label{C:quasi-wp-stable2} If $d>4(2g-2)$ and $[X\subset \P^r]\in \Ch^{-1}(\Chow_d^{ss})\subset \Hilb_d$ with $X$ connected then $X$ is a quasi-stable curve  and $\OO_X(1)$
is properly balanced.
\end{enumerate}
\end{coro}

%Note that part \eqref{C:quasi-wp-stable2} of the above Corollary \ref{C:quasi-wp-stable} was obtained in \cite[Prop. 1.0.12]{Gie}
%(see also \cite[Sec. 3.3]{Cap}) for $d\geq 10(2g-2)$.

Note that, by Proposition \ref{nefample}\eqref{ample} of the Appendix, we have the following Remark, which can be seen as a partial converse to Proposition \ref{P:bal-quasi}.

\begin{rmk}\label{R:propbal-ample}
A balanced line bundle of degree $d>\frac{3}{2}(2g-2)$ on a quasi-wp-stable curve $X$ is properly balanced if and only
if it is ample.
Therefore, for $d>\frac{3}{2}(2g-2)$, the set $B_X^d$ is the set of all the multidegrees of ample balanced line bundles on $X$.
\end{rmk}

\section{Stabilizer subgroups}\label{S:automo}

Let $[X\subset \P^r]$ be a Chow semistable point of $\Hilb_d$ with $X$ connected and  $d>2(2g-2)$.
Note that $X$ is a quasi-wp-stable curve by Corollary \ref{C:quasi-wp-stable}\eqref{C:quasi-wp-stable1}, $L:=\cO_X(1)$ is balanced
and $X$ is non-degenerate and linearly normal in $\P^r$ by the Potential pseudo-stability Theorem \ref{teo-pstab}.

%Moreover:
%\begin{itemize}
%\item If $4(2g-2)<d$ then $X$ is quasi-stable by Corollary \ref{C:quasi-wp-stable}\eqref{C:quasi-wp-stable2};
%\item If $2(2g-2)<d<\frac{7}{2}(2g-2)$ then $X$ is quasi-p-stable by Corollary \ref{C:quasi-p-stable}.
%\end{itemize}

The aim of this section is to describe the stabilizer subgroup of an element $[X\subset \P^r]\in \Hilb_d$ as above.
We denote by ${\rm Stab}_{\GL_{r+1}}([X\subset \P^r])$ the stabilizer subgroup
of $[X\subset \P^r]$ in $\GL_{r+1}$, i.e. the subgroup of $\GL_{r+1}$ fixing $[X\subset \P^r]$. Similarly, ${\rm Stab}_{\PGL_{r+1}}([X\subset \P^r])$ is the stabilizer
subgroup of $[X\subset \P^r]$ in $\PGL_{r+1}$. Clearly, ${\rm Stab}_{\PGL_{r+1}}([X\subset \P^r])={\rm Stab}_{\GL_{r+1}}([X\subset \P^r])/\Gm$, where $\Gm$ denotes the diagonal subgroup of $\GL_{r+1}$ which clearly belongs
to ${\rm Stab}_{\GL_{r+1}}([X\subset \P^r])$.

%Finally, recall that the group $\GL_{r+1}$ acts on $\Hilb_d$ via the projection $\GL_{r+1}\to \PGL_{r+1}$; given $[X\subset \P^r]\in \Hilb_d$,

It turns out that the stabilizer subgroup of $[X\subset \P^r]\in \Hilb_d$ is related to the automorphism group of the pair $(X,\OO_X(1))$, which is defined as follows.

%As described in \cite[Section 2.12]{BFV},
Given a variety $X$ and a line bundle $L$ on $X$, an automorphism of $(X, L)$ is given by a pair $(\sigma, \psi)$ such that $\sigma\in \Aut(X)$ and
$\psi$ is an isomorphism between the line bundles $L$ and $\sigma^*(L)$. The group of automorphisms of $(X,L)$  is naturally an algebraic group denoted
by $\Aut(X,L)$. We get a natural forgetful homomorphism
\begin{equation}\label{homo-auto}
\begin{aligned}
F: \Aut(X,L)&\to \Aut(X)\\
(\sigma,\psi)& \mapsto \sigma,\\
\end{aligned}
\end{equation}
whose kernel is the  multiplicative group $\Gm$,  acting
 as fiberwise multiplication on $L$, and whose image is the subgroup of $\Aut(X)$ consisting of automorphisms $\sigma$
 such that $\sigma^*(L)\cong L$.
%The automorphisms belonging to $\Gm\subseteq \Aut(C,L)$  are called  {\it scalar} automorphisms.
The quotient $\Aut(X,L)/\Gm$ is denoted by $\ov{\Aut(X,L)}$
and is called the {\it reduced} automorphism group of $(X, L)$.

The relation between the stabilizer subgroup of an embedded variety $X\subset \P^{r}$ and the automorphism group of the pair $(X,\OO_X(1))$ is provided by the following (probably well-known)
result.

\begin{lemma}\label{L:aut-stab}
Given a projective embedded variety $X\subset \P^r$ which is non-degenerate and linearly normal, there are isomorphisms
of algebraic groups
$$\begin{sis}
& {\rm Aut}(X,\OO_X(1))\cong {\rm Stab}_{\GL_{r+1}}([X\subset \P^r]), \\
& \ov{\Aut(X,\OO_X(1))}\cong  {\rm Stab}_{\PGL_{r+1}}([X\subset \P^r]).
\end{sis}$$
\end{lemma}
\begin{proof}
%This result is certainly well-known to experts. However, since we could not find a suitable reference, we sketch a proof for the reader's convenience.

Observe first that the natural restriction map $H^0(\P^r,\OO_{\P^r}(1))\to H^0(X,\OO_X(1))$ is an isomorphism because by assumption the embedding $X\subset \P^r$
is non-degenerate and linearly normal.
Therefore, we identify the above two vector spaces and we denote them by $V$. Note that $\P^r=\P(V^{\vee})$ and that
the standard coordinates on $\P^r$ induce a basis of $V$, which we call the standard basis of $V$.

Let us now define a homomorphism
\begin{equation}\label{E:isom1}
\eta:\Aut(X, \OO_X(1))\to \Stab_{\GL_{r+1}}([X\subset \P^r])\subseteq \GL_{r+1}=\GL(V^{\vee}).
\end{equation}
Given $(\sigma, \psi)\in \Aut(X,\OO_X(1))$, where $\sigma\in \Aut(X)$ and $\psi$ is an isomorphism between $\OO_X(1)$ and $\sigma^*\OO_X(1)$, we define $\eta((\sigma,\psi))\in
 \GL(V^{\vee})$ as the composition
$$\eta((\sigma,\psi)): V^{\vee}=H^0(X, \OO_X(1))^{\vee}\xrightarrow[\cong]{\widehat{\psi^{-1}}} H^0(X, \sigma^{*}\OO_X(1))^{\vee}\xrightarrow[\cong]{\widehat{\sigma^*}} H^0(X,\OO_X(1))^{\vee}=V^{\vee},
$$
where $\widehat{\psi^{-1}}$ is the dual of the isomorphism induced by $\psi^{-1}$ and $\widehat{\sigma^*}$ is the dual of the isomorphism induced by $\sigma^*$. Let us denote by $\phi_{|\OO_X(1)|}$ (resp. $\phi_{|\sigma^*\OO_X(1)|}$) the embedding
of $X$ in ${\mathbb P}^r$ given by the complete linear series $|\OO_X(1)|$ (resp. by $|\sigma^*\OO_X(1)|$)
with respect to the basis of $H^0(X, \OO_X(1))$ (resp. $H^0(X, \sigma^*\OO_X(1))$) induced by the standard basis of $V$ via the above isomorphisms. By construction, the following diagram commutes:
\begin{equation}
\label{autom}
\xymatrix{
X  \ar@{^{(}->}[rrr]^{\phi_{{|\OO_X(1)|}}} \ar[dd]^{\sigma}\ar@{_{(}->}[rrrd]_{\phi_{{|\sigma^*\OO_X(1)|}}}&&& \P(H^0(X, \OO_X(1))^{\vee})
\ar[d]^{\widehat{\psi^{-1}}}\\
&&& \P(H^0(X, \sigma^{*}\OO_X(1))^{\vee})\ar[d]^{\widehat{\sigma^*}}\\
X \ar@{^{(}->}[rrr]^{\phi_{{|\OO_X(1)|}}}&&& \P(H^0(X,\OO_X(1))^{\vee}).
}
\end{equation}
Thus we get that $\eta((\sigma,\psi))$ belongs to $\Stab_{\GL_{r+1}}([X\subset \P^r])\subseteq \GL(V^{\vee})$ and $\eta$ is well-defined.

Conversely, we define a homomorphism
\begin{equation}\label{E:isom2}
\tau: \Stab_{\GL_{r+1}}([X\subset \P^r]) \to \Aut(X,L)
 \end{equation}
 as follows. An element $g\in \Stab_{\GL_{r+1}}([X\subset \P^r])\subseteq \GL_{r+1}=\GL(V^{\vee})$ will send $X$ isomorphically onto itself, and thus induces an automorphism $\sigma\in \Aut(X)$. Consider now the isomorphism
$$\w{\psi}:V=H^0(X,\OO_X(1))\xrightarrow[\cong]{\widehat{g^{-1}}} V=H^0(X,\OO_X(1)) \xrightarrow[\cong]{\sigma^*} H^0(X,\sigma^*\OO_X(1)),$$
where $\widehat{g^{-1}}$ is the dual of $g^{-1}$ and $\sigma^*$ is the isomorphism induced by $\sigma$. The isomorphism $\w{\psi}$ induces an isomorphism $\psi$
between $\OO_X(1)$ and $\sigma^*\OO_X(1)$ making the following diagram commutative
$$\xymatrix{
H^0(X, \OO_X(1))\otimes \OO_X \ar@{->>}[r] \ar[d]_{\w{\psi}} & \OO_X(1) \ar[d]^{\psi} \\
H^0(X, \sigma^*\OO_X(1))\otimes \OO_X \ar@{->>}[r] & \sigma^*\OO_X(1).  \\
}
$$
We define $\tau(g):=(\sigma, \psi)\in \Aut(X, \OO_X(1))$.

We leave to the reader the task of checking that the homomorphisms $\eta$ and $\tau$ are induced by
morphisms of algebraic groups and that they are one the inverse of the other.

The map $\eta$ sends the subgroup $\Gm\subseteq \Aut(X,\OO_X(1))$ of scalar multiplications on $\OO_X(1)$ into the diagonal subgroup $\Gm\subset \GL_{r+1}$ and therefore it induces an isomorphism $\ov{\Aut(X,\OO_X(1))}\cong \Stab_{\PGL_{r+1}}([X\subset \P^r])$.
\end{proof}

In Theorem \ref{T:auto-grp} below, we describe the connected component $\Aut(X, L)^0$ of $\Aut(X,L)$ containing the identity for the pairs we will be interested in.
By Definition \ref{D:quasi-wp-stable}, recall that for a quasi-wp-stable curve $X$ we denote by
$\exc \subset X$ the subcurve of $X$ consisting of the union of the exceptional components $E$ of $X$, i.e., the subcurves $E\subset X$
such that $E\cong \P^1$ and $k_E=2$. We denote by $\w{X}:=\exc^c$ the complementary subcurve of $\exc$ and by $\gamma(\w{X})$ the number of connected components of $\w{X}$.
Certain elliptic tails of $X$ will play a special role in what follows; see \ref{Con:tails} for the relevant  terminology on elliptic tails.

\begin{defi}\label{D:elltails-types}
Let $F$ be an irreducible elliptic tail of $X$ (i.e., an irreducible subcurve of $X$ such that $g_F=1$ and $k_F=1$) and let $p$ denote the intersection point between $F$ and the complementary subcurve $F^c$.
%Let $L$ be an ample line bundle on $X$ and denote by $d_F=\deg_F L$ the degree of $L$ on $F$.
Given an ample line bundle $L$ on $X$, we can write $L_{|F}=\OO_F((d_F-1)p+q)$, where $d_F=\deg_F L$ denotes the degree of $L$ on $F$, for a uniquely determined smooth point $q$ of $F$.
We say that $F$ is \emph{special} with respect to $L$ if $q=p$ and \emph{non-special} otherwise. We denote by $\epsilon(X,L)$ the number of cuspidal elliptic tails of $X$ that are special with respect to $L$.
\end{defi}

\begin{rmk}
If $F$ is a reducible elliptic tail of $X$ (for example, reducible nodal or tacnodal), $F$ cannot be special. Indeed, using the same notation as in Definition \ref{D:elltails-types}, if $L_{|F}=\OO_F(d_F p)$, there exists an irreducible component $E\subset F$ such that $\deg\, L_{|E}=0$, hence $L$ is not ample.
\end{rmk}

Before stating Theorem \ref{T:auto-grp}, we introduce the following notation: we denote by $\tau(X)$ the number of tacnodal elliptic tails of $X$.

\begin{thm}\label{T:auto-grp}
Let $X$ be  either a quasi-stable curve of genus $g\geq 2$ or a quasi-wp-stable curve of genus $g\geq 3$ and let $L$ be a properly balanced line bundle of degree $d\in \Z$ on $X$.
Then the connected component ${\rm Aut}(X,L)^0$ of $\Aut(X,L)$ containing the identity is isomorphic to $\Gm^{\gamma(\w{X})+\epsilon(X,L)+\tau(X)}$.
\end{thm}

\begin{proof}

Consider the wp-stable reduction $X\to \wps(X)$ of $X$ (see Proposition \ref{P:wp-stab}).
Note that since $\wps(X)=\proj\oplus_{i\geq 0}H^0(X,\omega_X^i)$, an automorphism of $X$ naturally induces an automorphism of $\wps(X)$, so by composing the homomorphism
$F$ (see \eqref{homo-auto}) with the homomorphism $\Aut(X)\to \Aut(\wps(X))$
induced by the wp-stable reduction, we get a homomorphism
\begin{equation}\label{E:map-G}
G: \Aut(X,L)\longrightarrow \Aut(\wps(X)).
\end{equation}
We will determine the connected component $\ker(G)^0$ of the kernel of $G$ and  the connected component $\Im(G)^0$ of the image of $G$ in the two claims below.

\underline{CLAIM 1}: $\ker(G)=\ker(G)^0=\Gm^{\gamma(\w{X})}.$

Recall from Proposition \ref{P:wp-stab} that the wp-stable reduction $X\to \wps(X)$ is the contraction of every exceptional component $E\cong \P^1$ of $X$ to a node or a cusp if $E\cap E^c$ consists of two nodes or one tacnode, respectively. We can factor the wp-stable reduction of $X$ as
$$X\to Y \to \wps(X),$$
where $c:X\to Y$ is obtained by contracting all the exceptional components $E$ of $X$ such that $E\cap E^c$ consists of two nodes and $Y\to \wps(X)$ is obtained by contracting all the exceptional components $E$ of $Y$ such that
$E\cap E^c$ consists of a tacnode. Now, since an automorphism of $X$ must send exceptional components of $X$ meeting the rest of $X$ in two distinct points to exceptional components of the same type, we can factor the map $G$ of \eqref{E:map-G} as
$$G:\Aut(X,L)\stackrel{G_1}{\longrightarrow}\Aut(Y)\stackrel{G_2}{\longrightarrow} \Aut(\wps(X)).$$
This gives an exact sequence
\begin{equation}\label{E:seq-ker-G}
0\to \ker(G_1)\to \ker(G) \xrightarrow{{G_1}_{|\ker(G)}} \ker(G_2).
\end{equation}
The same proof of \cite[Lemma 2.11]{BFV} applied to the contraction map $X\to Y$ gives that
\begin{equation}\label{E:ker-G_1}
\ker(G_1)=\Gm^{\gamma(\w{X})}.
\end{equation}
Using \eqref{E:seq-ker-G} and \eqref{E:ker-G_1}, Claim 1 follows if we prove that
\begin{equation}\label{E:Im-ker}
\Im(G_1)\cap \ker(G_2)=\{\id\}.
\end{equation}
In order to prove \eqref{E:Im-ker}, we need first to describe explicitly $\ker(G_2)$.
Recall that, by construction, all the exceptional components $E\cong \P^1$ of $Y$ are such that $E\cap E^c$ consists of a tacnode of $Y$ and all of them are contracted to a cusp of $\wps(X)$ by the map $Y\to \wps(X)$.
Therefore, $\ker(G_2)$ consists of all the automorphisms $\gamma\in \Aut(Y)$ such that $\gamma$ restricts to the identity on $\ov{Y\setminus \cup E}$ where the union runs over all the exceptional subcurves $E$ of $Y$.
Consider one of these exceptional components $E\subset Y$ and let
$\{p\}=E\cap E^c$. Since $p$ is a tacnode of $Y$, there is an isomorphism (see \cite[Sec. 6.2]{HH2})
$$i:T_pE\stackrel{\cong}{\longrightarrow} T_pE^c,$$
where $T_pE$ is the tangent space of $E$ at $p$ and similarly for $T_p E^c$.
Any $\gamma\in \Aut(Y)$ preserves the isomorphism $i$. If moreover $\gamma\in \ker(G_2)\subseteq \Aut(Y)$
then $\gamma$ acts trivially on the irreducible component of $E^c$ containing $p$, hence it acts trivially also on
$T_p E^c$. Therefore, the restriction of $\gamma\in \ker(G_2)$ to $E$ will be an element $\phi\in \Aut(E)$  that fixes $p$ and induces the identity on $T_pE$. Fix the identification $(E,p)\cong (\P^1,0)$ and consider the transformations in $\Aut(\P^1)=\PGL_2$ of the form
\begin{equation}\label{E:Moebius}
\phi_{\lambda}(z)=\frac{z}{\lambda z+1}\qquad \psi_{\mu}(z)=\mu z
\end{equation}
for $\lambda\in k$ and $\mu\in k^*$.
All the elements that fix $p$ and induce the identity on $T_pE$ form a subgroup of $\Aut(E)$, which is isomorphic to the additive subgroup
$\Ga$ of $\Aut(\P^1)=\PGL_2$ given by all the transformations $\phi_{\lambda}$ (for $\lambda\in k$). Conversely, every such $\phi$ extends to an automorphism of $\Aut(Y)$, which is the identity on
$E^c$ and therefore lies in $\ker(G_2)$. From this discussion, we deduce that
\begin{equation}\label{E:ker-G2}
\ker(G_2)=\prod_E \Ga,
\end{equation}
where the product runs over all the exceptional components $E$ of $Y$.

We can now prove \eqref{E:Im-ker}. Take an element $(\sigma,\psi)\in \Aut(X,L)$ such that $G_1(\sigma,\psi)\in \ker(G_2)$. Consider an exceptional component $E$ of $Y$; let $\{p\}=E\cap E^c$
and let $C$ be the irreducible component of $E^c$ containing $p$.
By \eqref{E:ker-G2} and the discussion preceding it, we get that $G_1(\sigma,\psi)_{|E}=\phi_{\lambda}$ for some $\lambda\in k$ (as in \eqref{E:Moebius}) and $G_1(\sigma,\psi)_{|C}=\id_C$.
By construction, the map $c: X\to Y$ is an isomorphism in a neighborhood of $E\subset Y$. Therefore, by abuse of notation, we identify $E$ with its inverse image via $c$, similarly for $p$, and we
call $C'$ the irreducible component of $X$ such that $\{p\}=E\cap C'$. From the above properties of $G_1(\sigma,\psi)$, we deduce that  $\sigma_{|E}=\phi_{\lambda}$ and $\sigma_{|C}=\id_C$.
Consider now $\widehat{X}\cong E \coprod E^c$ the partial normalization of $X$ at $p$ and let $\nu:\widehat{X}\to X$ be the natural map. We have an exact sequence
$$0\to \Ga\to \Pic(X)\stackrel{\nu^*}{\longrightarrow} \Pic(\widehat{X})=\Pic(E)\times \Pic(E^{c})\to 0.$$
By looking at the gluing data defining line bundles on $X$, it is easy to check  that the above automorphism $\sigma\in \Aut(X)$ acts as the identity on $\Pic(\widehat{X})$ and that it acts on $\Ga$
by sending $\mu$ in $\mu+\lambda$.
Since, by assumption, there exists an isomorphism $\psi$ between $\sigma^*(L)$ and $L$, we must have that $\lambda=0$, or in other words that $\sigma_{|E}=\phi_0=\id_E$. Since this is true for
all the exceptional components $E$ of $Y$, from \eqref{E:ker-G2} we get that $G_1(\sigma,\psi)=\id$ and \eqref{E:Im-ker} is finally proved.

\underline{CLAIM} 2: $\Im(G)^0=\Gm^{\epsilon(X,L)+\tau(X)}.$

If $X$ is quasi-stable of genus $g\geq 2$, then $\wps(X)$ is stable of genus $g\geq 2$ and if $X$ is quasi-p-stable of genus $g\geq 3$, then $\wps(X)$ is p-stable of
genus $g\geq 3$. In both cases, $\Aut(\wps(X))$ is a finite group (see \cite{DM} for stable curves and  \cite[Proof of Lemma 5.3]{Sch} for p-stable curves); hence
$\Im(G)^0=\{\id\}$ and Claim 2 is proved.

In the general case, consider the p-stable reduction $\wps(X)\to \ps(\wps(X)):=\ps(X)$ of $\wps(X)$ (see Definition \ref{D:p-stab}) and the induced map
$$
H:\Aut(\wps(X))\to \Aut(\ps(X)).
$$
As recalled before, $\Aut(\ps(X))$ is a finite group if $g\geq 3$; hence we get that
\begin{equation}\label{E:aut0ker0}
\Aut(\wps(X))^0=\ker(H)^0.
\end{equation}
The p-stable reduction $\wps(X)\to \ps(X)$ contracts all the elliptic tails of $\wps(X)$ to cusps of $\ps(X)$. This easily implies that
\begin{equation}\label{E:ker0}
\ker(H)^0=\prod_F\Aut(F,p)^0,
\end{equation}
where the product is over all the elliptic tails $F$ of $\wps(X)$, $\{p\}=F\cap F^c$ and $\Aut(F,p)^0$ is the connected component of the automorphism group of the pointed curve $(F,p)$. There are 3 possibilities for the elliptic tails of the quasi-wp-stable $\wps(X)$ according to Figure 1 below.
\begin{figure}[h]
\begin{center}
\begin{picture}(268.5,30)(0,23)
\qbezier(113,22)(121.125,28.5)(121.75,33)
\qbezier(208.5,22)(216.625,28.5)(217.25,33)
\qbezier(21.5,22)(29.625,28.5)(30.25,33)
\qbezier(121.75,33)(120.25,37.5)(125.75,37)
\qbezier(217.25,33)(215.75,37.5)(221.25,37)
\qbezier(30.25,33)(28.75,37.5)(34.25,37)
\qbezier(125.75,37)(141.375,37.125)(130.5,52.75)
\qbezier(221.25,37)(236.875,37.125)(226,52.75)
\qbezier(34.25,37)(49.875,37.125)(39,52.75)
\qbezier(119.75,47)(159.875,47.125)(155.5,58.75)
\qbezier(155.75,58.25)(152.625,65)(148,57.75)
\qbezier(148,57.75)(144,48.5)(169,49.25)
\qbezier(35.75,45.75)(46.25,53.625)(52.75,55)
\qbezier(52.75,55)(65.75,58)(74.75,63)
\qbezier(222.75,47.25)(248.5,52.75)(244.25,61.25)
\qbezier(244.25,61.25)(248.125,53.125)(268.5,59.5)
\put(73.75,55.25){\makebox(0,0)[cc]{$F$}}
\put(266.75,53){\makebox(0,0)[cc]{$F$}}
\put(168,42){\makebox(0,0)[cc]{$F$}}
\put(49,10){\makebox(0,0)[cc]{Type I}}
\put(145.75,10){\makebox(0,0)[cc]{Type II}}
\put(246.5,10){\makebox(0,0)[cc]{Type III}}
\end{picture}
%\mbox{
%\epsfig{file=git.eps,width=10cm,height=0cm,angle=0}
%}
\end{center}
\caption{All the possible elliptic tails of a wp-stable curve.}
\label{codell}
\end{figure}
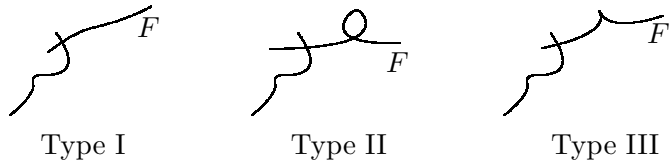

We claim that, for an elliptic tail $F$ of $\wps(X)$, the following holds
\begin{equation}\label{E:Aut0}
\Aut(F,p)^0=
\begin{cases}
\{\id\} & \text{ if $F$ is smooth or nodal (type I or II),} \\
\mathbb{G}_m & \text{ if $F$ is cuspidal (type III)}.
\end{cases}
\end{equation}
If $F$ is of type I, this follows from the well-known fact that a 1-pointed smooth curve of genus 1 has only finitely many automorphisms. If $F$ is of type II (resp. of type III), this follows from the identification of $\Aut(F,p)^0$ with the subgroup of automorphism of $F^{\nu}\cong \mathbb{P}^1$ fixing three points (resp. two points), namely the inverse image and that of $p$ and of the singular locus of $F$ via the normalization map $\nu: F^{\nu}\to F$.

Now, suppose that $F$ is an elliptic tail of type III. Obviously $F$ is the image of an elliptic tail $F'$ of X via the wp-stable reduction. Since the wp-stable reduction contracts the exceptional
subcurves of $X$, $F'$ can be chosen is such a way that $F'$ is cuspidal irreducible or tacnodal with two irreducible components. Using (\ref{E:aut0ker0}), (\ref{E:ker0}) and (\ref{E:Aut0}), Claim 2 follows if we prove that $\Aut(F,p)^0=\mathbb{G}_m \subset \Im(G)$ if and only if one of the following cases is satisfied:
\begin{enumerate}[(i)]
	\item $F'$ is cuspidal and special with respect to $L$,
	\item $F'$ is tacnodal.
\end{enumerate}
%
%$$\begin{aligned}
%& \text{Case (I): } F'\cong F \text{ and }L_{|F'}=\mathcal{O}_{F'}(d_{F'} p)\\
%& \text{Case (II): } F' \text{ is tacnodal}.\\
%\end{aligned}
%$$

If $F'$ is cuspidal, we can identify $F$ with $F'$ and clearly $\Aut(F,p)^0\subset \Im(G)$ if and only if $L_{|F}\in \Pic^{d_F}(F)$ is fixed by $\mathbb{G}_m$. Consider the $\mathbb{G}_m$-equivariant isomorphism $\rho:F_{\rm sm}\stackrel{\cong}{\longrightarrow} \Pic^{d_F}(F)$ which maps $r$ to $\mathcal{O}_F((d_F-1)p+r)$. The unique $\mathbb{G}_m$-fixed point is the point $p$, which is sent to $\mcl{O}_F(d_F p)$ by $\rho$. Therefore, $L_{|F}$ is fixed by $\Aut(F,p)^0=\mbb{G}_m$ if and only if $L_{|F}=\mathcal{O}_F(d_F p)$, or in other words when $F$ is special with respect to $L$.

Now, suppose that $F'$ is tacnodal, i.e. $F'$ is the union of two smooth rational subcurve $E_1$ and $E_2$ meeting in a tacnode. Let $\{p\}=F'\cap (F')^c$ and let $q$ be the tacnode; assume that $p\in E_2$. Consider
$$
\widehat{X}=(F')^{\nu} \coprod (F')^c=E_1 \coprod E_2 \coprod (F')^c
$$
the partial normalization of $X$, $\nu:\widehat{X}\to X$ the natural map and $\{q_1,q_2\}$ the inverse image of $q$ via $\nu$, where we assume that $q_1\in E_1$ and $q_2\in E_2$. The following holds:
$$
\Aut((F')^{\nu},q_1,q_2,\nu^{-1}(p))^0=\Aut(E_1,q_1)^o\times \Aut(E_2,q_2,\nu^{-1}(p))^0\cong(\mbb{G}_m \ltimes \mbb{G}_a) \times \mbb{G}_m.
$$
Indeed, if we fix the identifications $(E_1,q_1)\cong (\mbb{P}^1,0)$ and $(E_2,q_2,\nu^{-1}(p))\cong (\mbb{P}^1,0,\infty)$, we can consider the transformations of the form (\ref{E:Moebius}) and it is well-known that
\begin{enumerate}
	\item $\Aut(\mbb{P}^1,0)^0$ is generated by the automorphisms $\phi_{\lambda}$, $\psi_{\mu}\in \PGL_2$ for $\lambda\in k$ and $\mu\in k^*$,
	\item $\Aut(\mbb{P}^1,0,\infty)^0$ is generated by $\psi_{\mu}\in \PGL_2$ for $\mu\in k^*$.
\end{enumerate}
As explained in the proof of Claim 1, every $\gamma\in\Aut(X)$ preserves the isomorphism
$i:T_q E_1 \stackrel{\cong}{\longrightarrow} T_q E_2$, so that there is an identification of $\Aut(F',p)^0$ with the subgroup of $\Aut(E_1,q_1)\times \Aut(E_2,q_2,\nu^{-1}(p))^0$ corresponding to the elements
$(\psi_{\mu_1},\phi_{\lambda},\psi_{\mu_2})$ such that $\mu_1=\mu_2$. Hence
$$
\Aut(F',p)^0\cong\mbb{G}_m \ltimes \mbb{G}_a.
$$
Now, the wp-stable reduction $F'\to F$ induces a surjective map $\Aut(F',p)^0\to \Aut(F,p)^0$ and $(\psi_{\mu},\phi_{\lambda},\psi_{\mu})\in \Aut(F',p)^0$ is mapped to the identity
if and only if its restriction to $E_2$ is the identity, i.e. if and only if $\mu=1$.
We obtain an exact sequence
$$
0\to \Ga\to \Aut(F',p)^0\cong \mbb{G}_m \ltimes \mbb{G}_a\to \Aut(F,p)^0\cong \Gm \to 0,
$$
which allows one to identify $\Aut(F,p)^0$ with the subgroup of $\Aut(F',p)$ consisting of all the elements of the form
$(\psi_{\mu},{\rm id}, \psi_{\mu})$.
 If $L\in \Pic(X)$, for any such $\gamma=(\psi_{\mu},\id,\psi_{\mu})\in \Aut(F',p)^0\subset \Aut(X)$ we have that $\gamma^*L\cong L$ since, given the exact sequence
$$
0\to \Ga\to \Pic(X)\stackrel{\nu^*}{\longrightarrow} \Pic(\widehat{X})=\Pic(E_1)\times \Pic(E_2)\times \Pic((F')^c)\to 0,
$$
the automorphism $\gamma$ acts as the identity both on $\Pic(\widehat{X})\cong \mbb{Z}^2 \times \Pic((F')^c)$ and on the gluing data $\Ga$. Hence $\Aut(F,p)^0\subset \Im(G)$ and the claim 2 is completely proven.

\end{proof}

%By combining Corollary \ref{C:quasi-wp-stable}\eqref{C:quasi-wp-stable2}, Corollary \ref{C:quasi-p-stable}, Lemma \ref{L:aut-stab} and Theorem \ref{T:auto-grp}, we get the following

%\begin{coro}\label{C:aut-grp}
%Let $[X\subset \P^r]\in \Ch^{-1}(\Chow_d^{ss})\subset \Hilb_d$ with $X$ connected and assume that either $d>4(2g-2)$ or $2(2g-2)<d<\frac{7}{2}(2g-2)$ and $g\geq 3$.
%Then the connected component $\Stab_{\GL_{r+1}}([X\subset \P^r])^0$ of $\Stab_{\GL_{r+1}}([X\subset \P^r])$ containing the identity is isomorphic to $\Gm^{\gamma(\w{X})}$.
%\end{coro}

\section{Behavior at the extremes of the basic inequality}\label{S:extreme-inequ}

Recall from Corollary \ref{C:quasi-wp-stable}\eqref{C:quasi-wp-stable1} that if $[X\subset \P^r]\in \Hilb_d$ is Chow semistable  with $X$ connected and $d>2(2g-2)$,
then $X$ is quasi-wp-stable and $\OO_X(1)$ is properly balanced.

The aim of this section is to investigate the properties of the Chow semistable points $[X\subset \P^r]\in \Hilb_d$ such that $\OO_X(1)$ is stably balanced or strictly balanced
(see Definition \ref{D:bal-lb}).

%There is an equivalence relation of the set of balanced line bundles on a quasi-stable curve $C$.

%\begin{defi}\label{equiv-rel}
%Given two balanced line bundles $L$ and $L'$ on a quasi-stable curve $C$, we say that $L$ and $L'$ are {\rm %equivalent}, and we write $(C,L)\equiv (C,L')$, if $L_{|\w{C}}\cong L'_{|\w{C}}$. The equivalence class of a pair %$(C,L)$ is denoted by $[(C,L)]$.
%\end{defi}

%Note that the above equivalence relation $\equiv$ clearly preserves the multidegree of the line bundles, hence it %preserves the condition of being strictly balanced or
%stably balanced.

Our fist result is the following

\begin{thm}\label{T:stab-bal}
If $d>2(2g-2)$ and $[X\subset \P^r]\in \Hilb_d^s\subseteq \Hilb_d$ with $X$ connected, then $\OO_X(1)$ is stably balanced.
\end{thm}
\begin{proof}
The proof uses some ideas from \cite[Prop. 1.0.7]{Gie} and \cite[Lemma 3.1]{Cap}.

%We will prove that if $X$ is a connected curve belonging to $\Ch^{-1}(\Chow_d)$ with $d>2(2g-2)$ and $\OO_X(1)$ is not stably balanced then
%$X\not\in \Hilb_d^s$.

Let $[X\subset \P^r]\in \Hilb_d^s\subseteq \Hilb_d$ with $X$ connected and assume that $d>2(2g-2)$.
By the Potential pseudo-stability Theorem \ref{teo-pstab} and Corollary \ref{C:quasi-wp-stable}\eqref{C:quasi-wp-stable1}, we get that  $X$ is a
quasi-wp-stable curve and $L:=\OO_X(1)$ is properly balanced and non-special.

By contradiction, suppose that $\OO_X(1)$ is not stably balanced. Then, by Definition \ref{D:balanced} and Remark \ref{R:bala}, we can
find a connected subcurve $Y$ with connected complementary subcurve $Y^c$ such that
\begin{equation}\label{E:init-assum}
\begin{sis}
& Y^c\not\subset \exc \: \: \text{ or equivalently } g_{Y^c}=0\Longrightarrow k_{Y^c}=k_Y\geq 3, \\
&\deg_{Y^c} L=M_Y=\frac{d}{2g-2}\deg_{Y^c} \omega_X+\frac{k_{Y^c}}{2}=\frac{d}{2g-2}(2g_{Y^c}-2+k_{Y^c})+\frac{k_{Y^c}}{2},\\
&\deg_Y L=m_Y=\frac{d}{2g-2}\deg_Y \omega_X-\frac{k_Y}{2}=\frac{d}{2g-2}(2g_Y-2+k_Y)-\frac{k_Y}{2}.\\
\end{sis}
\end{equation}

In order to produce the desired contradiction,  we will use the numerical criterion for Hilbert stability (see Fact \ref{Hilb-crit}).
Let $V:=H^0(\P^r,\OO_{\P^r}(1))=H^0(X, \OO_X(1))$ and consider the vector subspace
$$
U:=\ker \left\{H^0(\P^r,\OO_{\P^r}(1))\to H^0(Y,L_{|Y})\right\}\subseteq V.
$$
Set $N+1:=\dim U$. Choose a basis $\{x_0, \ldots, x_N, \ldots, x_r\}$ of $V$ relative to the filtration
$U\subseteq V$, i.e., $x_i\in U$ if and only if $0\leq i\leq N$. Define a 1ps $\rho$ of $\GL_{r+1}$ by
$$\rho(t)\cdot x_i=
\begin{cases}
x_i & \text{ if } 0\leq i\leq N, \\
t x_i & \text{ if } N+1\leq i\leq r.
\end{cases}
$$
We will estimate the two polynomials appearing in  Fact \ref{Hilb-crit} for the 1ps $\rho$.

First of all, the total weight $w(\rho)$ of $\rho$ satisfies $w(\rho)=r-N=\dim V-\dim U\leq h^0(Y,L_{|Y})$.
Since $L$ is non-special and $H^0(X,L)\twoheadrightarrow H^0(Y,L_{|Y})$ because $X$ is a curve, we get that $h^0(Y,L_{|Y})=\deg_Y L+1-g_Y$.
Therefore, we conclude that
\begin{equation}\label{E:RHS}
\frac{w(\rho)}{r+1}mP(m)\leq \frac{h^0(Y,L_{|Y})}{r+1}m(dm+1-g)=\frac{\deg_Y L+1-g_Y}{d+1-g}
\left[d m^2+ (1-g) m\right].
\end{equation}
In order to compute the polynomial $W_{\rho}(m)$ for $m\gg 0$, consider the filtration of $H^0(\P^r, \OO_{\P^r}(m))$:
$$0\subseteq U^m\subseteq U^{m-1}V\subseteq \ldots \subseteq U^{m-i}V^i \subseteq \ldots \subseteq V^m=H^0(\P^r, \OO_{\P^r}(m)),$$
where  $U^{m-i}V^i$ is the subspace of $H^0(\P^r, \OO_{\P^r}(m))$ generated by the monomials containing at least
$(m-i)$-terms among the variables $\{x_0,\ldots,x_N\}$. Note that for a monomial $B$ of degree $m$, it holds that
\begin{equation}\label{E:monom}
B\in U^{m-i}V^i\setminus U^{m-i+1}V^{i-1} \Longleftrightarrow w_{\rho}(B)=i.
\end{equation}
Via the surjective restriction map $\mu_m:H^0(\P^r,\OO_{\P^r}(m))\twoheadrightarrow H^0(X, L^m)$, the above filtration on $H^0(\P^r,\OO_{\P^r}(m))$ induces a filtration
$$0\subseteq F^0 \subseteq F^{1}\subseteq \ldots \subseteq F^i \subseteq \ldots \subseteq F^m=H^0(X, L^m),
$$
where $F^i:=\mu_m(U^{m-i}V^i)$. Using \eqref{E:monom}, we get that
\begin{equation}\label{E:polyW}
W_{\rho}(m)=\sum_{i=1}^m i\left[ \dim(F^i)-\dim(F^{i-1})\right]=m \dim(F^m)-\sum_{i=1}^{m-1} \dim(F^i)=
\end{equation}
$$ =m(dm+1-g)-\sum_{i=0}^{m-1} \dim(F^i).$$
It remains to estimate $\dim F^i$ for $0\leq i\leq m-1$. To this aim, consider the partial normalization
$\tau:\hat{X}\to X$ of $X$ at the nodes laying on $Y\cap Y^c$. Observe that $\hat{X}$ is the disjoint union of $Y$ and $Y^c$.
We denote by $\w{D}$ the inverse image of $Y\cap Y^c$ via $\tau$. Since $Y\cap Y^c$
consists of $k_{Y}$ nodes of $X$, $\w{D}$ is the disjoint union of $D_Y$ and $D_{Y^c}$, where
$D_Y$ consists of $k_Y$ smooth points on $Y$ and $D_{Y^c}$ consists of $k_{Y}$ smooth points on $Y^c$.
Consider now the injective pull-back morphism
$$\tau^*:H^0(X, L^m)\hookrightarrow H^0(\hat{X},\tau^*L^m)=
H^0(Y,L^m_{|Y})\oplus H^0(Y^c,L^m_{|Y^c}),
$$
which clearly coincides with the restriction maps to $Y$ and $Y^c$.

Note that if $B$ is a monomial belonging to $U^{m-i}V^{i}\subseteq H^0(\P^r, \OO_{\P^r}(m))$ for some $i\leq m-1$, then $B$ contains at least $m-i\geq 1$ variables among the $x_j$'s such that $x_j\in U$; hence the order of vanishing of $B$ along the subcurve $Y$ is at least equal to $m-i$ . This implies that any $s\in F^{i}\subseteq H^0(X, L^m)$ with $i\leq m-1$ vanishes identically on $Y$ and vanishes on the points of $D_{Y^c}$ with order at least $(m-i)$.
We deduce that
\begin{equation}\label{E:incl}
\tau^*(F^i)\subseteq H^0(Y^c, L^m_{|Y^c}((i-m)D_{Y^c})) \: \text{ for } 0\leq i\leq m-1.
\end{equation}
\underline{CLAIM:}  $H^1(Y^c, L^m_{|Y^c}((i-m)D_{Y^c}))=0 \: \text{ for } 0\leq i\leq m-1 \text{ and } m\gg 0. $

Let us prove the claim. Clearly, if the claim is true for $i=0$, then it is true for every $i>0$; so we can assume that $i=0$. According to Fact \ref{numcrit}\eqref{numcrit1} of the Appendix,
it is enough to prove that for any connected subcurve $Z\subseteq Y^c$, we have that
\begin{equation}\label{E:cond1}
\deg_Z(L_{|Z}^m(-mD_Z)) >2g_Z-2 \: \text{ for } \: m\gg 0,
\end{equation}
where $D_Z:=D_{Y^c}\cap Z$. Indeed, \eqref{E:cond1} is equivalent to
\begin{equation}\label{E:cond2}
\deg_Z L \geq  |D_Z| \text{ with strict inequality if } g_Z\geq 1.
\end{equation}
Observe that, since each point of $D_Z$ is the intersection of $Z$ with $Y=\ov{X\setminus Y^c}$ and $Z\cap \ov{Y^c\setminus Z}\neq \emptyset$ unless $Z=Y^c$ because $Y^c$ is connected, the following holds:
\begin{equation}\label{E:D-k}
|D_Z|\leq k_Z \text{ with equality if and only if } Z=Y^c,
\end{equation}
where $k_Z$ is, as usual, the length of the schematic intersection of $Z$ with the complementary subcurve $\ov{X\setminus Z}$ in $X$. In order to prove \eqref{E:cond2}, we consider
different cases.

If $g_Z\geq 1$, then using the basic inequality \eqref{E:basineq-multideg} for $L$ relative to the subcurve $Z$ and the assumption
$d>2(2g-2)$, we compute
$$ \deg_Z L\geq \frac{d}{2g-2}\deg_Z \omega_X-\frac{k_Z}{2}>2(2g_Z-2+k_Z)-\frac{k_Z}{2}\geq \frac{3k_Z}{2} \geq \frac{3|D_Z|}{2}\geq |D_Z|,
$$
which shows that \eqref{E:cond2} holds in this case.

If $g_Z=0$ and $Z=Y^c$ then,  using that $\deg_{Y^c}L =M_{Y^c}$ and $k_{Y^c}\geq 3$ by \eqref{E:init-assum},
we get
$$ \deg_{Y^c} L=M_{Y^c}= \frac{d}{2g-2}(2g_{Y^c}-2+k_{Y^c})+\frac{k_{Y^c}}{2}>2(k_{Y^c}-2)+\frac{k_{Y^c}}{2}> k_{Y^c}=|D_{Y^c}|,$$
which shows that \eqref{E:cond2} holds also in this case.

It remains to consider the case $g_Z=0$ and $Z \subsetneq Y^c$.
If  $k_Z\leq 2$ then, since $X$ is quasi-wp-stable and $Z$ is connected,
we must have that $Z$ is an exceptional component of $X$, i.e., $Z\cong \P^1$ and $k_Z=2$. By Proposition \ref{P:bal-quasi}, it follows that $\deg_Z L=1$.
Since $|D_Z|\leq 1$ by \eqref{E:D-k}, we deduce that \eqref{E:cond2} is satisfied
also in this case.  Finally, assume that $k_Z\geq 3$. Consider the subcurve $W:=Z^c\cap Y^c\subset Y^c$.
It is easy to check that
\begin{equation}\label{E:kW}
k_{Y^c}-k_W=|Z\cap Y|-|W\cap Z|=|D_Z|-(k_Z-|D_Z|)=2|D_Z|-k_Z.
\end{equation}
Using the basic inequality of $L$ with respect to $W$ together with \eqref{E:init-assum}, \eqref{E:kW} and $k_Z\geq 3$, we get
$$\deg_Z L=\deg_{Y^c}L-\deg_W L\geq \frac{d}{2g-2}\deg_{Y^c} \omega_X+\frac{k_{Y^c}}{2}-\frac{d}{2g-2}\deg_W\omega_X-\frac{k_W}{2}=
$$
$$=\frac{d}{2g-2}\deg_Z \omega_X+|D_Z|-\frac{k_Z}{2}>2(k_Z-2)+|D_Z|-\frac{k_Z}{2}>|D_Z|.
$$
The claim is now proved.

Using the claim above, we get from \eqref{E:incl} that
\begin{equation}\label{E:Fi}
\dim F^i=\dim \tau^*(F^i)\leq m \deg_{Y^c}L+(i-m)k_Y+1-g_{Y^c} \: \text{ for } 0\leq i\leq m-1 \text{ and } m\gg 0.
\end{equation}
Combining \eqref{E:Fi} and \eqref{E:polyW}, we get that
$$W_{\rho}(m)\geq  m(dm+1-g)-\sum_{i=0}^{m-1} \left[ m \deg_{Y^c}L+(i-m)k_Y+1-g_{Y^c} \right]=
$$
$$= m(dm+1-g)-m \left[ m \deg_{Y^c}L-mk_Y+1-g_{Y^c} \right]-k_Y \frac{m(m-1)}{2}=$$
\begin{equation}\label{E:LHS}
= m^2\left[\deg_Y L+\frac{k_Y}{2} \right]+ m\left[1-g_Y-\frac{k_Y}{2} \right],
\end{equation}
where in the last equality we have used $d=\deg L=\deg_YL+\deg_{Y^c}L$ and $g=g_{Y}+g_{Y^c}+k_Y-1$.

Using that  $\deg_YL=m_Y$ by \eqref{E:init-assum}, we easily check that
\begin{equation}\label{E:1case}
\deg_YL+\frac{k_Y}{2}=d\: \frac{\deg_YL+1-g_Y}{d+1-g}
\end{equation}
and
\begin{equation}\label{E:2case}
1-g_Y-\frac{k_Y}{2} = (1-g)\: \frac{\deg_YL+1-g_Y}{d+1-g}.
\end{equation}
By combining \eqref{E:RHS}, \eqref{E:LHS}, \eqref{E:1case}, \eqref{E:2case}, we get for $m\gg 0$:
\begin{equation}\label{E:inequa}
W_{\rho}(m)\geq m^2\left[\deg_Y L+\frac{k_Y}{2} \right]+ m\left[1-g_Y-\frac{k_Y}{2} \right] =
\end{equation}
$$
= \frac{\deg_Y L+1-g_Y}{d+1-g}\left[d m^2+ (1-g) m\right]\geq \frac{w(\rho)}{r+1}mP(m),
$$
which contradicts the numerical criterion for Hilbert stability (see Fact \ref{Hilb-crit}).

\end{proof}

 \subsection{Closure of orbits}

Given a point $[X\subset \P^r]\in \Hilb_d$, denote by $\Orb([X\subset \P^r])$ the orbit of $[X\subset \P^r]$ under the action of $\SL(V)=\SL_{r+1}$. Clearly, $\Orb([X\subset \P^r])$ depends only on $X$ and on the line bundle $L:=\OO_X(1)$ and not on the chosen embedding $X\subset \P^r$.

The aim of this subsection is to investigate the following

 \begin{question}\label{Q:clos-orb}
 Given two points $[X\subset \P^r]$, $[X'\subset \P^r] \in \Ch^{-1}(\Chow_d^{ss})$ with $X$ and $X'$ connected, when does it hold that
 $$[X'\subset \P^r]\in \ov{\Orb([X\subset \P^r])}? $$
 \end{question}

%In order to partially answer to the above Question,
We start by introducing an order relation on the set of pairs $(X,L)$ where $X$ is a quasi-wp-stable curve and $L$ is a properly balanced line bundle on $X$ of degree $d$.

\begin{defi}\label{D:speciali}
Let $(X',L')$ and $(X,L)$ be two pairs consisting of a quasi-wp-stable curve together with a properly balanced line bundle of degree $d$ on it.
\begin{enumerate}[(i)]
\item  \label{D:speciali1} We say that $(X',L')$ is an \emph{elementary isotrivial specialization} of $(X,L)$, and we write $(X,L)\spel (X',L')$,  if there exists a proper connected subcurve $Z\subset X'$ with $\deg_{Z}L'=m_{Z}$, $Z^c$ connected and $Z\cap Z^c\subseteq \exc'$
such that $(X,L)$ is obtained from $(X',L')$ by smoothing some nodes of $Z\cap Z^c$, i.e., there exists a smooth pointed  curve $(B,b_0)$ and a flat projective morphism $\cX\to B$ together with a line bundle $\cL$ on $\cX$ such
that $(\cX,\cL)_{b_0}\cong (X',L')$  and $(\cX,\cL)_{b}\cong (X,L)$ for every $b_0\neq b\in B$.

%    if there exists a map $\sigma: X'\to X$ that contracts an exceptional
%subcurve $\P^1\cong E\subset X'$ to a node $n$ of $X$ and is an isomorphism outside $E$ (we say that $X'$ is %obtained from $X$ by \emph{blowing-up} the node $n$) so that
%\begin{itemize}
%\item \label{D:speciali1} $n\in Y\cap Y^c\setminus \exc$ for a subcurve $Y\subset X$ such that $\deg_Y L=M_Y$ (we %call such a node $n$ of $X$ \emph{unstable } with respect to $L$);
%\item \label{D:speciali2} $(\sigma^*L)_{|E^c}=L'_{|E^c}(-p)$ where $p$ is given by the intersection of $E$ with the %strict transform $\ov{\sigma^{-1}(Y)\setminus E}$ of $Y$.
%\end{itemize}
\item \label{D:speciali2} We say that $(X',L')$ is an \emph{isotrivial specialization} of $(X,L)$, and we write $(X,L)\sp (X',L')$ if
%either $(X',L')=(X,L)$ or
$(X',L')$ is obtained from $(X,L)$ via a sequence of elementary isotrivial specializations.
\end{enumerate}
\end{defi}

There is a close relation between the existence of isotrivial specializations and strictly balanced line bundles,
as explained in the following

\begin{lemma}\label{L:specia}
Notation as in Definition \ref{D:speciali}.
\begin{enumerate}[(i)]
\item \label{L:specia1} If $(X, L)\sp (X',L')$ then $L$ is not strictly balanced.
\item \label{L:specia2} If $L$ is not strictly balanced, then there exists an isotrivial specialization $(X,L)\sp
    (X',L')$ such that $L'$ is strictly balanced.
\end{enumerate}
\end{lemma}
\begin{proof}
Part \eqref{L:specia1}: clearly, it is enough to consider the case where $(X, L)\spel (X',L')$ is an elementary isotrivial specialization as in Definition \ref{D:speciali}\eqref{D:speciali1}. For $Z\subseteq X'$ as in Definition \ref{D:speciali}\eqref{D:speciali1}, decompose $Z^c$ as the union of all the exceptional components $\{E_i\}_{i=1,\cdots,k_Z}$ of $X'$ that meet $Z$ and a subcurve $W$. By applying Remark \ref{R:bala}(i) to the subcurve $E_1\cup\dots\cup E_{k_Z}$, where the basic inequality achieves its maximal value, it is easy to see that $\deg_W L'=m_{W}$. Let now $\w{W}$ be the subcurve of $X$ given by the union of the irreducible components of $X$ that specialize to an irreducible component of $W\subset X'$. Since $(X, L)$ is obtained from
$(X',L')$ by smoothing some nodes which belong to $Z\cap \cup_i E_i$ and therefore are not in $W$,
we clearly have that $\w{W}\cong W$, $k_{\w{W}}=k_W$ and $L_{\w{W}}\cong L'_W$.
Hence $\deg_{\w{W}}L'=m_{\w{W}}$ and, since
$\w{W}\cap \w{W}^c\not\subset \exc$, we conclude that $L$ is not strictly balanced.

Part \eqref{L:specia2}: if $L$ is not strictly balanced, we can find a subcurve $Y\subset X$ such that
$\deg_Y L=M_Y$ and $Y\cap Y^c \subsetneq \exc$. Using that $\deg_{Y}L=M_Y$, or equivalently that $\deg_{Y^c}L=m_{Y^c}$,  it is easy to check that if $n\in Y\cap Y^c \cap \exc$ then there exists a unique exceptional component $E$ of $X$ such that $n\in E\subset Y$.

Let us denote by $\{n_1,\ldots,n_r\}$ the points belonging to $Y\cap Y^c\setminus \exc$.
Let $X'$ be the bubbling of $X$ at $\{n_1,\ldots,n_r\}$ and let $E_Y:=E_1\cup\dots\cup E_r$ be the new exceptional components of $X'$. Given a subcurve $Z\subseteq X$ denote by $Z'$ the strict transform of $Z$ via the bubbling morphism and define $k^Y_{Z'}:=|Z'\cap E_Y\cap Y|$.
\setlength{\unitlength}{1.9pt}

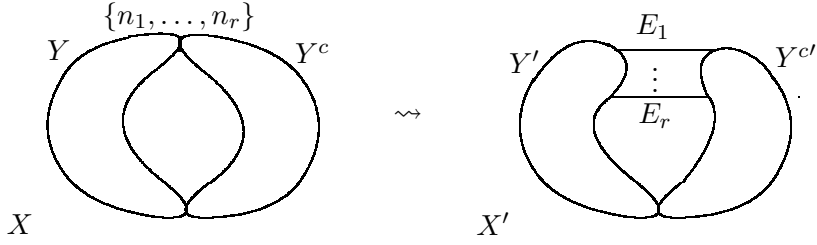
\begin{picture}(205.25,50.75)(10,124)
\qbezier(56.75,156)(43,145.875)(57.25,134.25)
\qbezier(67.25,133.25)(81,143.375)(66.75,155)
\qbezier(57.25,134.25)(68.25,125.625)(57.25,126.5)
\qbezier(151.25,134.25)(162.25,125.625)(151.25,126.5)
\qbezier(66.75,155)(55.75,163.625)(66.75,162.75)
\qbezier(57.25,126.5)(43.625,128.25)(38.5,135)
\qbezier(151.25,126.5)(137.625,128.25)(132.5,135)
\qbezier(66.75,162.75)(80.375,161)(85.5,154.25)
\qbezier(38.5,135)(31.875,144)(37.75,154)
\qbezier(132.5,135)(125.875,144)(131.75,154)
\qbezier(85.5,154.25)(92.125,145.25)(86.25,135.25)
\qbezier(179.5,154.25)(186.125,145.25)(180.25,135.25)
\qbezier(37.75,154)(41.125,160.375)(51,162.25)
\qbezier(86.25,135.25)(82.875,128.875)(73,127)
\qbezier(180.25,135.25)(176.875,128.875)(167,127)
\qbezier(57,156.25)(68.375,165.125)(51.25,162.5)
\qbezier(67,133.25)(55.625,124.375)(72.75,127)
\qbezier(161,133.25)(149.625,124.375)(166.75,127)
\thicklines
\put(61.25,160){\line(0,1){2}}
\put(61.25,162){\line(0,-1){1.75}}
\put(61.25,159.5){\line(0,1){2.5}}
%\emline(61.25,162)(61.5,160.25)
\multiput(61.25,162)(.03125,-.21875){8}{\line(0,-1){.21875}}
%\end
\put(37.25,159.75){\makebox(0,0)[cc]{$Y$}}
\put(87.75,159.25){\makebox(0,0)[cc]{$Y^c$}}
\put(61.25,166.5){\makebox(0,0)[cc]{$\{n_1,\dots,n_r\}$}}
\thinlines
\qbezier(166,151.25)(171.25,145.625)(161.5,133.5)
\qbezier(179.5,154)(176.625,159.375)(170.25,160.25)
\qbezier(170.25,160.25)(162.875,159.625)(166,151.5)
\qbezier(146.25,149.75)(139.25,144)(151.25,134.25)
\qbezier(146.5,161)(154.125,157.125)(146.25,149.75)
\put(184.5,150.5){\line(1,0){.25}}
\put(166.5,150.5){\line(-1,0){19.25}}
\put(168.25,159.75){\line(-1,0){20.5}}
\put(155.5,164.5){\makebox(0,0)[cc]{$E_1$}}
\put(156,146.75){\makebox(0,0)[cc]{$E_r$}}
\put(184,158.75){\makebox(0,0)[cc]{${Y^c}'$}}
\put(130,157.5){\makebox(0,0)[cc]{$Y'$}}
\put(106.75,146.75){\makebox(0,0)[cc]{$\sp$}}
\put(156,155.5){\makebox(0,0)[cc]{$\vdots$}}
\qbezier(131.75,153.75)(138.875,164.25)(146.5,160.75)
\put(29.5,125.25){\makebox(0,0)[cc]{$X$}}
\put(123.75,125.25){\makebox(0,0)[cc]{$X'$}}
\end{picture}

Define a multidegree $\underline d$ on $X'$ such that $\underline d_{E_i}=1$, for $i=1,\dots, r$ and, given an irreducible component $C$ of $X$, set $$\underline d_{C'}=
\deg_{C}L-k^Y_{C'}.$$
%$$\underline d_{C'}=\begin{cases}
%\deg_{C}L-k^Y_{C'} \mbox{ if } n\in C \mbox { and } C\subseteq Y \\
%\deg_C L \mbox{ otherwise.}
%\end{cases}
%$$
From  \cite[Important Remark 5.1.1]{Cap} we know that there is a flat and proper family $\mathcal X \to B$ over a pointed curve $(B,b_0)$ and a line bundle $\mathcal L$ over $\mathcal X$ such that $(\mathcal X_b,\mathcal L_{|{\mathcal X_b}})\cong(X,L)$ for $b\neq b_0$ and $(\mathcal X_{b_0},\mathcal L_{|\mathcal X_{b_0}})\cong(X',L')$ where $X'$ is the bubbling of $X$ at $\{n_1,\dots, n_r\}$ and $\underline\deg L'=\underline d$.

Let us check that $L'$ is properly balanced.
It is clear that the degree of $L'$  is equal to $1$ on all the exceptional components of $X'$.
%Denote by $Z'$ the strict transform of a subcurve $Z$ of $X$.
Let $W\subseteq X'$ and let us check that $L'$ satisfies the basic inequality \eqref{E:basineq-multideg}. Start by assuming that  $W=Z'$ for some $Z\subseteq Y$. Then we have that
\begin{equation}\label{basineqZ}
\deg_{Z'}L'=\deg_ZL-k^Y_{Z'}=\deg_YL-\deg_{\ov{Y\setminus Z}}L-k^Y_{Z'}=M_Y-\deg_{\ov{Y\setminus Z}}L-k^Y_{Z'} \geq  M_Y-M_{\ov{Y\setminus Z}}-k^Y_{Z'}
\end{equation}
$$ =M_Z-|Z\cap\ov{Y\setminus Z}|-k^Y_{Z'}=M_Z-k_Z+|Z'\cap {Y^c}'|=m_Z+|Z'\cap {Y^c}'|.$$
Suppose now that $W=Z_{Y^c}'\cup Z_Y'\cup E_W$ where $Z_{Y^c}\subseteq Y^c$, $Z_Y\subseteq Y$ and $E_W\subseteq E_Y$. Then,
$\deg_WL'={\deg}_{Z'_{Y^c}}L'+\deg_{Z'_Y}L'+|E_W|$ and,
by (\ref{basineqZ}), it follows that
$$\deg_WL'={\deg}_{Z_{Y^c}}L+m_Z+|Z'_Y\cap{Y^c}'|+|E_W|\geq$$ $$\frac{d\omega_W}{2g-2}-\frac{k_{Z_{Y^c}}} 2-\frac{k_{Z_Y}}2+|Z'_Y\cap{Y^c}'|+|E_W|
=m_W+|E_W|-|E_W\cap Z'_Y\cap Z'_{Y^c}|\geq m_W$$
Analogously, we can show that $\deg_WL'\leq M_W$, so we conclude that $L'$ is properly balanced.

Now, if $L'$ is strictly balanced we are done. If not, we repeat the same procedure and, after a finite number of steps, we will find the desired pair $(X'',L'')$ with $L''$ strictly balanced.

\end{proof}

We can now give a partial answer to Question \ref{Q:clos-orb}.

\begin{thm}\label{T:spec-clos-orb}
Let $[X\subset \P^r], [X'\subset \P^r] \in \Hilb_d$ and assume that $X$ and $X'$ are quasi-wp-stable curves
and $\OO_X(1)$ and $\OO_{X'}(1)$ are properly balanced and non-special. Suppose that
$(X,\OO_X(1))\sp (X', \OO_{X'}(1))$.  Then
\begin{enumerate}[(i)]
\item \label{T:spec-clos-orb1} $[X'\subset \P^r]\in \ov{\Orb([X\subset \P^r])}. $
\item \label{T:spec-clos-orb2} $[X \subset \P^r]\in \Ch^{-1}(\Chow_d^{ss})$ (resp. $\Hilb_d^{ss}$) if and only if $[X'\subset \P^r]\in \Ch^{-1}(\Chow_d^{ss})$ (resp. $\Hilb_d^{ss}$).
\end{enumerate}
\end{thm}
\begin{proof}
It is enough, in view of  Fact \ref{F:weight-basin}, to find a 1ps $\rho:\Gm\to \GL_{r+1}$ that stabilizes
$[X'\subset \P^r]$ and such that $\mu([X'\subset \P^r]_m,\rho)\leq 0$ for $m\gg 0$ and $[X\subset \P^r]\in A_{\rho}([X'\subset \P^r])$.

We can clearly assume that $(X,\OO_X(1))\spel (X', \OO_{X'}(1))$. Using the notation of Definition \ref{D:speciali}\eqref{D:speciali1}, this means that there exists a connected subcurve $Z\subset X'$ with $Z^c$ connected and $Z\cap Z^c\subset \exc'$ and $\deg_Z L'=m_{Z}$ such that $(X,\OO_X(1))$ is obtained from $(X',\OO_{X'}(1))$
by smoothing some of the nodes of $Z\cap Z^c$.
Moreover, we can decompose the connected complementary subcurve $Z^c$ as
$$Z^c=\bigcup_{1\leq i\leq k_Z} E_i\cup W,$$
where the $E_i$'s are the exceptional subcurves of $X'$ that meet the subcurve $Z$ and $W:=\ov{Z^c\setminus \cup_i E_i}$ is clearly connected as well. Since $\deg_{E_i}L'=1$, it follows from Remark \ref{R:bala} applied to the subcurve $E_1\cup\dots\cup E_{k_Z}$ that $\deg_W L'=m_W$.

The required 1ps $\rho$ of $\GL_{r+1}$ is similar to the 1ps
considered in the proof of Theorem \ref{T:stab-bal}.
%let $V:=H^0(\P^r,\OO_{\P^r}(1))=H^0(X', \OO_{X'}(1))$ (since $\OO_{X'}(1)$ is non-special and of degree $d=r+g$)
%and consider the vector subspace
%$$U:=\ker \left\{H^0(X',\OO_{X'}(1))\to H^0(W,L'_{|W})\right\}\subset V.$$
%Set $N+1:=\dim U$. Choose a basis $\{x_0, \ldots, x_N, \ldots, x_r\}$ of $V$ relative to the filtration
%$U\subseteq V$, i.e., $x_i\in U$ if and only if $0\leq i\leq N$. Define a 1ps $\rho$ of $\GL_{r+1}$ by
%$$\rho(t)\cdot x_i=
%\begin{cases}
%x_i & \text{ if } 0\leq i\leq N, \\
%t x_i & \text{ if } N+1\leq i\leq r.
%\end{cases}
%$$
%Another way to define $\rho$ is as follows.
More precisely, consider the restriction map
$${\rm res}: H^0(X',\OO_{X'}(1))\longrightarrow H^0(Z, \OO_Z(1))\oplus H^0(W, \OO_W(1)).
$$
The map ${\rm res}$ is injective because the complementary subcurve of $Z\cup W$ is made up of the exceptional components $E_i\cong \P^1$, each of which meets both $Z$ and $W$ in one point.
Moreover, since $\OO_{X'}(1)$ is non-special by assumption, which implies that also $\OO_Z(1)$ and $\OO_W(1)$
are non-special, we have that
$$\dim H^0(Z, \OO_Z(1))+ \dim H^0(W, \OO_W(1))=\deg_Z\OO_{X'}(1)-g_Z+1+\deg_W\OO_{X'}(1)-g_W+1= $$
$$=
m_Z-g_Z+1+m_W-g_W+1=d-g+1=\dim H^0(X',\OO_{X'}(1)),
$$
where we have used that $m_Z+m_W=d-k_Z$ and $g=g_W+g_Z+k_Z-1$.
This implies that ${\rm res}$ is an isomorphism.
Define now the 1ps $\rho:\Gm\to \GL_{r+1}$ so that
$$\begin{sis}
& \rho(t)_{|H^0(W,\OO_W(1))}=t \cdot \Id, \\
& \rho(t)_{|H^0(Z,\OO_Z(1))}= \Id.\\
\end{sis}
$$
%Let $\langle W\rangle$ and $\langle Z\rangle$ the linear span of
%$W$ and $Z$, respectively. Since $\OO_{X'}(1)$ is non-special, it follows that
%$$\begin{sis}
%& \dim \langle W\rangle = \deg_W \OO_{X'}(1)-g_W =m_W-g_W, \\
%& \dim \langle Z\rangle = \deg_Z \OO_{X'}(1)-g_Z =m_Z-g_Z. \\
%\end{sis}$$
Let us check that the above 1ps $\rho$ satisfies all the desired properties.

\un{CLAIM 1}: $\mu([X'\subset \P^r]_m,\rho)\leq 0$ for $m\gg 0$.

This is proved exactly as in Theorem \ref{T:stab-bal}: see \eqref{E:inequa} and the equation for $\mu([X\subset \P^r]_m,\rho)$
given in Fact \eqref{Hilb-crit}.

\un{CLAIM 2}: $\rho$ stabilizes $[X'\subset \P^r]\in \Hilb_d$.

Using Lemma \ref{L:aut-stab}, it is enough to check that
$$\Im \rho\subseteq \Aut(X',\OO_{X'}(1))\cong \Stab_{\GL_{r+1}}([X'\subset \P^r])\subseteq \GL_{r+1}.$$
Since the non exceptional subcurve $\w{X'}\subset X'$ is contained in  $Z\coprod W$, it follows from the proof of Theorem \ref{T:auto-grp} that $\Aut(X',\OO_{X'}(1))$ contains a subgroup $H$ isomorphic to $\Gm^2$
and such that $(\lambda,\mu)\in H\cong \Gm^2$ acts via multiplication by $\lambda$ on
$H^0(W,\OO_W(1))$ and by $\mu$ on $H^0(Z,\OO_Z(1))$. By construction, it follows that $\Im \rho\subseteq H$ and we are done.

\un{CLAIM 3}: $[X\subset \P^r]\in A_{\rho}([X'\subset \P^r])$.

Recall that, by assumption, $(X, \OO_X(1))$ is obtained from $(X',\OO_{X'}(1))$ by smoothing some of the nodes of
$Z\cap Z^c=\cup_i (Z\cap E_i)$. Denote by $n_i$ the node given by the intersection of $Z$ with $E_i$ and
by $\Def_{(X',n_i)}$ the functor of infinitesimal deformations of the complete local ring $\widehat{\OO}_{X', n_i}$ (see \cite[Sec. 2.4]{Ser}).
According to \cite[Cor. 3.1.2, Exa. 3.1.4(a)]{Ser},  if we write $\widehat{\OO}_{X',n_i}=k[[u_i,v_i]]/(u_iv_i)$, then
$\Def_{(X',n_i)}$ has a semiuniversal ring equal  to $k[[a_i]]$ with universal family  given by  $ k[[u_i,v_i,a_i]]/(u_iv_i-a_i)$.

Now, consider the local Hilbert functor $H_{X'}^{\P^r}$ parametrizing infinitesimal deformations of $X'$ in $\P^r$ (see \cite[Sec. 3.2.1]{Ser}).
Clearly, $H_{X'}^{\P^r}$ is pro-represented by the complete local ring of $\Hilb_d$ at $[X\subset \P^r]$.
Since $X'$ is a curve with locally complete intersection singularities and $\OO_{X'}(1)$ is non-special, from \cite[I.6.10]{Kol} we get  that
the natural morphism of functors
\begin{equation}\label{E:for-sm1}
H_{X'}^{\P^r}\longrightarrow  \Def_{X'}
\end{equation}
is formally smooth, where $\Def_{X'}$ is the functor of infinitesimal deformations of $X'$. It follows easily from
\cite[Thm. 2.4.1]{Ser} that the natural morphism of functors
\begin{equation}\label{E:for-sm2}
\Def_{X'}\longrightarrow \prod_i \Def_{(X', n_i)}
\end{equation}
is also formally smooth.
%According to \cite[Cor. 6.3]{HH2}, whose proof extends to our case since $X'$ has locally complete intersection singularities and $\OO_{X'}(1)$ is non-special,
%there exists an open neighborhood $[X'\subset \P^r]\in U\subset \Hilb_d$ such that the natural map
%$$U\longrightarrow \prod_i \Def(X',n_i),$$
%is formally smooth.
Moreover, since $\rho$ stabilizes $[X'\subset \P^r]$ by Claim 2, the above morphisms \eqref{E:for-sm1} and \eqref{E:for-sm2} are equivariant under the
natural action of $\rho$ on each functor.
%up to shrinking $U$, we can assume that $U$ is invariant under the action of $\rho$ and that the above map is $\rho$-equivariant.
Therefore, in order to prove that $[X\subset \P^r]\in A_{\rho}([X'\subset \P^r])$, it is enough to prove that $\rho$
acts on each $k[[a_i]]$ with positive weight (compare also with the proof of \cite[Lemma 4]{HMo} and of \cite[Cor. 7.9]{HH2}).

Fix a node $n_i=E_i\cap Z$ for some $1\leq i\leq k_Z$.
We can choose coordinates $\{x_1,\ldots, x_{r+1}\}$ of $V=H^0(\P^r,\OO_{\P^r}(1))=H^0(X',\OO_{X'}(1))$ so
that $x_i$ is the unique coordinate that does not vanish at $n_i$ and the exceptional component $E_i$ is given by the
linear span $\langle x_i, x_{i+1}\rangle$, as well as the tangent $T_{Z,n_i}$ of $Z$ at $n_i$ is given by the linear span
$\langle x_{i-1}, x_i\rangle$. Then the completion of the local ring $\OO_{X',n_i}$ is equal to $k[[u_i, v_i]]/(u_iv_i)$ where $u_i=x_{i-1}/x_i$ and $v_i=x_{i+1}/x_i$. Since $T_{Z, n_i}$ is contained in the linear span
$\langle Z\rangle$ of $Z$ and $\rho(t)_{|H^0(W,\OO_W(1))}=\Id$ by construction,
we have that $\rho(t)\cdot x_i= x_i$ and $\rho(t)\cdot x_{i-1}=x_{i-1}$; hence $\rho(t)\cdot u_i=u_i$. On the other hand, the point $q_i$ defined by $x_k=0$ for every $k\neq i+1$
is clearly the node given by the intersection of $E_i$ with $W$. Since $\rho(t)_{|H^0(W,\OO_W(1))}=t\cdot \Id$ by construction, we have that $\rho(t)\cdot x_{i+1}=tx_{i+1}$; hence
$\rho(t)\cdot v_i=t v_i$. Since the equation of the universal family over $k[[a_i]]$ is given by $u_iv_i-a_i=0$ and $\rho$ acts on this universal family, we deduce that
$\rho(t)\cdot a_i=ta_i$, which concludes our proof.

%Part \eqref{T:spec-clos-orb1}: Generalize \cite[Important Remark 5.1.1]{Cap}.

%Part \eqref{T:spec-clos-orb2} follows from part \eqref{T:spec-clos-orb1} and a well-know fact about projective GIT quotients, which we recall in Lemma \ref{L:fact-GIT} below.

\end{proof}

From the above theorem, we deduce now the following
\begin{coro}\label{C:poly-strictly-bal}
Let $[X\subset \P^r]\in \Hilb_d$ with $X$ connected and $d>2(2g-2)$.
If $[X\subset \P^r]$ is Chow polystable or Hilbert polystable
then $\OO_X(1)$ is strictly balanced.
%The same conclusion holds under the assumption that $[X\subset \P^r]$ is Hilbert polystable, i.e. $[X\subset \P^r]\in \Hilb_d^{ss}$ and
%the orbit  of $[X\subset \P^r]$ is closed in $\Hilb_d^{ss}$.
\end{coro}
\begin{proof}
Let us prove the statement for the Chow polystability;
%polystability;
 the Hilbert polystability being analogous.

Let $[X\subset \P^r]\in \Hilb_d$ for $d>2(2g-2)$ with $X$ connected and assume that $[X\subset \P^r]$ is Chow-polystable.
Recall that $X$ is quasi-wp-stable by Corollary \ref{C:quasi-wp-stable}\eqref{C:quasi-wp-stable1} and $\OO_X(1)$ is properly balanced by Theorem \ref{teo-pstab} and Proposition \ref{P:bal-quasi}. By Lemma \ref{L:specia}, we can find a pair $(X',L')$ consisting of a quasi-wp-stable curve $X'$ and a strictly balanced line bundle $L'$ on $X'$ such that
$(X,\OO_X(1)) \sp (X', L')$. Note that $L'$ is ample by Remark \ref{R:propbal-ample}; moreover $X'$ does not have elliptic
tails if $d <5/2(2g-2)$ because otherwise, by the basic inequality \eqref{E:basineq-multideg}, $L'$ would have
degree at most $2$ on each elliptic tail, hence it would not be very ample.
Therefore, we can apply Theorem \ref{bal-pos} which allows us to conclude that $L'$ is non-special and very ample; we get a point $[X'\stackrel{|L'|}{\hookrightarrow}\P^r]\in \Hilb_d$.
The above Theorem \ref{T:spec-clos-orb} gives that $[X'\subset \P^r]\in \ov{\Orb([X\subset \P^r])}$ and
$[X'\subset \P^r]\in \Ch^{-1}(\Chow_d^{ss})$. Since $[X\subset \P^r]$ is Chow polystable, we must have that
$[X'\subset \P^r]\in \Orb([X\subset \P^r])$; hence $X'=X$ and $\OO_X(1)=\OO_{X'}(1)=L'$ is strictly balanced.

\end{proof}

%We claim that it is enough to prove that for any such subcurve $Y\subseteq X$, it holds
%\begin{equation}\label{E:Gie}
%\frac{h^0(Y,L_{|Y})}{r+1}=\frac{h^0(Y,L_{|Y})}{h^0(X,L)}>\frac{\deg_Y L+k_Y/2}{d}.
%\end{equation}
%Indeed, using that $h^0(Y,L)=\deg_YL+1-g_Y$ and $h^0(X,L)=d+1-g$ (since $L$ is non-special), it is easily checked that \eqref{E:Gie} is equivalent to
%\begin{equation}\label{E:str1}
%\deg_YL>\frac{d}{2g-2}\deg_Y \omega_X-\frac{k_Y}{2}.
%\end{equation}
%Moreover, equation \eqref{E:Gie} for $Y^c$ together with the fact that $\deg_YL+\deg_{Y^c}L=d$ and $\deg_Y\omega_X+\deg_{Y^c}\omega_X=2g-2$
%implies that
%\begin{equation}\label{E:str2}
%\deg_YL<\frac{d}{2g-2}\deg_Y \omega_X+\frac{k_Y}{2}.
%\end{equation}
%Equations \eqref{E:str1} and \eqref{E:str2} imply that the basic inequality \eqref{E:basineq-multideg} is strict at $Y$.

\section{A criterion of stability for tails}\label{S:crit-tails}

In this section we would like to state a criterion of stability for tails based on the Hilbert-Mumford criterion and the parabolic group. Let $[X\inj \P^{r}]\in \Hilb_{d}$ with $d> 2(2g-2)$, where $X$ is the union of two curves $X_1$ and $X_2$ (of degrees $d_1,d_2$ and genus $g_1,g_2$) that intersect each other transversally in a single point $p$. By the Potential pseudo-stability Theorem \ref{teo-pstab}\eqref{teo-pstab2}, we can assume that $h^1(X,\OO_{X}(1))=0$, which implies that $h^0(X_i,\OO_{X_i}(1))=d_i+1-g_i=:r_i+1$. Hence, denoting by $\langle X_1\rangle$ and $\langle X_2\rangle$ respectively the linear spans of $X_1$ and $X_2$, we can find a system of coordinates $\{x_1,\ldots,x_{r+1}\}$ such that
\begin{eqnarray}\label{eq:coorcoda}
\langle X_1\rangle=\bigcap_{i=r_1+2}^{r+1}\{x_i=0\}\quad \textit{and}\quad \langle X_2\rangle=\bigcap_{i=1}^{r_1}\{x_i=0\}.
\end{eqnarray}
Using this type of coordinates to find destabilizing one-parameter subgroups is very convenient, because we can study the two subcurves separately, as the results below show.

Let $\rho$ be a 1ps of  $\GL_{r+1}$. By Proposition \ref{prop:sumsub}, we know that $e_{X,\rho}=e_{X_1,\rho}+e_{X_2,\rho}$, but in general we cannot say something similar for the Hilbert weight $W_{X,\rho}(m)$. If $\rho$ is diagonalized by coordinates of type \eqref{eq:coorcoda}, we can do it.
\begin{lemma}\label{lem:sumhilb}
Let $[X\subset \P^r]\in \Hilb_d$ as above and let $\rho$ be a 1ps of $\GL_{r+1}$ diagonalized by coordinates of type \eqref{eq:coorcoda}. %Denoting by
%	$$
%	[X'\subset \P^r]=\lim_{t\ra 0}\rho(t)[X\subset \P^{r}]\quad \text{and}\quad [X'_i\subset \P^r]=\lim_{t\ra 0}\rho(t)[X_i\subset \P^{r}]\quad\text{for }i=1,2,
%  $$
%$[X'_1\subset \P^r]=\lim_{t\ra 0}\rho(t)[X_1\subset \P^{r}]$ and $[X'_2\subset \P^r]=\lim_{t\ra 0}\rho(t)[X_2\subset \P^{r}]$,
%we have that $X'=X'_1\cup X'_2\subset \P^r$.
Then
\begin{equation}\label{eq:WW1W2}
W_{X,\rho}(m)= W_{X_1,\rho}(m)+W_{X_2,\rho}(m)-w_{r_1+1}m.
\end{equation}
\end{lemma}
\begin{proof}
%We use the notation of the proof of proposition \ref{prop:sumsub}.
Let $m$ be a positive integer and consider a monomial basis $\{B_1,\ldots,B_{P_1(m)}\}$ of $H^0(X_1,\OO_{X_1}(m))$. Since the point
$p=[x_1=0,\ldots,x_{r_1}=0,x_{r_1+1}=1,x_{r_1+2}=0,\ldots,x_{r+1}=0]$ belongs to $X_1$, there exists a monomial (for example $B_{P_1(m)}$) such that $B_{P_1(m)}=x_{r_1+1}^m$. The same holds for each monomial basis $\{B'_1,\ldots,B'_{P_2(m)}\}$ of $H^0(X_2,\OO_{X_2}(m))$ (for example $B'_{P_2(m)}=x_{r_1+1}^m$). By the exact sequence
$$
0\ra H^0(X,\OO_{X}(m)) \stackrel{(|_{X_1},|_{X_2})}{\lra} H^0(X_1,\OO_{X_1}(m))\oplus H^0(X_2,\OO_{X_2}(m))\lra H^0(X_1\cap X_2,\OO_{X_1\cap X_2}(m))\ra 0,
$$
we obtain that $\{B_1,\ldots,B_{P_1(m)},B'_1,\ldots,B'_{P_2(m)-1}\}$ is a monomial basis of $H^0(X,\OO_{X}(m))$.
Therefore, if we choose the monomial basis $\{B_1,\ldots,B_{P_1(m)}\}$  and $\{B'_1,\ldots,B'_{P_2(m)}\}$ so that
$$
W_{X_1,\rho}(m)=\sum_{i=1}^{P_1(m)}w_{\rho}(B_i)\quad \text{and}\quad
W_{X_2,\rho}(m)=\sum_{i=1}^{P_2(m)}w_{\rho}(B'_i),
$$
then we get
\begin{eqnarray}
W_{X,\rho}(m) &\leq& \sum_{i=1}^{P_1(m)}w_{\rho}(B_i)+\sum_{i=1}^{P_2(m)-1}w_{\rho}(B'_i)
=\sum_{i=1}^{P_1(m)}w_{\rho}(B_i)+\sum_{i=1}^{P_2(m)}w_{\rho}(B'_i)-w_{\rho}(B'_{P_2(m)})=\nonumber\\
&=&  W_{X_1,\rho}(m)+W_{X_2,\rho}(m)-w_{r_1+1}m.\nonumber
\end{eqnarray}
Now, we will prove the reverse inequality. Choose a monomial basis $\{B_1,\ldots,B_{P(m)}\}$ of $H^0(X,\OO_X(m))$ such that
$$
W_{X,\rho}(m)=\sum_{i=1}^{P(m)}w_{\rho}(B_i).
$$
The same argument used to prove the inequality $\geq$ of Proposition \ref{prop:sumsub} shows that for each monomial basis $\{B_1,\ldots,B_{P(m)}\}$ of $H^0(X,\OO_{X}(m))$,
we can reorder the monomials so that
\begin{enumerate}
	\item $\{B_1,\ldots,B_{P_1(m)}\}$ is a monomial basis of $H^0(X_1,\OO_{X_1}(m))$,
	\item $\{B_{P_1(m)},\ldots,B_{P(m)}\}$ is a monomial basis of $H^0(X_2,\OO_{X_2}(m))$,
	\item $B_{P_1(m)}=x^m_{r_1+1}$.
\end{enumerate}
We obtain
\begin{eqnarray}
W_{X,\rho}(m) &=&\sum_{i=1}^{P(m)}w_{\rho}(B_i)=\sum_{i=1}^{P_1(m)}w_{\rho}(B_i)+\sum_{i=P_1(m)}^{P(m)}w_{\rho}(B_i)-w_{\rho}(B_{P_1(m)})\nonumber\\
&\geq& W_{X_1,\rho}(m)+W_{X_2,\rho}(m)-w_{r_1+1}m\nonumber,
\end{eqnarray}
and we are done.
\end{proof}
Let $I$, $I_1$ and $I_2$ be the ideals of $X$, $X_1$ and $X_2$, respectively.
If $\rho$ is diagonalized by coordinates of type \eqref{eq:coorcoda}, we can compute easily the flat limit
$$
\lim_{t\ra 0}\rho(t)[X\subset \P^r]
$$
by computing the flat limits of $X_1$ and $X_2$ separately.
\begin{lemma}\label{lem:flatlimitcode}
Let $X=X_1\cup X_2\subset \P^r$ be a connected (possibly not reduced) curve and let $\{x_1,\ldots,x_{r+1}\}$ coordinates such that
$$
\{x_{i}\,|\,r_1+2\leq i\leq r+1\}\subset I_1,\quad \{x_{i}\,|\,1\leq i\leq r_1\}\subset I_2\quad\textit{and}\quad I_1+I_2=\langle x_i\,|\,i\neq r_1+1\rangle.
$$
Let $\rho$ be a 1ps of $\GL_{r+1}$  diagonalized by $\{x_1,\ldots,x_{r+1}\}$ and denote by $\prec$ a $\rho$-weighted lexicographical order in $k[x_1,\ldots,x_{r+1}]$ that refines $\prec_{\rho}$.
\begin{enumerate}[(i)]
\item \label{L:flatlim1} If $\{f_1,\ldots,f_n\}\subset k[x_1,\ldots,x_{r_1+1}]$ is a system of generators for $I_1\cap k[x_1,\ldots,x_{r_1+1}]$ and $\{g_1,\ldots,g_m\}\subset k[x_{r_1+1},\ldots,x_{r+1}]$ is a system of generators for $I_2\cap k[x_{r_1+1},\ldots,x_{r+1}]$, then
$$
I=\langle f_1,\ldots,f_n,g_1,\ldots,g_m,x_ix_j\,|\,1\leq i\leq r_1\text{ and } r_1+2\leq j\leq r+1\rangle.
$$
\item \label{L:flatlim2} Moreover, if $\{f_1,\ldots,f_n\}$ and $\{g_1,\ldots,g_m\}$ are Gr\"obner bases with respect to $\prec$, then
\begin{enumerate}[(a)]
\item $\{f_1,\ldots,f_n,x_{r_1+2},\ldots,x_{r+1}\}\quad\text{and}\quad \{x_1,\ldots,x_{r_1},g_1,\ldots,g_m\}$
are Gr\"obner bases respectively for $I_1$ and $I_2$;
\item $\{f_1,\ldots,f_n,g_1,\ldots,g_m,x_ix_j\,|\,1\leq i\leq r_1\text{ and } r_1+2\leq j\leq r+1\}$
is a Gr\"obner basis for $I$.
\end{enumerate}
\item \label{L:flatlim3}
We have that
$$
\In_{\prec}(I)=\In_{\prec}(I_1)\cap\In_{\prec}(I_2)\quad\text{and}\quad \In_{\prec_{\rho}}(I)=\In_{\prec_{\rho}}(I_1)\cap\In_{\prec_{\rho}}(I_2).
$$
\end{enumerate}
\end{lemma}
\begin{proof}
Let us first prove part \eqref{L:flatlim1}.
Consider $f\in I=I_1\cap I_2$. Since $f\in I_2$, there exist $p_1,\ldots,p_{r_1}\in k[x_1,\ldots,x_{r+1}]$ and $q_1,\ldots,q_{m}\in k[x_{r_1+1},\ldots,x_{r+1}]$ such that
$$
f=\sum_{i=1}^{r_1}x_ip_i+\sum_{k=1}^{m}q_kg_k.
$$
Let $\widetilde{p}_i\in k[x_1,\ldots,x_{r_1+1}]$ for $i=1,\ldots,r_1$ such that each monomial of $p_i-\widetilde{p}_i$ contains one among the coordinates $x_{r_1+2},\ldots,x_{r+1}$. Analogously, let $\widetilde{q}_k\in k[x_1,\ldots,x_{r_1+1}]$ for $k=1,\ldots,m$ such that each monomial of $q_k-\widetilde{q}_k$ contains one among the coordinates $x_{1},\ldots,x_{r_1}$. In this way for $i=1,\ldots,r_1$ and $j=r_1+2,\ldots,r+1$ there exist polynomials $l_{ij}$ which satisfy
$$
f=\sum_{i=1}^{r_1}x_i\widetilde{p}_i+\sum_{i=1}^{r_1}\sum_{j=r_1+2}^{r+1}x_ix_jl_{ij}+\sum_{k=1}^{m}\widetilde{q_k} g_k.
$$
Since in each term of the above summation the monomial $x_{r_1+1}^a$ does not appear, we get that
$$
\sum_{i=1}^{r_1}\sum_{j=r_1+2}^{r+1}x_ix_jl_{ij}+\sum_{k=1}^{m}\widetilde{q_k}g_k\in I_1.
$$
Moreover, since by assumption $f\in I_1$, we also get that
$$
\sum_{i=1}^{r_1}x_i\widetilde{p_i}\in I_1,
$$
hence there exist $h_1,\ldots,h_n\in k[x_1,\ldots,x_{r_1+1}]$ such that
$$
\sum_{i=1}^{r_1}x_i\widetilde{p_i}=\sum_{i=1}^{n}h_i f_i.
$$
Substituting into the above expression of $f$, we get
\begin{equation}\label{eq:fgen}
f=\sum_{i=1}^{n}h_i f_i+\sum_{i=1}^{r_1}\sum_{j=r_1+2}^{r+1}x_ix_jl_{ij}+\sum_{k=1}^{m}\widetilde{q}_kg_k,
\end{equation}
which shows that  $\{f_1,\ldots,f_n,g_1,\ldots,g_m,x_ix_j\,|\,1\leq i\leq r_1\text{ and } r_1+2\leq j\leq r+1\}$ is a system of generators for $I$, q.e.d.

Now, suppose that $\{f_1,\ldots,f_n\}$ and $\{g_1,\ldots,g_m\}$ are Gr\"obner bases with respect to $\prec$ and let us prove part \eqref{L:flatlim2}.
The assertion (a) follows easily from the  Buchberger's criterion (see Fact \ref{F:buch}).
% it is very easy to check that $\{f_1,\ldots,f_n,x_{r_1+2},\ldots,x_{r+1}\}$ and $\{x_1,\ldots,x_{r_1},g_1,\ldots,g_m\}$ are Gr\"obner bases for, respectively, $I_1$ and $I_2$ with respect to $\prec$.
In order to prove the assertion (b), consider an element $f\in I_1\cap I_2$ and write it as in \eqref{eq:fgen}.
By definition, the three polynomials
$$
F:=\sum_{i=1}^{n}h_i f_i,\quad G:=\sum_{i=1}^{r_1}\sum_{j=r_1+2}^{r+1}x_ix_jl_{ij}\quad \text{and}\quad H:=\sum_{k=1}^{m}\widetilde{q}_kg_k
$$
have no common similar monomials, so that
$$
\In_{\prec}(f)=\In_{\prec}(\In_{\prec}(F)+\In_{\prec}(G)+\In_{\prec}(H)).
$$
Obviously $\In_{\prec}(G)\in \langle x_ix_j\,|\,1\leq i\leq r_1\text{ and } r_1+2\leq j\leq r+1\rangle$. We know that $\{f_1,\ldots,f_n\}$ is a Gr\"obner basis, hence $\In_{\prec}(F)\in \langle \In_{\prec}(f_1),\ldots,\In_{\prec}(f_n)\rangle$. Similarly, $\In_{\prec}(H)\in \langle \In_{\prec}(g_1),\ldots,\In_{\prec}(g_n)\rangle$, hence
$$
\{f_1,\ldots,f_n,g_1,\ldots,g_m,x_ix_j\,|\,1\leq i\leq r_1\text{ and } r_1+2\leq j\leq r+1\}
$$
is a Gr\"obner basis for $I$, q.e.d.

Let us now prove part \eqref{L:flatlim3}.
According to \eqref{L:flatlim2}(a), the ideals $\In_{\prec}(I_1)$ and $\In_{\prec}(I_2)$ satisfy the hypothesis of \eqref{L:flatlim1} with respect to the generators $\{ \In_{\prec_{\rho}}(f_1),\ldots,\In_{\prec_{\rho}}(f_n)\}$ of $\In_{\prec}(I_1)\cap k[x_1,\ldots,x_{r_1+1}]$ and $\{\In_{\prec_{\rho}}(g_1),\ldots,\In_{\prec_{\rho}}(g_m)\}$ of $\In_{\prec}(I_2)\cap k[x_{r_1+1},\ldots,x_{r+1}]$.
Therefore, part \eqref{L:flatlim1} gives that
$$\{ \In_{\prec_{\rho}}(f_1),\ldots,\In_{\prec_{\rho}}(f_n), \In_{\prec_{\rho}}(g_1),\ldots,\In_{\prec_{\rho}}(g_m),x_ix_j\,|\,1\leq i\leq r_1\text{ and } r_1+2\leq j\leq r+1\}
$$
is a system of generators of $\In_{\prec}(I_1)\cap \In_{\prec}(I_2)$.
However, the above elements generate also $\In_{\prec}(I)$ by \eqref{L:flatlim2}(b), and the first assertion of part \eqref{L:flatlim3} follows.
The second assertion follows in a similar way, once we apply Fact \ref{T:groebflatlimit} and \eqref{L:flatlim2} to get
$$
\begin{sis}
& \In_{\prec_{\rho}}(I_1)=\langle \In_{\prec_{\rho}}(f_1),\ldots,\In_{\prec_{\rho}}(f_n),x_{r_1+2},\ldots,x_{r+1}\rangle,\\
& \In_{\prec_{\rho}}(I_2)=\langle \In_{\prec_{\rho}}(g_1),\ldots,\In_{\prec_{\rho}}(g_m),x_1, \ldots, x_{r_1} \rangle,\\
& \In_{\prec_{\rho}}(I)=\langle \In_{\prec_{\rho}}(f_1),\ldots,\In_{\prec_{\rho}}(f_n),\In_{\prec_{\rho}}(g_1),\ldots,\In_{\prec_{\rho}}(g_m),x_ix_j\,|\,1\leq i\leq r_1\text{ and } r_1+2\leq j\leq r+1\rangle.
\end{sis}
$$
\end{proof}
%If $\rho:\Gm\longrightarrow\GL_{r+1}$ is a one-parameter subgroup diagonalized by these coordinates, we can compute easily $W_{X,\rho}(m)$.

We are going to state a criterion of stability for tails, according to which coordinates of type \eqref{eq:coorcoda} diagonalize the one-parameter subgroups that give the ``worst'' weights.
\begin{prop}\label{prop:stab-tail}\textbf{\emph{(Criterion of stability for tails.)}}
Let $[X\subset \P^{r}]\in \rm{Hilb}_{d}$ as above. The following conditions are equivalent:
\begin{enumerate}
	\item $[X\subset \P^{r}]$ is Hilbert semistable (resp. polystable, stable);
	\item $[X\subset \P^{r}]$ is Hilbert semistable (resp. polystable, stable)
	with respect to any one-parameter subgroup $\rho:\Gm\ra \GL_{r+1}$ diagonalized by coordinates of type (\ref{eq:coorcoda});
	\item $[X\subset \P^{r}]$ is Hilbert semistable (resp. polystable, stable)
	with respect to any one-parameter subgroup $\rho:\Gm\ra \GL_{r+1}$ diagonalized by coordinates of type (\ref{eq:coorcoda}) with weights $w_1,\ldots,w_{r+1}$ such that
%\begin{enumerate}[(i)]
%	\item either $w_{1}=w_{2}=\ldots=w_{r_1+1}=0$
%	\item or $w_{r_1+1}=w_{r_1+2}=\ldots=w_{r+1}=0$
%\end{enumerate}
$$
w_{1}=w_{2}=\ldots=w_{r_1+1}=0\quad\text{or}\quad w_{r_1+1}=w_{r_1+2}=\ldots=w_{r+1}=0.
$$
\end{enumerate}
The same holds for the Chow semistability (resp. polystability, stability).
\end{prop}
\begin{proof}
The implications $(1) \Longrightarrow (2) \Longrightarrow (3)$ are clear for each type of stability.

Let us now prove the implication $(2)\Longrightarrow (1)$.  Let $X=(x_1,\ldots,x_{r+1})^t$ be a basis of coordinates of type (\ref{eq:coorcoda}).
By Corollary \ref{cor:base} applied to $(x_1,\ldots,x_{r_1},x_{r_1+2},\ldots,x_{r+1},x_{r_1+1})$, it is enough to consider a 1ps
$\rho:\Gm\ra \GL_{r+1}$ that  is diagonalized by the coordinates
\begin{equation}\label{eq:zax}
(z_1,\ldots,z_{r+1})^t=Z=AX
\end{equation}
where
\begin{equation}\label{eq:A}
%\left(
%\begin{array}{c}
%z_1 \\
%z_2\\
%\vdots\\
%z_{r_1}\\
%z_{r_1+1}\\
%z_{r_1+2}\\
%\vdots\\
%z_{r+1}\\
%\end{array}
%\right)=
A=
\left(
\begin{array}{ccccccccc}
1 & 0 & \cdots & 0 & 0 & 0 & 0 &\cdots & 0\\
a_{21} & 1 & \cdots & 0 & 0 & 0 & 0 &\cdots & 0\\
\vdots & \vdots & \ddots & \vdots & \vdots & \vdots & \vdots & \vdots & \vdots\\
a_{r_1,1} & a_{r_1,2} & \cdots & 1 & 0 & 0 & 0 & \cdots & 0\\
a_{r_1+1,1} & a_{r_1+1,2} & \cdots & a_{r_1+1,r_1}  & 1 & a_{r_1+1,r_1+2} & a_{r_1+1,r_1+3} & \cdots & a_{r_1+1,r+1}\\
a_{r_1+2,1} & a_{r_1+2,2} & \cdots & a_{r_1+2,r_1}  & 0 & 1 & 0 & \cdots & 0\\
a_{r_1+3,1} & a_{r_1+3,2} & \cdots & a_{r_1+3,r_1}  & 0 & a_{r_1+3,r_1+2} & 1 & \cdots & 0\\
\vdots & \vdots & \vdots & \vdots & 0 & \vdots & \vdots & \ddots & \vdots\\
a_{r+1,1} & a_{r+1,2} & \cdots & a_{r+1,r_1}  & 0 & a_{r+1,r_1+2} & a_{r+1,r_1+3} & \cdots & 1\\
\end{array}
\right).
%\left(
%\begin{array}{c}
%x_1 \\
%x_2\\
%\vdots\\
%x_{r_1}\\
%x_{r_1+1}\\
%x_{r_1+2}\\
%\vdots\\
%x_{r+1}\\
%\end{array}
%\right)
\end{equation}
%Denote by $A=(a_{ij})$ the matrix above and
Define the new matrix $A'=(a'_{ij})$ as follows
$$
a'_{ij}=
\left\{
\begin{array}{ll}
a_{ij} & \text{if }i\leq r_1+1\text{ or }j\geq r_1+1\\
0 & \text{if }i\geq r_1+2\text{ and }j\leq r_1\\
\end{array}
\right.
$$
so that\\
\begin{equation}\label{eq:AA}
A'=
\left(
\begin{array}{ccccccccc}
1 & 0 & \cdots & 0 & 0 & 0 & 0 &\cdots & 0\\
a_{21} & 1 & \cdots & 0 & 0 & 0 & 0 &\cdots & 0\\
\vdots & \vdots & \ddots & \vdots & \vdots & \vdots & \vdots & \vdots & \vdots\\
a_{r_1,1} & a_{r_1,2} & \cdots & 1 & 0 & 0 & 0 & \cdots & 0\\
a_{r_1+1,1} & a_{r_1+1,2} & \cdots & a_{r_1+1,r_1}  & 1 & a_{r_1+1,r_1+2} & a_{r_1+1,r_1+3} & \cdots & a_{r_1+1,r+1}\\
0 & 0 & \cdots & 0 & 0 & 1 & 0 & \cdots & 0\\
0 & 0 & \cdots & 0 & 0 & a_{r_1+3,r_1+2} & 1 & \cdots & 0\\
\vdots & \vdots & \vdots & \vdots & 0 & \vdots & \vdots & \ddots & \vdots\\
0 & 0 & \cdots & 0 & 0 & a_{r+1,r_1+2} & a_{r+1,r_1+3} & \cdots & 1\\
\end{array}
\right).
\end{equation}
\\
Now, set $(z'_1,\ldots,z'_{r+1})^t=:Z'=A'X$; the coordinates $Z'$ are of type (\ref{eq:coorcoda}). Consider the one-parameter subgroup $\rho'$ diagonalized by the coordinates $Z'$ with the same weights of $\rho$ (in particular $w(\rho)=w(\rho')$). Since $z'_i=z_i$ for $i=1,\ldots,r_1+1$, if $\{B_1(Z'),\ldots,B_{P_1(m)}(Z')\}$ is a monomial basis of $H^0(X_1,\OO_X(m))$, then $\{B_1(Z),\ldots,B_{P_1(m)}(Z)\}$ is again a monomial basis of $H^0(X_1,\OO_X(m))$, hence
\begin{equation}\label{eq:rho1}
W_{X_1,\rho}(m)\leq W_{X_1,\rho'}(m)\quad \text{and} \quad e_{X_1,\rho}\leq e_{X_1,\rho'}.
\end{equation}
Similarly, the set of monomial bases of the subcurve $X_2$ with respect to $Z$ and the one with respect to $Z'$ are the same, so that
\begin{equation}\label{eq:rho2}
W_{X_2,\rho}(m)= W_{X_2,\rho'}(m)\quad \text{and} \quad e_{X_2,\rho}=e_{X_2,\rho'}.
\end{equation}
Suppose that $[X\subset \P^r]$ is Chow semistable (resp. stable) with respect to $\rho'$, i.e.
$$
e_{X,\rho'}\leq \frac{2d}{r+1}\,w(\rho')\quad\text{(resp. $<$)}.
$$
Combining the formulas \eqref{eq:rho1} and \eqref{eq:rho2} with Proposition \ref{prop:sumsub}, we get
$$
e_{X,\rho}=e_{X_1,\rho}+e_{X_2,\rho}\leq e_{X_1,\rho'}+e_{X_2,\rho'}=e_{X,\rho'}\leq\frac{2d}{r+1}\,w(\rho')=\frac{2d}{r+1}\,w(\rho)\quad\text{(resp. $<$)}
$$
and the implication (2) $\Longrightarrow$ (1) for Chow semistability (resp. stability) follows.
%is done by the Hilbert-Mumford numerical criterion (Fact \ref{Chow-crit}).
We notice that this last step does not work for the Hilbert semistability (resp. stability) because in general
$$
W_{X,\rho}(m)\neq W_{X_1,\rho}(m)+W_{X_2,\rho}(m),
$$
as Lemma \ref{lem:sumhilb} shows. But the argument used to prove the part $\leq$ of \ref{eq:WW1W2} can be applied to $\rho$, so that
\begin{equation}\label{eq:WW1W2min}
W_{X,\rho}(m)\leq W_{X_1,\rho}(m)+W_{X_2,\rho}(m)-w_{r+1}m,
\end{equation}
By \eqref{eq:rho1}, \eqref{eq:rho2}, \eqref{eq:WW1W2min} and Lemma \ref{lem:sumhilb} we obtain
$$
W_{X,\rho}(m)\leq W_{X_1,\rho}(m)+W_{X_2,\rho}(m)-w_{r_1+1}m\leq W_{X_1,\rho'}(m)+W_{X_2,\rho'}(m)-w_{r_1+1}m =W_{X,\rho'}(m).
$$
and the implication (2) $\Longrightarrow$ (1) for the Hilbert semistability (resp. stability) follows.
%is done by the Hilbert-Mumford numerical criterion (Fact \ref{Hilb-crit}).

Now, we will prove the implication (2) $\Longrightarrow$ (1) for the Chow polystability (for the Hilbert polystability the argument is analogous using Lemma \ref{lem:sumhilb} instead of
Proposition \ref{prop:sumsub}). By what proved above, we get that $[X\subset \P^r]$ is Chow semistable. By Corollary \ref{cor:base}, it is enough to prove that $[X\subset \P^r]$ is Chow polystable
with respect a 1ps $\rho:\Gm\to \GL_{r+1}$ that is diagonalized by coordinates $Z$ as in \eqref{eq:zax}. We can assume that
\begin{equation}\label{eq:erhoeq}
e_{X,\rho}=\frac{2d}{r+1}\,w(\rho),
\end{equation}
because, if  $e_{X,\rho}>\frac{2d}{r+1}\,w(\rho)$ then there is nothing to prove.
As before,  consider the one-parameter subgroup $\rho'$ with the same weights of $\rho$ and diagonalized by the coordinates $Z'=A'X$, which are of type (\ref{eq:coorcoda}).
Denoting by $B=(b_{ij})$ the matrix $A(A')^{-1}$,  we have that
%$$
%\left\{
%\begin{array}{l}
%z_1=z'_1\\
%\ldots = \ldots\\
%z_{r_1+1}=z'_{r_1+1}\\
%\displaystyle z_{r_1+2}=z'_{r_1+2}+\sum_{j=1}^{r_1}b_{r_1+2,j}z'_{j}\\
%\ldots = \ldots\\
%\displaystyle z_{r+1}=z'_{r+1}+\sum_{j=1}^{r_1}b_{r+1,j}z'_{j}.\\
%\end{array}
%\right.
%$$
\\
$$
B=
\left(
\begin{array}{ccccccccc}
1 & 0 & \cdots & 0 & 0 & 0 & 0 & \cdots & 0\\
0 & 1 & \cdots & 0 & 0 & 0 & 0 & \cdots & 0\\
\vdots & \vdots & \ddots & \vdots & \vdots & \vdots & \vdots & \vdots & \vdots\\
0 & 0 & \cdots & 1 & 0 & 0 & 0 & \cdots & 0\\
0 & 0 & \cdots & 0 & 1 & 0 & 0 & \cdots & 0\\
b_{r_1+2,1} & b_{r_1+2,2} & \cdots & b_{r_1+2,r_1}  & 0 & 1 & 0 & \cdots & 0\\
b_{r_1+3,1} & b_{r_1+3,2} & \cdots & b_{r_1+3,r_1}  & 0 & 0 & 1 & \cdots & 0\\
\vdots & \vdots & \vdots & \vdots & 0 & \vdots & \vdots & \ddots & \vdots\\
b_{r+1,1} & b_{r+1,2} & \cdots & b_{r+1,r_1}  & 0 & 0 & 0 & \cdots & 1\\
\end{array}
\right)
$$
\\
and $Z=BZ'$.

\un{CLAIM}: If $r_1+2\leq j\leq r+1$, $1\leq i\leq r_1$ and $b_{ji}\neq 0$ then $w_j\geq w_i$.

Suppose by contradiction that $w_j<w_i$. Define a one-parameter subgroup $\widetilde{\rho}$ diagonalized by the new coordinates $Y=(y_1,\ldots,y_{r+1})^t$ where
\begin{equation}\label{eq:yz}
y_k=
\left\{
\begin{array}{ll}
z'_k & \text{if }k\neq i,\\
\displaystyle \sum_{l=1}^{r_1}b_{jl}z'_l & \text{if }k=i\\
\end{array}
\right.
\end{equation}
with weights
\begin{equation}\label{eq:wwtilde}
\widetilde{w}_{k}=
\left\{
\begin{array}{ll}
w_{k} & \text{if }k\neq i,\\
w_j & \text{if }k=i.\\
\end{array}
\right.
\end{equation}
Notice that $Y$ is a set of coordinates of type \eqref{eq:coorcoda} and $w(\widetilde{\rho})<w(\rho)$. Let $B_1,\ldots,B_{P_1(m)}$ be a monomial basis of $H^0(X_1,\OO_{X_1}(m))$ with respect to $Y$ such that
$$
e_{X_1,\widetilde{\rho}}=\text{n.l.c.}\bigg(\sum_{l=1}^{P_1(m)}w_{\rho}(B_l)\bigg),
$$
where n.l.c denotes the normalized leading coefficient.
The set of coordinates $Y$ is of type \eqref{eq:coorcoda}, hence $B_l\in k[y_1,\ldots,y_{r_1+1}]$ for $l=1,\ldots,P_1(m)$. Notice that ${y_i}_{|X_1}={z_j}_{|X_1}$ and $y_k=z_k$ for $k=1,\ldots,i-1,i+1,\ldots,r_1+1$.
Therefore, if
$$
\{B_1(y_1,\ldots,y_{r_1+1}),\ldots,B_{P_1(m)}(y_1,\ldots,y_{r_1+1})\}
$$
is a monomial basis of $H^0(X_1,\OO_{X_1}(m))$, then the same holds for
$$
\{B_1(z_1,\ldots,z_{i-1},z_j,z_{i+1}\ldots,z_{r_1+1}),\ldots,B_{P_1(m)}(z_1,\ldots,z_{i-1},z_j,z_{i+1}\ldots,z_{r_1+1})\};$$
hence
\begin{equation}\label{eq:eroerotild1}
e_{X_1,\rho}\leq\,\text{n.l.c.}\bigg(\sum_{l=1}^{P_1(m)}w_{\rho}(B_l)\bigg)= e_{X_1,\widetilde{\rho}}.
\end{equation}
Moreover, since ${y_l}_{|X_2}={z_{l}}_{|X_2}$ for $r_1+1\leq l\leq r+1$, the monomial bases of $H^0(X_2,\OO_{X_2}(m))$ with respect to $Y$ and $Z$ are the same, hence
\begin{equation}\label{eq:eroerotild2}
e_{X_1,\rho}=e_{X_1,\widetilde{\rho}}.
\end{equation}
Since we already know that $[X\subset \P^r]$ is Chow semistable, Fact \ref{Chow-crit} gives that
\begin{equation}\label{eq:rotildsemistab}
e_{X,\widetilde{\rho}}\leq \frac{2d}{r+1}\,w(\widetilde{\rho}).
\end{equation}
Combining \eqref{eq:eroerotild1}, \eqref{eq:eroerotild2}, \eqref{eq:rotildsemistab} and Proposition \ref{prop:sumsub}, we get
$$
e_{X,\rho}=e_{X_1,\rho}+e_{X_2,\rho}\leq e_{X_1,\widetilde{\rho}}+e_{X_2,\widetilde{\rho}}=e_{X,\widetilde{\rho}}\leq \frac{2d}{r+1}\,w(\widetilde{\rho})<\frac{2d}{r+1}\,w(\rho),
$$
which contradicts \eqref{eq:erhoeq} and the Claim is proved.

Consider the weighted graded orders $\prec_{\rho}$, $\prec_{\rho'}$ and two weighted graded lexicographical orders $\prec$ and $\prec'$ that refine respectively $\prec_{\rho}$ and $\prec_{\rho'}$ and are induced by the lexicographical orders $z_1<z_2<\ldots<z_{r+1}$ and $z'_1<z'_2<\ldots<z'_{r+1}$. Denote by $I$, $I_1$ and $I_2$ the ideals of $X$, $X_1$ and $X_2$ respectively. Let $f_1,\ldots,f_n\in k[z'_1,\ldots,z'_{r_1+1}]$ and $g_1,\ldots,g_m\in k[z'_{r_1+1},\ldots,z'_{r+1}]$ such that $\{f_1,\ldots,f_n,z'_{r_1+2},\ldots,z'_{r+1}\}$ and $\{z'_{1},\ldots,z'_{r_1},g_1,\ldots,g_m\}$ are Gr\"obner bases respectively of $I_1$ and $I_2$ with respect to $\prec'$. Fact \ref{T:groebflatlimit} implies that
$$\begin{sis}
& I'_1:=\In_{\prec_{\rho'}}(I_1)=\langle\In_{\prec_{\rho'}}(f_1),\ldots,\In_{\prec_{\rho'}}(f_n),z'_{r_1+2},\ldots,z'_{r+1}\rangle, \\
& I'_2:=\In_{\prec_{\rho'}}(I_2)=\langle z'_{1},\ldots,z'_{r_1},\In_{\prec_{\rho'}}(g_1),\ldots,\In_{\prec_{\rho'}}(g_m)\rangle.
\end{sis}$$
Applying Lemma \ref{lem:flatlimitcode}, we obtain that
$$
\{\In_{\prec_{\rho'}}(f_1),\ldots,\In_{\prec_{\rho'}}(f_n),\In_{\prec_{\rho'}}(g_1),\ldots,\In_{\prec_{\rho'}}(g_m),z'_{i}z'_{j}\,|\,1\leq i\leq r_1\text{ and }r_1+2\leq j\leq r+1\}
$$
is a system of generators of $I':=\In_{\prec_{\rho}}(I)$. Denoting by $X':=V(I')\subset \P^r$ and applying Fact \ref{T:flatlimit}, we get
$$
[X'\subset\P^r]=\lim_{t\ra 0}\rho'(t)[X \subset \P^r].
$$
Now, define
$$
[X''\subset \P^r]=\lim_{t\ra 0}\rho(t)[X \subset \P^r]
$$
and consider the matrix $B'=(b'_{ji})$ defined as follows:
$$
b'_{ji}=
\left\{
\begin{array}{lll}
b_{ji} & \text{if }1\leq j\leq r_1+1 \text{ or }r_1+1\leq i\leq r+1,\\
b_{ji} & \text{if }r_1+2\leq j\leq r+1 \text{, }1\leq i\leq r_1\text{ and }w_j=w_i,\\
0 & \text{if }r_1+2\leq j\leq r+1 \text{, }1\leq i\leq r_1\text{ and }w_j>w_i.\\
\end{array}
\right.
$$
By the above CLAIM, we have
$$
\In_{\prec_{\rho}}\bigg(z_k-\sum_{i=1}^{r_1}b_{ki}z_{i}\bigg)=z_k-\sum_{i=1}^{r_1}b'_{ki}z_{i}
$$
for $k=r_1+2,\ldots,r+1$. Since $z'_i=z_i$ for $i=1,\ldots,r_1+1$, we get that $f_i(Z)=f_i(B^{-1}Z)$, hence, by Buchberger's criterion (see Fact \ref{F:buch}), the system of generators
$$
\bigg\{f_1(B^{-1}Z),\ldots,f_n(B^{-1}Z),z_{r_1+2}-\sum_{k=1}^{r_1}b_{r_1+2,k}z_k,\ldots,z_{r+1}-\sum_{k=1}^{r_1}b_{r+1,k}z_k\bigg\}
$$
is a Gr\"obner basis of $I_1$ with respect to $\prec$, so that
$$
\In_{\prec}(I_1)=\langle\In_{\prec}(f_1(B^{-1}Z)),\ldots,\In_{\prec}(f_n(B^{-1}Z)),z_{r_1+2}\ldots,z_{r+1}\rangle.
$$

Now, consider $I_2$. For each $j=1,\ldots,m$ there exists $h_j\in k[z_1,\ldots,z_{r+1}]$ such that each of its monomials contains one of the coordinates $z_1,\ldots,z_{r_1}$ and the following holds:
\begin{equation}\label{eq:gjBZ}
g_j(B^{-1}Z)=g_j(Z)+h_j(Z);
\end{equation}
hence
$$
\langle z_{1},\ldots,z_{r_1},g_1(B^{-1}Z),\ldots,g_m(B^{-1}Z)\rangle=
\langle z_{1},\ldots,z_{r_1},g_1(Z),\ldots,g_m(Z)\rangle.
$$
Applying $\In_{\prec}$ to \eqref{eq:gjBZ}, we obtain $\In_{\prec}(g_j(B^{-1}Z))=\In_{\prec}(g_j(Z))$; hence
$$
\begin{aligned}
\In_{\prec}(\langle z_{1},\ldots,z_{r_1},g_1(B^{-1}Z),\ldots,g_m(B^{-1}Z)\rangle)& =
\In_{\prec}(\langle z_{1},\ldots,z_{r_1},g_1(Z),\ldots,g_m(Z)\rangle)=\\
=\langle z_{1},\ldots,z_{r_1},\In_{\prec}(g_1(Z)),\ldots,\In_{\prec}(g_m(Z))\rangle
& =\langle z_{1},\ldots,z_{r_1},\In_{\prec}(g_1(B^{-1}Z)),\ldots,\In_{\prec}(g_m(B^{-1}Z))\rangle.
\end{aligned}
$$
By definition $\{z_{1},\ldots,z_{r_1},g_1(B^{-1}Z),\ldots,g_m(B^{-1}Z)\}$ is a Gr\"obner basis of $I_2$ with respect to $\prec$.
We notice that $\In_{\prec}(I)\subset \In_{\prec}(I_1)\cap \In_{\prec}(I_2)$.
Applying Lemma \ref{lem:flatlimitcode} to the ideals $\In_{\prec}(I_1)$ and $\In_{\prec}(I_2)$ we deduce that
$$
\{\In_{\prec}(f_1(B^{-1}Z)),\ldots,\In_{\prec}(f_n(B^{-1}Z)),\In_{\prec}(g_1(B^{-1}Z)),\ldots,\In_{\prec}(g_m(B^{-1}Z)),z_iz_j\}_{1\leq i\leq r_1,r_1+2\leq j\leq r+1}
$$
is a system of generators for $\In_{\prec}(I_1)\cap \In_{\prec}(I_2)$; hence
$$
\bigg\{f_1(B^{-1}Z),\ldots,f_n(B^{-1}Z),g_1(B^{-1}Z),\ldots,g_m(B^{-1}Z),z_i\bigg(z_{j}-\sum_{k=1}^{r_1}b_{jk}z_k\bigg)\bigg\}_{1\leq i\leq r_1, r_1+2\leq j\leq r+1}
$$
is a Gr\"obner basis for $I$ with respect to $\prec$. By Fact \ref{T:groebflatlimit}, we obtain that
$$
\bigg\{\In_{\prec_{\rho}}(f_1(B^{-1}Z)),\ldots,\In_{\prec_{\rho}}(f_n(B^{-1}Z)),\In_{\prec_{\rho}}(g_1(B^{-1}Z)),
\ldots,\In_{\prec_{\rho}}(g_m(B^{-1}Z)),z_i\bigg(z_{j}-\sum_{k=1}^{r_1}b'_{jk}z_k\bigg)\bigg\}
$$
generate $\In_{\prec_{\rho}}(I)$
for $1\leq i\leq r_1$ and $r_1+2\leq j\leq r+1$. Let $M_i$ be the monomials in $Z'$ such that
$$
g_j=\sum M_i
$$
and the sum is not redundant. Denoting by $\widetilde{w}=\max_i\{w_{\rho'}(M_i)\}$, we have that
\begin{eqnarray}\label{eq:ingB'}
\In_{\prec_{\rho}}(g_j(B^{-1}Z))&=&\In_{\prec_{\rho}}\bigg(\sum M_i(B^{-1}Z)\bigg)\nonumber\\
&=& \sum_{i\,|\,w_{\rho'}(M_i)=\widetilde{w}} \In_{\prec_{\rho}}(M_i(B^{-1}Z))\nonumber\\
&=& \sum_{i\,|\,w_{\rho'}(M_i)=\widetilde{w}} M_i((B')^{-1}Z))\nonumber\\
&=& \In_{\prec_{\rho'}}(g_j)((B')^{-1}Z)
\end{eqnarray}
Moreover, as pointed out before, $B$ and $B'$ do not change the first $r_1+1$ coordinates, hence
\begin{equation}\label{eq:infB'}
\In_{\prec_{\rho'}}(f)((B')^{-1}Z)=\In_{\prec_{\rho}}(f((B)^{-1}Z)).
\end{equation}
Combining \eqref{eq:ingB'} and \eqref{eq:infB'}, we deduce that $[X''\subset \P^r]
\in \Orb([X'\subset \P^r])$. By our hypothesis, $[X\subset \P^r]$ is Chow polystable with respect to $\rho'$;  thus, there exists $C\in \GL_{r+1}$ such that
\begin{equation}\label{eq:X'CX}
[X'\subset \P^r]=C[X\subset \P^r]
\end{equation}
We deduce that $[X''\subset\P^r]\in \Orb([X\subset \P^r])$, and we are done.

Let us finally prove the implication (3) $\Longrightarrow$ (2). Consider $\rho'$ as above (which we will rename $\rho$). Up to translating the weights, we can assume that $w_{r_1+1}=0$. Define $\rho_1$ and $\rho_2$ with weights respectively $w_1^1,\ldots,w_{r+1}^1$ and $w_1^2,\ldots,w_{r+1}^2$ so that
\begin{equation}\label{eq:w1iw2i}
w_{i}^1=
\left\{
\begin{array}{ll}
w_{i} & \text{if }i\leq r_1\\
0 & \text{if }i\geq r_1+1\\
\end{array}
\right.
\quad \text{and}\quad
w_{i}^2=
\left\{
\begin{array}{ll}
0 & \text{if }i\leq r_1\\
w_{i} & \text{if }i\geq r_1+1\\
\end{array}
\right.
\end{equation}
so that $w_{i}^1+w_{i}^2=w_{i}$ for all $i$ and $w(\rho_1)+w(\rho_2)=w(\rho)$. Now, notice that
\begin{equation}\label{eq:xirhoi}
W_{X_1,\rho}(m)=W_{X,\rho_1}(m),\quad W_{X_2,\rho}(m)=W_{X,\rho_2}(m), \quad e_{X_1,\rho}=e_{X,\rho_1} \quad\text{and}\quad e_{X_2,\rho}=e_{X,\rho_2}.
\end{equation}
If $[X\subset \P^r]$ is Chow semistable (resp. stable) with respect to $\rho_1$ and $\rho_2$, i.e.,
%by the Hilbert-Mumford numerical criterion (Fact \ref{Chow-crit})
$$
e_{X,\rho_1}\leq \frac{2d}{r+1}\,w(\rho_1)\quad\text{(resp. $<$)} \quad \text{and} \quad
e_{X,\rho_2}\leq \frac{2d}{r+1}\,w(\rho_2)\quad\text{(resp. $<$)};
$$
then, applying Proposition  \ref{prop:sumsub} and \eqref{eq:xirhoi}, we get that
$$
e_{X,\rho}=e_{X_1,\rho}+e_{X_2,\rho}=e_{X,\rho_1}+e_{X,\rho_2}\leq \frac{2d}{r+1}\,w(\rho_1)+\frac{2d}{r+1}\,w(\rho_2)=\frac{2d}{r+1}\,w(\rho)\quad\text{(resp. $<$)},
$$
or, in other words, that $[X\subset \P^r]$ is Chow semistable (resp. stable) with respect to $\rho$.
The same argument goes through for the Hilbert semistability (resp. stability) replacing Proposition \ref{prop:sumsub} with Lemma \ref{lem:sumhilb}.

It remains to prove  the implication (3) $\Longrightarrow$ (2) for the polystability. Suppose that there exist matrices $A_1,A_2\in \GL_{r+1}$ such that
$$
\lim_{t\ra 0}\rho_1(t)[X\subset \P^r]=A_1[X\subset \P^r]\quad\text{and}\quad \lim_{t\ra 0}\rho_2(t)[X\subset \P^r]=A_2[X\subset \P^r].
$$
By Lemma \ref{lem:flatlimitcode}, the following holds:
$$
\lim_{t\ra 0}\rho(t)[X\subset \P^r] = \lim_{t\ra 0}\rho_2(t)\bigg(\lim_{t\ra 0}\rho_1(t)[X\subset \P^r]\bigg).
$$
Moreover,  each point of $X_2$ is fixed by $A_1$, hence $\rho_2(t)A_1=A_1 \rho_2(t)$. We deduce that
\begin{eqnarray}
\lim_{t\ra 0}\rho(t)[X\subset \P^r] &=& \lim_{t\ra 0}\rho_2(t)\bigg(\lim_{t\ra 0}\rho_1(t)[X\subset \P^r]\bigg)=\lim_{t\ra 0}\rho_2(t)(A_1[X\subset \P^r])\nonumber\\
&=& A_1 \lim_{t\ra 0}\rho_2(t)[X\subset \P^r]=A_1A_2[X\subset \P^r]\nonumber
\end{eqnarray}
and we are done.
\end{proof}

The proof of this criterion suggests us an important remark. First, we need a definition.

\begin{defi}\label{D:replace}
Let $X$ be a quasi wp-stable curve such that $X=X_1\cup X_2$, $k_{X_{1}}=1$, and denote by $p$ the intersection point of $X_1$ and $X_2$. Let $(X'_1,q)$ be a pointed curve, where $q$ is a smooth point, and define the new curve $X'$ by gluing $X'_1$ and $X_2$ so that $p$ is identified with $q$ and represents a separating node of the new curve.
We say that $X'$
\textbf{{is obtained from}} $X$ \textbf{{by replacing}} $X_1$ \textbf{{with}} $(X_1',q)$. Since $\text{Pic}(X)=\text{Pic}(X_1)\times \text{Pic}(X_2)$, it makes sense to introduce another definition. Let $(X,L)$ and $(X'_1,q,L'_1)$ be a couple and a triple where $X$, $X'$ and $q$ are as above, $L\in \text{Pic}(X)$ and $L'_1\in \text{Pic}(X'_1)$. Consider the new curve $X'$ as above and the line bundle $L'$ defined as follows:
$$
L':=(L_1,L_{|X_2})\in\text{Pic}(X'_1)\times \text{Pic}(X_2)=\text{Pic}(X').
$$
We say that the couple $(X',L')$ \textbf{{is obtained from}} $(X,L)$ \textbf{{by replacing}} $X_1$ \textbf{{with}} $(X_1',q,L'_1)$. If $L$ is very ample we will identify $(X,L)$ and $[X\stackrel{|L|}{\inj}\P^{\dim|L|}]$.
\end{defi}

\begin{rmk}\label{rmk:replweight}
Let $X$, $\rho_1$ and $\rho_2$ be as in the proof of the implication (3) $\Longrightarrow$ (2) in Proposition \ref{prop:stab-tail}, and denote by $L:=\OO_X(1)$ (we keep the same notation). Suppose that the system of coordinates $\{z_1,\ldots,z_{r+1}\}$ is of type \eqref{eq:coorcoda} and that it diagonalizes $\rho_1$ and $\rho_2$. By formulas (\ref{eq:xirhoi}), $W_{X,\rho_1}(m)$ (hence $e_{X,\rho_1}$) depends only on the curve $X_1$ and the embedding $L_1:=L_{|X_1}$ in $\cap_{i=r_1+2}^{r+1}\{x_i=0\}$. In other words, if we replace $X_2$ with another curve $X'_2$ so that the embedding $L_1$ in $\cap_{i=r_1+2}^{r+1}\{x_i=0\}$ is the same, then $W_{X,\rho_1}(m)$ does not change.
\end{rmk}

We can use this remark in order to prove a useful result.

\begin{coro}\label{C:stab-repl}
Let $[X_1\subset\P^{r_1}]\in \Hilb_{d_1,g_1}$, $[X_2\subset\P^{r_2}]\in \Hilb_{d_2,g_2}$ and $[X_3\subset\P^{r_3}]\in \Hilb_{d_3,g_3}$ such that $X_i=C_i\cup D_i$ and $k_{C_i}=1$ for $i=1,2,3$. Denote by
$$
\{p_i\}=C_i\cap D_i\quad\text{,}\quad L_i=\OO_{X_i}(1)\quad \text{and}\quad \nu_i=\frac{d_i}{2g_i-2}
$$
for $i=1,2,3$. Suppose that $\nu_1=\nu_2=\nu_3$ and $(X_3,L_3)$ is obtained from $(X_1,L_1)$ by replacing $D_1$ with $(D_2,p_2,{L_2}_{|D_2})$. If $[X_1\subset\P^{r_1}]$ and $[X_2\subset\P^{r_2}]$ are Chow semistable (resp. polystable, stable), then $[X_3\subset\P^{r_3}]$ is Chow semistable (resp. polystable, stable). The same holds for the Hilbert semistability (resp. polystability, stability).
\end{coro}
\begin{proof}
We will first prove the statement for the Chow (semi-, poly-) stability; the case of Hilbert (semi-, poly-) stability is completely analogous. Denoting by $s_i=h^0(C_i,{L_i}_{|C_i})-1$ for $i=1,2,3$, we have that $s_1=s_3$. By Proposition \ref{prop:stab-tail}, it suffices to consider one-parameter subgroups $\rho_1$ and $\rho_2$ diagonalized by coordinates $(x_1,\ldots,x_{r_3+1})$ of $\P^{r_3}$ such that
$$
\langle C_3\subset \P^{r_3}\rangle=\bigcap_{i=s_3+2}^{r_3+1}\{x_i=0\}\quad\text{and}\quad\langle D_3\subset \P^{r_3}\rangle=\bigcap_{i=1}^{s_3}\{x_i=0\}
$$
with weights
$$
\rho_1(t)\cdot x_i=
\left\{
\begin{array}{ll}
t^{w_{i}}x_i & \text{if }i\leq s_3\\
x_i & \text{if }i\geq s_3+1\\
\end{array}
\right.
\quad\text{and}\quad
\rho_2(t)\cdot x_i=
\left\{
\begin{array}{ll}
x_i & \text{if }i\leq s_3+1\\
t^{w_{i}}x_i & \text{if }i\geq s_3+2\\
\end{array}
\right.
$$
By hypothesis $C_1=C_3$ and $L_{|C_1}=L_{|C_3}$, hence we can find coordinates $(x'_1,\ldots,x'_{r_1+1})$ in $\P^{r_1}$ such that for each $i=1,\ldots,s_1+1$
$$
{x'_i}_{|C_1}={x_i}_{|C_3}
$$
where we identify %$V_{C_1}:=
$\langle C_1\subset \P^{r_1}\rangle$ with %$V_{C_3}:=
$\langle C_3\subset \P^{r_3}\rangle$. If $\rho'_1$ is a one-parameter subgroup diagonalized by $(x'_1,\ldots,x'_{r_1+1})$ with weights
$$
w'_{i}=
\left\{
\begin{array}{ll}
w_{i} & \text{if }i\leq r_1\\
0 & \text{if }i\geq r_1+1\\
\end{array}
\right.
$$
then $w(\rho)=w(\rho')$ and
$$
e_{X_1,\rho'_1}=e_{C_1,\rho'_1}=e_{C_3,\rho_1}=e_{X_3,\rho_1},
$$
since the embeddings $C_1\inj \langle C_1\subset \P^{r_1}\rangle$ and $C_3\inj \langle C_3\subset \P^{r_3}\rangle$ are the same (see Remark \ref{rmk:replweight}).
Notice that the equalities $\nu_1=\nu_2=\nu_3$ are equivalent to
$$
\frac{d_1}{r_1+1}=\frac{d_2}{r_2+1}=\frac{d_3}{r_3+1}.
$$
Since $[X_1\subset \P^{r_1}]$ is Chow semistable (resp. stable) by assumption, the Hilbert-Mumford numerical criterion (Fact \ref{Chow-crit}) gives that
$$
e_{X_1,\rho'_1}\leq \frac{2d_1}{r_1+1}\,w(\rho'_1)\quad\text{(resp. $<$)}.
$$
Combining this inequality with the previous relations, we obtain
$$
e_{X_3,\rho_1}=e_{X_1,\rho'_1}\leq \frac{2d_1}{r_1+1}\,w(\rho'_1)=\frac{2d_3}{r_3+1}\,w(\rho_1)\quad\text{(resp. $<$)},
$$
i.e. that  $[X_3\subset \P^{r_3}]$ is Chow semistable (resp. stable) with respect to $\rho_1$. Analogously, using that $[X_2\subset \P^{r_2}]$ is Chow semistable (resp. stable), we obtain that
$[X_3\subset \P^{r_3}]$ is Chow semistable (resp. stable) with respect to $\rho_2$.

It remains to prove the polystability. Suppose that $[X_3\subset \P^{r_3}]$ is strictly Chow semistable, so that there exists a one-parameter subgroup, for example $\rho_1$ as above, such that
$$
e_{X_3,\rho_1}=\frac{2d_3}{r_3+1}\,w(\rho_1)
$$
and consider the one-parameter subgroup $\rho'_1$ as above. Let $I_1$ and $I_3$ be the ideals of $C_1$ and $C_3$ respectively in $\P^{r_1}$ and $\P^{r_3}$. It is easy to check that $I_1\cap k[x_1,\ldots,x_{s_1+1}]=I_3\cap k[x_1,\ldots,x_{s_3+1}]$,
hence $\In_{\prec_{\rho_1'}}(I_1)=\In_{\prec_{\rho_1}}(I_3)$. Since $[X_1\subset \P^{r_1}]$ is Chow polystable by hypothesis, there exists $A=(a_{ij})\in \GL_{r_1+1}$ such that
$$
\lim_{t\ra 0}\rho'_1(t)[X_1\subset \P^{r_1}]=A[X_1\subset \P^{r_1}].
$$
Define the matrix $A'=(a'_{ij})\in \GL_{r_3+1}$ as follows:
$$
a'_{ij}=
\left\{
\begin{array}{lll}
a_{ij} & \text{if }1\leq i\leq s_1\text{ and }1\leq j\leq s_1,\\
1 & \text{if }(s_1+1\leq i\leq r_3+1\text{ or }s_1+1\leq j\leq r_3+1)\text{ and }i=j,\\
0 & \text{if }(s_1+1\leq i\leq r_3+1\text{ or }s_1+1\leq j\leq r_3+1)\text{ and }i\neq j.\\
\end{array}
\right.
$$
Now, we notice that $\rho_1$ fixes each point of $D_3\subset \P^{r_3}$. Moreover,  the actions of $\rho_1$ and $\rho'_1$ on $C_3\subset \P^{r_3}$ coincide, hence by Lemma \ref{lem:flatlimitcode} and Corollary \ref{C:flatlimit} we get
$$
\lim_{t\ra 0}\rho_1(t)[X_3\subset \P^{r_3}]=A'[X_3\subset \P^{r_1}].
$$
This implies that $[X_3\subset \P^{r_3}]$ is Chow polystable with respect to $\rho_1$.
Analogously, using that $[X_2\subset \P^{r_2}]$ is Chow polystable, we obtain that   $[X_3\subset \P^{r_3}]$ is Chow polystable with respect to $\rho_2$, and we are done.

\end{proof}

\section{Elliptic tails and Tacnodes with a line}\label{S:ell-tails}

According to Corollary \ref{C:quasi-wp-stable}\eqref{C:quasi-wp-stable1}, if $[X\subset \P^r]\in \Hilb_d$ is Chow semistable  with $X$ connected and $2(2g-2)<d \leq 4(2g-2)$, then $X$ is quasi-wp-stable. The aim of this section is to investigate whether $X$ can have elliptic tails or tacnodes with a line.

Our first result concerns special elliptic tails (in the sense of Definition \ref{D:elltails-types}).

%We begin this section by investigating the relation between the presence of cusps and the presence of special elliptic tails (in the sense of Definition \ref{D:elltails-types}).
%We refer to \cite{HMo}.

\begin{thm}\label{T:spec-ell}
Let $[X\subset \P^r] \in  \Hilb_d$ with $X$ connected and $2(2g-2)<d$. Assume that one of the following two conditions is satisfied, namely:
\begin{enumerate}[(i)]
\item \label{T:spec-ell1} $d<4(2g-2)$ and $[X\subset \P^r] \in \Ch^{-1}(\Chow_d^{ss})$;
\item \label{T:spec-ell2} $d=4(2g-2)$ and $[X\subset \P^r] \in \Hilb_d^{ss}$.
\end{enumerate}
Then $X$ does not have any special elliptic tails.
\end{thm}
\begin{proof}
%The proof is inspired by \cite[Lemmas 1 and 4]{HMo}.
According to the hypothesis on $[X\subset \P^r]$, we get that $X$ is quasi-wp-stable
by Corollary \ref{C:quasi-wp-stable}\eqref{C:quasi-wp-stable2} and that $L:=\OO_X(1)$ is very ample, non-special and balanced of degree $d$
by the Potential pseudo-stability Theorem \ref{teo-pstab}.

Suppose that $X$ has a special elliptic tail, i.e., $X=F\cup C$ where $F\subseteq X$ is an irreducible subcurve of arithmetic genus $1$, $C\subseteq X$ is a connected subcurve of
arithmetic genus $g-1$ and $p:=F\cap C$ is a nodal point of $X$ which is a smooth point of both $F$ and $C$, and
$L_{|F}=\OO_F(\nu p)$ for some $\nu\in \N$.
We want to show, by contradiction,  that $[X\subset \P^r]\not\in \Ch^{-1}(\Chow_d^{ss})$ (resp. $\Hilb_d^{ss}$) if \eqref{T:spec-ell1} (resp. \eqref{T:spec-ell2}) holds.
Since $\Ch^{-1}(\Chow_d^{ss})$ and $\Hilb_d^{ss}$ are open in $\Hilb_d$, we can assume that $F$ is a smooth elliptic tail.

Note that the basic inequality \eqref{E:basineq-multideg} applied to $F$ gives that $\nu:=\deg L_{|F}\leq 4$ if \eqref{T:spec-ell1} holds and $\nu=4$ if \eqref{T:spec-ell2} holds.

Consider the linear spans $V_F:=\langle F\rangle $ of $F$ and $V_C:=\langle C\rangle$ of $C$ on $\P^{r}=\P(V)$. It follows from Riemann-Roch theorem,
using that $L$ (hence $L_{|C}$ and $L_{|F}$) is non-special, that $V_F$ has dimension $\nu-1$ and $V_C$ has dimension $d-\nu-(g-1)=r-\nu+1$.
Therefore, we can choose a system of coordinates $\{x_1,\ldots, x_{r+1}\}$ of type \eqref{eq:coorcoda}, i.e. such that
\begin{equation}\label{eq:coorcodaspec}
V_F=\bigcap_{i=\nu+1}^{r+1}\{x_i=0\}\quad \text{and}\quad V_C=\bigcap_{i=1}^{\nu-1}\{x_i=0\}.
\end{equation}
Hence $p$ is the point where all the $x_i$'s vanish except $x_{\nu}$. For $1\leq i\leq \nu$, we will identify  $x_i$ with the section of
$H^0(F, L_{|F})$ it determines, and we will denote by $\ord_{p}(x_i)$
the order of vanishing of $x_i$ at $p$. By Riemann-Roch theorem applied to
the line bundles $L_{|F}(-ip)=\OO_F((\nu-i)p)$ for $i=0, \ldots, \nu$, we may choose the first $\nu$ coordinates $\{x_1,\ldots,x_{\nu}\}$ of $V$ so
that
\begin{equation}\label{E:order2}
\ord_{p}(x_i)=\begin{cases}
\nu &\text{ if } i=1, \\
\nu-i &\text{ if } 2\leq i\leq \nu. \\
\end{cases}
%\end{aligned}
\end{equation}
Consider the one-parameter subgroup $\rho: \Gm \to \GL(V)$ which, in the above coordinates, has the diagonal form
$\rho(t)\cdot x_i = t^{w_i}x_i \:  \text{ for } i=1,\ldots, r+1,$ with weights $w_i$ such that
\begin{equation}\label{E:weights2}
\begin{sis}
& w_1=w_{\rho}(x_1)=0 & \\
& w_i=w_{\rho}(x_i)=i & \:\text{ for } 2\leq i\leq \nu, \\
& w_j=w_{\rho}(x_j)=\nu& \: \: \text{ for } j\geq \nu+1.
\end{sis}
\end{equation}
%Note that, combining \eqref{E:order} and \eqref{E:weights}, we have that
%\begin{equation}\label{E:ord-weights}
%\ord_P(x_i)+w_{\rho}(x_i)=\nu  \:\text{ for } 2\leq i\leq \nu.
%\end{equation}
The proof of \cite[Lemma 1]{HMo} extends verbatim to our case and gives that
$$W_{X,\rho}(m)=m^2\left[d\nu-\frac{\nu^2}{2}\right]+m\left[\frac{3\nu}{2}-g\nu\right]-1\text{ for any } m\geq 2.$$
In particular, the normalized leading coefficient of $W_{X,\rho}$ is equal to
\begin{equation}\label{E:lead-coeff}
e_{X,\rho}=2d\nu-\nu^2.
\end{equation}
From \eqref{E:weights2}, it is easy to compute that the total weight of $\rho$ is equal to
\begin{equation}\label{E:tot-weight2}
w(\rho)=\sum_{i=2}^{\nu} i +\nu(r+1-\nu)=\nu(r+1)+\frac{-\nu^2+\nu-2}{2}.
\end{equation}
If \eqref{T:spec-ell1} holds, i.e. if $v:=\frac{d}{2g-2}<4$, combining
\eqref{E:lead-coeff}, \eqref{E:tot-weight2} and the fact that $r=d-g$, we get that
\begin{equation*}\label{E:equiv-ineq}
2d\frac{w(\rho)}{r+1}=2d\nu+\frac{2v}{2v-1}(-\nu^2+\nu-2)<2 d\nu+\frac{8}{7}(-\nu^2+\nu-2)=2d\nu-\nu^2-\frac{(\nu-4)^2}{7}\leq
e_{X,\rho}.
\end{equation*}
This implies that $[X\subset \P^r]\not \in \Ch^{-1}(\Chow_d^{ss})$ by Fact \ref{Chow-crit}.

On the other hand, if \eqref{T:spec-ell2} holds, i.e. if $v:=\frac{d}{2g-2}=4$ (hence $\nu=4$), then
$$W_{X,\rho}(m)=(4d-8)m^2+(6-4g)m-1> (4d-8)m^2+(5-4g)m=\frac{w(\rho)}{r+1}mP(m) \:\: \text{�for } \: m\gg 0.$$
This implies that $[X\subset \P^r]\not \in \Hilb_d^{ss}$ by Fact \ref{Hilb-crit}.

\end{proof}

As a corollary, if $d=4(2g-2)$ and $X\subset\P^r$ admits a special elliptic tail, then $[X\subset \P^r]$ is strictly Chow semistable with respect to the 1ps $\rho$ as in \eqref{E:weights2}. It would be interesting to find the equations of $F$ in its linear span $\langle F\rangle$ in order to determine the flat limit
$$
[X_0\subset \P^r]=\lim_{t\ra 0}\rho(t)[X\subset \P^r]
$$
by using Corollary \ref{C:flatlimit}. This is not very difficult to do (we leave it to the reader as an exercise). One obtains that $X_0$ is given by the union of $C$ and a special cuspidal elliptic tail, which we denote by $F_0$. Here we do not use this fact, we consider directly $[X_0\subset \P^r]\in \Hilb_d$ where $X_0=F_0\cup C$ and $F_0$ is cuspidal and special. Using the same system of coordinates $\{x_1,\ldots,x_{r+1}\}$ in $\P^{r}$ as in Theorem \ref{T:spec-ell}, we can parameterize $F_0$ by
$$
[s,t]\in\P^1\mapsto [s^4,s^2t^2,st^3,t^4,0,\ldots,0],
$$
so $F$ is special since $\ord_p(x_1)=4$, the cusp $q$ is the point $[1,0,0,\ldots,0]$ and $\rho$ stabilizes $[X_0\subset \P^r]$. We now compute the basins of attraction (see \eqref{Sec:bas-attra}) of $[X_0\subset \P^r]$ with respect to $\rho$ and to $\rho^{-1}$.

%are ready to show explicitly the relation between the cusps and the special elliptic tails we outlined at the beginning of this section by studying $A_{\rho}([X_0\subset \P^r])$ and $A_{\rho^{-1}}([X_0\subset \P^r])$ (see \cite[Lemma 4]{HMo}).

\begin{thm}\label{T:basin-cusps}
Let $[X_0\subset \P^r]\in \Hilb_d$ as above and let $\rho$ as in \eqref{E:weights2}. Then $A_{\rho}([X_0\subset \P^r])$ (resp. $A_{\rho^{-1}}([X_0\subset \P^r])$) contains smoothings of the cusp (resp. separating node), but not smoothings of the separating node (resp. cusp).
\end{thm}
\begin{proof}
The proof is similar to \cite[Lemma 4]{HMo} and uses
the same techniques used to prove CLAIM 3 in Theorem \ref{T:spec-clos-orb}. We have that the tangent space $T_{X,q}$ is given by $\langle x_2,x_3\rangle$, so that the completion of the local ring $\OO_{X,q}$ is given by $k[[u,v]]/(u^2-v^3)$, where $u=x_3/x_1$ and $v=x_2/x_1$. Since $\rho(t)\cdot x_1=x_1$, $\rho(t)\cdot x_2=t^2x_2$ and $\rho(t)\cdot x_3=t^3x_3$, we obtain that $\rho(t)\cdot u=t^3 u$ and $\rho(t)\cdot v=t^2 v$. We recall that $\Def_{(X,q)}$ has a semiuniversal ring equal to $k[[a,b]]$ with universal family $k[[u,v,a,b]]/(u^2-v^3-av-b)$. This implies that $\rho(t)\cdot a=t^4a$ and $\rho(t)\cdot b=t^6b$, so that $A_{\rho}([X_0\subset \P^r])$ contains smoothings of $q$. If we consider the action of $\rho$ on the universal family of the separating node $p$, we obtain non-positive weights, hence $A_{\rho^{-1}}([X_0\subset \P^r])$ does not contain smoothings of $p$. The converse result holds if we consider $A_{\rho^{-1}}([X_0\subset \P^r])$.
\end{proof}

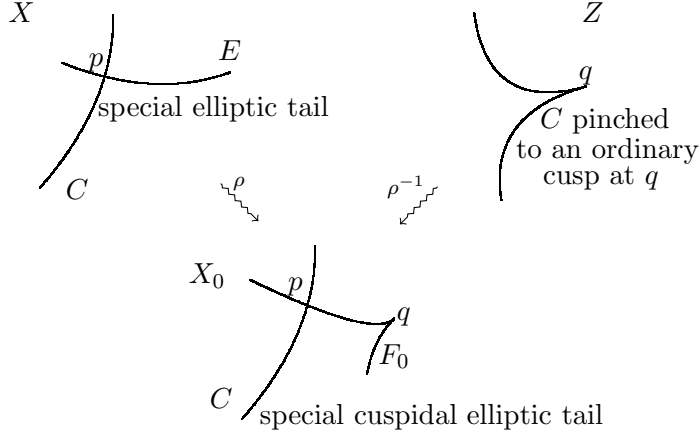
\begin{figure}[h]
\begin{center}
\unitlength .5mm % = 2.845pt
\linethickness{0.4pt}
\ifx\plotpoint\undefined\newsavebox{\plotpoint}\fi % GNUPLOT compatibility
\begin{picture}(172,110)(25,80)
\qbezier(81.75,119.5)(114.625,103.625)(120,109.25)
\qbezier(120,109.25)(114.625,104.125)(112.75,94.5)
\qbezier(79.5,82.5)(99.25,105)(99,128.5)
\qbezier(25.75,144)(45.5,166.5)(45.25,190)
\qbezier(31.75,177.25)(56.5,167.5)(76.25,174.75)
\qbezier(141.25,190.5)(144.875,163.75)(171,171)
\qbezier(171,171)(144.75,163.625)(148.5,140.75)
\put(21,190.75){\makebox(0,0)[cc]{$X$}}
\put(172.75,190.75){\makebox(0,0)[cc]{$Z$}}
\put(70.25,120.75){\makebox(0,0)[cc]{$X_0$}}
\put(76.25,180.75){\makebox(0,0)[cc]{$E$}}
\put(72,165){\makebox(0,0)[cc]{special elliptic tail}}
\put(40.75,177.25){\makebox(0,0)[cc]{$p$}}
\put(122.5,109.75){\makebox(0,0)[cc]{$q$}}
\put(171,174){\makebox(0,0)[cc]{$q$}}
\put(130,83){\makebox(0,0)[cc]{special cuspidal elliptic tail}}
\put(35.75,143.75){\makebox(0,0)[cc]{$C$}}
\put(176,162){\makebox(0,0)[cc]{$C$ pinched}}
\put(178,154){\makebox(0,0)[cc]{to an ordinary }}
\put(175,146.75){\makebox(0,0)[cc]{cusp at $q$}}
\put(120,99.25){\makebox(0,0)[cc]{$F_0$}}
\put(93.75,117.75){\makebox(0,0)[cc]{$p$}}
\put(74,88.25){\makebox(0,0)[cc]{$C$}}
\put(75,141){\makebox(0,0)[cc]{$\xxrsquigarrow{\phantom{aaa}\rho}$}}
\put(122.5,140.75){\makebox(0,0)[cc]{$\xxlsquigarrow{\phantom{a}\rho^{-1}}$}}
\end{picture}
\caption{The basin of attraction of a curve $X_0$ with a special cuspidal elliptic tail $F_0$.}
\label{F:basin-cusptail}
 %with respect to $\rho$ and $\rho^{-1}$.}
\end{center}
\end{figure}

%The original result of this section is that there is a similar relation between tacnodes with a line and non-special elliptic tails.
Our second result concerns tacnodes with a line (in the sense of \ref{Con:singu}).

\begin{thm}\label{T:tac-line}
If $\frac{7}{2}(2g-2)<d$ and $[X\subset \P^r] \in \Ch^{-1}(\Chow_d^{ss})\subset \Hilb_d$ with $X$ connected, then
$X$ does not have tacnodes with a line.
\end{thm}
\begin{proof}
This follows from \cite[Prop. 1.0.6, Case 2]{Gie}; however, for the reader's convenience and also because we will need
it later, we give a sketch of the proof.

Using the hypothesis on $[X\subset \P^r]$, we get that $X$ is quasi-wp-stable
by Corollary \ref{C:quasi-wp-stable}\eqref{C:quasi-wp-stable2} and $L:=\OO_X(1)$ is very ample, non-special and balanced of degree $d$ by the Potential pseudo-stability Theorem \ref{teo-pstab}.
Suppose that $X$ has a tacnode with a line, i.e. we can write $X=Y\cup E$ with $E\cong \P^1$, $\{p\}= E\cap Y$ is a tacnode
of $X$ and $\deg L_{|E}=1$. We want to show, by contradiction, that $[X\subset \P^r]\not \in \Ch^{-1}(\Chow_d^{ss})$
if $\frac{7}{2}(2g-2)<d$.

Since $E$ and $Y$ are tangent in $p\in E\cap Y$, we can choose coordinates $\{x_1,\ldots, x_{r+1}\}$ of $H^0(X, L)$
so that
$$\begin{sis}
& \ord_p({x_i}_{|E})\geq 2 \: \text{ and } \: \ord_p({x_i}_{|Y})\geq 2 & \text{ for any } 1\leq i\leq r-1, \\
& \ord_p({x_r}_{|E})=\ord_p({x_r}_{|Y})= 1,\\
& \ord_p({x_{r+1}}_{|E})=\ord_p({x_{r+1}}_{|Y})= 0,\\
\end{sis}$$
where ${x_i}_{|E}$ (resp. ${x_i}_{|Y}$) denotes the image of $x_i\in H^0(X,L)$ via the restriction map
$H^0(X,L)\to H^0(E, L_{|E})$ (resp. $H^0(X,L)\to H^0(Y,L_{|Y})$), and $\ord_p$ denotes the order of vanishing of a section
at the point $p$ (considered as a smooth point of $E$ and of $Y$).

Consider now the one-parameter subgroup $\rho: \Gm \to \GL(V)$ which, in the above coordinates, has the diagonal form
$\rho(t)\cdot x_i = t^{w_i}x_i \:  \text{ for } i=1,\ldots, r+1,$ with weights $w_i$ such that
\begin{equation}\label{E:weights0}
\begin{sis}
& w_i=w_{\rho}(x_i)=0 \:  \: \text{ for } 1\leq i\leq r-1, \\
& w_r=w_{\rho}(x_r)=1, \\
& w_{r+1}=w_{\rho}(x_{r+1})=2.
\end{sis}
\end{equation}
Clearly, the total weight of $\rho$ is equal to $w(\rho)=3.$
The proof of \cite[Prop. 1.0.6, Case 2]{Gie} gives that
\begin{equation}\label{E:Gie-peso}
e_{X,\rho}\geq 7.
\end{equation}
Therefore, if $\frac{7}{2}<v:=\frac{d}{2g-2}$ then we have that
$$2w(\rho) \frac{d}{r+1}= 6\frac{2v}{2v-1}<6\cdot \frac{7}{6}\leq e_{X,\rho}.$$
which implies that $[X\subset \P^r]\not \in \Ch^{-1}(\Chow_d^{ss})$ by Fact \ref{Chow-crit}.
\end{proof}

Combining  Corollary \ref{C:quasi-wp-stable}\eqref{C:quasi-wp-stable1} with Theorem \ref{T:tac-line} and Theorem \ref{T:spec-ell}, we get the following

\begin{coro}\label{C:quasi-wp-inter}
Let $[X\subset \P^r] \in  \Hilb_d$ with $X$ connected and $2(2g-2)<d$. Assume that one of the following two conditions is satisfied
\begin{enumerate}[(i)]
\item  $\frac{7}{2}(2g-2)<d<4(2g-2)$ and $[X\subset \P^r] \in \Ch^{-1}(\Chow_d^{ss})$;
\item  $d=4(2g-2)$ and $[X\subset \P^r] \in \Hilb_d^{ss}$.
\end{enumerate}
Then $X$ is a quasi-wp-stable curve without tacnodes nor special elliptic tails.
\end{coro}

We turn now our attention to the stability of arbitrary elliptic tails.

\begin{rmk}\label{rmk:code3}
Suppose that $[X\subset \P^r]$ has an elliptic tail, i.e. we can write $X=Y\cup F$ where $F\subseteq X$ is a connected subcurve of arithmetic genus $1$, $Y\subseteq X$ is a connected subcurve of arithmetic genus $g-1$ and $F\cap Y=\{p\}$ where $p$ is a nodal point of $X$ which is smooth in $F$ and $Y$. Our goal is to determine under which hypothesis $[X\subset \P^r]$ is Hilbert or Chow semistable. Using Corollary \ref{C:quasi-wp-stable}\eqref{C:quasi-wp-stable1} and the Potential pseudo-stability Theorem \ref{teo-pstab}, we can assume that $X$ is quasi-wp-stable and $L:=\OO_X(1)$ is very ample, non-special and balanced of degree $d$.

Let $\nu:=\deg L_{|F}$. Since $L$ (and hence $L_{|F}$) is very ample by construction, we must have $\nu\geq 3$. On the other hand, by applying the basic inequality
\eqref{E:basineq-multideg} to the subcurve $F\subseteq X$ we get
\begin{equation*}\label{E:d-nu}
\left|\nu-\frac{d}{2g-2}\right|\leq \frac{1}{2},
\end{equation*}
so that
$$
\left\{
\begin{array}{ll}
\nu \leq 3 & \text{if } d<\frac{7}{2}(2g-2),\\
\nu = 3,4 & \text{if } d=\frac{7}{2}(2g-2),\\
\nu \geq 4 & \text{if } d>\frac{7}{2}(2g-2).\\
\end{array}
\right.
$$
If $\nu=4$ then there exists an isotrivial specialization $(X,L)\sp (X', L')$ (in the sense of Definition \ref{D:speciali}), where $X'=Y\cup E\cup F$ is obtained from $X$ by blowing
up the node $p$ (i.e. inserting an exceptional component $E\cong \P^1$ meeting $Y$ and $F$ in one point) and $L'$ is a properly balanced line bundle on $X'$
such that $\deg L'_{|Y}=\deg L_{|Y}= d-4$, $\deg L'_{|E}=1$ and  $\deg L'_{|F}=3$. Using Theorem \ref{bal-pos} from the Appendix, it is easy to see that $L'$ is non-special and
very ample; therefore there exists $[X'\subset \P^r]\in \Hilb_d$ such that $\OO_{X'}(1)=L'$. Thus the basic inequality \eqref{E:basineq-multideg} and Theorem \ref{T:spec-clos-orb} imply
\begin{enumerate}
	\item $[X'\subset \P^r]\in \overline{\Orb([X\subset \P^r])}$;
	\item if $d=\frac{7}{2}(2g-2)$ then $[X\subset \P^r]\in \Hilb_d^{ss}$ (resp. $[X\subset \P^r]\in\Ch^{-1}(\Chow_d^{ss})$) if and only if $[X'\subset \P^r]\in \Hilb_d^{ss}$ (resp. $[X\subset \P^r]\in\Ch^{-1}(\Chow_d^{ss})$).
\end{enumerate}
\end{rmk}

%\begin{coro}\label{C:coda3}
%Let $[X\subset \P^r]\in \Hilb_d$ and $[X'\subset \P^r]\in \Hilb_d$ as in Remark \ref{rmk:code3}. Then $[X'\subset \P^r]\in \overline{\Orb([X\subset \P^r])}$. Moreover if $d=\frac{7}{2}(2g-2)$ then $[X\subset \P^r]\in \Hilb_d^{ss}$ (resp. $[X\subset \P^r]\in\Ch^{-1}(\Chow_d^{ss})$) if and only if $[X'\subset \P^r]\in \Hilb_d^{ss}$ (resp. $[X\subset \P^r]\in\Ch^{-1}(\Chow_d^{ss})$).
%\end{coro}

\begin{thm}\label{T:ell-curves}
Let $[X\subset \P^r] \in  \Hilb_d$ with $X$ connected and assume that one of the following conditions is satisfied
\begin{enumerate}[(i)]
\item \label{T:ell-curves1} $2(2g-2)<d<\frac{7}{2}(2g-2)$ and $[X\subset \P^r] \in \Ch^{-1}(\Chow_d^{ss})$;
\item \label{T:ell-curves2} $d=\frac{7}{2}(2g-2)$ and $[X\subset \P^r] \in \Hilb_d^{ss}$.
\end{enumerate}
Then $X$ does not have elliptic tails, i.e. $X_{\rm ell}=\emptyset$.
\end{thm}
\begin{proof}
Denote by $L:=\OO_X(1)$. We want to show, by contradiction, that
$[X\subset \P^r]\not\in \Ch^{-1}(\Chow_d^{ss})$ (resp. $[X\subset \P^r]\not\in \Hilb_d^{ss}$) if \eqref{T:ell-curves1} (resp. \eqref{T:ell-curves2}) holds.
Note that since $\Ch^{-1}(\Chow_d^{ss})$ and $\Hilb_d^{ss}$ are open in $\Hilb_d$, we can assume that $F$ is a generic connected curve of arithmetic genus one, and in particular that it is a
smooth elliptic curve. Moreover,  by Remark \ref{rmk:code3} we can assume that $\deg\,L_{|F}=3$, so that we can write
\begin{equation}\label{E:rest-F}
L_{|F}=\OO_F(2p+q)
\end{equation}
for some (uniquely determined) $q\in F$. By our generic assumption on $F$, we can assume that $q\neq p$.
%\footnote{This is where our proof differs from \cite[Lemma 1]{HMo}, which deals only with the case $q=p$.}

Consider now the linear spans $V_F:=\langle F\rangle $ of $F$ and $V_Y:=\langle Y\rangle$ of $Y$ on $\P^{r}=\P(V)$. It follows from Riemann-Roch theorem,
using that $L$ (hence $L_{|Y}$ and $L_{|F}$) is non-special,  that  $V_F$ has dimension $2$ and $V_Y$ has dimension $d-3-(g-1)=r-2$.
Therefore, we can choose coordinates $\{x_1,\ldots, x_{r+1}\}$ of $V$ such that
$$
V_F=\bigcap_{i=4}^{r+1}\{x_i=0\}\quad\text{,}\quad V_Y=\bigcap_{i=1}^{2}\{x_i=0\}
$$
and $p$ is the point where all the $x_i$'s vanish except $x_{3}$. For $1\leq i\leq 3$, we will identify  $x_i$ with the section of $H^0(F, L_{|F})$ it determines and we will denote by $\ord_p(x_i)$ the order
of vanishing of $x_i$ at $p$. By Riemann-Roch theorem applied to
the line bundles $L_{|F}(-ip)$ for $i=0, \ldots, 3$ and using \eqref{E:rest-F} with $q\neq p$, we may choose the first three coordinates $\{x_1,\ldots,x_{3}\}$ of $V$ so
that
\begin{equation}\label{E:order}
%\begin{aligned}
%& \ord_p(x_1)=\begin{cases}
%\nu & \text{ if } q=p,\\
%\nu-1 & \text{ if } q\neq p, Ê
%\end{cases} \\
%&
\ord_p(x_i)=3-i \text{ for } 1\leq i\leq 3. \\
%\end{aligned}
\end{equation}
Consider the one-parameter subgroup $\rho: \Gm \to \GL(V)$ which, in the above coordinates, has the diagonal form
$\rho(t)\cdot x_i = t^{w_i}x_i \:  \text{ for } i=1,\ldots, r+1,$ with weights $w_i$ such that
\begin{equation}\label{E:weights}
\begin{sis}
& w_1=w_{\rho}(x_1)=-2, \\
& w_2=w_{\rho}(x_2)=-1, \\
& w_j=w_{\rho}(x_j)=0& \: \: \text{ for } j\geq 3.
\end{sis}
\end{equation}
%Note that, combining \eqref{E:order} and \eqref{E:weights}, we have that
%\begin{equation}\label{E:ord-weights}
%\ord_p(x_i)+w_{\rho}(x_i)=\nu  \:\text{ for } 2\leq i\leq \nu.
%\end{equation}
The total weight of  $\rho$ is equal to
\begin{equation}\label{E:tot-weight}
w(\rho)=-2-1=-3.
\end{equation}
We want now to compute the polynomial $W_{X,\rho}(m)$. By Lemma \ref{lem:sumhilb}
$$
W_{X,\rho}(m)=W_{F,\rho}(m)+W_{Y,\rho}(m)
$$
since $w_3=0$. Moreover,  each coordinate of $V_Y=\{x_1=x_{2}=0\}$ has weight 0, hence $W_{Y,\rho}(m)=0$ and
\begin{equation}\label{E:form-W}
W_{X,\rho}(m)=W_{F,\rho}(m)
\end{equation}

In order to compute the polynomial $W_{F,\rho}(m)$, consider the embedding of $F$ as a cubic
curve in $\P^2=\P(H^0(F, L_{|F})^{\vee})$ given by the complete linear system $|L_{|F}|$ .
Let $f\in k[x_1,x_2,x_3]_3$ be a homogenous polynomial of degree $3$ defining $F$.
The conditions \eqref{E:order} on the order at $p$ of the coordinates $\{x_1,x_2,x_3\}$ translate directly into
conditions on the polynomial $f$. More specifically, the point $p$ has coordinates $(0,0,1)$ and
$p\in F$ if and only if the coefficient of $x_3^3$ in $f$ is equal to zero. The condition that $\ord_p x_1\geq 2$
says that the tangent space of $F$ at $p$ must have equation $\{x_1=0\}$ which translates into the fact that the coefficient of $x_3^2x_2$ in $f$ is zero while the coefficient of $x_3^2x_1$  is not zero. Finally, we have that $\ord_p(x_1)=2$ (i.e. $p$ is not a flex point of $F$) if and only if the coefficient of $x_2^2x_3$ in $f$ is not zero.
Summing up, every polynomial $f$ such that the coordinates $\{x_1,x_2,x_3\}$ satisfy \eqref{E:order}
is of the form
\begin{equation}\label{E:pol-f}
f=a_{300} x_1^3+a_{210}x_1^2x_2+a_{201}x_1^2x_3+a_{120}x_1x_2^2+a_{102}x_1x_3^2+a_{111}x_1x_2x_3+
a_{030}x_2^3+a_{021}x_2^2x_3,
\end{equation}
where $a_{102}\neq 0$ and $a_{021}\neq 0$.

Because of the choice \eqref{E:weights} of the one-parameter subgroup $\rho$, it is easy to see that the monomial $x_2^2 x_3$ has the maximal $\rho$-weight among all
the monomials appearing in the above equation \eqref{E:pol-f} of $f$. Moreover, the same monomial
appears with non-zero coefficient in $f$. Therefore, a collection of $3m$ monomials that compute the polynomial
$W_{F,\rho}(m)$
%, according to the formula \eqref{E:fun-W-F},
is represented by the monomials which are not divisible
by $x_2^2x_3$, namely
$$\left\{\{x_1^{m-k}x_3^k\}_{0\leq k\leq m}, \{x_2 x_1^{m-1-h}x_3^h\}_{0\leq h\leq m-1},
\{x_2^2 x_1^{m-2-j}x_2^j\}_{0\leq j\leq m-2}\right\}.$$
We get
\begin{equation*}
W_{F,\rho}(m)=\sum_{k=0}^m [w_1(m-k)+kw_3] +\sum_{h=0}^{m-1} [w_2+(m-1-h)w_1+hw_3]+\sum_{j=0}^{m-2}[(j+2)w_2+(m-2-j)w_1]=
\end{equation*}
\begin{equation}\label{E:expl-WF}
=\left[\frac{3}{2}w_1+\frac{1}{2}w_2+w_3\right]m^2 +\left[-\frac{3}{2}w_1+\frac{3}{2}w_2 \right] m +[w_1-w_2]=-\frac{7}{2}\,m^2 +\frac{3}{2}\,m-1.
\end{equation}Combining \eqref{E:expl-WF} with \eqref{E:form-W}, we get
\begin{equation}\label{E:expl-W}
W_{X,\rho}(m)=-\frac{7}{2}\,m^2 +\frac{3}{2}\,m-1.
\end{equation}
In particular, the normalized leading coefficient of $W_{X,\rho}(m)$ is equal to
\begin{equation}\label{E:3-nlm-WF}
e_{X, \rho}=-7.
\end{equation}

Let us first assume that condition \eqref{T:ell-curves1}� holds, and in particular that $\displaystyle v:=\frac{d}{2g-2}<\frac{7}{2}$.
The right hand side of the numerical criterion for Chow (semi)stability (see Fact \ref{Chow-crit}) can be bounded above as follows:
\begin{equation}\label{E:3-RHS}
2d\,\frac{w(\rho)}{r+1}=-\frac{6d}{r+1}=-6\,\frac{d}{d-g+1}=-6\,\frac{2v}{2v-1}<-7.
\end{equation}
From \eqref{E:3-nlm-WF} and \eqref{E:3-RHS}, we deduce that the chosen 1ps $\rho$ satisfies
$$
e_{X,\rho}>2d\,\frac{w(\rho)}{r+1}.
$$
In other words, $\rho$ violates the numerical criterion for Chow semistability of $[X\subset \P^r]$
(see Fact \ref{Chow-crit}); hence
$[X\subset \P^r]\not\in \Ch^{-1}(\Chow_d^{ss})$ which is the desired contradiction.

Finally, let us assume that condition \eqref{T:ell-curves2}� holds, namely $\displaystyle v:=\frac{d}{2g-2}=\frac{7}{2}$. One of the two polynomials appearing in
the numerical criterion for Hilbert (semi)stability (see Fact \ref{Hilb-crit}) is equal to
\begin{equation}\label{E:expl-Pm}
\frac{w(\rho)}{r+1}\,mP(m)=-\frac{3}{r+1}\,m(md+1-g)=
 -\frac{3d}{r+1}\,m^2+\frac{3}{r+1}\,(g-1)m= \frac{7}{2}\,m^2+\frac{1}{2}\,m.
\end{equation}
From \eqref{E:expl-WF} and \eqref{E:expl-Pm}, it follows that
$$\frac{w(\rho)}{r+1}\,mP(m)-W_{X,\rho}(m)<0 \: \: \text{�for }�\: m \gg 0,
$$
which implies that $[X\subset \P^r]\not \in \Hilb_d^{ss}$ by Fact \ref{Hilb-crit}.
\end{proof}

Combining the previous Theorem \ref{T:ell-curves} with Corollary \ref{C:quasi-wp-stable}\eqref{C:quasi-wp-stable1} and Definition \ref{D:quasi-wp-stable}\eqref{D:quasi2}, we get the following

\begin{coro}\label{C:quasi-p-stable}
Let $[X\subset \P^r] \in  \Hilb_d$ with $X$ connected and assume that one of the following two conditions is satisfied
\begin{enumerate}[(i)]
\item  $2(2g-2)<d<\frac{7}{2}(2g-2)$ and $[X\subset \P^r] \in \Ch^{-1}(\Chow_d^{ss})$;
\item  $d=\frac{7}{2}(2g-2)$ and $[X\subset \P^r] \in \Hilb_d^{ss}$.
\end{enumerate}
Then $X$  is a quasi-p-stable curve.
\end{coro}

The Chow semistable locus for $\displaystyle d=\frac{7}{2}(2g-2)$ is very interesting. In the proof of the last theorem we said that if $[X\subset \P^r]$ admits an elliptic tail $F$ that satisfies \eqref{E:rest-F} (we use the same notation of the last proof) and $\rho:\Gm\ra \GL_{r+1}$ is the one-parameter subgroup defined above with weights \eqref{E:weights}, then
$$
e_{X,\rho}=\frac{7}{3}\,w(\rho)
$$
Thus $[X\subset \P^r]$ is Chow strictly semistable with respect to $\rho$. In the next theorem, we determine
$$
[X_0\subset \P^r]:=\lim_{t\ra 0}\rho(t)\cdot [X\subset \P^r]
$$
and we study the basins of attraction (see \eqref{Sec:bas-attra}) of $[X_0\subset \P^r]$ with respect to $\rho$ and $\rho^{-1}$.

%$A_{\rho}([X_0\subset \P^r])$ and $A_{\rho^{-1}}([X_0\subset \P^r])$.

\begin{thm}\label{T:basintacn}
Let $X$ and $\rho$ be as above. Then $X_0=F_0\cup Y$, where $F_0$ is a tacnodal elliptic tail (in particular $[X\subset \P^r]\in A_{\rho}([X_0\subset \P^r])$). Moreover,  $A_{\rho^{-1}}([X_0\subset \P^r])$ contains smoothings of the separating node, but not smoothings of the tacnode.
\end{thm}
\begin{proof}
Denote by $I$ the homogeneous ideal defined by $X\subset \P^r$. In the proof of Theorem \ref{T:ell-curves}, we said that $F\subset V_F$ satisfies the equation
$$
f=a_{300} x_1^3+a_{210}x_1^2x_2+a_{201}x_1^2x_3+a_{120}x_1x_2^2+a_{102}x_1x_3^2+a_{111}x_1x_2x_3+
a_{030}x_2^3+a_{021}x_2^2x_3=0,
$$
where $a_{102}\neq 0$ and $a_{021}\neq 0$. Let us consider the $\rho$-weighted graded order $\In_{\prec_{\rho}}$ as in \S \ref{S:limGrob} .We obtain that
$g:=\In_{\prec_{\rho}}(f)=a_{102}x_1x_3^2+a_{021}x_2^2x_3$, which is the equation of an elliptic tacnodal curve $F_0$ in $V_F\cong \P^2$.
Denote by $I_F$ and $I_Y$ respectively the ideals of $F$ and $C$, respectively.
Suppose that $h_1,\ldots,h_n\in k[x_3,x_4\ldots,x_{r+1}]$ is a Gr\"obner basis for $I(Y)\cap k[x_3,\ldots,x_{r+1}]$: by Lemma \ref{lem:flatlimitcode} we deduce that
$$
\{f,h_1,\ldots,h_n,x_ix_j\,|\,1\leq i\leq 2\text{ and }4\leq j\leq r+1\}
$$
is a Gr\"obner basis for $I=I_Y\cap I_F$. Applying the definition we get
$$
\In_{\prec_{\rho}}(I)=\langle g,h_1,\ldots,h_n,x_ix_j\,|\,1\leq i\leq 2\text{ and }4\leq j\leq r+1\rangle,
$$
which is exactly the ideal defining the variety $X_0=F_0\cup Y\subset\P^r$. Now, it is enough to apply Corollary \ref{C:flatlimit}.

In order to show the last part of the theorem, we use the same techniques used to prove CLAIM 3 in the proof of Theorem \ref{T:spec-clos-orb}. Consider the separating node
$\{p\}=\{[0,0,1,0,\ldots,0]\}=V_{F_0}\cap V_Y$. We can assume that the tangent space $T_{X,p}$ is given by $\langle x_2,x_4\rangle$, hence the completion of the local ring $\OO_{X,p}$ is given by
$k[[u,v]]/(uv)$, where $u=x_2/x_3$ and $v=x_4/x_3$. Since $\rho(t)^{-1}\cdot x_2=tx_2$ and $\rho(t)^{-1}\cdot x_4=x_4$, we get $\rho(t)^{-1}\cdot u=tu$ and $\rho(t)^{-1}\cdot v=v$. We recall that
$\Def_{(X,p)}$ has a semiuniversal ring equal to $k[[a]]$ with universal family $k[[u,v,a]]/(uv-a)$. This implies that $\rho(t)^{-1}\cdot a=ta$ and $A_{\rho^{-1}}([X_0\subset \P^r])$ contains smoothings
of $p$. If we consider the action of $\rho$ on the universal family of the tacnode, we obtain non-positive weights, hence $A_{\rho^{-1}}([X_0\subset \P^r])$ does not contain any smoothings of the tacnode.
\end{proof}

\begin{figure}[h!]
\begin{center}
\unitlength .5mm % = 2.845pt
\linethickness{0.35pt}
\ifx\plotpoint\undefined\newsavebox{\plotpoint}\fi % GNUPLOT compatibility
\begin{picture}(176,93)(0,91)
\qbezier(77.5,87.75)(97.25,110.25)(97,133.75)
\qbezier(25.75,144)(45.5,166.5)(45.25,190)
\qbezier(31.75,177.25)(56.5,167.5)(76.25,174.75)
\put(21,183.75){\makebox(0,0)[cc]{$X$}}
\put(172.75,183.75){\makebox(0,0)[cc]{$Z$}}
\put(74.25,186){\makebox(0,0)[cc]{$F$}}
\put(80,166){\makebox(0,0)[cc]{non-special elliptic tail}}
\put(40.75,177.25){\makebox(0,0)[cc]{$p$}}
\put(35.75,143.75){\makebox(0,0)[cc]{$Y$}}
\put(91.75,123){\makebox(0,0)[cc]{$p$}}
\put(72,93.5){\makebox(0,0)[cc]{$Y$}}
\put(73,142){\makebox(0,0)[cc]{$\xxrsquigarrow{\phantom{aaa}\rho}$}}
\put(120.5,144){\makebox(0,0)[cc]{$\xxlsquigarrow{\rho^{-1}}$}}
\qbezier(133.25,184.25)(176,165.75)(128.75,149.25)
\put(153.5,186.25){\line(0,-1){39}}
\put(187,169.25){\makebox(0,0)[cc]{tacnode with a line}}
\put(105,94){\makebox(0,0)[cc]{$F_0$}}
\put(117.5,130){\makebox(0,0)[cc]{$E$}}
\put(151,110.5){\makebox(0,0)[cc]{tacnodal elliptic tail}}
\put(161,148.25){\makebox(0,0)[cc]{$E$}}
\put(134.75,90.75){\makebox(0,0)[cc]{$X_0$}}
\qbezier(82.25,125.75)(90.625,117.25)(101.5,120.75)
\qbezier(101.5,120.75)(113.875,124.375)(114.75,116.5)
\qbezier(114.75,116.5)(116.125,109.625)(107,107.25)
\qbezier(107,107.25)(99,106.375)(97,103)
\put(115.25,127.5){\line(0,-1){28.25}}
\end{picture}
\caption{The basin of attraction of a curve $X_0$ with a tacnodal elliptic tail $F_0$.}
 %with respect to the 1ps $\rho$ and $\rho^{-1}$.}
\end{center}
\end{figure}
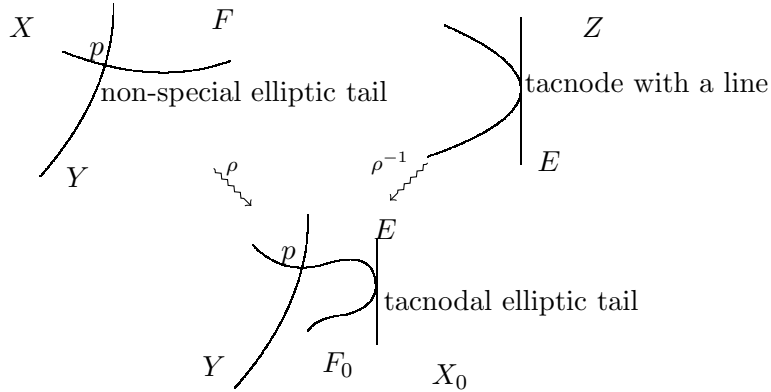

\section{A stratification of  the semistable locus}\label{S:stratifica}

Consider the following sublocus of $\Ch^{-1}(\Chow_d^{ss})\subset \Hilb_d$:
\begin{equation}\label{Hd}
\Ch^{-1}(\Chow_d^{ss})^o:=\{[X\subset \P^r]\in \Ch^{-1}(\Chow_d^{ss})\subset \Hilb_d \: :\: X \text{ is connected}\}.
\end{equation}
If $d>2(2g-2)$, the condition of being connected is both closed and open in
$\Ch^{-1}(\Chow_d^{ss})$: it is closed because of its natural interpretation as a topological condition;
it is open because $[X\subset \P^r] \in \Ch^{-1}(\Chow_d^{ss})$ is a reduced curve by the Potential pseudo-stability Theorem \ref{teo-pstab} and therefore $X$ is connected if and only if $h^0(X,\OO_X)=1$,
which is an open condition by upper-semicontinuity.
Thus, $\Ch^{-1}(\Chow_d^{ss})^o$ is both open and closed in $\Ch^{-1}(\Chow_d^{ss})$; or, in other words, it is
a disjoint union of connected components of $\Ch^{-1}(\Chow_d^{ss})$.

Inspired by \cite[Sec. 5]{Cap}, we introduce in this section an $\SL_{r+1}$-invariant stratification of $\Ch^{-1}(\Chow_d^{ss})^o$  and we establish some properties of it.

%the inclusion relations between the closures of these strata.

Recall that for any $[X\subset \P^r] \in \Ch^{-1}(\Chow_d^{ss})^o$ so that $d>2(2g-2)$, the curve $X$ is quasi-wp-stable and  $\OO_X(1)$ is properly balanced
by Corollary \ref{C:quasi-wp-stable}\eqref{C:quasi-wp-stable1}.
Recall also that $B_X^d$ denotes the set of multidegrees of properly balanced  line bundles on $X$ of total degree $d$ (see Definition \ref{D:prop-bal}).

Following \cite[Sec. 5.1]{Cap}, consider, for any  quasi-wp-stable curve $X$ of genus $g$ and any $\un d\in B_X^d$, the (locally closed) stratum of $\Ch^{-1}(\Chow_d^{ss})^o$:
\begin{equation}\label{E:strata}
M_X^{\un d}:=\{[Y\subset \P^r]\in \Ch^{-1}(\Chow_d^{ss})^o\: : \: \exists \text{ an isomorphism } \phi: X\to Y \text{ such that } \un{\deg}\: \phi^*\cO_Y(1)=\un d\}.
\end{equation}
Note, in particular, that the isomorphism $\phi$ between the abstract curve $X$ and the embedded curve $Y$ is not specified. However, with a slight abuse of notation, we will often represent points of $M_X^{\un d}$ by
$[X\subset \P^r]$.

Each stratum $M_X^{\un d}$ is $\SL_{r+1}$-invariant since $\SL_{r+1}$ acts on $\Ch^{-1}(\Chow_d^{ss})^o$ by changing the embedding of $X$ inside $\P^r$ and thus it preserves $X$ and the multidegree $\un d$.
 Note that $M_X^{\un d}$ may be empty for certain pairs $(X, \un d)$ as above.

\subsection{Specializations of strata}\label{SS:specia-str}

The aim of this subsection is to describe all  pairs $(X',\un d')$ with $X'$ quasi-wp-stable of genus $g$ and $\un d'\in B_{X'}^d$ such that $M_{X'}^{\un d'} \subseteq \ov{M_{X}^{\un d}}$.

Generalizing the refinement relation of \cite[Sec. 5.2]{Cap}, we define an order relation on the set of pairs $(X,\un d)$, where $X$ is a quasi-wp-stable curve of genus $g$ and
%$\un d$ is any multidegree on $X$ of total degree $d$ (although we will mostly be interested in the case
$\un d\in B_X^d$.

\begin{defi}\label{D:order-rela} Let $(X',d')$ and $(X'',d'')$ be such that $X'$ and $X''$ are two quasi-wp-stable curves of genus $g$ and $\un d'\in B^d_{X'}$, $\un d''\in B^d_{X''}$.
\noindent
%\begin{enumerate}[(i)]
%\item \label{D:order-rela1}
%We say that $X''\preceq X'$  if they have the same wp-stable reduction $X=\wps(X')=\wps(X'')$
%and there exists a surjective morphism $\sigma: X''\to X'$ commuting with their wp-stable  reduction morphisms
%$$\xymatrix{X''\ar[rr]^{\sigma} \ar^{\phi''}[dr] & & X' \ar[ld]_{\phi'}\\
%& \wps(X'')=X=\wps(X')&
%}$$
%\item \label{D:order-rela2} Assume moreover that .
 %is a multidegree on $X'$ (resp. $X''$) of total degree $d$.
We say that $(X'',\un d'')\preceq (X',\un d')$   if $(X'',d'')$ can be obtained from $(X',d')$ via a sequence of elementary operations as depicted in Figures \ref{blowUpNode}, \ref{blowUpCusp},  \ref{CuspidalCode}, \ref{CuspidalTail}, \ref{TacnodalCode} and \ref{TacnodalTail} below.
%$X''\preceq X'$  and there exists a surjective morphism $\sigma: X''\to X'$ as before such that for every subcurve $Y'\subseteq X'$ there exists a subcurve $Y''\subseteq X''$ with $Y'=\sigma(Y'')$ and $\un d'_{Y'}=\un d''_{Y''}$.
%\end{enumerate}
%\end{defi}
%The order relation $\preceq$ can be described in terms of elementary operations as follows.
%\begin{lemma}\label{L:rela-step}
%With the same notation as in the above Definition \ref{D:order-rela}, we have that
%$(X'',\un d'')\preceq (X',\un d')$ if and only if $X''$ is obtained from $X'$ via a sequence of bubblings of  nodes and cusps of $X'$ and the  the multidegree $\un d''$ is obtained
%from $\un d'$ at each step according to the rules depicted in Figures .
%\vspace{-1cm}
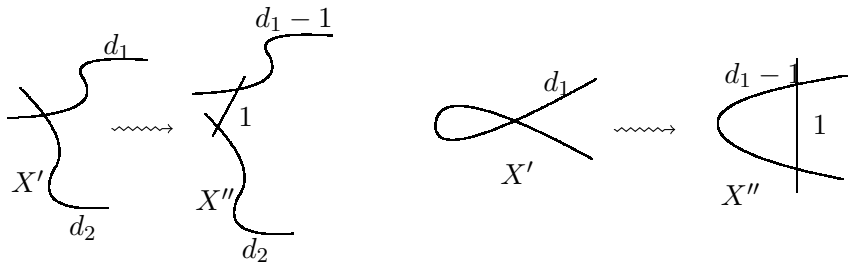
\begin{figure}[!h]
\begin{tabular}{cc}
\unitlength .3mm % = 2.845pt
\linethickness{0.4pt}
\ifx\plotpoint\undefined\newsavebox{\plotpoint}\fi % GNUPLOT compatibility
\begin{picture}(157.5,72)(20,100)
\qbezier(16.75,148.5)(59,150)(50.25,165.5)
\qbezier(98.75,159.25)(141,160.75)(132.25,176.25)
\qbezier(50.25,165.5)(41.75,176.625)(78.25,174.25)
\qbezier(132.25,176.25)(123.75,187.375)(160.25,185)
\put(31.75,152.75){\makebox(0,0)[cc]{}}
\qbezier(22,161.75)(46,138.25)(37,125.75)
\qbezier(104,150.25)(128,126.75)(119,114.25)
\qbezier(37,125.75)(27,106.25)(61,108.75)
\qbezier(119,114.25)(109,94.75)(143,97.25)
\multiput(107.25,140.25)(.0336879433,.0626477541){423}{\line(0,1){.0626477541}}
\put(65,180.25){\makebox(0,0)[cc]{$d_1$}}
\put(49.75,100){\makebox(0,0)[cc]{$d_2$}}
\put(126.25,90){\makebox(0,0)[cc]{$d_2$}}
\put(143.25,192.5){\makebox(0,0)[cc]{$d_1-1$}}
\put(109,227){\line(-1,0){.25}}
\put(122,149){\makebox(0,0)[cc]{$1$}}
\put(25.5,120.5){\makebox(0,0)[cc]{$X'$}}
\put(108.5,112.5){\makebox(0,0)[cc]{$X''$}}
\put(76,146.75){\makebox(0,0)[cc]{$\xrsquigarrow{\phantom{aaaa}}$}}
\end{picture}
\unitlength .4mm % = 2.845pt
\linethickness{0.4pt}
\ifx\plotpoint\undefined\newsavebox{\plotpoint}\fi % GNUPLOT compatibility
\begin{picture}(155,75.25)(0,115)
\qbezier(68.25,164.25)(15.375,134.625)(15,148.5)
\qbezier(15,148.5)(15,166.375)(67,137.75)
\qbezier(150.5,131)(108.375,139.125)(108.75,149.75)
\qbezier(108.75,149.75)(108.75,158.875)(151.75,165.5)
\put(135.25,171.25){\line(0,-1){44.75}}
\put(55.5,163){\makebox(0,0)[cc]{$d_1$}}
\put(140.5,146.25){$1$}
\put(111.028,163){$d_1-1$}
\put(116.5,126){\makebox(0,0)[cc]{$X''$}}
\put(42.25,133.25){\makebox(0,0)[cc]{$X'$}}
\put(84.5,149.75){\makebox(0,0)[cc]{$\xrsquigarrow{\phantom{aaaa}}$}}
\end{picture}
\end{tabular}
\caption{Bubbling of a node: external and internal cases.}
\label{blowUpNode}
\end{figure}
\begin{figure}[!h]
\begin{center}
\unitlength .35mm % = 2.845pt
\linethickness{0.4pt}
\ifx\plotpoint\undefined\newsavebox{\plotpoint}\fi % GNUPLOT compatibility
\begin{picture}(151.75,30.25)(10,110)
\qbezier(59.5,149.75)(46.5,134.375)(16.5,131.5)
\qbezier(16.5,131.5)(44.5,126.25)(55.5,109)
\qbezier(139.75,154.25)(98.125,128.375)(143,101)
\put(119.5,108.75){\line(0,1){41.5}}
\put(35.75,144){\makebox(0,0)[cc]{$d$}}
\put(109,129.75){\makebox(0,0)[cc]{$1$}}
\put(146.75,144.25){\makebox(0,0)[cc]{$d-1$}}
\put(24.75,112.5){\makebox(0,0)[cc]{$X'$}}
\put(155,110.5){\makebox(0,0)[cc]{$X''$}}
\put(81,129){\makebox(0,0)[cc]{$\xrsquigarrow{\phantom{aaaa}}$}}
\end{picture}
\end{center}
\caption{Bubbling of a cusp.}
\label{blowUpCusp}
\end{figure}
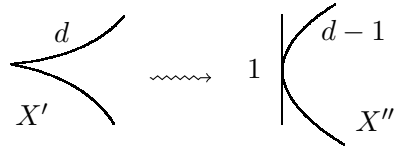
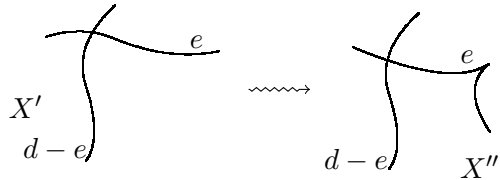
\begin{figure}[!h]
\begin{center}
\unitlength .4mm % = 2.845pt
\linethickness{0.4pt}
\ifx\plotpoint\undefined\newsavebox{\plotpoint}\fi % GNUPLOT compatibility
\begin{picture}(202.25,35)(0,130)
\qbezier(23.25,167)(35.625,170.75)(45.5,166.5)
\qbezier(45.5,166.5)(64.875,158.625)(80.75,162.25)
\put(73.25,165.25){\makebox(0,0)[cc]{$e$}}
\put(26.25,129.75){\makebox(0,0)[cc]{$d-e$}}
\put(126.25,126){\makebox(0,0)[cc]{$d-e$}}
\qbezier(36.5,125.75)(42.125,133.5)(36.25,151.25)
\qbezier(137.25,123)(142.875,130.75)(137,148.5)
\qbezier(36.25,151.25)(32.25,164.625)(46.25,177.5)
\qbezier(137,148.5)(133,161.875)(147,174.75)
\put(167.5,123.75){\makebox(0,0)[cc]{$X''$}}
\put(16.5,143.75){\makebox(0,0)[cc]{$X'$}}
\put(100.75,151.75){\makebox(0,0)[cc]{$\xrsquigarrow{\phantom{aaaa}}$}}
\qbezier(125.25,163.5)(158.125,149.75)(170.5,158)
\qbezier(170.5,158)(160.875,150.375)(170.75,135.25)
\put(163.25,160.5){\makebox(0,0)[cc]{$e$}}
\end{picture}
\end{center}
\caption{Replacing an elliptic tail by a cuspidal elliptic tail.}
\label{CuspidalCode}
\end{figure}
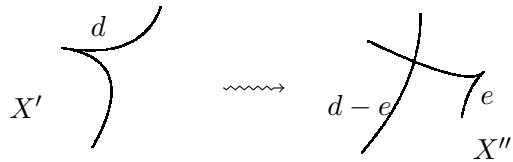
\begin{figure}[!h]
\begin{center}
\unitlength .4mm % = 2.845pt
\linethickness{0.4pt}
\ifx\plotpoint\undefined\newsavebox{\plotpoint}\fi % GNUPLOT compatibility
\begin{picture}(202.25,30.5)(0,150)
\qbezier(131.75,175.75)(164.625,159.875)(170,165.5)
\qbezier(170,165.5)(164.625,160.375)(162.75,150.75)
\put(42,180.75){\makebox(0,0)[cc]{$d$}}
\put(171.25,157){\makebox(0,0)[cc]{$e$}}
\put(18,154){\makebox(0,0)[cc]{$X'$}}
\put(173,140.25){\makebox(0,0)[cc]{$X''$}}
\qbezier(40,140.5)(56.5,169.5)(30,173.5)
\qbezier(30,173.5)(56.125,169.125)(62.75,187.25)
\qbezier(129.5,138.75)(149.25,161.25)(149,184.75)
\put(128.75,154.25){\makebox(0,0)[cc]{$d-e$}}
\put(93.5,162.5){\makebox(0,0)[cc]{$\xrsquigarrow{\phantom{aaaa}}$}}
\end{picture}
\end{center}
\caption{Replacing a cuspidal singularity by a cuspidal elliptic tail.}
\label{CuspidalTail}
\end{figure}
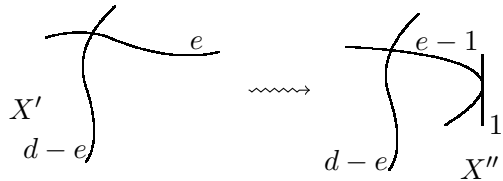
\begin{figure}[!h]
\begin{center}
\unitlength .4mm % = 2.845pt
\linethickness{0.4pt}
\ifx\plotpoint\undefined\newsavebox{\plotpoint}\fi % GNUPLOT compatibility
\begin{picture}(202.25,35)(0,133)
\qbezier(23.25,167)(35.625,170.75)(45.5,166.5)
\qbezier(45.5,166.5)(64.875,158.625)(80.75,162.25)
\put(73.25,165.25){\makebox(0,0)[cc]{$e$}}
\put(26.25,129.75){\makebox(0,0)[cc]{$d-e$}}
\put(126.25,126){\makebox(0,0)[cc]{$d-e$}}
\put(156.5,165.75){\makebox(0,0)[cc]{$e-1$}}
\put(172.75,138.25){\makebox(0,0)[cc]{$1$}}
\qbezier(36.5,125.75)(42.125,133.5)(36.25,151.25)
\qbezier(137.25,123)(142.875,130.75)(137,148.5)
\qbezier(36.25,151.25)(32.25,164.625)(46.25,177.5)
\qbezier(137,148.5)(133,161.875)(147,174.75)
\qbezier(123,164)(192,160)(156,138)
\put(168.25,137.75){\line(0,1){23.5}}
\put(167.5,123.75){\makebox(0,0)[cc]{$X''$}}
\put(16.5,143.75){\makebox(0,0)[cc]{$X'$}}
\put(100.75,151.75){\makebox(0,0)[cc]{$\xrsquigarrow{\phantom{aaaa}}$}}
\end{picture}
\end{center}
\caption{Replacing an elliptic tail by a tacnodal elliptic tail.}
\label{TacnodalCode}
\end{figure}
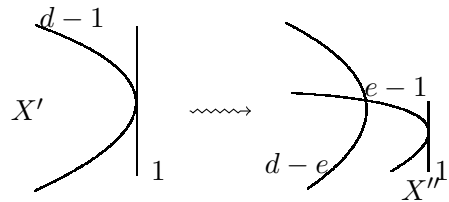
\begin{figure}[!h]
\begin{center}
\unitlength .4mm % = 2.845pt
\linethickness{0.4pt}
\ifx\plotpoint\undefined\newsavebox{\plotpoint}\fi % GNUPLOT compatibility
\begin{picture}(202.25,42.5)(10,133)
\put(126.25,132){\makebox(0,0)[cc]{$d-e$}}
\put(159.25,157.25){\makebox(0,0)[cc]{$e-1$}}
\put(174.5,129.75){\makebox(0,0)[cc]{$1$}}
\qbezier(124.75,155.5)(193.75,151.5)(157.75,129.5)
\put(170,129.25){\line(0,1){23.5}}
\put(167.5,123.75){\makebox(0,0)[cc]{$X''$}}
\put(100.75,151.75){\makebox(0,0)[cc]{$\xrsquigarrow{\phantom{aaaa}}$}}
\qbezier(39.5,178)(106.125,154)(39.25,123)
\put(73,177.5){\line(0,-1){49.5}}
\put(36.75,150.25){\makebox(0,0)[cc]{$X'$}}
\put(80.25,129.75){\makebox(0,0)[cc]{$1$}}
\put(51.5,180.5){\makebox(0,0)[cc]{$d-1$}}
\qbezier(130,123.75)(172.125,153.25)(122.75,178.75)
\end{picture}
\end{center}
\caption{Replacing a tacnode with a line by a tacnodal elliptic tail.}
\label{TacnodalTail}
\end{figure}
\end{defi}

%\end{lemma}

%\begin{proof}
%Using the explicit description of the wp-stable reduction of Proposition \ref{P:wp-stab}, it is easy to see that
%$X''\preceq X'$  if and only if $X''$ is obtained from $X'$ via a sequence of bubblings of nodes  and cusps. According to Definition \ref{D:order-rela}\eqref{D:order-rela2}, it is now clear that
%at each bubbling $\un d''$ must be obtained from $\un d'$ according to the rules depicted in Figures \ref{blowUpNode} and \ref{blowUpCusp}.

%\end{proof}

\begin{rmk}\label{D:order-rela1}
Given two quasi-wp-stable curves $X'$ and $X''$ (not necessarily endowed with any multidegree), we can also say that $X''\preceq X'$   if $X''$ can be obtained from $X'$ via a sequence of the elementary operations depicted in Figures \ref{blowUpNode}, \ref{blowUpCusp}, \ref{TacnodalCode}, \ref{CuspidalCode}, \ref{CuspidalTail} and \ref{TacnodalTail}, ignoring the degrees.
\end{rmk}

From the above description it is easy to see that there is a relation between the isotrivial specialization introduced in Definition \ref{D:speciali} and the order relation $\preceq$. More precisely, the following holds:

\begin{rmk}
Let $(X',L')$ and $(X'',L'')$ be two pairs consisting of a quasi-wp-stable  curve of genus $g$ and a properly balanced line bundle of degree $d$.
If $(X',L')\sp (X'',L'')$, then $(X'',\un{\deg}L'')\preceq (X', \un{\deg}L')$.
\end{rmk}

The following elementary property of the order relation $\preceq$ will be used in what follows.

\begin{lemma}\label{L:order-rela}
Notation as in Definition \ref{D:order-rela}.
%\begin{enumerate}[(i)]
%\item \label{L:order-rela1} If $(X'',\un d'')\preceq (X',\un d')$ and  $\un d''\in B_{X''}^d$ then  $\un d'\in B_{X'}^d$.
%\item \label{L:order-rela2}
If $X''\preceq X'$ and $\un d''\in B_{X''}^d$ then there exists $\un d'\in B_{X'}^d$
    such that $(X'',\un d'')\preceq (X',\un d')$.
%\end{enumerate}
\end{lemma}

\begin{proof}
%By Lemma \ref{L:rela-step} above it is enough to assume that $X''$ is obtained from $X'$ by blowing-up a node (which can be internal or external) or a cusp.
Start by assuming that $X''$ is obtained from $X'$ by bubbling an external node $n$, as in the picture on the left of Figure \ref{blowUpNode}.

Denote by $\{C_1', \ldots , C_{\gamma}'\}$ the irreducible components of $X'$, by
$\{C_1'',\ldots,C''_{\gamma}\}$ their proper transforms in $X''$ and by $E$ the exceptional component that is contracted to the node $n$ by the map $\sigma: X''\to X'$. Assume that $C_1'$ and $C_2'$ are the two irreducible components of $X'$ that contain the node $n$.
Define a multidegree $\un d'$ on $X'$ in the following way:
\begin{equation*}
\un d'_{C_i'}:=
\begin{sis}
&\un d''_{C_i''} &\mbox{ for } i\neq 1,\\
&\un d''_{C_1''}+1 &\mbox{ for } i=1.
\end{sis}
\end{equation*}
It is clear that $|\underline d'|=d$, so we must check that $\underline d'$ satisfies the basic inequality \eqref{E:basineq-multideg}. Given a subcurve $Z'$ of $X'$, we denote by $Z''$ the subcurve of $X''$ that is the proper transform of $Z'$
under the bubbling map $X''\to X'$. Define $W_{Z'}$ to be the subcurve  of $X''$ such that $W_{Z'}=Z''$ if $C_1'\subsetneq Z'$ and  $W_{Z'}=Z''\cup E$ if $C_1'\subseteq Z'$. Then it is easy to see that
$$\begin{sis}
& d'_{Z'}=d''_{W_{Z''}},\\
& g_{Z'}=g_{W_{Z''}},\\
& k_{Z'}=k_{W_{Z''}}.
\end{sis}$$
 Hence the basic inequality \eqref{E:basineq-multideg} for $\un d'$ relative to the subcurve
$Z'$ is the same as the basic inequality for $\un d''$ relative to the subcurve $W_{Z''}$.
We conclude that if $d''\in B^{d}_{X''}$ then $d'\in B^{d}_{X'}$.

The remaining cases are similar (and easier) and are therefore left to the reader.
\end{proof}

We will now prove that the above order relation $\preceq$  determines the inclusion relations among the closures of the strata $M_X^{\un d}\subset \Ch^{-1}(\Chow_d^{ss})^o$ of \eqref{E:strata}.
The following result is a generalization of \cite[Prop. 5.1]{Cap}.

\begin{prop}\label{P:deg-strata}
 Assume that $d>2(2g-2)$ and moreover that $g\geq 3$ if $d\leq 4(2g-2)$.
 Let $X'$ and $X''$ be two quasi-wp-stable curves of genus $g$ and let $\un d'\in B_{X'}^d$ and $\un d''\in B_{X''}^d$.
 Assume that $M_{X''}^{\un d''}\neq \emptyset$. Then
$$M_{X''}^{\un d''} \subseteq \ov{M_{X'}^{\un d'}} \Longleftrightarrow (X'',\un d'')\preceq (X',\un d').$$
\end{prop}
\begin{proof}

$\Longleftarrow$
We will start by showing that if $X''\preceq X'$ then there is a family $u:\mathcal X\to B$ over a smooth curve $B$ whose geometric fiber $\cX_b$ over a point $b\in B$ is such that $\mathcal X_b\cong X'$ for all $b\neq b_0$ and $\mathcal {X}_{b_0}\cong X''$.
% From Lemma \ref{L:rela-step} above, it is enough to assume that $X''$ is obtained from $X'$ by bubbling  a node or a cusp.

Start by assuming that $X''$ is obtained from $X'$ by bubbling  a node, say $n$.
 %and take $\un d''\in B^{d}_{X''}$. By Lemma \ref{L:order-rela}, $\exists \un d'\in B^d_{X'}$ such that $(X'',\un d'')$
Let $B$ be a smooth curve and consider the trivial family $X'\times B$ over $B$. By blowing up the surface $X\times B$ at the node $n$ belonging to the fiber over a point $b_0\in B$, we get a family $u:\mathcal X\to B$ whose geometric fiber $\cX_b$ over a point $b\in B$ is such that $\mathcal X_b\cong X'$ for all $b\neq b_0$ and $\mathcal {X}_{b_0}\cong X''$ as in the figure below (where we have depicted an external node, but the case of an internal node is completely similar).

%\begin{figure}[h]
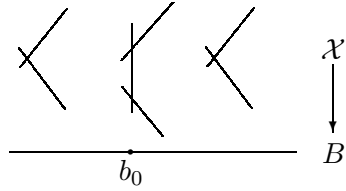
\begin{center}
\begin{picture}(154.75,70)(0,105)
\put(32,117.75){\line(1,0){108.5}}
%\emline(74.25,152.25)(94,174.5)
\multiput(74.25,152.25)(.0337030717,.0379692833){586}{\line(0,1){.0379692833}}
%\end
%\emline(74.75,142.25)(90,123.75)
\multiput(74.75,142.25)(.0337389381,-.0409292035){452}{\line(0,-1){.0409292035}}
%\end
\put(78.5,166.25){\line(0,-1){33.5}}
%\emline(106.75,150)(124.5,172.25)
\multiput(106.75,150)(.0336812144,.0422201139){527}{\line(0,1){.0422201139}}
%\end
%\emline(36,149.75)(53.75,172)
\multiput(36,149.75)(.0336812144,.0422201139){527}{\line(0,1){.0422201139}}
%\end
%\emline(106.25,157.5)(123.75,134.25)
\multiput(106.25,157.5)(.0337186898,-.0447976879){519}{\line(0,-1){.0447976879}}
%\end
%\emline(35.5,157.25)(53,134)
\multiput(35.5,157.25)(.0337186898,-.0447976879){519}{\line(0,-1){.0447976879}}
%\end
\put(154.5,156.75){\makebox(0,0)[cc]{$\mathcal X$}}
\put(154.75,117.5){\makebox(0,0)[cc]{$B$}}
\put(78.25,117.75){\circle*{.5}}
\put(78,110){\makebox(0,0)[cc]{$b_0$}}
\put(78,118){\circle*{1.414}}
\put(154.25,150.75){\vector(0,-1){25.5}}
\put(77.75,117.75){\circle*{2.121}}
\end{picture}

\end{center}
%\caption{Degenerating }
%\label{F:deg-blowup-node}
%\end{figure}
When $X''$ is obtained from $X'$ by bubbling  a cusp, we proceed in the same way as in the previous case: we consider the trivial family $X'\times B$ over $B$ and by blowing up the surface $X\times B$ on the cusp $p$ belonging to the fiber over a point $b_0\in B$, we get a family $u:\mathcal X\to B$ whose geometric fiber $\cX_b$ over a point $b\in B$ is such that $\mathcal X_b\cong X'$ for all $b\neq b_0$ and $\mathcal {X}_{b_0}\cong X''$ as in the figure below.
%we consider a family $u:\mathcal X \to B$ such that $\mathcal X_{b_0}\cong X''$ and $\mathcal X_b\cong X'$ for $b\neq b_0$ as in the figure below

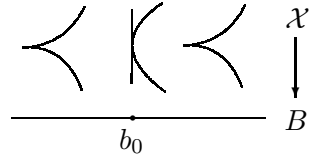
\begin{center}
\begin{picture}(141.5,60)(0,115)
\put(33.25,122.75){\line(1,0){96}}
\qbezier(121.75,165.75)(114.375,151)(98.5,150.25)
\qbezier(61,165)(53.625,150.25)(37.75,149.5)
\qbezier(120.5,134.5)(114.5,151)(98.5,150.5)
\qbezier(59.75,133.75)(53.75,150.25)(37.75,149.75)
\qbezier(90.25,166.25)(67.125,150.5)(91.5,132.75)
%\emline(78.75,136)(79,164)
\multiput(78.75,136)(.03125,3.5){8}{\line(0,1){3.5}}
%\end
\put(141.5,160.75){\makebox(0,0)[cc]{$\mathcal X$}}
\put(141,122){\makebox(0,0)[cc]{$B$}}
\put(140.75,153.5){\vector(0,-1){23}}
\put(79.25,123){\circle*{2.062}}
\put(78.75,114.25){\makebox(0,0)[cc]{$b_0$}}
\end{picture}
\end{center}

The cases when $X''$ is a cuspidal elliptic tail as in Figures \ref{CuspidalCode} and \ref{CuspidalTail} are direct consequences of Remark \ref{R:3stacks}\eqref{R:3stacks1}.
The cases depicted in Figures \ref{TacnodalCode} and \ref{TacnodalTail} can then be obtained from these as follows. Let $u:\mathcal X\to B$ be a family such that, for $b\neq b_0$, $\mathcal X_b\cong X'$ is a curve with an elliptic tail and such that $\mathcal X_{b_0}$ is a curve with a cuspidal elliptic tail as in Figure \ref{CuspidalCode}.  By bubbling the surface $\mathcal X$ at the cusp $p$ of the central fiber $\cX_{b_0}$, we get a new family $u':\mathcal X'\to B$ such that $\mathcal X'_b\cong X'$ is a curve with an elliptic tail as before and such that $\mathcal X'_{b_0}$ has a tacnodal elliptic tail as in Figure \ref{TacnodalCode}.
Finally, in order to deal with the situation depicted in Figure \ref{TacnodalTail}, we consider an isotrivial family $u:\mathcal X\to B$ where for $b\neq b_0$, $\mathcal X_b$ is a curve with a cuspidal singularity and such that $\mathcal X_{b_0}$ has a cuspidal tail, as in Figure \ref{CuspidalTail}.
The locus in $\cX$ corresponding to the cusp in each fiber $\mathcal X_b$ of $u$ over $B$ is a Weil divisor on the surface $\cX$; by blowing up this divisor, we get a new family $u':\mathcal X'\to B$ such that, for $b\neq b_0$, $\mathcal X'_b\cong X'$ is a curve having a tacnode with a line while
$\mathcal X'_{b_0}$ has a tacnodal elliptic tail, as in Figure \ref{TacnodalTail}.

Consider now the relative Picard scheme
$\pi:\Pic_{\cX/B}\to B$ of the family $u:\cX\to B$, which exists by a well-known result of Mumford
(see \cite[Sec. 8.2, Thm. 2]{BLR}). Since $H^2(\cX_b, \OO_{\cX_b})=0$ for any $b\in B$ because $\cX_b$ is a curve,
we get that $\pi:\Pic_{\cX/B}\to B$ is smooth by \cite[Sec. 8.4, Prop. 2]{BLR}.

Let now $[X''\subset \P^r=\P(V)]\in M_{X''}^{\un d''}$ and set $L''=\OO_{X''}(1)\in \Pic^{\un d''}(X'')$. Note that the embedding $X''\subset \P^r$ defines an isomorphism $\phi:H^0(X'', L'')\stackrel{\cong}{\to} V$.

We can view $L''$ as a geometric point of
$(\Pic_{\cX/B})_{b_0}\cong \Pic(X'')$. Since the morphism $\pi: \Pic_{\cX/B}\to B$ is smooth, up to shrinking $B$ (i.e., replacing it with  an \'etale open neighborhood of $b_0$), we can find a section $\sigma$ of $\pi$ such that $\sigma(b_0)=L''$.
Moreover, by definition of the order relation $\preceq$ (see Figures \ref{blowUpNode}, \ref{blowUpCusp}, \ref{CuspidalCode}, \ref{CuspidalTail}, \ref{TacnodalCode} and \ref{TacnodalTail} above), it is clear that we can choose the section $\sigma$ so that $\sigma(b)$ is a line bundle of multidegree $\un d'$ on $\cX_{b}\cong X'$ for every $b\neq b_0$.

Up to shrinking $B$ again, we can assume that the section $\sigma$ corresponds to a line bundle $\mathcal L$ over $\mathcal X$ such that  $\mathcal L_{|{\mathcal X}_{b_0}}\cong L''$ and  $\mathcal L_{|{\mathcal X}_{b}}$ has multidegree $\un d'$ for $b\neq b_0$. Since $L''$ is very ample and non-special
and these conditions are open, up to shrinking $B$ once more, we can assume that
$\mathcal L$ is relatively very ample and we can fix an isomorphism $\Phi:
u_*\mathcal L\stackrel{\cong}{\to} \OO_B\otimes V$ of sheaves on $B$ such that $\Phi_{|b_0}=\phi$.
Via the isomorphism $\Phi$, the relatively very ample line bundle $\mathcal{L}$ defines an embedding
$$\xymatrix{
\mathcal X \ar@{^{(}->}[rr]^{i}\ar[dr]_u & & \P(\OO_B\otimes V)=\P^r_B\ar[ld],\\
& B &
}
$$
whose restriction over $b_0\in B$ is the embedding $X''\subset \P^r$.
The family $u:\mathcal X\to B$ together with the embedding $i$, defines a morphism $f:B\to \Ch^{-1}(\Chow_d^{ss})^o$ such that $f(b_0)=[X''\subset \P^r]\in M_{X''}^{\un d''}$ and $f(b)\in M_{X'}^{\un d'}$ for every $b\neq b_0$,
so we conclude that $M_{X''}^{\un d''} \subseteq \ov{M_{X'}^{\un d'}}$.

$\Longrightarrow$
%Consider the map (see Theorem \ref{T:comp-Pic}):
%$$\phi^{\rm ps}: \Ch^{-1}(\Chow_d^{ss})^o\to \ov{Q}_{d,g}:=\Ch^{-1}(\Chow_d^{ss})^o/SL_{r+1} \stackrel{\Phi^{\rm ps}}{\longrightarrow} \Mgp.$$
%Clearly, $M_{X'}^{\un d'}$ is contained in the fiber $(\phi^{\rm ps})^{-1}(X)$ where $X:=\ps(X')\in \Mgp$.
Suppose now that $M_{X''}^{\un d''}\subseteq \ov{M_{X'}^{\un d'}}$.
Then we can find a smooth curve $B$ and a morphism $f:B\to \Ch^{-1}(\Chow_d^{ss})^o$  such that
$f(b_0)\in M_{X''}^{\un d''}$ for some $b_0\in B$ and $f(b)\in M_{X'}^{\un d'}$ for every $b_0\neq b\in B$.
By pulling back the universal family above $\Ch^{-1}(\Chow_d^{ss})^o$ along the morphism $f$, we get a family
$$\xymatrix{
\cX \ar@{^{(}->}[r] \ar[d]^u & B\times \P^r\\
B&
}$$
such that $\cX_{b_0}=X''$ and $\cX_{b}=X'$ for every $b\neq b_0$.
In particular, $u$ yields an isotrivial specialization of $X'$ into $X''$.
Let $\ov\cX\to B$ be the wp-stable reduction of $u$; for $b\neq b_0$, $\ov X':=\ov\cX_b$ is the wp-stable reduction of $X'$ while $\ov X'':=\ov\cX_{b_0}$ is the wp-stable reduction of $X''$.
According to Remark \ref{R:3stacks}\eqref{R:3stacks1}, $\ov X'$ and $\ov X''$ may differ by replacing elliptic tails by cuspidal elliptic tails or by replacing cuspidal singularities by cuspidal elliptic tails as in Figures \ref{CuspidalCode} and \ref{CuspidalTail}, so $\ov X''\preceq \ov X'$.
Then, as $\cX$ is a family of quasi-wp-stable curves, it is obtained from $\ov\cX$ in two steps: first by blowing up the surface $\ov\cX$ on the locus of some nodal or cuspidal singularities along all the fibers of $\ov\cX$ giving rise to a new family $\widetilde\cX$, and then by further blowing up $\widetilde\cX$ on nodal or cuspidal singularities of the fiber over $b_0$.
Denote $\widetilde X':=\widetilde\cX_{b}$ for $b\neq b_0$ and $\widetilde X'':=\widetilde\cX_{b_0}$. Then it is easy to see that $\widetilde X''\preceq \widetilde X'$: the only situation, which needs some care, is the blow up $\ov\cX$ at a cuspidal singularity when $\ov X''$ is obtained from $\ov X'$ by replacing a cuspidal singularity by a cuspidal elliptic tail as in Figure \ref{CuspidalTail}. But in this case, we easily see that we get a situation as described on Figure \ref{TacnodalTail}, so $\widetilde X''\preceq \widetilde X'$.
Finally, by blowing up $\widetilde\cX$ on nodes or cusps of $\widetilde\cX_{b_0}$ we get the situations described in Figures \ref{blowUpNode}, \ref{blowUpCusp} and \ref{TacnodalCode}, so $X''\preceq X'$.

%here is a family $u:\mathcal X\to B$ of curves and line bundles $L_b$, for $b\in B$ such that for $b\neq b_0$, $\mathcal X_b\cong X'$ and $\un\deg(L_b)=d'$ while for $b=b_0$, $\mathcal X_{b_0}\cong X''$ and $\un\deg(L_{b_0})=d''$.%\subseteq (\phi^{\rm ps})^{-1}(X)$, which implies that $\ps(X'')=X=\ps(X')$.

%By assumption, The ps-stable reduction $\ps(u)$ of $u$ is an isotrivial family of p-stable curves by what observed before. This implies that $\cX$ is obtained from
%an isotrivial family $\ov{\cX}$ with fibers isomorphic to $X'$ by bubbling  some nodes and cusps of the central fiber
%$\ov{\cX}_{b_0}=X'$. We get  a surjective morphism $\sigma:X''=\cX_{b_0}\to \ov{\cX}_{b_0}=X'$ which clearly commutes with the p-stable reduction morphisms; in other words

Consider now the line bundles $L_{b_0}:=\OO_{\cX}(1)_{|\cX_{b_0}}\in \Pic^{\un d''}(X'')$ and $L_b:=\OO_{\cX}(1)_{|\cX_{b}}\in \Pic^{\un d'}(X')$  for any $b_0\neq b\in B$.
Let $Y'\subseteq X'$ be a subcurve
of $X'$. Consider the subcurve $Y''\subseteq X''=\cX_{b_0}$ given by the union of all the irreducible components $C_i$ of $X''$ for which there exists a section $s$ of $u:\cX\to B$ such that
$s(b_0)\in C_i$ and $s(b)\in Y'\subseteq X'=\cX_{b'}$ for every $b\neq b_0$. By construction, we get that
$\un d'_{Y'}=\deg_{Y'} L_{b}=\deg_{Y''}L_{b_0}=\un d''_{Y''}$. According to Definition \ref{D:order-rela}, this yields that $(X'',\un d'')\preceq (X',\un d')$.

\end{proof}

\subsection{A completeness result}\label{SS:complete}

Given any  quasi-wp-stable curve $X$ of genus $g$ and a multidegree $\un d\in B_X^d$, consider the subgroup of the automorphism group $\Aut(X)$ of $X$ given by
\begin{equation}\label{E:aut-d}
\Aut^{\un d}(X)=\{\phi\in \Aut(X)\: : \: \phi^*L\in \Pic^{\un d}(X) \: \text{ for any }
L\in \Pic^{\un d}(X)\}.
\end{equation}
Note that given a point $[Y\stackrel{i}{\hookrightarrow} \P^r]$ belonging to the stratum $M_X^{\un d}$ as in \eqref{E:strata}, the line bundle $\phi^*\OO_Y(1)\in \Pic^{\un d}(X)$ is only well-defined up to the action
of $\Aut^{\un d}(X)$. We denote by $[\OO_X(1)]$ the class of this line bundle in the quotient
$\Pic^{\un d}(X)/\Aut^{\un d}(X)$.
Therefore, we have a well-defined (set-theoretic) map
\begin{equation}\label{E:map-strata-Pic}
\begin{aligned}
p: M_X^{\un d} & \to \Pic^{\un d}(X)/\Aut^{\un d}(X) \\
[X\subset \P^r] & \mapsto [\OO_X(1)].
\end{aligned}
\end{equation}
%where, given $L\in \Pic^{\un d}(X)$, we denote by $[L]$ its class in $\Pic^{\un d}(X)/\Aut(X)$.
% If we endow the quotient space $\Pic^{\un d}(X)/\Aut(X)$ with the quotient topology with respect to the projection $\pi:\Pic^{\un d}(X)\to \Pic^{\un d}(X)/\Aut(X)$,
%then we can view $p$ as a continuos map of topological spaces.
Note that the fibers of the map $p$ are exactly the $\SL_{r+1}$-orbits on $M_X^{\un d}$.
The image of $p$ can be nicely described using the following useful result about the relation between the automorphism group of $X$ and the stability of $[X\subset \P^r]$.

\begin{lemma}\label{L:autstable}
Let $[X\subset \P^r]\in \Hilb_d$ where $X$ is non-degenerate and linearly normal in $\P^r$. Set $L=\OO_X(1)$. If $\phi\in \Aut(X)$, then $[X\subset \P^r]$ is Chow semistable (resp. stable) if and only if $[X \stackrel{|\phi^*L|}{\hookrightarrow} \P^{r}]$ is Chow semistable (resp. stable). The same holds for the Hilbert (semi)stability.
\end{lemma}
\begin{proof}
For $m\gg 0$ we have the commutative diagram
$$\xymatrix{
S^m H^0(X, L) \ar@{->}[r]^{\phi^*} \ar@{->>}[d] & S^m H^0(X, \phi^*L) \ar@{->>}[d]\\
H^0(X, L^m)\ar@{->}[r]^{\phi^*} & H^0(X,\phi^*L^m),  \\
}
$$
which allows us to identify the  monomial bases of $H^0(X,L^m)$ and of $H^0(X, \phi^* L^m)$. More precisely, if we fix a system of coordinates $\{x_1,\ldots,x_{r+1}\}$ and $\{B_1,\ldots,B_{P(m)}\}$ is a monomial basis of $H^0(X, L^m)$ with respect to $\{x_1,\ldots,x_{r+1}\}$, then $\{\phi^*(B_1),\ldots,$ $\phi^*(B_{P(m)})\}$ is a monomial basis of $H^0(X, \phi^*L^m)$ with respect to the system of coordinates $\{\phi^*(x_1), \ldots,\phi^*(x_{r+1})\}$. Now, let $\rho$ be a one-parameter subgroup diagonalized by $\{x_1,\ldots,x_{r+1}\}$ with weights $w_1,\ldots,w_{r+1}$ and define another one-parameter subgroup $\rho^*$ diagonalized by $\{\phi^*(x_1),\ldots,\phi^*(x_{r+1})\}$ with the same ordered weights. By the identification of monomial bases, we have that $W_{X,\rho}(m)=W_{X,\rho^*}(m)$ for $m\gg 0$\footnote{Here there is an abuse of notation: $W_{X,\rho}(m)$ is referred to $[X\subset \P^r]$, while $W_{X,\rho^*}(m)$ is referred to $[X \stackrel{|\phi^*L|}{\hookrightarrow} \P^{r}]$.}. Now, it suffices to apply the Hilbert-Mumford criterion (Fact \ref{Hilb-crit}) and our statement is proved.
\end{proof}

\begin{coro}\label{C:Im-p}
Let $X$ be a quasi-wp-stable curve and $\un d\in B_X^d$. Let $L\in \Pic^{\un d}(X)$ and assume that $L$ is very ample and non-special.
Consider the point $[X\stackrel{|L|}{\hookrightarrow}\P^r] \in \Hilb_d$, which is well-defined up to the action of $\SL_{r+1}$.
Then
$$[L]\in \Im(p) \Leftrightarrow [X\stackrel{|L|}{\hookrightarrow}\P^r]\in \Ch^{-1}(\Chow_d^{ss})^o.$$
\end{coro}

The aim of this subsection is to prove the following completeness result, which generalizes \cite[Prop. 5.2]{Cap}.

\begin{prop}\label{P:completeness}
Let $X$ be a quasi-wp-stable curve and $\un d\in B_X^d$. Assume that one of the following conditions is satisfied:
\begin{enumerate}[(i)]
	\item \label{compl1} $d>4(2g-2)$;
	\item \label{compl2} $2(2g-2)<d<\frac{7}{2}(2g-2)$ and $g\geq 3$;
	\item \label{compl3} $X$ is quasi-p-stable, $\frac{7}{2}(2g-2)<d<4(2g-2)$ and $g\geq 3$.
\end{enumerate}
Then either $M_X^{\un d}=\emptyset$ or $p: M_X^{\un d}\to \Pic^{\un d}(X)/\Aut^{\un d}(X)$ is surjective.
\end{prop}
 \begin{proof}
%The proof is a generalization of the proof of \cite[Prop. 5.2]{Cap}. However, since the extension
%of the result of loc. cit. to our situation is non trivial (in particular, we have to rely on some results contained in the Appendix), we include a full proof.

Note that the statement of the Proposition is equivalent to the fact that either $\theta^{-1}(\Im(p))=\emptyset$ or $\theta^{-1}(\Im(p))=\Pic^{\un d}(X)$, where
$\theta:\Pic^{\un d}(X)\to \Pic^{\un d}(X)/\Aut^{\un d}(X)$ is the projection to the quotient.
We first make the following two reductions.

\un{Reduction 1:} We can assume that if $d<\frac{5}{2}(2g-2)$ then $X$ does not contain  elliptic tails;  in this case every $L\in \Pic^{\un d}(X)$ is non-special and very ample.

Indeed, according to Theorem \ref{bal-pos}\eqref{bal-ns}, $L\in \Pic^{\un d}(X)$ is non-special since $X$ is quasi-wp-stable, hence $G$-semistable, and
$\deg L=d>2(2g-2)>2g-2$ (recall that $g\geq 2$).   Now, if $d<\frac{5}{2}(2g-2)$ and $X$ contains some elliptic tail $F$, then from the basic inequality it follows easily that $\un d_F =2$.
But no line bundle of degree $2$ on a curve of genus $1$ is very ample, hence no
line bundle of multidegree $\un d$ on $X$  can be very ample;
 %because a line bundle of degree $2$ on a curve of genus $1$ is not very ample; therefore $M_X^{\un d}=\emptyset$ and we are done.
otherwise, since any $L\in \Pic^{\un d}(X)$ is ample by
Remark \ref{R:propbal-ample}, it follows from Theorem \ref{bal-pos}\eqref{bal-va} that $L$ is very ample, q.e.d.

\un{Reduction 2:} We can assume that  $\un d$ is strictly balanced.

Indeed, suppose the proposition is true for all strictly balanced line bundles on quasi-wp-stable curves and let us show that
it is true for our multidegree $\un d$ on $X$, assuming that $\un d$ is not strictly balanced. Let $L\in \Pic^{\un d}(X)$. Since $\un d$ is not strictly balanced, by Lemma \ref{L:specia}\eqref{L:specia2}
there exists an isotrivial specialization $(X,L)\sp (X',L')$ such that $\un d':=\deg L'$ is a strictly balanced multidegree on $X'$. Moreover, from the proof of the cited Lemma, it follows easily that
the curve $X'$ and the multidegree $\un d'$ depend only on $X$ and $\un d$ and not on $L\in \Pic^{\un d}(X)$.
Note that, since $X'$ is obtained from $X$ by bubbling  some nodes of $X$, then $X$ has some elliptic tails
if and only if $X'$ has some elliptic tails. Therefore, according to Reduction 1, $L$ and $L'$ are non-special and very ample.
Up to the choice of a basis of $H^0(X,L)$ and of $H^0(X',L')$, we get two points of $\Hilb_d$, namely
$[X\stackrel{|L|}{\hookrightarrow} \P^r]$ and $[X'\stackrel{|L'|}{\hookrightarrow} \P^r]$.
These two points are indeed well-defined only up to the action of the group $\SL_{r+1}$. From Corollary \ref{C:Im-p}, we get that
$[L]\in \Im(M_X^{\un d}\stackrel{p}{\to} \Pic^{\un d}(X)/\Aut^{\un d}(X))$ if and only if $[X\stackrel{|L|}{\hookrightarrow} \P^r]\in \Ch^{-1}(\Chow_d^{ss})^o$, and similarly that
$[L']\in \Im(M_{X'}^{\un d'}\stackrel{p'}{\to} \Pic^{\un d'}(X')/\Aut^{\un d'}(X'))$ if and only if $[X'\stackrel{|L'|}{\hookrightarrow} \P^r]\in \Ch^{-1}(\Chow_d^{ss})^o$.
Therefore, Theorem \ref{T:spec-clos-orb}\eqref{T:spec-clos-orb2} gives that $[L]\in \Im(p)$ if and only if $[L']\in \Im(p')$.
In other words, we have defined a set-theoretic map
$$\begin{aligned}
\Upsilon: \Pic^{\un d}(X) &\to \Pic^{\un d'}(X') \\
L & \mapsto L'
\end{aligned}
$$
such that $\Upsilon^{-1}({\theta'}^{-1}(\Im(p')))=\theta^{-1}(\Im(p))$., where $\theta:\Pic^{\un d}(X)\to \Pic^{\un d}/\Aut^{\un d}(X)$ and $\theta':\Pic^{\un d'}(X')\to \Pic^{\un d'}/\Aut^{\un d'}(X')$ are the projection maps.
The proposition for $\un d'$ is equivalent to the fact that either
${\theta'}^{-1}(\Im(p'))=\emptyset$ or ${\theta'}^{-1}(\Im(p'))=\Pic^{\un d'}(X')$. Using the above map $\Upsilon$,
it is easy to see that the above properties hold also for $\un d$,  q.e.d.

\vspace{.1cm}
We now prove the proposition for a pair $(X,\un d)$ satisfying the properties of Reduction 1 and Reduction 2 and one of the conditions \eqref{compl1}, \eqref{compl2} and \eqref{compl3}.
Assume that $M_X^{\un d}\neq \emptyset$, for otherwise there is nothing to prove. Let us first prove the following

\underline{CLAIM:} $\theta^{-1}(\Im(p))\subset \Pic^{\un d}(X)$  is open and dense.

Consider a Poincar\'e line bundle $\cP$ on $X\times \Pic^{\un d}(X)$, i.e., a line bundle $\cP$ such that $\cP_{|X\times \{L\}}\cong L$ for every
$L\in \Pic^{\un d}(X)$ (see \cite[Ex. 4.3]{Kle}). By Reduction 1, it follows that
%According to Theorem \ref{bal-pos}\eqref{bal-ns}, every $L\in \Pic^{\un d}(X)$ is non-special since $X$ is quasi-wp-stable, %hence $G$-semistable, and $\deg L=d>2(2g-2)>2g-2$
%(recall that $g\geq 2$). Moreover, if $d<\frac{5}{2}(2g-2)$ then $X$ does not contain any elliptic tails by
%Theorem \ref{T:ell-curves}. Since every $L\in \Pic^{\un d}(X)$ is ample
%by the assumption that $M_X^{\un d}\neq \emptyset$ and the fact that ampleness is a numerical condition,
%we can apply Theorem \ref{bal-pos}\eqref{bal-va} to conclude that every $L\in \Pic^{\un d}(X)$ is very ample.
$\cP$ is relatively very ample with respect to the projection $\pi_2: X\times \Pic^{\un d}(X)\to \Pic^{\un d}(X)$ and that
$(\pi_2)_*(\cP)$ is locally free of rank equal to $r+1=d-g$ .
We can therefore find a Zariski open cover $\{U_i\}_{i\in I}$ of $\Pic^{\un d}(X)$ such that $(\pi_2)_*(\cP)_{|U_i}\cong \OO_{U_i}^{\oplus (r+1)}$ and the line bundle $\cP$ induces an embedding
$$\xymatrix{
X\times U_i \ar@{^{(}->}[rr]^{\eta_i}\ar[dr]_{\pi_2} & & \P(\OO_{U_i}^{r+1})=\P^r_{U_i}\ar[ld],\\
& U_i &
}
$$
The above embedding corresponds to a map $f_i:U_i\to \Hilb_d$ and, using Corollary \ref{C:Im-p}, we get  that
$$\theta^{-1}(\Im(p)) =\bigcup_i f_i^{-1}(\Ch^{-1}(\Chow_d^{ss})^o).$$
Since $\Ch^{-1}(\Chow_d^{ss})^o$ is open inside $\Chow^{-1}(\Chow_d^{ss})$ (by the discussion at the beginning of Section \ref{S:stratifica}) and  $\Ch^{-1}(\Chow_d^{ss})$ is open in $\Hilb_d$
(because any GIT-semistability condition is open),
it follows that $f_i^{-1}(\Ch^{-1}(\Chow_d^{ss})^o)$ is open inside $U_i$; hence $\theta^{-1}(\Im(p)) $ $\subseteq \Pic^{\un d}(X)$ is open  as well.
Moreover, since $\Pic^{\un d}(X)$ is irreducible and $M_X^{\un d}\neq \emptyset$, we get that $\theta^{-1}(\Im(p))\subseteq \Pic^{\un d}(X)$ is also dense, q.e.d.

\vspace{0.1cm}

In order to finish the proof, it remains to show that $\theta^{-1}(\Im(p))\subseteq \Pic^{\un d}(X)$ is closed.
Since $\theta^{-1}(\Im(p))$ is open by the CLAIM, it is enough to prove that $\theta^{-1}(\Im(p))$ is closed under specializations (see \cite[Ex. II.3.18(c)]{Har}), i.e.,
if $B\subseteq \Pic^{\un d}(X)$ is a smooth curve such that $B\setminus \{b_0\}\subseteq \theta^{-1}(\Im(p))$ then $b_0\in \theta^{-1}(\Im(p))$.
The same construction as in the proof of the Claim gives, up to shrinking $B$ around $b_0$, a map $f:B\to \Hilb_d$ such that $f(B\setminus \{b_0\})\subset \Ch^{-1}(\Chow_d^{ss})^o\subseteq \Ch^{-1}(\Chow_d^{ss})$. We denote by $\mathcal L$ the relatively ample line bundle on $\pi_1:\cX:=X\times B\to B$ which gives the embedding into $\P^r_B$.

We can now apply a fundamental result in GIT, called \emph{polystable replacement property} (see e.g. \cite[Thm. 4.5]{HH2}), which implies that, up to replacing $B$ with a finite cover ramified over $b_0$, we can find two maps
$g:B\to \Ch^{-1}(\Chow_d^{ss})^o$ and $h:B\setminus \{b_0\}\to \SL_{r+1}$ such that
\begin{equation}\label{E:2fam-equiv}
f(b)=h(b)\cdot g(b)  \text{ for every } b_0\neq b\in B,
\end{equation}
\begin{equation}\label{E:poly-limit}
g(b_0) \text{ is Chow polystable. }
\end{equation}
We denote by $\pi_2: \cY\to B$ the pull-back of the universal family over $\Ch^{-1}(\Chow_d^{ss})^o$ via the map $g$ and by $\cM$ the line bundle on $\cY$ which is the pull-back of the universal line bundle via $g$. Property \eqref{E:2fam-equiv} implies that
$X\cong \cY_b$ and $\un{\deg}\cM_{|\cY_b}=\un d$ for every $b_0\neq b\in B$. Moreover, if we set
$Y:=\cY_{b_0}$, $M:=\cM_{|Y_0}$ and $\un d':=\un{\deg} M$, then Proposition \ref{P:deg-strata} implies that
$(Y,\un d')\preceq (X, \un d)$. Observe also that \eqref{E:poly-limit} together with Corollary \ref{C:poly-strictly-bal} imply that
$M$ is strictly balanced.

\vspace{0.1cm}
Assume now that \eqref{compl1} holds. By Corollary \ref{C:quasi-wp-stable}\eqref{C:quasi-wp-stable2}, $X$ and $Y$ are quasi-stable curves. Therefore, $(Y,\un d')$ is obtained from
$(X,\un d)$ via a sequence of bubbling  of nodes, as depicted in Figure \ref{blowUpNode}. In particular, there exists
a surjective map $\sigma:Y\to X$ that contracts the new exceptional components produced by bubbling  some of the nodes of $X$.
Hence there exists a map $\Sigma:\cY\to \cX$ over $B$ which is an isomorphism away from the fiber over $b_0$ and whose restriction over $b_0\in B$ is the contraction map $\sigma:Y\to X$.
%of Definition \ref{D:order-rela}\eqref{D:order-rela1}.
Consider the line bundle
$\w{\cL}:=\Sigma^*(\cL)$ on $\cY$ and set $\w{L}:=\w{\cL}_{|\cY_{b_0}}=\sigma^*(L)$ and $\w{\un d}=\un{\deg}(\w{L})$.
%It is easy to check (and left to the reader) that $\w{\un d}$ is a balanced multidegree on $Y$ (see Definition
%\ref{D:bala-multdeg}) and moreover, clearly, $\w{\un d}_E=0$ for every exceptional component $E\cong \P^1$ of $Y$ that gets contracted by the map $\sigma:Y\to X$.
 Property \eqref{E:2fam-equiv} implies that, up to shrinking $B$ around $b_0$, $\w{\cL}$ and $\cM$ are isomorphic away
from the central fiber $\cY_{b_0}=Y$; hence, by Lemma \ref{L:twister}, we can find a Cartier divisor $T$ on $\cY$ supported on the central fiber $\cY_{b_0}=Y$ such that
\begin{equation}\label{E:twist}
\w{\cL}=\cM\otimes \OO_{\cY}(T).
\end{equation}
This implies that the multidegrees $\un d'$ and $\w{\un d}$ on $Y$ are equivalent in the sense of Definition \ref{D:equiv-mult}.
Since $\un d$ is strictly balanced by Reduction 1,
we can now apply Lemma \ref{L:equiv-lb} (with $Z=Y$ and $\sigma'=\id$) in order to conclude that $X=Y$ or,
equivalently, $\cX=\cY$.
Since we have already observed that $(Y,\un d')\preceq (X,\un d)$, we must have that $\un d=\un d'$. Combining this with \eqref{E:twist}, we get that
$L:=\cL_{\cX_{b_0}}=\cM_{\cX_{b_0}}=M$. We deduce that $b_0=L=M\in \theta^{-1}(\Im(p))$ by combining Corollary \ref{C:Im-p} with \eqref{E:poly-limit}, q.e.d.

\vspace{0.1cm}
Next, assume that \eqref{compl2} holds. By Corollary \ref{C:quasi-p-stable}, $X$ and $Y$ are quasi-p-stable. Therefore, $(Y,\un d')$ is obtained from
$(X,\un d)$ via a sequence of bubbling  of nodes and cusps, as depicted in Figures \ref{blowUpNode} and \ref{blowUpCusp}.
  Thus we get again a map $\Sigma:\cY\to \cX$ over $B$ with the same properties as in case \eqref{compl1} and the argument is completely analogous to the previous one.

\vspace{0.1cm}
Finally, assume  that \eqref{compl3} holds. In this case, the curve $X$ does not contain elliptic tails by assumption, whereas the curve $Y$ might contain some elliptic tails.
 Denote by $M'$ the line bundle on $Y$ such that $M'_{|Y_{\text{ell}}^c}=M_{|Y_{\text{ell}}^c}$ and $M'_{|F}$ is special for each elliptic tail $F$ of $Y$. As in the proof of ($\Leftarrow$) in
 Proposition \ref{P:deg-strata}, up to shrinking $B$ around $b_0$, we can find a relatively very ample line bundle $\cM'$ on $\cY$ so that  $\cM'_{|Y}=M'$ and $\cM'_{|\cY_b}$ has the same
 multidegree as $\cM_{|\cY_b}$ for each $b\neq b_0$. Let $F_1,\ldots,F_n$ be the elliptic tails that compose the elliptic locus $Y_{\text{ell}}$, denote by $p_i$ the intersection point of $F_i$ with
 $(F_i)^c$ and define the line bundle $\cN$ as follows:
$$
\cN=\cM'\otimes \OO_{\cY}(4F_1+\ldots+4F_n).
$$
Since $\OO_{\cY}(4F_1+\ldots+4F_n)_{|F_i}=\OO_{F_i}(-4p_i)$ and $\cM'_{|F_i}=\OO_{F_i}(4p_i)$, we deduce that $N:=\cN_{|Y}$ is trivial on each $F_i$.
Therefore, the line bundle $\cN$ is relatively globally generated and the induced map
$$\xymatrix{
\cY \ar[rr]^{\phi_{|\cN|}}\ar[dr]_{\pi_2} & & \P^r_B\ar[ld],\\
& B &
}
$$
embeds $\pi_2^{-1}(B\setminus \{b_0\})$ in $\P^r_B$ and contracts exactly $Y_{\text{ell}}\subset Y$ over $b_0$. Denote by $\cZ$ the image of $\cY$ in $\P^r_B$ via $\phi_{|\cN|}$ and
$\pi_3:\cZ\to B$ the restriction of $\P^r_B\to B$ to $\cZ\subset \P^r_B$. Setting $Z=\cZ_{b_0}$ and $\un d'':=\un{\deg}\, \OO_{\P^r_B}(1)_{|Z}$, Proposition \ref{P:deg-strata} implies that
$(Z,\un d'')\preceq (X,\un d)$. Since $Z$ does not contain elliptic tails, $(Z,\un d'')$ is obtained from
$(X,\un d)$ via a sequence of bubbling  of nodes and cusps, as depicted in Figures \ref{blowUpNode} and \ref{blowUpCusp}. Hence, as in part \eqref{compl1},
there exists a map $\Sigma:\cZ\to \cX$ which is an isomorphism away from the central fiber and whose restriction to the central fiber is the contraction of the exceptional components of $Z$
produced by blowing up some nodes and cusps of $X$. Summing up, we have the commutative diagram
$$\xymatrix{
\cY \ar[rr]^{\phi_{|\cN|}}\ar[rdr]_{\pi_2} & & \cZ \ar[rr]^{\Sigma} \ar[d]^{\pi_3} & & \cX \ar[ldl]^{\pi_1},\\
& & B & &
}
$$
Composing $\phi_{|\cN|}$ with $\Sigma$, we obtain a map $\Sigma':\cY\to \cX$ over $B$, whose restriction over $b_0$ contracts $Y_{\text{ell}}$ and possibly some exceptional components of $Y$.
As above, consider the line bundle $\w{\cL}:=(\Sigma')^*\cL$ and let $T$ be the Cartier divisor on $\cY$, supported on $\cY_{b_0}=Y$, such that
\begin{equation}\label{eq:MLT}
\w{\cL}=\cM\otimes \OO_{\cY}(T).
\end{equation}
If we prove that $Y_{\text{ell}}=\emptyset$, we are done since $\Sigma':\cY \ra \cX$ contracts only exceptional components as in the case  \eqref{compl1} and \eqref{compl2}. Suppose,
by contradiction, that $Y$ admits an elliptic tail, which we denote by $F$. Set $C=F^c$ and denote by $p$ the intersection point of $F$ with $F^c$. Note that
 $\OO_{\cY}(T)_{|F}$ is equal to the line bundle associated to some multiple of $p$, i.e. $F$ is special with respect to $\OO_{\cY}(T)_{|F}$ (see Definition \ref{D:elltails-types}).
On the other hand, since  $F$ is contracted by $\Sigma'$ and $\w{\cL}$ is the pull-back of $\cL$ via $\Sigma'$, we have that
$\w{\cL}_{|F}=\OO_{F}$. Therefore, using \eqref{eq:MLT}, we deduce that $F$ is also special with respect to $M$.
%Set $C=F^c$ and denote by $p$ the intersection point of $F$ with $F^c$.
%Up to tensoring the equality (\ref{eq:MLT})  with $\OO_{\cY}(m(F+C))$ for a suitable integer $m$, we can assume that the support of $T$ does not contain $F$.
%Since $\deg M_{|F}=4$ and $\w{\cL}_{|F}=\OO_{F}$, from equality \eqref{eq:MLT} we have that $\deg\, \OO_{\cY}(T)_{|F}=4$. Moreover, since $F$ intersects $F^c$ only in $p$,
%we must have that $\OO_{\cY}(T)_{|F}=\OO_{F}(-4p)$. Therefore, from \eqref{eq:MLT} we get that
%$$
%M_{|F}=\cM_{|F}=\w{\cL}_{|F}\otimes \OO_{\cY}(-T)_{|F}=\OO_F(4p).
%$$
This is absurd because the point $[Y\stackrel{|M|}{\inj} \P^r]$ is Chow semistable and  cannot have special elliptic tails by Theorem \ref{T:spec-ell}.

\end{proof}

The following well-known Lemma  (see e.g. the proof of \cite[Prop. 6.1.3]{Ray}) was used in the above proof of Proposition \ref{P:completeness}.

%generalization of \cite[Lemma 5.2]{Cap} to our setting.

\begin{lemma} \label{L:twister}
Let $B$ be a smooth curve and let $f:\mathcal X\to B$ be a flat and proper morphism. Fix a point $b_0\in B$ and set $B^*=B\setminus
\{b_0\}$.
Let $\cL$ and $\cM$ be two line bundles on $\cX$ such that $\cL_{|f^{-1}(B^*)}= \cM_{|f^{-1}(B^*)}$.
Then
$$\cL=\cM\otimes \cO_{\cX}(D),$$
 where $D$ is a Cartier divisor on $\cX$ supported on $f^{-1}(b_0)$.
%Let $({\mathcal X}, {\mathcal L})$ and $({\mathcal Y}, {\mathcal M})$ be
%two generic, locally trivial one-parameter families of quasi-wp-stable curves over a smooth pointed base $(B, b_0)$. Assume that over $B^*=B\setminus
%\{b_0\}$ there exists an isomorphism of polarized families. Then there exists a family ${\mathcal W} \rightarrow (B, b_0)$ of quasi-wp-stable
%curves having a birational morphism onto the previous families and biregular away from $b_0$, and such that the pull-backs of ${\mathcal L}$ and
%${\mathcal M}$ to ${\mathcal W}$ differ by a Cartier divisor supported on the central fiber.
 \end{lemma}

\begin{rmk}\label{R:poly-completeness}
Since in the proof of Proposition \ref{P:completeness} we applied the \emph{polystable replacement property}, a stronger result holds: if $[X\subset \P^r]$ is Hilbert (resp. Chow) semistable, $\OO_X(1)$ is strictly balanced and one of the conditions of Proposition \ref{P:completeness} is satisfied, then $[X\subset \P^r]$ is Hilbert (resp. Chow) polystable. This result may be viewed as a partial converse to Corollary \ref{C:poly-strictly-bal}.
\end{rmk}

\begin{rmk}
\noindent
\begin{enumerate}[(i)]
  \item Proposition \ref{P:completeness} is false in the case $\frac{7}{2}(2g-2)\leq d<4(2g-2)$ if the curve $X$ is not assumed to be quasi-p-stable
  (see Remark \ref{rmk:code} and also Theorem \ref{T:semistable3}\eqref{T:semistable3c} and Theorem \ref{T:semistable4}).
	\item The careful reader will have noticed that we do not say anything if $d=\frac{7}{2}(2g-2)$ nor $4(2g-2)$. Actually, Proposition \ref{P:completeness} can be extended to the cases
	$d=4(2g-2)$ (for $X$ quasi-wp-stable) and $d=\frac{7}{2}(2g-2)$ (for $X$ quasi-p-stable).
	In our presentation, we will only use the extension to $d=4(2g-2)$, but we are not ready yet to prove it. Its proof requires the analysis of stability of elliptic tails and will be dealt with later (see Proposition \ref{P:completeness3}).
\end{enumerate}
\end{rmk}

The following result is an immediate consequence of Proposition \ref{P:completeness}.

\begin{coro}\label{C:semist-num}
Let $[X\stackrel{i}{\hookrightarrow} \P^r], [X\stackrel{i'}{\hookrightarrow} \P^r ]$ points in $\Hilb_d$. Assume that one of the conditions of Proposition \ref{P:completeness} is satisfied and
$\un \deg \: i^*\OO_{\P^r}(1)=\un \deg \: {i'}^*\OO_{\P^r}(1)$. Then $[X\stackrel{i}{\hookrightarrow} \P^r]$ belongs to $\Ch^{-1}(\Chow_d^{ss})$ (resp. $\Hilb_d^{ss}$) if and only if
$[X\stackrel{i'}{\hookrightarrow} \P^r]$ belongs to $\Ch^{-1}(\Chow_d^{ss})$ (resp. $\Hilb_d^{ss}$).

%Let $\un d=\un{\deg}\: \OO_X(1)$ and $\un d':=\un{\deg} \:\OO_{X'}(1)$.
%Assume that
%\begin{enumerate}[(i)]
%\item If $d>4(2g-2)$ then $(X,\un d)\preceq (X',\un d')$;
%\item If $2(2g-2)<d<4(2g-2)$ then $(X,\un d)\preceq (X',\un d')$.
%\end{enumerate}
%Then if $[X\subset \P^r]\in \Ch^{-1}(\Chow_d^{ss})$ then $[X'\subset \P^r]\in \Ch^{-1}(\Chow_d^{ss})$.
\end{coro}
\begin{proof}
Let us first prove the statement for the Chow semistability.
Assume that $[X\stackrel{i}{\hookrightarrow} \P^r]\in \Ch^{-1}(\Chow_d^{ss})$. This is equivalent to saying that $[X\stackrel{i}{\hookrightarrow} \P^r]\in M_X^{\un d}$
where $\un d:=\un \deg \: i^*\OO_{\P^r}(1)=\un \deg \: {i'}^*\OO_{\P^r}(1)$. In particular, $M_X^{\un d}\neq \emptyset$; hence, from Proposition \ref{P:completeness} and Corollary \ref{C:Im-p},
we deduce that  $[X\stackrel{i'}{\hookrightarrow} \P^r]\in M_X^{\un d}$, or in other words $[X'\stackrel{i'}{\hookrightarrow} \P^r]\in \Ch^{-1}(\Chow_d^{ss})$, q.e.d.

%there exists $[X\stackrel{j}{\hookrightarrow} \P^r]\in M_X^{\un d}$ such that $j^*\OO_{\P^r}(1)=(i')^*\OO_{\P^r}(1)$. However, this implies that
%$[X\stackrel{j}{\hookrightarrow} \P^r]$ is in the orbit of $[X\stackrel{i'}{\hookrightarrow} \P^r]$. Since each stratum $M_X^{\un d}$ is $\SL_{r+1}$-invariant, we get that
% $[X\stackrel{i'}{\hookrightarrow} \P^r]\in M_X^{\un d}$, q.e.d.

The proof for the Hilbert semistability is similar: we can define a stratification of
$$\Hilb_d^{ss,o}:=\{[X\subset \P^r]\in \Hilb_d^{ss}\: :\: X \text{ is connected}\},$$
whose strata are given by
$$\w{M}_X^{\un d}=\{[X\subset \P^r]\in \Hilb_d^{ss,o}\: : \:\un{\deg} \: \OO_X(1)=\un d\}\subseteq M_X^{\un d}.$$
It is clear that Propositions \ref{P:deg-strata} and \ref{P:completeness} remain valid if we substitute $M_X^{\un d}$ with $\w{M}_X^{\un d}$. Therefore,
the above proof for the Chow semistability extends verbatim to the Hilbert semistability.

%By assumption, we have that $[X\subset \P^r]\in M_X^{\un d}$ (see \eqref{E:strata}); in particular $M_X^{\un d}\neq \emptyset$.
%Using Proposition \ref{P:deg-strata}, together with our assumptions, we deduce that $M_{X'}^{\un d'}\neq \emptyset$.
%Therefore, Proposition \ref{P:completeness} tells us that $\OO_{X'}(1)\in \Pic^{\un d'}(X')$ belongs to the image of the map
%$M_{X'}^{\un d'}\to \Pic^{\un d'}(X')$, which is indeed equivalent to say that $[X'\subset \P^r]\in \Ch^{-1}(\Chow_d^{ss})$.
\end{proof}

%\begin{rmk}\label{R:semist-num-Hilb}
%Under the same assumptions of Corollary \ref{C:semist-num}, if we moreover assume that  $[X\subset \P^r]\in \Hilb_d^{ss}$ then we can conclude that
%$[X'\subset \P^r]\in \Hilb_d^{ss}$. Indeed, we can define a stratification of $\w{H_d}:=\{[X\subset \P^r]\in \Hilb_d^{ss}\: :\: X \text{ is connected}\}$ by setting
%$$\w{M}_X^{\un d}=\{[X\subset \P^r]\in \w{H_d}\: : \: \deg \: \OO_X(1)=\un d\}.$$
%It is clear that Propositions \ref{P:deg-strata} and \ref{P:completeness} remain valid if we substitute $M_X^{\un d}$ with $\w{M}_X^{\un d}$. Therefore,
%the proof of Corollary \ref{C:semist-num} extends verbatim if we substitute $\Ch^{-1}(\Chow_d^{ss})$ with $\Hilb_d^{ss}$.
%\end{rmk}

\section{Semistable, polystable and stable points (part I)}\label{S:semistab1}

The aim of this section is to describe the points of $\Hilb_d$ that are Hilbert or Chow semistable, polystable and stable for
$$
2(2g-2)<d\leq \frac{7}{2}(2g-2) \quad \text{and}\quad d>4(2g-2).
$$
The range $\frac{7}{2}(2g-2)<d\leq 4(2g-2)$ will be investigated later.

Let us begin with the semistable points.
\begin{thm}\label{T:semistable}
Consider a point $[X\subset \P^r]\in \Hilb_d$ and assume that $X$ is connected.
\begin{enumerate}
\item \label{T:semistable1} If $d>4(2g-2)$ then the following conditions are equivalent:
\begin{enumerate}[(i)]
\item \label{T:semistable1i} $[X\subset \P^r]$ is Hilbert semistable;
% \in \Hilb_d^{ss}$;
\item \label{T:semistable1ii} $[X\subset \P^r]$ is Chow semistable;
%\in \Ch^{-1}(\Chow_d^{ss})$, i.e., $[X\subset \P^r]\in H_d$;
\item \label{T:semistable1iii} $X$ is quasi-stable, non-degenerate and linearly normal in $\P^r$ and $\OO_X(1)$ is properly balanced and non-special;
\item \label{T:semistable1iiibis} $X$ is quasi-stable and $\OO_X(1)$ is properly balanced;
\item \label{T:semistable1iv} $X$ is quasi-stable and $\OO_X(1)$ is balanced.
\end{enumerate}
\item \label{T:semistable2} If $2(2g-2)<d<\frac{7}{2}(2g-2)$ and $g\geq 3$ then the following conditions are equivalent:
\begin{enumerate}[(i)]
\item \label{T:semistable2i} $[X\subset \P^r]$ is Hilbert semistable;
% \in \Hilb_d^{ss}$;
\item \label{T:semistable2ii} $[X\subset \P^r]$ is Chow semistable;
%\in \Ch^{-1}(\Chow_d^{ss})$, i.e., $[X\subset \P^r]\in H_d$;
\item \label{T:semistable2iii} $X$ is quasi-p-stable, non-degenerate and linearly normal in $\P^r$ and $\OO_X(1)$ is properly balanced and non-special;
\item \label{T:semistable2iiibis} $X$ is quasi-p-stable and $\OO_X(1)$ is properly balanced;
\item  \label{T:semistable2iv} $X$ is quasi-p-stable and $\OO_X(1)$ is balanced.
\end{enumerate}
\end{enumerate}
\end{thm}
\begin{proof}
Let us first prove part \eqref{T:semistable1}.
%Our proof is a generalization of the proof of \cite[Prop. 6.1]{Cap}; however, since the extensions of the results of loc. cit. to our setting is
%non-trivial, we include full details.

\eqref{T:semistable1i} $\Longrightarrow$ \eqref{T:semistable1ii} follows from  Fact \ref{HilbtoChow}.

\eqref{T:semistable1ii} $\Longrightarrow$ \eqref{T:semistable1iii} follows from the potential stability theorem (see Fact \ref{F:GM2}) and Corollary \ref{C:quasi-wp-stable}\eqref{C:quasi-wp-stable2}.

\eqref{T:semistable1iii} $\Longrightarrow$ \eqref{T:semistable1iiibis} is clear.

\eqref{T:semistable1iiibis} $\Longleftrightarrow$ \eqref{T:semistable1iv} follows from Remark \ref{R:propbal-ample}, using that $\OO_X(1)$ is ample.

\eqref{T:semistable1iiibis} $\Longrightarrow$ \eqref{T:semistable1i} First of all, we make the following

\un{Reduction: } We can assume that $\OO_X(1)$ is strictly balanced.

Indeed, by Lemma \ref{L:specia}\eqref{L:specia2}, there exists an isotrivial specialization $(X,\OO_X(1))\sp (X',L')$ such that $X'$ is quasi-stable and $L'$ is a strictly balanced
line bundle on $X'$ of total degree $d$. According to Theorem \ref{bal-pos} and using that $d>4(2g-2)$, we conclude that $L'$ is very ample and non-special.
Therefore, by choosing a basis of $H^0(X',L')$, we get a point $[X'\stackrel{|L'|}{\hookrightarrow}\P^r]\in \Hilb_d$. According to Theorem \ref{T:spec-clos-orb}, $[X\subset \P^r]\in \Hilb_d^{ss}$
if and only if $[X'\subset \P^r]\in \Hilb_d^{ss}$. Therefore, up to replacing $X$ with $X'$, we can assume that $\OO_X(1)$ is strictly balanced, q.e.d.

Now, since $X$ is quasi-stable, we can find a smooth curve $B\stackrel{f}{\hookrightarrow} \Hilb_d$ and a point $b_0\in B$
such that, if we denote by $\P^r \times B \stackrel{i}{\hookleftarrow} \cX \stackrel{\pi}{\to} B$ the pull-back via $f$ of the universal family over $\Hilb_d$ and we set
$\cL:=i^*(\OO_{\P^r}(1)\boxtimes \OO_B)$,
then $[\cX\stackrel{i}{\hookrightarrow} \P^r\times B]_{b_0}=[X\subset \P^r]$  and $\cX_{|\pi^{-1}(b)}$ is a connected smooth curve for every $b\in B\setminus \{b_0\}$.
Note that, by construction, $\pi$ is a family of
quasi-stable curves of genus $g$. As in the proof of Proposition \ref{P:completeness},
we can now apply the \emph{semistable replacement property}, which implies that, up to replacing $B$ with a finite cover ramified over $b_0$, we can find two maps
$g:B\to \Hilb_d$ and $h:B\setminus \{b_0\}\to \SL_{r+1}$ such that
\begin{equation}\label{E:2fam-equivbis}
f(b)=h(b)\cdot g(b)  \text{ for every } b_0\neq b\in B,
\end{equation}
\begin{equation}\label{E:poly-limitbis}
g(b_0) \text{ is Hilbert polystable. }
\end{equation}
We denote by $\P^r \times B \stackrel{i'}{\hookleftarrow} \cY \stackrel{\pi'}{\to} B$ the pull-back via $g$  of the universal family over $\Hilb_d$ and we set
$\cM:=(i')^*(\OO_{\P^r}(1)\boxtimes \OO_B)$.
Property \eqref{E:2fam-equivbis} implies that, up to shrinking again $B$ around $b_0$, we have that
\begin{equation}\label{E:iso-fam}
(\cX,\cL)_{|\pi^{-1}(B\setminus \{b_0\})}\cong (\cY,\cM)_{|(\pi')^{-1}(B\setminus \{b_0\})}.
\end{equation}
Note that this fact together with \eqref{E:poly-limitbis} and the potential stability Theorem (Fact \ref{F:GM2}) implies that $\pi'$ is also a family of quasi-stable curves of genus $g$.

Consider now the stable reductions $\s(\pi):\s(\cX)\to B$ of $\pi:\cX\to B$ and $\s(\pi'):\s(\cY)\to B$ of  $\pi':\cY\to B$ (see Remark \ref{R:wp-stab}). From \eqref{E:iso-fam}, it follows that
$\s(\pi)$ and $\s(\pi')$ are two families of stable curves, which are isomorphic away from the fibers over $b_0$. Since the stack $\MMg$ of stable curves is separated, we conclude that
\begin{equation}\label{E:iso-stab}
\xymatrix{
\s(\cX)\ar[rr]^{\cong}\ar[rd]_{\s(\pi)}& & \s(\cY)\ar[ld]^{\s(\pi')}\\
& B &
}
\end{equation}
Therefore, $\pi$ and $\pi'$ are two families of quasi-stable curves with the same stable reduction (from now on, we identify $\s(\cX)\stackrel{\s(\pi)}{\longrightarrow} B$ and
$\s(\cY)\stackrel{\s(\pi')}{\longrightarrow} B$ via the above isomorphism). If we blow-up all the nodes of the fiber over $b_0$ of the stable reduction $\s(\pi)=\s(\pi')$, we get
a new family of  quasi-stable curves $\w{\pi}:\cZ\to B$ with the same stable reduction as that of $\pi$ and of $\pi'$, which moreover dominates $\pi$ and $\pi'$,
i.e., such that there exists a commutative diagram
\begin{equation}\label{E:dom-family}
\xymatrix{
& \cZ \ar[dd]^{\w{\pi}} \ar[dr]^{\Sigma'}\ar[dl]_{\Sigma}& \\
\cX\ar[dr]_{\pi} & & \cY \ar[dl]^{\pi'}\\
& B &
}
\end{equation}
where the morphisms $\Sigma$ and $\Sigma'$ induce an isomorphism of the corresponding  stable reductions. Equivalently,
the maps $\Sigma$ and $\Sigma'$ are obtained by contracting some of the exceptional components of the fiber of $\cZ$ over $b_0$.
%Moreover, if there exists an exceptional component
%$E\subset \cZ_{|b_0}$ which is blown down both by $\Sigma$ and $\Sigma'$, then we can replace $\cZ$ by the contraction of $E$ and we still have the same commutative diagram as
%in \eqref{E:dom-family}. Therefore, we can assume that $\cZ$ does not have exceptional components in the fiber over $b_0$ which are contracted by both $\Sigma$ and $\Sigma'$.
If we set $\w{\cL}:=\Sigma^*(\cL)$ and $\w{\cM}:=(\Sigma')^*(\cM)$, then \eqref{E:iso-fam} gives that
$$\w{\cL}_{\w{\pi}^{-1}(B\setminus b_0)}\cong \w{\cM}_{\w{\pi}^{-1}(B\setminus b_0)}.$$
Lemma \ref{L:twister} now gives that there exists a Cartier divisor $D$ on $\cZ$ supported on $\w{\pi}^{-1}(b_0)$ such that
\begin{equation}\label{E:twist-Z}
\w{\cL}=\w{\cM}\otimes \OO_{\cZ}(D).
\end{equation}
We now set $(X,L):=(\cX, \cL)_{b_0}$ and $\un d:=\un{\deg} L$, $(Y, M):=(\cY, \cM)_{b_0}$ and $\un d':=\un{\deg} M$, $Z:=\cZ_{b_0}$, $\w{L}:=\w{\cL}_{b_0}$ and $\w{\un d}:=\deg \w{L}$,
$\w{M}:=\w{\cM}_{b_0}$ and $\w{\un d'}:=\deg \w{M}$. Equation \eqref{E:twist-Z} gives that $\w{\un d}$ and $\w{\un d'}$ are equivalent on $Z$. Moreover,
$\un d$ is strictly balanced by the above Reduction and \un d' is strictly balanced by the assumption \eqref{E:poly-limitbis} together with
Corollary \ref{C:poly-strictly-bal}. Therefore, we can apply Lemma \ref{L:equiv-lb} twice  to conclude that $\cX=\cY$.
Now, the relation \eqref{E:iso-fam} together with the Lemma \ref{L:twister} imply that there exists a Cartier divisor $D'$ on $\cX=\cY$ supported on $\pi^{-1}(b_0)$ such that
\begin{equation}\label{E:twist-XY}
\cL=\cM\otimes \OO_{\cX}(D').
\end{equation}
In particular, we get that $\un d$ is equivalent to $\un d'$. Since $\un d$ and $\un d'$ are strictly balanced, Lemma \ref{L:equiv-strict} implies that $\un d=\un d'$.
Since $[\cY\stackrel{i'}{\hookrightarrow} \P^r\times B]_{b_0}=[Y\stackrel{}{\hookrightarrow} \P^r]\in \Hilb_d^{ss}$ by assumption \eqref{E:poly-limitbis}, Corollary \ref{C:semist-num} gives that
$[X\subset \P^r]\in \Hilb_d^{ss}$, q.e.d.

The proof of part \eqref{T:semistable2} is similar: it is enough to replace quasi-stable curves by quasi-p-stable curves (using Corollary \ref{C:quasi-p-stable}), to replace the stable
reduction by the p-stable reduction and using the fact that the stack $\MMgp$ of p-stable curves of genus $g\geq 3$ is separated.

\end{proof}

From the above Theorem \ref{T:semistable}, we can deduce a description of the Hilbert and Chow polystable and stable points of $\Hilb_d$.

\begin{coro}\label{C:polystable}
Consider a point $[X\subset \P^r]\in \Hilb_d$ and assume that $X$ is connected.
\begin{enumerate}
\item \label{C:polystable1} If $d>4(2g-2)$ then the following conditions are equivalent:
\begin{enumerate}[(i)]
\item \label{C:polystable1i} $[X\subset \P^r] $ is Hilbert polystable;
\item \label{C:polystable1ii} $[X\subset \P^r]$ is Chow polystable;
\item \label{C:polystable1iii} $X$ is quasi-stable, non-degenerate and linearly normal in $\P^r$ and $\OO_X(1)$ is strictly balanced and non-special;
\item \label{C:polystable1iv} $X$ is quasi-stable and $\OO_X(1)$ is strictly balanced.
\end{enumerate}
\item \label{C:polystable2} If $2(2g-2)<d<\frac{7}{2}(2g-2)$ and $g\geq 3$ then the following conditions are equivalent:
\begin{enumerate}[(i)]
\item \label{C:polystable2i} $[X\subset \P^r] $ is Hilbert polystable;
\item \label{C:polystable2ii} $[X\subset \P^r]$ is Chow polystable;
\item \label{C:polystable2iii} $X$ is quasi-p-stable, non-degenerate and linearly normal in $\P^r$ and $\OO_X(1)$ is strictly balanced and non-special;
\item  \label{C:polystable2iv} $X$ is quasi-p-stable and $\OO_X(1)$ is strictly balanced.
\end{enumerate}
\end{enumerate}
\end{coro}
\begin{proof}
Let us prove part \eqref{C:polystable1}.

\eqref{C:polystable1i} $\Longleftrightarrow$ \eqref{C:polystable1ii}:  from Theorem \ref{T:semistable}\eqref{T:semistable1} we get that the Hilbert semistable locus inside $\Hilb_d$ is equal to the
Chow  semistable locus. Since a point of $\Hilb_d$ is Hilbert (resp. Chow) polystable if and only if it is Hilbert (resp. Chow) semistable and its orbit is closed inside the
Hilbert (resp. Chow)  semistable locus, we conclude that also the locus of Hilbert polystable points is equal to the locus of Chow polystable points.

\eqref{C:polystable1ii} $\Longrightarrow$ \eqref{C:polystable1iii} follows from the potential stability theorem (see Fact \ref{F:GM2}), Corollary \ref{C:quasi-wp-stable}\eqref{C:quasi-wp-stable2}
and Corollary \ref{C:poly-strictly-bal}.

\eqref{C:polystable1iii} $\Longrightarrow$ \eqref{C:polystable1iv} is obvious.

\eqref{C:polystable1iv} $\Longrightarrow$ \eqref{C:polystable1i} follows from Theorem \ref{T:semistable} and Remark \ref{R:poly-completeness}.
\end{proof}

\begin{coro}\label{C:stable}
Consider a point $[X\subset \P^r]\in \Hilb_d$ and assume that $X$ is connected.
\begin{enumerate}
\item \label{C:stable1} If $d>4(2g-2)$ then the following conditions are equivalent:
\begin{enumerate}[(i)]
\item \label{C:stable1i} $[X\subset \P^r]$ is Hilbert stable;
% \in \Hilb_d^s$;
\item \label{C:stable1ii} $[X\subset \P^r]$ is Chow stable;
%\in \Ch^{-1}(\Chow_d^s)$;
\item \label{C:stable1iii} $X$ is quasi-stable, non-degenerate and linearly normal in $\P^r$ and $\OO_X(1)$ is stably balanced and non-special;
\item \label{C:stable1iv} $X$ is quasi-stable and $\OO_X(1)$ is stably balanced.
\end{enumerate}
\item \label{C:stable2} If $2(2g-2)<d<\frac{7}{2}(2g-2)$ and $g\geq 3$ then the following conditions are equivalent:
\begin{enumerate}[(i)]
\item \label{C:stable2i} $[X\subset \P^r]$ is Hilbert stable;
%\in \Hilb_d^s$;
\item \label{C:stable2ii} $[X\subset \P^r]$ is Chow stable;
%\in \Ch^{-1}(\Chow_d^s)$;
\item \label{C:stable2iii} $X$ is quasi-p-stable, non-degenerate and linearly normal in $\P^r$ and $\OO_X(1)$ is stably balanced and non-special;
\item  \label{C:stable2iv} $X$ is quasi-p-stable and $\OO_X(1)$ is stably balanced.
\end{enumerate}
\end{enumerate}
\end{coro}
\begin{proof}
Let us prove part \eqref{C:stable1}.

\eqref{C:stable1ii} $\Longrightarrow$ \eqref{C:stable1i} follows from Fact \ref{HilbtoChow}.

\eqref{C:stable1i} $\Longrightarrow$ \eqref{C:stable1iii} follows from the Potential stability theorem (see Fact \ref{F:GM2}) and Theorem \ref{T:stab-bal}.

\eqref{C:stable1iii} $\Longrightarrow$ \eqref{C:stable1iv} is obvious.

\eqref{C:stable1iv} $\Longrightarrow$ \eqref{C:stable1ii}: from Corollary \ref{C:polystable}\eqref{C:polystable1}, we get that $[X\subset \P^r]$ is Chow polystable.
Lemma \ref{compa-bal} gives that $\w{X}:=\ov{X\setminus \exc}$ is connected; hence, combining Lemma \ref{L:aut-stab} and Theorem \ref{T:auto-grp},
we deduce that $\Stab_{\PGL_{r+1}}([X\subset \P^r])$
 is a finite group. This implies that $[X\subset \P^r]\in \Ch^{-1}(\Chow_d^s)$ since a point of $\Hilb_d$ is Hilbert (resp. Chow) stable if and only if it is Hilbert (resp. Chow) polystable
 and it has finite stabilizers with respect to the action of $\PGL_{r+1}$.

The proof of part \eqref{C:stable2} is similar, using the Potential pseudo-stability Theorem \ref{teo-pstab} and Corollary \ref{C:polystable}\eqref{C:polystable2}.
\end{proof}

The characterization of the GIT semistable locus for $\frac{7}{2}(2g-2)\leq d\leq 4(2g-2)$ is a bit more intricate and requires other arguments, as the following Remark points out.

%We can understand this by the following remark.

\begin{rmk}\label{rmk:code}
Let $X=C\cup E$ be a curve of genus $g\geq 3$ whose only singularity is a tacnode with the line $E$ and let us fix a balanced line bundle $L$ of degree $\frac{7}{2}(2g-2)\leq d\leq 4(2g-2)$. Consider a point $[X\subset \P^r]\in \text{Hilb}_d$ with $\OO_X(1)=L$ and let us try going over again the argument of the proof of Theorem \ref{T:semistable}\eqref{T:semistable1}. Using the same notation, since $X$ is quasi-p-stable, we can find a polarized family $(\cX\ra B,\cL)$ over a smooth curve $B\stackrel{f}{\hookrightarrow} \Hilb_d$ such that $[\cX\stackrel{i}{\hookrightarrow} \P^r\times B]_{b_0}=[X\subset \P^r]$  and $\cX_{|\pi^{-1}(b)}$ is a connected smooth curve for every $b\in B\setminus \{b_0\}$. We apply the \emph{polystable replacement property} and we obtain a new polarized family $(\cY\ra B,\cM)$. Consider now the p-stable reductions $\ps(\pi):\ps(\cX)\ra B$ and $\ps(\pi):\ps(\cY)\ra B$. Up to shrinking $B$ around $b_0$ we have
\begin{equation}\label{E:iso-fam2}
(\cX,\cL)_{|\pi^{-1}(B\setminus \{b_0\})}\cong (\cY,\cM)_{|(\pi')^{-1}(B\setminus \{b_0\})}
\end{equation}
so that $\ps(\pi):\ps(\cX)\ra B$ and $\ps(\pi):\ps(\cY)\ra B$ are isomorphic away from the fibers over $b_0$, hence isomorphic everywhere since the stack $\MMgp$ of p-stable curves
is separated for $g\geq 3$. In particular $\ps(Y)\cong\ps(X)$. There are three cases:
\begin{enumerate}
	\item $Y\cong X$;
	\item $Y\cong \wps(X)$, or, in other words, $Y$ is irreducible with only a cusp and no nodes;
	\item $Y$ admits an elliptic tail.
\end{enumerate}
We claim that only (3) occurs. Indeed, case (1) is absurd because $[Y\stackrel{|M|}{\inj}\P^r]$ is Chow polystable by construction but the tacnodes with a line are not Chow polystable for
$d=\frac{7}{2}(2g-2)$ by Theorem \ref{T:basintacn}, and they are Chow unstable for $d>\frac{7}{2}(2g-2)$ by Theorem \ref{T:tac-line}. Suppose, by contradiction, that (2) occurs.
We have a map
\begin{equation}\label{E:iso-tacncusp}
\xymatrix{
\cX\ar[rr]^{\text{$\wps$}}\ar[rd]_{\pi}& & \cY\ar[ld]^{\pi'}\\
& B &
}
\end{equation}
Denote by $\w{\cL}$ the pull-back of $\cM$ via $\wps$, $L'=\w{\cL}_{|b_0}$, $\un d=\un \deg{L}=(\deg_C L,\deg_E L)$ and $\un d'=\un \deg{L'}$. By Lemma \ref{L:twister}, there exists a Cartier divisor $T$ on $\cX$ such that $\w{\cL}=\cL\otimes \OO_{\cX}(T)$. This implies that
$$
(d-1,1)=\un d\equiv \un d'=(d,0),
$$
which is absurd since $|C\cap E|=2$. We conclude that in $\Hilb_d$ there are examples of Chow semistable points that admit elliptic tails.
This fact is the origin of some new difficulties in the range $\frac{7}{2}(2g-2)\leq d\leq 4(2g-2)$.
So far, our techniques worked well since the stacks $\MMg$ and $\MMgp$ (for $g\geq 3$) are separated, but for $\frac{7}{2}(2g-2)\leq d\leq 4(2g-2)$ they are not enough to determine the semistable
locus of $\Hilb_d$ because we have to work with the stack $\MMgwp$ of wp-stable curves, which is not separated. Notice also that,
if we could use the same techniques  in the range $\frac{7}{2}(2g-2)\leq d\leq 4(2g-2)$ successfully,
we would prove, for instance, the completeness result of Proposition \ref{P:completeness} for every quasi-wp-stable curve, which is false since special elliptic curves are Chow unstable by
Theorem \ref{T:spec-ell}.
\end{rmk}

To conclude this section we study the extremal case $d=\frac{7}{2}(2g-2)$, a very interesting case, because the semistable loci with respect to Hilbert stability and Chow stability are different.
\begin{thm}\label{T:semistable3}
Consider a point $[X\subset \P^r]\in \Hilb_d$ with $d=\frac{7}{2}(2g-2)$ and $g\geq 3$ and assume that $X$ is connected.
\begin{enumerate}
\item \label{T:semistable3h} The following conditions are equivalent:
\begin{enumerate}[(i)]
\item \label{T:semistable3hi} $[X\subset \P^r]$ is Hilbert semistable;
% \in \Hilb_d^{ss}$;
\item \label{T:semistable3hiii} $X$ is quasi-p-stable, non-degenerate and linearly normal in $\P^r$ and $\OO_X(1)$ is properly balanced and non-special;
\item \label{T:semistable3hiiibis} $X$ is quasi-p-stable and $\OO_X(1)$ is properly balanced;
\item \label{T:semistable3hiv} $X$ is quasi-p-stable and $\OO_X(1)$ is balanced.
\end{enumerate}
\item \label{T:semistable3c} The following conditions are equivalent:
\begin{enumerate}[(i)]
\item \label{T:semistable3ci} $[X\subset \P^r]$ is Chow semistable;
%\in \Ch^{-1}(\Chow_d^{ss})$, i.e., $[X\subset \P^r]\in H_d$;
\item \label{T:semistable3cii} $X$ is quasi-wp-stable without special elliptic tails, non-degenerate and linearly normal in $\P^r$ and $\OO_X(1)$ is properly balanced and non-special;
\item \label{T:semistable3ciibis} $X$ is quasi-wp-stable without special elliptic tails and $\OO_X(1)$ is properly balanced;
\item  \label{T:semistable3ciii} $X$ is quasi-wp-stable without special elliptic tails and $\OO_X(1)$ is balanced.
\end{enumerate}
\end{enumerate}
\end{thm}
\begin{proof}
The proof of (\ref{T:semistable3h}) is analogous to Theorem \ref{T:semistable}\eqref{T:semistable2} since for $d=\frac{7}{2}(2g-2)$ the elliptic tails are Hilbert unstable by Theorem \ref{T:ell-curves}. Let us prove (\ref{T:semistable3c}).

(\ref{T:semistable3ci}) $\Longrightarrow$ (\ref{T:semistable3cii}) follows from the Potential pseudo-stability Theorem \ref{teo-pstab} and Theorem \ref{T:spec-ell}.

(\ref{T:semistable3cii}) $\Longrightarrow$ (\ref{T:semistable3ciibis}) is clear.

(\ref{T:semistable3ciibis}) $\Longrightarrow$ (\ref{T:semistable3ciii}) is obvious.

(\ref{T:semistable3ciii}) $\Longrightarrow$ (\ref{T:semistable3ci}). By Theorem \ref{T:basintacn} and Theorem \ref{T:spec-clos-orb} we can assume that
\begin{enumerate}[(a)]
	\item \label{T:semistable3c-ass1} each elliptic tail $F$ of degree $4$ contains an elliptic tail $F'$ of degree 3 as a subcurve;
	\item \label{T:semistable3c-ass2} each elliptic tail $F$ of degree $3$ is tacnodal and $F^c$ consists of the union of subcurves $C$ and $E\cong \P^1$, where $E$ meets $C$ and $F$ in one point;
\end{enumerate}
Let $n$ be the number of elliptic tails of degree 3. We prove our statement by induction on $n$. If $n=0$, then $[X\subset \P^r]$ is Chow semistable by (\ref{T:semistable3h}). Suppose that $n>0$. Consider an elliptic tail $F$ of degree 3
and the 1ps $\rho$ as in \eqref{E:weights} which, as observed before Theorem \ref{T:basintacn}, satisfies $\displaystyle e_{X,\rho}=2 d\frac{w(\rho)}{r+1}=\frac{7}{3}w(\rho)$.
 By Theorem \ref{T:basintacn} there exists $[Y\subset \P^r]\in A_{\rho^{-1}}([X\subset \P^r])$ that satisfies (\ref{T:semistable3ciii}) and contains $n-1$ elliptic tail. By induction $[Y\subset \P^r]$ is Chow semistable and Fact \ref{F:weight-basin} implies that also $[X\subset \P^r]$ is Chow semistable.
\end{proof}
\begin{coro}\label{C:polystable3}
Consider a point $[X\subset \P^r]\in \Hilb_d$ with $d=\frac{7}{2}(2g-2)$ and $g\geq 3$ and assume that $X$ is connected.
\begin{enumerate}
\item \label{C:polystable3h} The following conditions are equivalent:
\begin{enumerate}[(i)]
\item \label{C:polystable3hi} $[X\subset \P^r] $ is Hilbert polystable;
\item \label{C:polystable3hii} $X$ is quasi-p-stable, non-degenerate and linearly normal in $\P^r$ and $\OO_X(1)$ is strictly balanced and non-special;
\item \label{C:polystable3hiii} $X$ is quasi-p-stable and $\OO_X(1)$ is strictly balanced.
\end{enumerate}
\item \label{C:polystable3c} The following conditions are equivalent:
\begin{enumerate}[(i)]
\item \label{C:polystable3ci} $[X\subset \P^r]$ is Chow polystable;
\item \label{C:polystable3cii} $X$ is quasi-wp-stable, each elliptic tail of $X$ is tacnodal, each tacnode is contained in an elliptic tail, $X$ is non-degenerate and linearly normal in $\P^r$, $\OO_X(1)$ is strictly balanced and non-special;
\item \label{C:polystable3ciii} $X$ is quasi-wp-stable, each elliptic tail is tacnodal, each tacnode is contained in an elliptic tail and $\OO_X(1)$ is strictly balanced.
\end{enumerate}
\end{enumerate}
\end{coro}
\begin{proof}
Since for $d=\frac{7}{2}(2g-2)$ the elliptic tails are Hilbert unstable by Theorem \ref{T:ell-curves}, the argument of Corollary \ref{C:polystable}\eqref{C:polystable2} goes through for (\ref{C:polystable3h}). Let us prove (\ref{C:polystable3c}).

(\ref{C:polystable3ci}) $\Longrightarrow$ (\ref{C:polystable3cii}) is implied by Theorem \ref{T:basintacn}.

(\ref{C:polystable3cii}) $\Longrightarrow$ (\ref{C:polystable3ciii}) is clear.

(\ref{C:polystable3ciii}) $\Longrightarrow$ (\ref{C:polystable3ci}).
Let $X$ and $L:=\OO_X(1)$ be as in (\ref{C:polystable3ciii}). By Theorem \ref{T:basintacn} and Theorem \ref{T:spec-clos-orb} we have to work under the assumptions \eqref{T:semistable3c-ass1} and \eqref{T:semistable3c-ass2} of the proof of Theorem \ref{T:ell-curves}. Let $n$ be the number of elliptic tails of degree 3. We prove our statement by induction over $n$. (For a sketch of the proof strategy, see Construction \ref{C:replellcode}.) If $n=0$, $[X\subset \P^r]$ is Hilbert polystable by (\ref{C:polystable3h}), hence also Chow polystable. Suppose now that $n>0$. Consider an elliptic tail $F$ of degree 3 and denote by $C_1=F^c$ and $\{p\}=F\cap C_1$. Let $C_2$ be a smooth curve of genus $g$, $q$ a point of $C_2$ and $L'_{C_2}\in \text{Pic}^{d+3}(C_2)$. Denote by $(X',L')$ the couple consisting of a curve $X'$ of genus $g'$ and a line bundle $L'$ on $X'$ obtained from $(X,L)$ by replacing $F$ with $(C_2,q,L'_{C_2})$, as in Definition \ref{D:replace}. The line bundle $L'$ has degree $d'=2d$ and is very ample, hence we can consider the point $[X'\subset\P^{r'}]\in \text{Hilb}_{d',g'}$ with $\OO_{X'}(1)=L'$. We notice that
\begin{equation}\label{eq:nunu2}
\nu':=\frac{d'}{2g'-2}=\frac{d}{2g-2}=:\nu.
\end{equation}
Now, we claim that $L'$ is strictly balanced. As we noticed in Remark \ref{R:bala}, it suffices to check the basic inequality \eqref{E:basineq-multideg} for each connected subcurves such that its complementary is connected. Let $D\subset X'$ a connected subcurve. If $D=C_2$ then obviously the basic inequality \eqref{E:basineq-multideg} is satisfied. Otherwise, up to replacing $D$ with $D^c$, we can assume that $D$ does not contain $C_2$ as a subcurve. This implies that $D$ can be seen as a subcurve of $X$. Since $\deg L_{|D}=\deg L'_{|D}$ and $|D\cap \overline{X\setminus D}|=|D\cap \overline{X'\setminus D}|$, the basic inequality \eqref{E:basineq-multideg} is satisfied by (\ref{eq:nunu2}). Now, the point $[X'\subset\P^{r'}]$ admits $n-1$ elliptic tails, hence it is Chow polystable. Consider now $[Y\subset \P^r]\in \Hilb_d$ such that $Y=F\cup E\cup C$, where $C$ is smooth, $E\cong \P^1$, $E$ meets $F$ and $C$ in one point and $\OO_Y(1)$ is balanced. By Theorem \ref{T:semistable3} this point is Chow semistable. Let $[Y'\subset \P^r]\in \overline{\Orb([Y\subset P^r])}\cap \Ch^{-1}(\Chow_d^{ss})$. Denoting by $\un d$ and $\un d'$ the multidegrees of $\OO_Y(1)$ and $\OO_{Y'}(1)$ respectively, by Proposition \ref{P:deg-strata} we get that $(Y',\un d')\preceq (Y,\un d)$, so that $Y\cong Y'$ and $\dim( \Stab_{\PGL_{r+1}}([Y\subset \P^r]))=\dim (\Stab_{\PGL_{r+1}}([Y\subset \P^r]))$. This implies that $[Y\subset \P^r]$ is Chow polystable. Since \eqref{eq:nunu2} holds and $(X,L)$ can be obtained again from $(X',L')$ by replacing $C_2$ with $(F,p,L|_F)$, $[X\subset \P^r]$ is Chow polystable by Corollary \ref{C:stab-repl} and we are done.
\end{proof}
\begin{coro}\label{C:stable3}
Consider a point $[X\subset \P^r]\in \Hilb_d$ with $d=\frac{7}{2}(2g-2)$ and $g\geq 3$ and assume that $X$ is connected.
\begin{enumerate}
\item \label{C:stable3h} The following conditions are equivalent:
\begin{enumerate}[(i)]
\item \label{C:stable3hi} $[X\subset \P^r]$ is Hilbert stable;
% \in \Hilb_d^{ss}$;
\item \label{C:stable3hiii} $X$ is quasi-p-stable, non-degenerate and linearly normal in $\P^r$ and $\OO_X(1)$ is stably balanced and non-special;
\item \label{C:stable3hiv} $X$ is quasi-p-stable and $\OO_X(1)$ is stably balanced.
\end{enumerate}
\item \label{C:stable3c} The following conditions are equivalent:
\begin{enumerate}[(i)]
\item \label{C:stable3ci} $[X\subset \P^r]$ is Chow stable;
%\in \Ch^{-1}(\Chow_d^{ss})$, i.e., $[X\subset \P^r]\in H_d$;
\item \label{C:stable3cii} $X$ is quasi-p-stable without tacnodes, non-degenerate and linearly normal in $\P^r$ and $\OO_X(1)$ is stably balanced and non-special;
\item  \label{C:stable3ciii} $X$ is quasi-p-stable without tacnodes and $\OO_X(1)$ is stably balanced.
\end{enumerate}
\end{enumerate}
\end{coro}
\begin{proof}
Since for $d=\frac{7}{2}(2g-2)$ the elliptic tails are Hilbert unstable by Theorem \ref{T:ell-curves}, the argument of Corollary \ref{C:stable}\eqref{C:stable2} goes through for (\ref{C:stable3h}).
Let us prove (\ref{C:stable3c}).

(\ref{C:stable3ci}) $\Longrightarrow$ (\ref{C:stable3cii}) follows from Corollary \ref{C:polystable3} and Theorem \ref{T:ell-curves}.

(\ref{C:stable3cii}) $\Longrightarrow$ (\ref{C:stable3ciii}) is clear.

(\ref{C:stable3ciii}) $\Longrightarrow$ (\ref{C:stable3ci}):  from Corollary \ref{C:polystable3}\eqref{C:polystable3c}, we get that $[X\subset \P^r]$ is Chow polystable.
Since $\OO_X(1)$ is stably balanced,  Lemma \ref{compa-bal} gives that $\w{X}:=\ov{X\setminus \exc}$ is connected; hence, combining Lemma \ref{L:aut-stab} and Theorem \ref{T:auto-grp},
we deduce that $\Stab_{\PGL_{r+1}}([X\subset \P^r])$ is a finite group.
This implies that $[X\subset \P^r]\in \Ch^{-1}(\Chow_d^s)$ since a point of $\Hilb_d$ is Hilbert (resp. Chow) stable if and only if it is Hilbert (resp. Chow) polystable and it has finite stabilizers
with respect to the action of $\PGL_{r+1}$.
\end{proof}

\section{Stability of elliptic tails}\label{S:stab-elltails}

In this section, we will use the criterion of stability for tails (Proposition \ref{prop:stab-tail}) in order to study the stability of elliptic curves for $\frac{7}{2}(2g-2)<d\leq 4(2g-2)$.
We notice that in this range - by the basic inequality  \eqref{E:basineq-multideg} - it suffices to consider the elliptic curves of degree 4. In particular if $F$ is an elliptic curve of $[X\subset \P^r]$, then $r_1:=h^0(F,\OO_X(1)|_F)-1=3$.
\begin{lemma}\label{lem:complcode}
Let $\frac{7}{2}(2g-2)<d\leq 4(2g-2)$ and let $[X\subset\P^r]\in \rm{Hilb}_{d}$ such that $X=F\cup C$ where $F$ is an elliptic tail (smooth, nodal, cuspidal or reducible nodal). Denote by $\{p\}=F\cap C$ and
$$
\OO_X(1)=(\OO_X(1)|_F,L_2:=\OO_X(1)|_C)\in \emph{Pic}^4(F)\times \emph{Pic}^{d-4}(C).
$$
Let $(F',q)$ be a pointed elliptic curve and denote by $X'$ be the curve obtained from $X$ by replacing $F$ with $(F',q)$, as in Definition \ref{D:replace}.
\begin{enumerate}
\item If $[X\subset\P^r]$ is Hilbert semistable (resp. stable), then $[X'\stackrel{|L|}{\inj}\P^r]$ is Hilbert semistable (resp. stable) for each properly balanced line bundle
  $$
  L\in (\emph{Pic}^4(F')\setminus \{\OO_{F'}(4p)\})\times \{L_2\}
  $$
\item If $[X\subset\P^r]$ is Chow semistable (resp. stable) then
\begin{enumerate}[(i)]
  \item If $\frac{7}{2}(2g-2)<d<4(2g-2)$, then $[X'\stackrel{|L|}{\inj}\P^r]$ is Chow semistable (resp. stable) for each properly balanced line bundle
  $$
  L\in (\emph{Pic}^4(F')\setminus \{\OO_{F'}(4p)\})\times \{L_2\}
  $$
  \item If $d=4(2g-2)$, then $[X'\stackrel{|L|}{\inj}\P^r]$ is Chow semistable for each properly balanced line bundle
  $$
  L\in \emph{Pic}^4(F')\times \{L_2\}
  $$
  (resp. Chow stable if $F'$ is not cuspidal and $L\in (\emph{Pic}^4(F')\setminus \{\OO_{F'}(4p)\})\times \{L_2\}$).
\end{enumerate}
\end{enumerate}
\end{lemma}
\begin{proof}
Consider $X'$ and a properly balanced line bundle $L=(L_1,L_2)\in \text{Pic}^4(F')\times \text{Pic}^{d-4}(C)$. By Theorem \ref{bal-pos}\eqref{bal-va1} the line bundle $L$ is very ample and non-special, hence we can consider the point $[X\stackrel{|L|}{\inj}\P^r]\in \Hilb_d$.
Let $\rho_1$ and $\rho_2$ be two one-parameter subgroups diagonalized by a system of coordinates $\{x_1,\ldots,x_{r+1}\}$ of type (\ref{eq:coorcoda}), i.e. such that
$$
\langle F'\rangle=\bigcap_{i=5}^{r+1}\{x_i=0\}\quad \text{and}\quad \langle C\rangle=\bigcap_{i=1}^{3}\{x_i=0\},
$$
and having weights
\begin{equation}\label{eq:rho1rho2}
\rho_1(t)\cdot x_i=
\left\{
\begin{array}{ll}
t^{w_{i}}x_i & \text{if }i\leq 3,\\
x_i & \text{if }i\geq 4,\\
\end{array}
\right.
\quad\text{and}\quad
\rho_2(t)\cdot x_i=
\left\{
\begin{array}{ll}
x_i & \text{if }i\leq 4,\\
t^{w_{i}}x_i & \text{if }i\geq 5.\\
\end{array}
\right.
\end{equation}
By Proposition \ref{prop:stab-tail}, it suffices to prove that $[X'\subset \P^r]$ is Chow or Hilbert  (semi-)stable with respect to any such $\rho_1$ and $\rho_2$.

By Remark \ref{rmk:replweight} and the Hilbert-Mumford criterion (see Fact \ref{Hilb-crit} and Fact \ref{Chow-crit}), if $[X\subset\P^r]$ is Hilbert semistable (resp. Chow semistable) we have
$$
W_{X',\rho_2}(m)=W_{X,\rho_2}(m)\leq\frac{w(\rho_2)}{r+1}\,mP(m)\quad \bigg(\text{resp. }e_{X',\rho_2}= e_{X,\rho_2}\leq\frac{2d}{r+1}\,w(\rho_2)\bigg),
$$
while if $[X\subset\P^r]$ is Hilbert stable (resp. Chow stable) then
$$
W_{X',\rho_2}(m)=W_{X,\rho_2}(m)<\frac{w(\rho_2)}{r+1}\,mP(m)\quad \bigg(\text{resp. }e_{X',\rho_2}=e_{X,\rho_2}<\frac{2d}{r+1}\,w(\rho_2)\bigg).
$$
This proves the Hilbert or Chow (semi-)stability of $[X'\subset \P^r]$ with respect to $\rho_2$.

The Hilbert or Chow (semi-)stability of $[X'\subset \P^r]$ with respect to $\rho_1$ will follow from the next lemma, that completes our proof.
\end{proof}

\begin{lemma}\label{lem:code}
Let $[X\subset\P^r]\in \rm{Hilb}_{d}$ such that $X=F\cup C$ where $F$ is an elliptic tail (smooth, nodal, cuspidal or reducible nodal) and the line bundle
$$
L:=\OO_X(1)=(L_1:=L_{|F},L_2:=L_{|C})\in \emph{Pic}^4(F)\times \emph{Pic}^{d-4}(C)
$$
is properly balanced. Let $\rho_1$ be a one-parameter subgroup as in (\ref{eq:rho1rho2}). Then
\begin{enumerate}[(i)]
	\item if $\frac{7}{2}(2g-2)<d<4(2g-2)$ and $L_1\in \text{\emph{Pic}}^4(F)\setminus \{\OO_F(4p)\}$ then
	$$
	e_{F,\rho_1}<\frac{2d}{r+1}\,w(\rho_1)
	$$
	\item if $d=4(2g-2)$ then
	$$
	e_{F,\rho_1}\leq\frac{2d}{r+1}\,w(\rho_1)=\frac{16}{7}\,w(\rho_1).
	$$
	Moreover,  if $L_1\in \text{\emph{Pic}}^4(F)\setminus \{\OO_F(4p)\}$ then
	\begin{eqnarray}\label{eq:efwcusp}
	e_{F,\rho_1}<\frac{16}{7}\,w(\rho_1)\quad & \text{if} & \text{$F$ is not cuspidal,}\\
	W_{F,\rho_1}(m)<\frac{w(\rho_1)}{7}\,mP(m)\: \text{ for }�m \gg 0\quad & \text{if} & \text{$F$ is cuspidal.}\nonumber
	\end{eqnarray}
\end{enumerate}
\end{lemma}
\begin{proof}
Since $e_{F,\rho_1}$ does not depend on $C$, we can prove these two claims by considering $F$ as an elliptic tail of polarized curves whose semistability is known.

Firstly assume that $F$ is smooth, nodal or reducible nodal. Let $C$ be a smooth curve of genus 2 and consider the new curve $X=F\cup C$ with $\{p\}=F\cap C$ embedded in $\P^{11}$ via a properly balanced line bundle $M=(M_1,M_2)$ (indeed $M$ is very ample and non-special by Theorem \ref{bal-pos}\eqref{bal-va1}) with $M_1\in \text{Pic}^4(F)\setminus \{\mathcal{O}_F(4p)\}$, $\deg\,M_{|F}=4$ and $\deg\,M_{|C}=10$. Since the curve $X$ is quasi-wp-stable, $\displaystyle \frac{d}{2g-2}=\frac{7}{2}$
%$\displaystyle \frac{d}{r+1}=\frac{7}{6}$
and $M$ is properly balanced, by Theorem \ref{T:semistable3} we know that $[X\subset\P^{11}]$ with $M=\OO_X(1)$ is Chow semistable; hence
\begin{eqnarray}\label{eq:ellineq1}
e_{F,\rho_1}\leq \frac{2d}{r+1}w(\rho_1)=\frac{7}{3}\,w(\rho_1)
\end{eqnarray}
by the Hilbert-Mumford numerical criterion (Fact \ref{Chow-crit}). In the same way we can consider another properly balanced line bundle $M'$ such that
$\deg\,M'_{|F}=4$ and $\deg\,M'_{|C}=13$. Since the curve $X$ is quasi-stable, $\displaystyle 4<\frac{d}{2g-2}=\frac{17}{4} <4,5$
%\frac{d}{r+1}=\frac{17}{15}<\frac{8}{7}$
and $M'$ is stably balanced, by Corollary \ref{C:stable}\eqref{C:stable1} we know that
$[X\subset\P^{14}]$ with $M'=\OO_X(1)$ is Chow stable; hence
\begin{eqnarray}\label{eq:ellineq2}
e_{F,\rho_1}< \frac{2d}{r+1}w(\rho_1)=\frac{34}{15}\,w(\rho_1).
\end{eqnarray}
Now, consider a point $[X\subset \P^r]\in \Hilb_d$ that satisfies the hypothesis of our lemma. Assume that $\frac{7}{2}(2g-2)<d\leq 4(2g-2)$ and $L_1\in \text{Pic}^4(F)\setminus \{\OO_F(4p)\}$. Since $\displaystyle \frac{17}{15}<\frac{8}{7}\leq \frac{d}{r+1}<\frac{7}{6}$, combining (\ref{eq:ellineq1}) and (\ref{eq:ellineq2}) we deduce that
$$
\text{if }w(\rho_1)\geq 0\text{, then }e_{F,\rho_1}<\frac{34}{15}\,w(\rho_1)\leq \frac{2d}{r+1}\,w(\rho_1),
$$
and
$$
\text{if }w(\rho_1)<0\text{, then }e_{F,\rho_1}\leq \frac{7}{3}\,w(\rho_1)< \frac{2d}{r+1}\,w(\rho_1),
$$
so that (i) and (ii) are proved for smooth, nodal and reducible nodal elliptic tails under the hypothesis that $L_1\neq \OO_{F}(4p)$.

Let $X=F\cup C$ be a curve as above, with $F$ an irreducible elliptic tail (smooth, nodal, or cuspidal). By \cite[Proposition 6]{HMo}, we know that $[X\subset \P^{13}]$ with $\OO_{X}(1)=\omega_X^{\otimes 4}$ is strictly Chow semistable. Hence if $d=4(2g-2)$ and $L_1=\OO_F(4p)=(\omega_X^{\otimes 4})_{|F}$ we get
$$
e_{F,\rho_1}\leq\frac{2d}{r+1}\,w(\rho_1)=\frac{16}{7}\,w(\rho_1),
$$
and the first part of (ii) is proved.

It remains to prove (i) and (ii) for the cuspidal case. Suppose that $\frac{7}{2}(2g-2)<d<4(2g-2)$ and $F$ is cuspidal. In order to prove (i), it suffices to exhibit a non-special line bundle $L_1$ for which the inequality \eqref{eq:efwcusp} is satisfied. Indeed, $\Aut(F,p)$ acts transitively on $\text{Pic}^4(F)\setminus \{\OO_{F}(4p)\}$ and we can apply Lemma  \ref{L:autstable}. Consider the Chow semistable point $[Y\subset \P^r]\in \Hilb_d$ obtained in Remark \ref{rmk:code} and denote by $F$ its elliptic tail. Since the semistability is an open condition, up to smoothing arbitrarily $Y$, we can assume that $F$ is smooth. Now, let $B\subseteq \Pic^{\un d}(Y)$ be a smooth curve such that
$$
B\setminus \{b_0\}\subseteq (\text{Pic}^4(F)\setminus \{\OO_{F}(4p)\})\times \{\OO_{C}(1)\} \quad\text{and}\quad b_0=\{\OO_{F}(4p)\}\times \{\OO_{C}(1)\}.
$$
Consider the trivial family $\cY=Y\times B \ra B$ and denote by $\cL$ the Poincar\'e bundle $\cP$ on $Y\times \text{Pic}^{\un d}(Y)$ restricted to $\cY$. As in the proof of Proposition \ref{P:completeness}, up to shrinking $B$ around $b_0$, we obtain an embedding $\cY \hookrightarrow \P^r_B$, which yields a map $f:B\to \Hilb_d$ such that $f(B\setminus \{b_0\})\subset \Ch^{-1}(\Chow_d^{ss})^o$. Now, apply the \emph{polystable replacement property}. Up to replacing $B$ with a finite cover ramified over $b_0$, we get a polarized family $(\cZ\ra B,\cM)$ such that, denoting $Z:=\cZ_{b_0}$ and $M:=\cM_{|Z}$, the point $[Z\subset\P^r]$ with $M=\OO_Z(1)$ is Chow polystable. Denote by $F'$ the elliptic tail of $Z$. Since $\cZ$ is an isotrivial family of curves over $B$, either $Z\cong Y$ or $F'$ is cuspidal. If $Z\cong Y$, then $\OO_{Z}(1)_{|F}=\OO_{F}(4p)$, which is absurd by Theorem \ref{T:spec-ell}, hence the second case occurs. Since $F'\subset Z$ is not special, $\Stab_{\PGL_{r+1}}([Z\subset\P^r])$ is finite by Theorem \ref{T:auto-grp} and Lemma \ref{L:aut-stab}, hence $[Z\subset\P^r]$ is Chow stable. This proves the inequality
$$
e_{F,\rho_1}<\frac{2d}{r+1}\,w(\rho_1)
$$
if $F$ is cuspidal. The last inequality of (ii) can be proved in the same way by applying the \emph{polystable replacement property} for the Hilbert stability.
\end{proof}

\begin{coro}\label{C:esistcode}
Let $X=F\cup C$ be a connected curve where $F$ is an elliptic tail (smooth, nodal, cuspidal or reducible nodal) and $C$ is smooth. Denote by $p$ the intersection point of $F$ with $C$ and consider a properly balanced line bundle $L\in \text{\emph{Pic}}^d(X)$ with $\frac{7}{2}(2g-2)<d\leq 4(2g-2)$. Then there exists $M\in \text{\emph{Pic}}^{d-4}(C)$ such that
\begin{enumerate}
	\item \label{C:esistcode1} if $\frac{7}{2}(2g-2)<d<4(2g-2)$, $L_{|C}=M$ and $L_{|F}\in \text{\emph{Pic}}^4(F)\setminus\{\OO_F(4p)\}$, then $[X\stackrel{|L|}{\hookrightarrow} \P^r]$ is Chow stable;
	\item \label{C:esistcode2} if $d=4(2g-2)$, $L_{|C}=M$, $L_{|F}\in \text{\emph{Pic}}^4(F)\setminus\{\OO_F(4p)\}$ and $F$ is cuspidal (resp. not cuspidal), then $[X\stackrel{|L|}{\hookrightarrow} \P^r]$ is Hilbert (resp. Chow) stable; moreover if $\OO_{X}(1)_{|F}=\OO_F(4p)$ and $F$ is cuspidal, then $[X\stackrel{|L|}{\hookrightarrow} \P^r]$ is Chow polystable.
\end{enumerate}
\end{coro}
\begin{proof}
For \eqref{C:esistcode1} and the first statement of \eqref{C:esistcode2}, it suffices to consider the curve $Y$ obtained in Remark \ref{rmk:code} by applying the \emph{polystable replacement property} to a quasi-wp-stable curve $X'=C\cup E$, where $C$ and $E\cong \P^1$ meet together in a tacnode (in this case $p\in C$) and apply Lemma \ref{lem:complcode}. For the last statement of \eqref{C:esistcode2}, we notice that $X$ is a closed point in the stack $\MMgwp$ by Remark \ref{R:3stacks}\eqref{R:3stacks1}. Hence, if $\rho$ is a one-parameter subgroup such that
$$
e_{X,\rho}=\frac{16}{7}\,w(\rho)
$$
then, setting $[X_0\subset \P^r]=\lim_{t\ra 0}\rho(t)[X\subset \P^r]$, we have that $X\cong X_0$ and $$\dim\Stab_{\PGL_{r+1}}([X\subset \P^r])\geq \dim\Stab_{\PGL_{r+1}}([X_0\subset \P^r])$$
by Theorem \ref{T:auto-grp} and Lemma \ref{L:aut-stab}. This implies that $[X_0\subset \P^r]\in \Orb([X\subset \P^r])$ and we are done.
\end{proof}

Now, we are ready to extend the completeness result of Proposition \ref{P:completeness} to the case $d=4(2g-2)$.
\begin{prop}\label{P:completeness3}
Let $X$ be a quasi-wp-stable curve and $\un d\in B_X^d$. Assume that $d=4(2g-2)$. Then either $M_X^{\un d}=\emptyset$ or the map
$p:M_X^{\un d}\to \emph{Pic}^{\un d}(X)/\emph{Aut}(X)$ is surjective.
\end{prop}
\begin{proof}
Assume that $M_X^{\un d}\neq \emptyset$, for otherwise there is nothing to prove.
According to Corollary \ref{C:Im-p}, the surjectivity of $p$ is equivalent to the fact
 that $[X\stackrel{|L|}{\inj} \P^r]$ is Chow semistable for every $L\in \text{Pic}^{\un d}(X)$. To this aim, let $E=\{F_1,\ldots,F_k\}$ be the set of the elliptic tails of $X$, set $C=X_{\rm ell}^c$ and denote by $p_i$ the intersection point of $F_i$ with $(F_i)^c$ for each $i=1,\ldots,k$.
By standard arguments of basins of attraction, we can assume that:
\begin{enumerate}
  \item the multidegree $\un d$ is strictly balanced (same proof as that of \un{Reduction 2} in Proposition \ref{P:completeness});
	\item each elliptic tails $F$ with $\OO_X(1)|_F=\OO_F(4p)$ is cuspidal (by Theorem \ref{T:basin-cusps} and Fact \ref{F:weight-basin});
	\item each cusp is contained in an elliptic curve (same proof as that of (2)).
\end{enumerate}
Therefore, we have a curve like in the picture below:
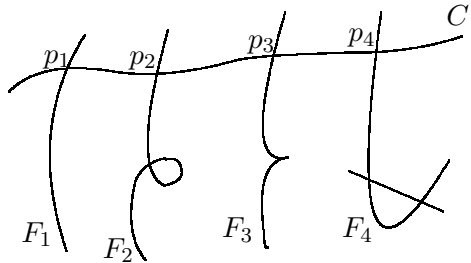
\begin{center}
%TeXCAD Picture [EllipticTails1.pic]. Options:
%\grade{\on}
%\emlines{\off}
%\epic{\off}
%\beziermacro{\on}
%\reduce{\on}
%\snapping{\off}
%\pvinsert{% Your \input, \def, etc. here}
%\quality{8.000}
%\graddiff{0.005}
%\snapasp{1}
%\zoom{4.0000}
\unitlength 0.6mm % = 2.845pt
\linethickness{0.4pt}
\ifx\plotpoint\undefined\newsavebox{\plotpoint}\fi % GNUPLOT compatibility
\begin{picture}(135,55)(0,118)
\qbezier(19.75,156.25)(26.625,163.5)(43,160.75)
\qbezier(43,160.75)(55,159)(70,163.25)
\qbezier(70,163.25)(74.375,164.375)(95.25,165)
\qbezier(95.25,165)(108.375,164.75)(120,168.5)
\qbezier(36.5,168.75)(23.5,148.125)(32.5,121)
\qbezier(81,174)(70.75,139.125)(81.5,141.75)
\qbezier(81.5,141.75)(74.125,140.625)(76.25,125)
\qbezier(102.25,174)(92.25,99.5)(117.25,141)
%\emline(116,129.75)(95,138.75)
\multiput(116,129.75)(-.0786516854,.0337078652){267}{\line(-1,0){.0786516854}}
%\end
\qbezier(55,170)(46.375,138.75)(54.25,135.5)
\qbezier(54.25,141.5)(49.5,140.5)(47.75,136.5)
\qbezier(47.75,136.5)(45.125,124.75)(50,119)
\qbezier(76.25,125)(76.375,121.125)(77,121.75)
\qbezier(54.25,135.5)(58.375,136.375)(58,138.75)
\qbezier(58,138.75)(57.75,142)(54.5,141.25)
\put(44,121.25){\makebox(0,0)[cc]{$F_2$}}
\put(70.5,125.75){\makebox(0,0)[cc]{$F_3$}}
\put(97,125.75){\makebox(0,0)[cc]{$F_4$}}
\put(26,124.75){\makebox(0,0)[cc]{$F_1$}}
\put(119.25,173.25){\makebox(0,0)[cc]{$C$}}
\put(30.5,163.5){\makebox(0,0)[cc]{$p_1$}}
\put(49.5,163.25){\makebox(0,0)[cc]{$p_2$}}
\put(75.75,166.75){\makebox(0,0)[cc]{$p_3$}}
\put(98,168.25){\makebox(0,0)[cc]{$p_4$}}
\end{picture}

\end{center}
Let $F=E_1\cup E_2$ be a reducible nodal elliptic curve where $E_1$ and $E_2$ are two smooth rational curves. Consider a smooth point $p\in E_1\subset F$ and a line bundle $M\in \text{Pic}^4(F)$ such that $\un \deg\,M=(\deg\,M_{|E_1},\deg\,M_{|E_2})=(3,1)$. By Lemma \ref{lem:complcode}, if we replace each elliptic tail $F_i$ with a pointed polarized curve $(F'_i,p'_i,M_i)\cong (F,p, M)$, we obtain a new curve $X'$ (see the picture below)
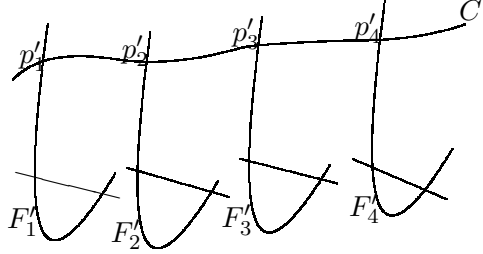
\begin{center}
%TeXCAD Picture [EllipticTails2.pic]. Options:
%\grade{\on}
%\emlines{\off}
%\epic{\off}
%\beziermacro{\on}
%\reduce{\on}
%\snapping{\off}
%\pvinsert{% Your \input, \def, etc. here}
%\quality{8.000}
%\graddiff{0.005}
%\snapasp{1}
%\zoom{4.0000}
\unitlength .6mm % = 2.845pt
\linethickness{0.4pt}
\ifx\plotpoint\undefined\newsavebox{\plotpoint}\fi % GNUPLOT compatibility
\begin{picture}(130.25,55)(0,115)
\qbezier(19.75,156.25)(26.625,163.5)(43,160.75)
\qbezier(43,160.75)(55,159)(70,163.25)
\qbezier(70,163.25)(74.375,164.375)(95.25,165)
\qbezier(95.25,165)(108.375,164.75)(120,168.5)
\qbezier(102.25,174)(92.25,99.5)(117.25,141)
\qbezier(27.5,168.5)(17.5,94)(42.5,135.5)
\qbezier(50.25,166.75)(40.25,92.25)(65.25,133.75)
\qbezier(75,170.25)(65,95.75)(90,137.25)
%\emline(116,129.75)(95,138.75)
\multiput(116,129.75)(-.0786516854,.0337078652){267}{\line(-1,0){.0786516854}}
%\end
\put(90,133){\line(0,1){0}}
\put(20.5,135.75){\line(4,-1){23}}
\put(43.5,130){\line(0,1){0}}
%\emline(45,136.75)(67.75,130.25)
\multiput(45,136.75)(.117875648,-.033678756){193}{\line(1,0){.117875648}}
%\end
%\emline(70,138.75)(91.75,133.25)
\multiput(70,138.75)(.132621951,-.033536585){164}{\line(1,0){.132621951}}
%\end
\put(121.25,171.75){\makebox(0,0)[cc]{$C$}}
\put(97.75,127.75){\makebox(0,0)[cc]{$F_4'$}}
\put(69.5,125.25){\makebox(0,0)[cc]{$F_3'$}}
\put(45,122){\makebox(0,0)[cc]{$F_2'$}}
\put(22,125){\makebox(0,0)[cc]{$F_1'$}}
\put(24.25,161.5){\makebox(0,0)[cc]{$p_1'$}}
\put(47,162.75){\makebox(0,0)[cc]{$p_2'$}}
\put(71.5,166.5){\makebox(0,0)[cc]{$p_3'$}}
\put(98.5,167.75){\makebox(0,0)[cc]{$p_4'$}}
\end{picture}

\end{center}
and a multidegree $\un d'$ for which there exists a properly balanced line bundles $L'\in \text{Pic}^{\un d'}(X')$ such that $[X'\stackrel{|L'|}{\inj} \P^r]$ is Chow semistable. We notice that
$X'$ is quasi-stable and each Chow semistable isotrivial specialization is again a quasi-stable curve, so that, by the proof of Proposition \ref{P:completeness} (case $d>4(2g-2)$), our
statement is true for $X'$ and $\un d'$. In order to complete the proof, it is enough to replace again each $F'_i$ with $(F_i,p_i,\OO_X(1)_{|F_i})$ and to apply Lemma \ref{lem:complcode}.
\end{proof}

\section{Semistable, polystable and stable points (part II)}\label{S:semistab2}

The aim of this section is to describe the points of $\Hilb_d$ that are Hilbert or Chow semistable, polystable and stable for
$$
\frac{7}{2}(2g-2)<d\leq 4(2g-2) \quad \text{�and }�\quad g\geq 3.
$$
The GIT analysis in this range is based on a nice numerical trick that uses the following

\begin{cons}\label{C:replellcode}
%The GIT analysis of the case $\frac{7}{2}(2g-2)\leq d\leq 4(2g-2)$ is based on a nice numerical trick.
Given a quasi-wp-stable curve $[X\subset \P^{d-g}]\in \Hilb_d$ which admits a non-special elliptic tail $F$, we define a new polarized curve $X'$ by replacing the polarized subcurve $F$ with a polarized smooth curve $Y$ of genus $g$ and degree $d-d_F$ so that $Y$ and $X \setminus F$ intersect again in one node, as in Definition \ref{D:replace}. If we denote by $d'$ and $g'$ the degree of the new line bundle $L'$ and the genus of $X'$ respectively, one can consider the Hilbert point $[X'\subset \P^{d'-g'}]\in \Hilb_{d'}$. We easily check that
$$
\nu':=\frac{d'}{2g'-2}=\frac{2d}{2(2g-1)-2}=\frac{d}{2g-2}=:\nu.
$$
Moreover, we claim that
$$
\OO_X(1)\text{ is balanced }\Longleftrightarrow \OO_{X'}(1) \text{ is balanced.}
$$
Let us prove the implication $\Longrightarrow$. As we noticed in Remark \ref{R:bala}, it suffices to check the basic inequality \eqref{E:basineq-multideg} for each connected subcurve such that its complementary subcurve is connected. Let $D\subset X'$ be a connected subcurve. If $D=Y$, then obviously the basic inequality \eqref{E:basineq-multideg} is satisfied. Otherwise, up to replacing $D$ with $D^c$, we can assume that $D$ does not contain $Y$ as a subcurve. This implies that $D$ can be seen as a subcurve of $X$. Since $\nu'=\nu$, $\deg L_{|D}=\deg L'_{|D}$ and $|D\cap \overline{X\setminus D}|=|D\cap \overline{X'\setminus D}|$, the basic inequality \eqref{E:basineq-multideg} is satisfied. The proof of the reverse implication $\Longleftarrow$ is analogous.

We notice that from $X$ to $X'$ the number of non-special elliptic tails decreases by 1. Applying the results about elliptic tails of the previous section and Corollary \ref{C:stab-repl}, one proves that
$$
[X'\subset \P^{d'-g'}]\in \text{ is semistable} \Longrightarrow [X\subset \P^{d-g}]\text{ is semistable},
$$
so that the GIT analysis can be completed by an induction argument on the number of non-special elliptic tails of $X$. (For $d=4(2g-2)$ the induction argument will be on the number of all elliptic tails of $X$). Applying arguments based on specializations of strata (Proposition \ref{P:deg-strata}) and results of completeness (Proposition \ref{P:completeness} and Proposition \ref{P:completeness3}), one can prove the basis of the induction as well. Notice that we have already used this construction in the proof of Corollary \ref{C:polystable3}.
\end{cons}

Let us begin with the case $\frac{7}{2}(2g-2)<d<4(2g-2)$.

\begin{thm}\label{T:semistable4}
Consider a point $[X\subset \P^r]\in \Hilb_d$ with $\frac{7}{2}(2g-2)<d<4(2g-2)$ and $g\geq 3$ and assume that $X$ is connected. The following conditions are equivalent.
\begin{enumerate}[(i)]
\item \label{T:semistable4i} $[X\subset \P^r]$ is Hilbert semistable;
% \in \Hilb_d^{ss}$;
\item \label{T:semistable4ii} $[X\subset \P^r]$ is Chow semistable;

\item \label{T:semistable4iii} $X$ is quasi-wp-stable without tacnodes nor special elliptic tails, non-degenerate and linearly normal in $\P^r$ and $\OO_X(1)$ is properly balanced and non-special;
\item \label{T:semistable4iv} $X$ is quasi-wp-stable without tacnodes nor special elliptic tails and $\OO_X(1)$ is properly balanced;
\item \label{T:semistable4v} $X$ is quasi-wp-stable without tacnodes nor special elliptic tails and $\OO_X(1)$ is balanced.
\end{enumerate}
\end{thm}
\begin{proof}
The implications (\ref{T:semistable4i}) $\Rightarrow$ (\ref{T:semistable4ii}), (\ref{T:semistable4iii}) $\Rightarrow$ (\ref{T:semistable4iv}) and (\ref{T:semistable4iv}) $\Rightarrow$ (\ref{T:semistable4v}) are clear.

(\ref{T:semistable4ii}) $\Rightarrow$ (\ref{T:semistable4iii}) follows from Corollary \ref{C:quasi-wp-stable}\eqref{C:quasi-wp-stable1} and Corollary \ref{C:quasi-wp-inter}.

(\ref{T:semistable4v}) $\Rightarrow$ (\ref{T:semistable4i}). The proof is based on Construction \ref{C:replellcode}. Let $X$ and $L:=\OO_X(1)$ be as in (\ref{T:semistable4v}) and let $n$ be the number of elliptic tails of $X$. We will prove our statement by induction over $n$.

Assume first that each cusp of $X$ is contained in an elliptic tail of $X$. If $n=0$, then $X$ is quasi-stable without elliptic tail and the same argument used to prove Theorem \ref{T:semistable} (case $d>4(2g-2)$) goes through. Suppose that $n>0$. Consider an elliptic tail $F$ (which  is non-special by assumption) and denote by $C_1=F^c$ and $\{p\}=F\cap C_1$. Let $C_2$ be a smooth curve of genus $g$, $q$ a point of $C_2$ and $L'_{C_2}\in \text{Pic}^{d+4}(C_2)$. Denote by $(X',L')$ the couple consisting of a curve $X'$ of genus $g'$ and a line bundle $L'$ on $X'$ obtained from $(X,L)$ by replacing $F$ with $(C_2,q,L'_{C_2})$. The line bundle $L'$ is ample of degree $d'=2d$, moreover we have
\begin{equation}\label{eq:nunu}
\nu':=\frac{d'}{2g'-2}=\frac{2d}{2(2g-1)-2}=\frac{d}{2g-2}=:\nu.
\end{equation}
By the same argument used in the proof of Corollary \ref{C:polystable3} and Construction \ref{C:replellcode}, $L'$ is properly balanced, therefore $L'$ is non-special and very ample by Theorem \ref{bal-pos}; hence we can consider the point $[X'\subset\P^{r'}]\in \text{Hilb}_{d',g'}$ with $\OO_{X'}(1)=L'$. Now, $X'$ contains $n-1$ elliptic tails and $L'$ is balanced, hence by induction $[X'\subset\P^{r'}]$ is Hilbert semistable. By Corollary \ref{C:esistcode}\eqref{C:esistcode1},
 there exists a Hilbert semistable point
$[Y\subset \P^r]\in \Hilb_d$ such that $Y$ admits the elliptic tail $F$ with $\OO_Y(1)_{|F}=L_{|F}$. Since \eqref{eq:nunu} holds and $(X,L)$ can be obtained again from $(X',L')$ by replacing $C_2$ with
$(F,p,L_{|F})$, $[X\subset \P^r]$ is Hilbert semistable by Corollary \ref{C:stab-repl}.

Consider, now, the general case, where $X$ can have cusps that not contained in an elliptic tail.
As before, we prove our statement by induction over $n$. If $n=0$, then $X$ is quasi-p-stable and, by Corollary \ref{C:semist-num}, it is enough to prove that for each balanced multidegree $\underline{d}$ there exists a line bundle $L$ of multidegree $\underline{d}$ such that $[X \subset \P^r]$ is Hilbert semistable with $\OO_X(1)=L$. By Proposition \ref{P:deg-strata}, the curve $X$ specializes isotrivially to a curve $X'$ such that each cusp in contained in an elliptic tail. Let $F_1,\ldots,F_m$ be the elliptic tails of $X'$ and denote by $p_i$ the intersection point of $F_i$ with $F_i^c$. Replacing each cuspidal elliptic tail $F_i$ with a pointed reducible nodal one $(F'_i,q_i)$, we obtain a quasi-stable curve $X''$, which is Hilbert semistable for each balanced polarization $L''$ by the argument above. If we replace again each reducible nodal elliptic tail $F'_i$ with the pointed polarized curve $(F_i,p_i,L_{|F_i})$, by Lemma \ref{lem:complcode} $[X'\subset \P^r]$ is Hilbert semistable. The semistability is an open condition, so that the theorem is true for a generic element of $\text{Pic}^{\underline{d}}(X)$. If $n>0$, we apply the same argument based on replacement of elliptic tails used above and the proof is complete.
\end{proof}

\begin{coro}\label{C:polystable4}
Consider a point $[X\subset \P^r]\in \Hilb_d$ with $\frac{7}{2}(2g-2)<d<4(2g-2)$ and $g\geq 3$ and assume that $X$ is connected.
The following conditions are equivalent:
\begin{enumerate}[(i)]
\item \label{C:polystable4i} $[X\subset \P^r] $ is Hilbert polystable;
\item \label{C:polystable4ii} $[X\subset \P^r]$ is Chow polystable;
\item \label{C:polystable4iii} $X$ is quasi-wp-stable without tacnodes nor special elliptic tails, non-degenerate and linearly normal in $\P^r$ and $\OO_X(1)$ is strictly balanced and non-special;
\item \label{C:polystable4iv} $X$ is quasi-wp-stable without tacnodes nor special elliptic tails and $\OO_X(1)$ is strictly balanced.
\end{enumerate}
\end{coro}
\begin{proof}
The implications (\ref{C:polystable4i}) $\Rightarrow$ (\ref{C:polystable4ii}) $\Rightarrow$ (\ref{C:polystable4iii}) are proved in the same way as in Corollary \ref{C:polystable}.

(\ref{C:polystable4iii}) $\Rightarrow$ (\ref{C:polystable4iv}) is obvious.

(\ref{C:polystable4iv}) $\Rightarrow$ (\ref{C:polystable4i}): the proof of this implication is based on Construction \ref{C:replellcode} and is very similar that used to prove Corollary \ref{C:polystable3}\eqref{C:polystable3c}.
Denote by $L=\OO_X(1)$ and $n$ the number of elliptic tails of $X$. We will prove our corollary by induction on $n$. If $n=0$, then $X$ is quasi-p-stable and the same argument used to prove Proposition \ref{P:completeness} shows that $[X\subset \P^r]$ is Hilbert polystable. Suppose that $n>0$. Consider an elliptic tail $F$ and denote by $C_1=F^c$ and $\{p\}=F\cap C_1$. As in the proof of Corollary \ref{C:polystable3}, let $C_2$ be a smooth curve of genus $g$, $q$ a point of $C_2$ and $L'_{C_2}\in \text{Pic}^{d+4}(C_2)$. Denote by $(X',L')$ the couple consisting of a curve $X'$ of genus $g'$ and an ample line bundle $L'$ of degree $d'=2d$ on $X'$ obtained from $(X,L)$ by replacing $F$ with $(C_2,q,L'_{C_2})$. By construction we have that
\begin{equation}\label{eq:same-nu}
\nu':=\frac{d'}{2g'-2}=\frac{2d}{2(2g-1)-2}=\frac{d}{2g-2}=:\nu.
\end{equation}
By the same argument used in the proof of Corollary
\ref{C:polystable3} and Construction \ref{C:replellcode}, the line bundle is strictly balanced. Therefore, the line bundle $L'$ is very ample and non-special by Theorem \ref{bal-pos}, and we can consider the point  �
$[X'\subset\P^{r'}]\in \text{Hilb}_{d',g'}$, with $\OO_{X'}(1)=L'$. Now, $[X'\subset\P^{r'}]$ admits $n-1$ elliptic tails, hence it is Hilbert polystable by induction. By Corollary \ref{C:esistcode}, there exists a Hilbert stable (hence polystable) point $[Y\subset \P^r]\in \Hilb_d$ such that $Y$ admits the elliptic tail $F$ with $\OO_Y(1)|_F=L|_F$. Since  \ref{eq:same-nu} holds and $(X,L)$ can be obtained from $(X',L')$ by replacing $C_2$ with $(F,p,L|_F)$, we get that  $[X\subset \P^r]$ is Hilbert polystable by Corollary \ref{C:stab-repl}.
\end{proof}

%\begin{thm}\label{T:stable4}
%Consider a point $[X\subset \P^r]\in \Hilb_d$ with $d=4(2g-2)$ and assume that $X$ is connected. The following conditions are equivalent:
%\begin{enumerate}[(i)]
%\item $[X\subset \P^r]$ is Chow stable;
%\item $X$ is quasi-stable without special elliptic tails, non-degenerate and linearly normal in $\P^r$ and $\OO_X(1)$ is stably balanced and non-special.
%\end{enumerate}
%\end{thm}
%\begin{proof}
%\end{proof}
\begin{coro}\label{C:stable4}
Consider a point $[X\subset \P^r]\in \Hilb_d$ with $\frac{7}{2}(2g-2)<d<4(2g-2)$ and $g\geq 3$ and assume that $X$ is connected. The following conditions are equivalent:
\begin{enumerate}[(i)]
\item \label{C:stable4i} $[X\subset \P^r]$ is Hilbert stable;
% \in \Hilb_d^{ss}$;
\item \label{C:stable4ii} $[X\subset \P^r]$ is Chow stable;

\item \label{C:stable4iii} $X$ is quasi-wp-stable without tacnodes and special elliptic tails, non-degenerate and linearly normal in $\P^r$ and $\OO_X(1)$ is stably balanced and non-special;
\item \label{C:stable4iv} $X$ is quasi-wp-stable without tacnodes and special elliptic tails and $\OO_X(1)$ is stably balanced.
\end{enumerate}
\end{coro}
\begin{proof}
The implications (\ref{C:stable4ii}) $\Rightarrow$ (\ref{C:stable4i}) and (\ref{C:stable4iii}) $\Rightarrow$ (\ref{C:stable4iv}) are clear.

(\ref{C:stable4i}) $\Rightarrow$ (\ref{C:stable4iii}) follows from Theorem \ref{T:semistable4} and Theorem \ref{T:stab-bal}.

(\ref{C:stable4iv}) $\Rightarrow$ (\ref{C:stable4ii}). By Corollary \ref{C:polystable4}, $[X\subset \P^r]$ is Chow polystable; hence it suffices to prove that
 ${\rm Stab}_{\PGL_{r+1}}([X\subset \P^r])$ is a finite group.
 Since the line bundle $\OO_X(1)$ is stably balanced,
Lemma \ref{compa-bal} gives that $\w{X}:=\ov{X\setminus \exc}$ is connected; hence, combining Lemma \ref{L:aut-stab} and Theorem \ref{T:auto-grp},
we deduce that $\Stab_{\PGL_{r+1}}([X\subset \P^r])$ is a finite group
%\footnote{There is another proof that do not use Corollary \ref{C:polystable4}. As in theorem \ref{T:semistable4}, the
% proof is by induction over $n$, the number of elliptic tails. Suppose that $n=0$. By theorem \ref{T:semistable4}, we know that $[X\subset \P^r]$ is Chow semistable. The same argument of %proposition \ref{P:completeness} shows that $[X\subset \P^r]$ is Chow polystable. By Theorem \ref{T:auto-grp} $\Aut(X,\OO_X(1))^0=\Gm$, hence $[X\subset \P^r]$ is Chow stable by Lemma \ref
%{L:aut-stab}. Suppose now $n>0$. The same argument of the proof of theorem \ref{T:semistable4} (based on replacing elliptic tails with smooth curves of genus $g$) goes through and we are
% done.}
\end{proof}

To conclude this section, we study the extremal case $d=4(2g-2)$, where the Chow semistable locus differs from the Hilbert semistable locus.

\begin{thm}\label{T:semistable5}
Consider a point $[X\subset \P^r]\in \Hilb_d$ with $d=4(2g-2)$ and $g\geq 3$ and assume that $X$ is connected.
\begin{enumerate}
\item \label{T:semistable5h} The following conditions are equivalent:
\begin{enumerate}[(i)]
\item \label{T:semistable5hi} $[X\subset \P^r]$ is Hilbert semistable;
% \in \Hilb_d^{ss}$;
\item \label{T:semistable5hiii} $X$ is quasi-wp-stable without tacnodes nor special elliptic tails, non-degenerate and linearly normal in $\P^r$ and $\OO_X(1)$ is properly balanced and non-special;
\item \label{T:semistable5hiiibis} $X$ is quasi-wp-stable without tacnodes nor special elliptic tails and $\OO_X(1)$ is properly balanced;
\item \label{T:semistable5hiv} $X$ is quasi-wp-stable without tacnodes nor special elliptic tails and $\OO_X(1)$ is balanced.
\end{enumerate}
\item \label{T:semistable5c} The following conditions are equivalent:
\begin{enumerate}[(i)]
\item \label{T:semistable5ci} $[X\subset \P^r]$ is Chow semistable;
%\in \Ch^{-1}(\Chow_d^{ss})$, i.e., $[X\subset \P^r]\in H_d$;
\item \label{T:semistable5cii} $X$ is quasi-wp-stable without tacnodes, non-degenerate and linearly normal in $\P^r$ and $\OO_X(1)$ is properly balanced and non-special.
\item \label{T:semistable5ciibis} $X$ is quasi-wp-stable without tacnodes and $\OO_X(1)$ is properly balanced;
\item  \label{T:semistable5ciii} $X$ is quasi-wp-stable without tacnodes and $\OO_X(1)$ is balanced.
\end{enumerate}
\end{enumerate}
%\begin{enumerate}[(i)]
% \item $[X\subset \P^r]$ is Chow semistable;
% $\in \Hilb_d^{ss}$;
% \item $X$ is quasi-wp-stable without tacnodes, non-degenerate and linearly normal in $\P^r$ and $\OO_X(1)$ is properly balanced and non-special.
%\end{enumerate}
\end{thm}
\begin{proof}
The proof of \eqref{T:semistable5h} is the same as the proof of Theorem \ref{T:semistable4}, using the fact that Corollary \ref{C:quasi-wp-inter} does hold true also in the present case.
Let us prove \eqref{T:semistable5c}.

\eqref{T:semistable5ci} $\Rightarrow$ \eqref{T:semistable5cii} follows from Theorem \ref{teo-pstab}, Corollary \ref{C:quasi-wp-stable} and Theorem \ref{T:tac-line}.

\eqref{T:semistable5cii} $\Rightarrow$ \eqref{T:semistable5ciibis} $\Rightarrow$ \eqref{T:semistable5ciii} are clear.

\eqref{T:semistable5ciii} $\Rightarrow$ \eqref{T:semistable5ci} is proved with the same argument used to prove the implication \eqref{T:semistable4v}$\Rightarrow$\eqref{T:semistable4i}:
the only difference is that we do not assume that the elliptic tails are non-special and we use Corollary \ref{C:esistcode}\eqref{C:esistcode2} instead of
Corollary \ref{C:esistcode}\eqref{C:esistcode1}.
\end{proof}

\begin{coro}\label{C:polystable5}
Consider a point $[X\subset \P^r]\in \Hilb_d$ with $d=4(2g-2)$ and $g\geq 3$ and assume that $X$ is connected.
The following conditions are equivalent:
\begin{enumerate}
\item \label{C:polystable5h} The following conditions are equivalent:
\begin{enumerate}[(i)]
\item \label{C:polystable5hi} $[X\subset \P^r] $ is Hilbert polystable;
\item \label{C:polystable5hii} $X$ is quasi-wp-stable without tacnodes and special elliptic tails, non-degenerate and linearly normal in $\P^r$ and $\OO_X(1)$ is strictly balanced and non-special;
\item \label{C:polystable5hiii} $X$ is quasi-wp-stable without tacnodes and special elliptic tails and $\OO_X(1)$ is strictly balanced.
\end{enumerate}
\item \label{C:polystable5c} The following conditions are equivalent:
\begin{enumerate}[(i)]
\item \label{C:polystable5ci} $[X\subset \P^r]$ is Chow polystable;
\item \label{C:polystable5cii} $X$ is quasi-wp-stable without tacnodes, each special elliptic tail of $X$ is cuspidal, each cuspidal elliptic tail of $X$ is special, each cusp of $X$ is contained in an elliptic tail, $X$ is non-degenerate and linearly normal in $\P^r$, $\OO_X(1)$ is strictly balanced and non-special;
\item \label{C:polystable5ciii} $X$ is quasi-wp-stable without tacnodes, each special elliptic tail of $X$ is cuspidal, each cuspidal elliptic tail of $X$ is special, each cusp of $X$ is contained in an elliptic tail and $\OO_X(1)$ is strictly balanced.
\end{enumerate}
\end{enumerate}
\end{coro}
\begin{proof}
The same argument of Corollary \ref{C:polystable4} proves (1). Let us prove (2).

(\ref{C:polystable5ci}) $\Rightarrow$ (\ref{C:polystable5cii}) follows from Theorem \ref{T:semistable5}\eqref{T:semistable5c}, Theorem \ref{teo-pstab} and Theorem \ref{T:basin-cusps}.

(\ref{C:polystable5cii}) $\Rightarrow$ (\ref{C:polystable5ciii}) is clear.

(\ref{C:polystable5ciii}) $\Rightarrow$ (\ref{C:polystable5ci}). We use Construction \ref{C:replellcode} again, as in the proof of Corollary \ref{C:polystable4}. Let $n$ be the number of elliptic tails of $X$. If $n=0$, then $X$ is quasi-stable and the proof of Corollary \ref{C:polystable} goes through. Suppose that $n>0$ and consider the point $[X'\subset\P^{r'}]\in \Hilb_{d',g'}$ obtained from $(X,L)$ by replacing an elliptic tail $F$ with a smooth curve $C_2$ of genus $g$ and degree $d+4$.
%We have
%\begin{equation}\label{eq:nunu4}
%\nu':=\frac{d'}{2g'-2}=\frac{d}{2g-2}=:\nu,
%\end{equation}
The line bundle $L':=\OO_{X'}(1)$ is strictly balanced on $X'$ and the point $[X'\subset \P^{r'}]$ is Chow polystable since $X$ admits $n-1$ elliptic tails. By Corollary \ref{C:esistcode}\eqref{C:esistcode2}, there exists a Chow polystable point $[Y\subset \P^r]$ such that $Y$ admits $F$ as an elliptic tail and $\OO_Y(1)_{|F}=\OO_X(1)_{|F}$. Finally, applying Corollary \ref{C:stab-repl} to the points $[X'\subset \P^{r'}]$, $[Y\subset \P^r]$ and $[X \subset\P^{r}]$, we deduce that $[X \subset\P^{r}]$ is Chow polystable.
\end{proof}

\begin{coro}\label{C:stable5}
Consider a point $[X\subset \P^r]\in \Hilb_d$ with $d=4(2g-2)$ and $g\geq 3$ and assume that $X$ is connected.
\begin{enumerate}
\item \label{C:stable5h} The following conditions are equivalent:
\begin{enumerate}[(i)]
\item \label{C:stable5hi} $[X\subset \P^r]$ is Hilbert stable;
% \in \Hilb_d^{ss}$;
\item \label{C:stable5hii} $X$ is quasi-wp-stable without tacnodes and special elliptic tails, non-degenerate and linearly normal in $\P^r$ and $\OO_X(1)$ is stably balanced and non-special;
\item \label{C:stable5hiii} $X$ is quasi-wp-stable without tacnodes and special elliptic tails and $\OO_X(1)$ is stably balanced.
\end{enumerate}
\item \label{C:stable5c} The following conditions are equivalent:
\begin{enumerate}[(i)]
\item \label{C:stable5ci} $[X\subset \P^r]$ is Chow stable;
%\in \Ch^{-1}(\Chow_d^{ss})$, i.e., $[X\subset \P^r]\in H_d$;
\item \label{C:stable5cii} $X$ is quasi-stable without special elliptic tails, non-degenerate and linearly normal in $\P^r$ and $\OO_X(1)$ is stably balanced and non-special.
\item  \label{C:stable5ciii} $X$ is quasi-stable without special elliptic tails and $\OO_X(1)$ is stably balanced.
\end{enumerate}
\end{enumerate}
\end{coro}
\begin{proof}
The same argument of Corollary \ref{C:stable4} proves \eqref{C:stable5h}.

Let us now prove \eqref{C:stable5c}. Note that  $[X\subset \P^r]$ is Chow stable if and only if it is Chow polystable and its stabilizer $\Stab_{\PGL_{r+1}}([X\subset \P^r])$ is a finite group.
Lemma \ref{L:aut-stab} and Theorem \ref{T:auto-grp} give that a Chow polystable point $[X\subset \P^r]$ as in
Corollary \ref{C:polystable5}\eqref{C:polystable5c} has finite stabilizer subgroup if and only if
\begin{itemize}
\item $X$ does not have special cuspidal elliptical tails;
\item  $\w{X}:=\ov{X\setminus \exc}$ is connected.
\end{itemize}
The first condition is equivalent to
the fact that $X$ does not have cusps (hence it is quasi-stable) nor special elliptic tails. The second condition
is equivalent to the fact that $\OO_X(1)$ is stably balanced by Lemma \ref{compa-bal}.
Part \eqref{C:stable5c} follows now from this fact together with
 Corollary \ref{C:polystable5}\eqref{C:polystable5c}.

\end{proof}

%\begin{rmk}
%The ideas behind the proofs of Theorem \ref{T:semistable5} and Corollary \ref{C:stable5} allow us to characterize $\Ch^{-1}(\Chow_d^{ss})$ and $\Ch^{-1}(\Chow_d^{s})$ for $d=4(2g-2)$ and $g=2$.
%\end{rmk}

\section{Geometric properties of the GIT quotient}
\label{S:map-pstable}

For any $d>2(2g-2)$, consider the open and closed subscheme $\Ch^{-1}(\Chow_d^{ss})^o$ of the Chow-semistable locus $\Ch^{-1}(\Chow_d^{ss})\subset \Hilb_d$ consisting of connected curves, see \eqref{Hd}.
From now on, in order to shorten the notation, we set
\begin{equation}\label{E:Hd2}
H_d:=\Ch^{-1}(\Chow_d^{ss})^o\subset \Hilb_d
\end{equation}
and we call $H_d$ the \emph{main component} of the Chow-semistable locus. Similarly, the locus
\begin{equation}\label{E:Hdt}
\wt{H}_d:=\Hilb_d^{ss,o}:=\{[X\subset \P^r]\in \Hilb_d^{ss}\: :\: X \text{ is connected}\}
\end{equation}
is an open and closed subscheme of $\Hilb_d^{ss}$, that we call the \emph{main component} of the Hilbert semi-stable locus.
Note that $\wt{H}_d$ is an open subset of $H_d$ by Fact \ref{HilbtoChow}
 and that $\wt{H}_d=H_d$ if and only if $d \not \in \{\frac 7 2(2g-2), 4(2g-2)\}$ by Theorems \ref{T:semistable}, \ref{T:semistable3},
\ref{T:semistable4} and \ref{T:semistable5}. The name ``main component" is justified by the fact that
 $H_d$ (resp. $\wt{H}_d$) is an irreducible component of
$\Ch^{-1}(\Chow_d^{ss})$ (resp. $\Hilb_d^{ss}$), as we will prove in Corollary \ref{C:irr-quot},
together with the fact that for some values of $d$ and $g$ there might exist other irreducible components of $\Chow_d^{ss}$ (resp. $\Hilb_d^{ss}$) made of non-connected curves
 (see Section \ref{sec:extra}).

Since $H_d$ is clearly an $\SL_{r+1}$-invariant closed and open subscheme of $\Ch^{-1}(\Chow_d^{ss})$, GIT
tells us that there exists a projective scheme
\begin{equation}\label{E:Chowquot}
\ov{Q}_{d,g}^c:=H_d/\!\!/\SL_{r+1}
\end{equation}
which is a good categorical quotient of $H_d$ by $\SL_{r+1}$ (see e.g. \cite[Sec. 6.1]{Dol}). Similarly, there exists a projective scheme
\begin{equation}\label{E:Hilbquot}
\ov{Q}_{d,g}^h:=\wt{H}_d/\!\!/\SL_{r+1}
\end{equation}
which is a good categorical quotient of $\wt{H}_d$ by $\SL_{r+1}$. Moreover,  since $\wt{H}_d\subseteq H_d$, there exists a projective morphism
\begin{equation}\label{E:HilbChowquot}
\Xi: \ov{Q}_{d,g}^h=\wt{H}_d/\!\!/\SL_{r+1}\to H_d/\!\!/\SL_{r+1}=\ov{Q}_{d,g}^c.
\end{equation}
If $d \not \in \{\frac{7}{2}(2g-2), 4(2g-2)\}$ then $\wt{H}_d=H_d$ (as observed before), which implies that $\Xi$ is an isomorphism. We will therefore set
\begin{equation}\label{E:eq-quot}
\ov{Q}_{d,g}:=\ov{Q}_{d,g}^h=\ov{Q}_{d,g}^c \: \text{ if } d \not \in \left\{\frac{7}{2}(2g-2), 4(2g-2)\right\}.
\end{equation}
Indeed, we will prove that $\Xi$ is an isomorphism if $d=\frac{7}{2}(2g-2)$ (see Proposition \ref{P:fib3.5}\eqref{P:fib3.5A}), whereas it is not an isomorphism if $d=4(2g-2)$ (see Proposition \ref{P:fib4}\eqref{P:fib4A}).

\begin{rmk}\label{R:closedpts}
By the well-known properties of GIT quotients (see \cite[Cor. 6.1]{Dol}), it follows that the closed points
of $\ov{Q}_{d,g}^h=\wt{H}_d/\!\!/\SL_{r+1}$ (resp. $\ov{Q}_{d,g}^c=H_d/\!\!/\SL_{r+1}$) correspond bijectively to orbits of Hilbert polystable points $[X\subset \P^r]$ in $\wt{H}_d$ (resp. Chow polystable points in $H_d$).
Moreover, note that the orbit of a point $[X\subset \P^r]\in \Hilb_d$ only determines the curve $X$ and the line bundle $\OO_X(1)$ up to automorphisms of $X$ (compare with the discussion at the beginning of \S\ref{SS:complete}).

\end{rmk}

\vspace{0.1cm}

We now focus on the geometric properties of $\ov{Q}_{d,g}^h$ and $\ov{Q}_{d,g}^c$. We begin with the following result, which says that the singularities of $\ov{Q}_{d,g}^h$ and $\ov{Q}_{d,g}^c$ are not too bad.

\begin{prop}\label{P:sing-GITquot}
Assume that $d>2(2g-2)$  and, moreover, that $g\geq 3$ if $d\leq 4(2g-2)$.
Then:
\begin{enumerate}[(i)]
 \item \label{T:comp-Pic4} $H_d$ (resp. $\wt{H}_d$) is non-singular of pure dimension $r(r+2)+4g-3$.
 \item \label{T:comp-Pic5} $\ov{Q}_{d,g}^c$ (resp. $\ov{Q}_{d,g}^h$) is reduced and normal of pure dimension $4g-3$. Moreover, if ${\rm char}(k)=0$, then $\ov{Q}_{d,g}^c$ (resp. $\ov{Q}_{d,g}^h$) has rational singularities, hence it is Cohen-Macauly.
\end{enumerate}
\end{prop}
\begin{proof}
Part \eqref{T:comp-Pic4}: it is enough to prove the statement for $H_d$, since $\wt H_d\subseteq H_d$ is an open subset. Consider a point $[X\subset \P^r]\in H_d$ and let $N_{X/\P^r}={\mathcal Hom}(\cI_X/\cI_X^2, \cO_X)$ be its normal sheaf, where $\cI_X$ is the ideal sheaf of $X$ inside $\P^r$.
 By the Potential pseudo-stability theorem \ref{teo-pstab}, $X$ is a (reduced) curve with locally complete intersection singularities, so that $X\subset \P^r$ is regular embedding.   Therefore, the tangent space of $H_d$ at $[X\subset \P^r]$ is $H^0(X, N_{X/\P^r})$ by \cite[Prop. 3.2.1]{Ser} and an obstruction for the local Hilbert functor of $X$ inside $\P^r$ is $H^1(X, N_{X/Y})$ by \cite[Prop. 3.2.6]{Ser}.
Dualizing the exact sequence (which is exact on the left since $X\subset \P^r$ is a regular embedding)
$$0\to \cI_X/\cI_X^2\to (\Omega^1_{\P^r})_{|X}\to \Omega^1_X\to 0,$$
we get the sequence
$$0\to T_X={\mathcal Hom}(\Omega_X^1,\cO_X)\to (T_{\P^r})_{|X}\to N_{X/\P^r}\to T_X^1:=
{\mathcal Ext}^1(\Omega_X^1, \cO_X)\to 0.$$
Since $T_X^1$ is a skyscraper sheaf (supported on the singular locus of the reduced curve $X$), we get a surjection
$$H^1(X, (T_{\P^r})_{|X})\twoheadrightarrow H^1(X, N_{X/\P^r}).$$
This, together, with the Euler sequence for $\P^r$ gives that
$$h^1(X, N_{X/\P^r})\leq h^1(X, (T_{\P^r})_{|X})\leq h^1(X, \cO_X(1)^{\oplus r+1})=0,$$
where we have used in the last equality that $\cO_X(1)$ is non-special, as it follows from the Potential pseudo-stability Theorem \ref{teo-pstab}. From this, we deduce that $H_d$ is smooth at $[X\subset \P^r]$ of dimension equal to
$$H^0(X,n_{X/\P^r})=\chi(X, N_{X/\P^r})= r(r+2)+4g-3,$$
as it follows by applying Riemann-Roch to the locally free sheaf $N_{X/\P^r}$.

Part \eqref{T:comp-Pic5}: $\ov{Q}_{d,g}^c$ is reduced and normal because $H_d$ is such (see e.g. \cite[Prop. 3.1]{Dol}). The dimension of $\ov{Q}_{d,g}^c$ is $4g-3$ because $H_d$ has dimension
$r(r+2)+4g-3$, the group $\SL_{r+1}$ has dimension $r(r+2)$ and the action of $\SL_{r+1}$ has generically finite stabilizers. If ${\rm char}(k)=0$, then $\ov{Q}_{d,g}^c$ has rational singularities
by  \cite{Bou}, using that $H_d$ is smooth. This implies that $\ov{Q}_{d,g}^c$ is Cohen-Macauly since, in characteristic zero, a variety having rational singularities is Cohen-Macauly (see \cite[Lemma 5.12]{KM}).
Alternatively, the fact that $\ov{Q}_{d,g}^c$ is Cohen-Macauly follows from \cite{HR}, using the fact that $H_d$ is smooth. The same argument works for $\ov{Q}_{d,g}^h$.
\end{proof}

We mention that, if ${\rm char}(k)=0$, $d>4(2g-2)$ and $g\geq 4$, then $\ov{Q}_{d,g}$ is known to have canonical singularities (see \cite{BFV} in the case where $\gcd(d+1-g,2g-2)=1$ and \cite{CMKV2} in the general case).
This result has been used in loc. cit. to compute the Kodaira dimension and the Iitaka fibration of $\ov{Q}_{d,g}$.

\vspace{0.1cm}

The GIT quotient $\ov{Q}_{d,g}^c$ admits a modular morphism to the moduli space $\Mgp$ of p-stable curves.

\begin{thm}\label{T:comp-Pic}
Assume that $d>2(2g-2)$  and, moreover, that $g\geq 3$ if $d\leq 4(2g-2)$.
 Then the following hold: 
\begin{enumerate}[(i)]
\item \label{T:comp-Pic1} There exists a surjective natural map $\Phi^{\rm ps}: \ov{Q}_{d,g}^c\to \Mgp$.
\item \label{T:comp-Pic2} If $d>4(2g-2)$ then the above map $\Phi^{\rm ps}$ factors as
$$\Phi^{\rm ps}: \ov{Q}_{d,g}^c \stackrel{\Phi^{\rm s}}{\longrightarrow} \Mg \stackrel{T}{\longrightarrow} \Mgp,$$
%where $\Phi^{\rm s}$ is the map constructed by Caporaso in \cite[Sec. 2.1]{Cap}
where $T$ is the map of Remark \ref{R:contr-HH}.
%constructed in \cite[Thm. 1.1]{HH1}.
\item \label{T:comp-Pic3} We have that
$$(\Phi^{\rm ps})^{-1}(M_g^o)\cong J_{d,g}^o,$$
where $M_g^o$ is the open subset of $M_g$ parametrizing curves without non-trivial
automorphisms and $J_{d,g}^o$ is the degree $d$ universal Jacobian   over $M_g^o$.
In particular, $(\Phi^{\rm ps})^{-1}(C)\cong \Pic^d(C)$
for every geometric point $C\in M_g^o\subset  \Mgp$.

If $d>4(2g-2)$ then the same conclusions hold for the morphism $\Phi^{\rm s}$.
\end{enumerate}
\end{thm}
%Indeed, the proof below will show that the morphism $\Phi^s$ exists also for $g=2$ (and $d>4(2g-2)$).
\begin{proof}
The proof is an adaptation of the ideas from \cite[Sec. 2]{Cap}.

Part \eqref{T:comp-Pic1}:
consider the restriction to $H_d$ of the universal family over $\Hilb_d$ and denote it by
$$\xymatrix{
\cC_d \ar@{^{(}->}[r] \ar[d]_{u_d} & H_d\times \P^r\\
H_d &
}$$
The morphism $u_d$ is flat, proper and its geometric fibers are quasi-wp-stable curves by Corollary \ref{C:quasi-p-stable}\eqref{C:quasi-wp-stable2}.
Consider the p-stable reduction of $u_d$ (see Definition \ref{D:p-stab}):
$$\xymatrix{
\cC_d\ar[rr]\ar[rd]_{u_d} && \ps(\cC_d) \ar[ld]^{\ps(u_d)} \\
& H_d &
}$$
The morphism $\ps(u_d)$ is flat, proper and its geometric fibers are p-stable curves of genus $g$. Therefore, by the modular properties of $\Mgp$,
the family $\ps(u_d)$ induces a modular map $\phi^{\rm ps}:H_d\to \Mgp$. Since the group $\SL_{r+1}$ acts on the family $\cC_d$ by only changing the
embedding of the fibers of $u_d$  into $\P^r$, the map $\phi^{\rm ps}$ is $\SL_{r+1}$-invariant and therefore it factors via a map $\Phi^{\rm ps}:\ov{Q}_{d,g}^c\to \Mgp$.

Let us show that $\Phi^{\rm ps}$ is surjective.
%Since  the map $\Phi^{\rm ps}$ is projective ($\ov{Q}_{d,g}^c$ and $\Mgp$ being projective), it is sufficient
%to prove that $\Phi^{\rm ps}$ is dominant.
Let $C$ be any connected smooth curve over $k$ of genus $g\geq 2$ and $L$ be any line bundle on $C$ of degree $d>2(2g-2)$. Note that $d=\deg L\geq 2g+1$ since $g\geq2$. Hence $L$ is very ample and non-special and therefore it  embeds $C$ in $\P^r=\P^{d-g}$. By Fact \ref{F:stab-smooth}, the corresponding point $\displaystyle [C\stackrel{|L|}{\hookrightarrow} \P^r ]\in \Hilb_d$ belongs to $H_d$ and it is clearly mapped to $C\in M_g\subset \Mgp$ by $\Phi^{\rm ps}$.
We conclude that the image of $\Phi^{\rm ps}$ contains the open
dense subset $M_g\subset \Mgp$. Moreover, $\Phi^{\rm ps}$ is projective since  $\ov{Q}_{d,g}^c$ is projective. Therefore, being projective and dominant, $\Phi^{\rm ps}$ has to be surjective.
This finishes the proof of part \eqref{T:comp-Pic1}.

Now, consider Part \eqref{T:comp-Pic2}. If $d>4(2g-2)$, then the potential stability Theorem (see Fact \ref{F:GM2})  says that the geometric fibers of the morphism $u_d$ are quasi-stable curves. From Definition \ref{D:p-stab} and Proposition \ref{P:wp-stab}, it follows that the p-stable reduction $\ps(u_d)$ of $u_d$ factors through the wp-stable reduction
 $\wps(u_d)$ of $u_d$ and that the latter one is a family of stable curves.
This implies that the map $\Phi^{\rm ps}:\ov{Q}_{d,g}^c\to \Mgp$ factors via a map $\Phi^{\rm s}:\ov{Q}_{d,g}^c \to \Mg$ followed  by the contraction map $T:\Mg\to \Mgp$.

Part \eqref{T:comp-Pic3}: let us prove the result for $\Phi^{\rm ps}$, the case of $\Phi^{\rm s}$ being analogous. Observe that the open subset $(\Phi^{\rm ps})^{-1}(M_g^o)\subset \ov{Q}_{d,g}^c=H_d/\!\!/\SL_{r+1}$ is isomorphic to $H_d^o/\!\!/\SL_{r+1}$, where
$$H_d^o:=\{[X\subset \P^r]\in H_d\: :\: X \in M_g^o\}.$$
Over $H_d^o$ there is a universal family $\cC_d^o\to H_d^o$, obtained by restriction of the universal family over $\Hilb_d$, which is endowed with a line bundle $\OO_{\cC_d^o}(1)$ of relative degree $d$ coming from the natural embedding $\cC_d^o\subset H_d^o\times \P^r$. By the universal property of $J_{d,g}^o$,
 we get a morphism $H_d^o\to J_{d,g}^o$ which is clearly invariant under the action of $\SL_{r+1}$, hence it descends to a morphism
$F:(\Phi^{\rm ps})^{-1}(M_g^o)=H_d^o/\!\!/\SL_{r+1}\to J_{d,g}^o$. In order to show that $F$ is an isomorphism, we will construct an inverse of it. Let $f:\cF\to J_{d,g}^o$ be the universal family over $J_{d,g}^o$ endowed with a universal line bundle $\cL$ of relative degree $d$. Since $d>2(2g-2)\geq 2g-2$, we have that $R^1f_*(\cL)=0$ and $f_*(\cL)$ is a locally free sheaf on $J_{d,g}^o$ of rank $d+1-g$.
Moreover, $\cL$ is relatively very ample so that we have an embedding $\eta: \cF\hookrightarrow \P(f_*(\cL))$. We can choose an open covering $J_{d,g}^o=\bigcup_i U_i$ together with trivializations $f_*(\cL)_{|U_i}\cong \cO_{U_i}^{\oplus d-g+1}$ such that the restriction of the embedding $\eta$ to $\cF_{|U_i}:=f^{-1}(U_i)$ becomes equal to $\eta_{|U_i}:\cF_{|U_i}\hookrightarrow \P(f_*(\cL)_{|U_i})\cong U_i\times \P^{d-g}$. By the universal property of $H_d^o\subset \Hilb_d$, we get a morphism $g_i: U_i\to H_d^o$. The composition $\ov g_i:U_i\stackrel{g_i}{\longrightarrow} H_d^o\to H_d^o/\!\!/\SL_{r+1}$ is independent of the chosen trivialization. Therefore, the morphisms $\{\ov g_i\}_i$ glue together to a morphism $G:J_{d,g}^o\to H_d^o/\!\!/\SL_{r+1}=(\Phi^{\rm ps})^{-1}(M_g^o)$, which by construction is an inverse of $F$.

\end{proof}

%In the case where either $d>4(2g-2)$ or $2(2g-2)<d<\frac{7}{2}(2g-2)$, the morphisms $\Phi^{\rm s}$ and $\Phi^{\rm ps}$ admit nice properties.

We now determine the dimension of the fibers of the morphisms $\Phi^{\rm s}$ and $\Phi^{\rm ps}$, starting from the cases $d\not \in \{\frac{7}{2}(2g-2), 4(2g-2)\}$.
%and as a Corollary it will follow that $\ov{Q}_{d,g}^c$ is irreducible.

%From the above results, we can deduce the irreducibility of the GIT quotient $\ov{Q}_{d,g}:=H_d/\!\!/SL_{r+1}$ and further properties of  the maps $\Phi^{\rm s}: \ov{Q}_{d,g}\to \Mg$ for $d>4(2g-2)$ and of $\Phi^{\rm ps}: \ov{Q}_{d,g}\to \Mgp$ for $2(2g-2)<d<\frac{7}{2}(2g-2)$ and $g\geq 3$.

\begin{prop}\label{P:irreduci}
\noindent
\begin{enumerate}[(i)]
\item \label{P:irreduci1} Assume that $d>4(2g-2)$. The morphism $\Phi^{\rm s}: \ov{Q}_{d,g}
\to \Mg$ has equidimensional fibers of dimension $g$ and, if ${\rm char}(k)=0$, $\Phi^{\rm s}$ is flat over the smooth locus of $\Mg$.
%the restriction $$\Phi^{\rm s}:(\Phi^{\rm s})^{-1}((\Mg)_{\rm sm})\to (\Mg)_{\rm sm}$$ is flat, where $(\Mg)_{\rm sm}$ is the smooth locus of $\Mg$.
\item \label{P:irreduci2} Assume that $2(2g-2)<d<\frac{7}{2}(2g-2)$ and $g\geq 3$.
The morphism $\Phi^{\rm ps}: \ov{Q}_{d,g}
\to \Mgp$ has equidimensional fibers of dimension $g$ and, if ${\rm char}(k)=0$, $\Phi^{\rm ps}$ is flat over the smooth locus of $\Mgp$.
%the restriction $\Phi^{\rm ps}:(\Phi^{\rm ps})^{-1}((\Mgp)_{\rm sm})\to (\Mgp)_{\rm sm}$ is flat, where $(\Mgp)_{\rm sm}$ is the smooth locus of $\Mgp$.
\item \label{P:irreduci3} Assume that $\frac{7}{2}(2g-2)<d$, $d\neq 4(2g-2)$ and $g\geq 3$. The fiber of
$\Phi^{\rm ps}:\ov{Q}_{d,g}\to \Mgp$ over a p-stable curve $X$ has
dimension equal to the sum of $g$ and the number of cusps of $X$.
\end{enumerate}
%In both cases, we get that $\ov{Q}_{d,g}$ (hence $H_d$) is irreducible.
\end{prop}
\begin{proof}
%The proof is a  generalization of \cite[Cor. 5.1, Lemma 6.2, Thm. 6.1(2)]{Cap}.
%We include a proof for the reader's convenience.

The flatness assertions in \eqref{P:irreduci1} and \eqref{P:irreduci2} follow from  the equidimensionality of the fibers and the fact that $\ov{Q}_{d,g}$ is Cohen-Macauly if ${\rm char}(k)=0$ (see
Theorem \ref{T:comp-Pic}\eqref{T:comp-Pic5}) by using the following well-know flatness's criterion.

\un{Fact} (see \cite[Cor. of Thm 23.1, p. 179]{Mat}):
Let $f:X\to Y$ be a dominant morphism between irreducible varieties. If $X$ is Cohen-Macauly, $Y$ is smooth and $f$ has equidimensional fibers of the same dimension, then
$f$ is flat.

\vspace{0.1cm}

Let us now prove the statements about the dimension of the fibers.

\vspace{0.1cm}

Assume first that $d>4(2g-2)$.
%Consider the map (see Theorem \ref{T:comp-Pic})
%$$\phi^{\rm s}:H_d\to \ov{Q}_{d,g} \stackrel{\Phi^{\rm s}}{\to} \Mg.$$
 By Corollary \ref{C:quasi-wp-stable}\eqref{C:quasi-wp-stable2}, the fiber of the morphism
$$\displaystyle \phi^{\rm s}:H_d\to \ov{Q}_{d,g} \stackrel{\Phi^{\rm s}}{\to} \Mg,$$ over a stable curve $X\in \Mg$, is equal to
\begin{equation*}
(\phi^{\rm s})^{-1}(X)=\bigcup_{\stackrel{\s(X')=X}{\un d'\in B_{X'}^d}} M_{X'}^{\un d'}
\end{equation*}
where the union runs over the \emph{quasi-stable} curves $X'$ whose stable reduction $\s(X')=\wps(X')$ is equal to $X$
and $\un d'\in B_{X'}^d$. Since every such $X'$ is obtained from $X$ by bubbling some of the nodes of $X$, we have that $X'\preceq X$ (see Remark \ref{D:order-rela1}).
Therefore, Lemma \ref{L:order-rela} implies that, for every pair $(X',\un d')$ appearing in the above decomposition, there exists $\un d\in B_X^d$ such that $(X',\un d')\preceq(X, \un d)$.
This implies that
$$(\phi^{\rm s})^{-1}(X)=\ov{\bigcup_{\un d\in B_X^d} M_X^{\un d}}\cap H_d.$$
We deduce that the fiber $(\Phi^{\rm s})^{-1}(X)$ contains an open dense subset isomorphic to
$$\left(\bigcup_{\un d\in B_X^d} M_X^{\un d}\right) / SL_{r+1}= \bigcup_{\un d\in B_X^d} M_X^{\un d}/ \SL_{r+1}.$$
For any $\un d\in B_X^d$ the  map $p:M_X^{\un d}\to \Pic^{\un d}(X)/\Aut^{\un d}(X)$ of \eqref{E:map-strata-Pic} is surjective by Theorem \ref{T:semistable}\eqref{T:semistable1}, and its fibers are exactly the $\SL_{r+1}$-orbits on $M_X^{\un d}$.  Therefore, we have
$$\dim M_X^{\un d}/SL_{r+1}= {\dim \Pic^{\un d}(X)/\Aut^{\un d}(X)}= g,$$
where we used that $\Aut^{\un d}(X)\subseteq \Aut(X)$ is a finite group because $X$ is a stable curve.
% is $\SL_{r+1}$-equivariant, the natural map $M_X^{\un d}\subset H_d \to \ov{Q}_{d,g}$ factors through it. Therefore, $\dim  M_X^{\un d}/SL_{r+1}\leq \dim \Pic^{\un d}(X)=g$; hence all the irreducible components of $(\Phi^{\rm s})^{-1}(X)$ have dimension at most $g$. On the other hand, since the general fiber of $\Phi^{\rm s}$ has dimension $g$ by Theorem \ref{T:comp-Pic}\eqref{T:comp-Pic3}, all the irreducible components of $(\Phi^{\rm s})^{-1}(X)$ must have dimension at least $g$ by the upper semicontinuity of the dimension of the fibers.
We conclude that $(\Phi^{\rm s})^{-1}(X)$ is of pure dimension $g$, i.e. part \eqref{P:irreduci1} is proved.

\vspace{0.1cm}

Assume now that $2(2g-2)<d<\frac{7}{2}(2g-2)$ and $g\geq 3$. By Corollary \ref{C:quasi-p-stable}, the fiber of the morphism
$$\displaystyle \phi^{\rm ps}:H_d\to \ov{Q}_{d,g} \stackrel{\Phi^{\rm ps}}{\to} \Mgp,$$
over a p-stable curve  $X\in \Mgp$, is given by
\begin{equation*}
(\phi^{\rm ps})^{-1}(X)=\bigcup_{\stackrel{\wps(X')=X}{\un d'\in B_{X'}^d}} M_{X'}^{\un d'},
\end{equation*}
where the union is over the possible \emph{quasi-p-stable} curves $X'$ whose wp-stable reduction $\wps(X')$
(which coincides with the p-stable reduction $\ps(X')$) is equal to $X$
and $\un d'\in B_{X'}^d$.  Since every such $X'$ is obtained from $X$ by bubbling  some nodes or cusps of $X$, we have that $X'\preceq X$ (see Remark \ref{D:order-rela1}).
Therefore, Lemma \ref{L:order-rela} implies that, for every pair $(X',\un d')$ appearing in the above decomposition, there exists $\un d\in B_X^d$ such that $(X',\un d')\preceq(X, \un d)$.
This implies that
$$(\phi^{\rm ps})^{-1}(X)=\ov{\bigcup_{\un d\in B_X^d} M_X^{\un d}}\cap H_d.$$
We now conclude the proof of part \eqref{P:irreduci2} arguing as before (using Theorem \ref{T:semistable}\eqref{T:semistable2}).

\vspace{0.1cm}

Assume finally that $\frac{7}{2}(2g-2)< d$, $d\neq 4(2g-2)$ and $g\geq 3$. By Corollary \ref{C:quasi-wp-stable}\eqref{C:quasi-wp-stable1}, the fiber of the morphism
$$\displaystyle \phi^{\rm ps}:H_d\to \ov{Q}_{d,g} \stackrel{\Phi^{\rm ps}}{\to} \Mgp,$$
over a p-stable curve  $X\in \Mgp$, is given by
\begin{equation}\label{E:fib-phi}
(\phi^{\rm ps})^{-1}(X)=\bigcup_{\stackrel{\ps(X')=X}{\un d'\in B_{X'}^d}} M_{X'}^{\un d'},
\end{equation}
where the union is over the possible \emph{quasi-wp-stable} curves $X'$ whose p-stable reduction $\ps(X')$ is equal to $X$ and $\un d'\in B_{X'}^d$. Every such a curve $X'$ is obtained from $X$ by bubbling  some of the nodes or cusps of $X$ and by replacing some of the cusps of $X$ by elliptic tails.

We want, now, to rewrite \eqref{E:fib-phi} in a more convenient way.
%describe the maximal curves, among the ones appearing in \eqref{E:fib-phi}, with respect to the order relation $\preceq$ of Remark \ref{D:order-rela1}.
With this in mind, let us introduce some notation.
Let $\{c_1,\cdots, c_l\}$ be the cusps of $X$. For any subset $\emptyset \subseteq S\subseteq [l]:=\{1,\cdots, l\}$, consider the family of wp-stable curves $\eta^S: \cX^S\to V^S:=(\ov{M}_{1,1})^{S}$ such the fiber of $\eta^S$ over a point $(F_i,p_i)_{i\in S}\in (\ov{M}_{1,1})^{S}$ is the wp-stable curve obtained
from $X$ by replacing the cusp $c_i$ with the $1$-pointed stable elliptic tail $(F_i,p_i)$ for every $i\in S$.
Note that $\eta^S: \cX^S\to V^{S}$ is a family of wp-stable curves whose p-stabilization
is the trivial family $X\times V^{S}$. For a point $t\in V^S$, set $\cX^S_t:=(\eta^S)^{-1}(t)$.
We can canonically identify the properly balanced multidegrees of total degree $d$ on $\cX^S_t$ as $t$ varies in $V^S$; we therefore set $^{S}B^d:=B_{\cX^S_t}^d$ for any $t\in V_S$. Moreover, for any given $\un d\in {}^S B^d$, we consider the locally closed subset of $H_d$ given by
$${}^S M^d=\bigcup_{t\in V^S} M_{\cX_t^S}^{\un d}\subset H_d.$$

From Definition \ref{D:order-rela} it follows that, among the quasi-wp-stable curves appearing in
\eqref{E:fib-phi}, the maximal curves with respect to the order relation $\preceq$ (see Remark \ref{D:order-rela1})
are those of type $\cX^S_t:=(\eta^S)^{-1}(t)$ for some $t\in V^S$
%belonging to one of the families  $\eta^S: \cX^S\to V^{S}$,
with $\emptyset \subseteq S \subseteq [l]$. Using this and Lemma \ref{L:order-rela}, we can rewrite \eqref{E:fib-phi}
as
\begin{equation*}\label{E:fib-phi2}
(\phi^{\rm ps})^{-1}(X)=\ov{\bigcup_{\stackrel{\emptyset\subseteq S\subseteq [l]}{\un d\in {}^S B^d}}
{}^S M^{\un d}}\cap H_d,
\end{equation*}
from which it follows that
%We deduce that the fiber $(\Phi^{\rm ps})^{-1}(X)$ contains an open dense subset isomorphic to
\begin{equation}\label{E:fiber-Phi}
\bigcup_{\stackrel{\emptyset\subseteq S\subseteq [l]}{\un d\in {}^S B^d}}
{}^S M^{\un d}/SL_{r+1} \text{ is open and dense in } (\Phi^{\rm ps})^{-1}(X).
\end{equation}
Using the map $p$ of \eqref{E:map-strata-Pic} and the fact that $\Aut(\cX^S_t)$ is a finite group since $\cX^S_t$ is wp-stable, we get for any $\un d\in {}^S B^d$ and any $t\in V^S$:
 \begin{equation}\label{E:estim1}
 \dim M_{\cX^S_t}^{\un d}/SL_{r+1}\leq  {\dim \Pic^{\un d}(\cX^S_t)/\Aut^{\un d}(\cX^S_t)}= g.
 \end{equation}
% with equality if and only if $M_{\cX^S_t}^{\un d}\neq \emptyset$.
We deduce that
 \begin{equation}\label{E:estim2}
 \dim {}^S M^{\un d}/SL_{r+1}\leq \dim V^S+g= |S|+ g,
 \end{equation}
% with equality if and only if $M_{\cX^S_t}^{\un d}\neq \emptyset$ for the generic element $t\in V^S$.
which by \eqref{E:fiber-Phi} implies that $\dim (\Phi^{\rm ps})^{-1}(X)\leq g+l$.

Consider, now, the special case where $S=[l]$. In this case, the curves $\cX_t^{[l]}$ are stable for any $t\in V^{(l)}$ and, for a generic $L_t\in \Pic^{\un d}(\cX_t^{[l]})$,
any element of the form $[\cX_t^{[l]}\stackrel{|L_t|}{\hookrightarrow} \P^r]$  is Chow (or equivalently Hilbert) semistable by Theorems
\ref{T:semistable}\eqref{T:semistable1} and \ref{T:semistable4}.
% the maps $p:M_{\cX_t^{[l]}}^{\un d}\to \Pic^{\un d}(\cX_t^{[l]})/\Aut^{\un d}(\cX_t^{(l)})$ are dominant for  any $\un d\in {}^S B^d$ and for $d\geq \frac{7}{2}(2g-2)$ and $g\geq 3$ by Theorems
%\ref{T:semistable}\eqref{T:semistable1} and \ref{T:semistable4}.
Therefore, for $S=[l]$ equality does hold in \eqref{E:estim1} and \eqref{E:estim2} and we deduce that $\dim (\Phi^{\rm ps})^{-1}(X)= g+l$.
\end{proof}

Let us study the dimension of the fibers of the morphisms $\Xi:\ov{Q}_{d,g}^h\to \ov{Q}_{d,g}^c$  and
$\Phi^{\rm ps}: \ov{Q}_{d,g}^c \to \Mgp$, as well as of their composition, in the two special cases $d \not \in \{\frac{7}{2}(2g-2), 4(2g-2)\}$.

\begin{prop}\label{P:fib3.5}
Assume that $d=\frac{7}{2}(2g-2)$ and that $g\geq 3$.
\begin{enumerate}[(i)]
\item \label{P:fib3.5A} The morphism $\Xi: \ov{Q}_{d,g}^h\to \ov{Q}_{d,g}^c$ is an isomorphism.
\item \label{P:fib3.5B} The morphisms $\Phi^{\rm ps}: \ov{Q}_{d,g}^c\to \Mgp$ and $\Phi^{\rm ps}\circ \Xi: \ov{Q}_{d,g}^h\to \Mgp$ have equidimensional fibers of dimension $g$ and,
if ${\rm char}(k)=0$, then $\Phi^{\rm ps}$ and $\Phi^{\rm ps}\circ \Xi$ are flat over the smooth locus of $\Mgp$.
\end{enumerate}
\end{prop}
\begin{proof}

In order to prove part \eqref{P:fib3.5A}, by applying the Zariski's main theorem in the form \cite[(4.4.9)]{EGAIII1}, it is enough to check that $\ov{Q}_{d,g}^c$ is reduced and normal  and that $\Xi$ is
birational and injective.

The fact that $\ov{Q}_{d,g}^c$ is reduced and normal follows from Proposition \ref{P:sing-GITquot}.

Consider the open and dense $\SL_{r+1}$-invariant subset $\Hilb_d^{s}\cap \wt{H}_d\subseteq \wt{H}_d$.  GIT tells us that there exists a good geometric quotient
$(\Hilb_d^{s}\cap \wt{H}_d)/\SL_{r+1}$ (in the sense of \cite[Sec. 6.1]{Dol}) which is an open subset  of $\wt{H}_d/\!\!/\SL_{r+1}=\ov{Q}_{d,g}^h$. Moreover, since $\Ch^{-1}(\Chow_d^s)\cap H_d$ is an open
and dense $\SL_{r+1}$-invariant subset of $\Hilb_d^{s}\cap \wt{H}_d$ by Fact \ref{HilbtoChow},  the properties of the good geometric quotients (see \cite[Sec. 6.1]{Dol}) ensure that there exists
a good geometric quotient $(\Ch^{-1}(\Chow_d^s)\cap H_d)/\SL_{r+1}$ which is an open and dense subset of $(\Hilb_d^{s}\cap \wt{H}_d)/\SL_{r+1}$.
Clearly, $\Xi$ is an isomorphism over $(\Ch^{-1}(\Chow_d^s)\cap H_d)/\SL_{r+1}$, which shows that $\Xi$ is birational.

Finally, let us show that $\Xi$ is injective, which will conclude the proof.
Consider a point of $\ov{Q}_{d,g}^h$ represented by the orbit of an Hilbert polystable point
$[X\subset \P^r]\in \wt{H}_d$ (see Remark \ref{R:closedpts}).
According to Corollary \ref{C:polystable3}\eqref{C:polystable3h}, this is equivalent to the fact that $X$
is quasi-p-stable and that $\OO_X(1)$ is strictly balanced.
From Theorem \ref{T:basintacn} and Corollary \ref{C:polystable3}\eqref{C:polystable3c} it follows that
$\Xi([X\subset \P^r])$ is represented by the orbit of any Chow polystable
point $[Y\subset \P^r]$ of $H_d$ such that:
\begin{itemize}
\item Let $\{q_1,\ldots,q_n\}$ be the tacnodes with a line of $X$; denote by $E_i$ the line contained in $X$ and passing through $q_i$ (for any $i=1,\ldots, n$) and let $\wh{X}$ be the complement of
the lines $E_i$ in $X$. Then $Y$ is obtained from $\wh{X}$ by gluing at each point $q_i$ an elliptic tail $F_i=F_i^1\cup F_i^2\cup F_i^3$, where $F_i^j\cong \P^1$ (for each $j=1,2,3$), $F_i^1$ is
joined nodally to $\wh{X}$ in $q_i$ and to $F_i^2$ while $F_i^2$ and $F_i^3$ meets in a tacnode. Note that $F_i^2\cup F_i^3$ is a tacnodal elliptic tail for each $i$.
\item $\OO_Y(1)$ is a strictly balanced line bundle on $Y$ such that $\OO_Y(1)_{|F_i^1}=\OO_{F_i^1}(1)$, $\OO_Y(1)_{|F_i^2}=\OO_{F_i^2}(2)$, $\OO_Y(1)_{|F_i^3}=\OO_{F_i^3}(1)$
and $\OO_Y(1)_{|\wh{X}}=\OO_X(1)_{|\wh{X}}$.
\end{itemize}
Note that two line bundles $\OO_Y(1)$ as above differ by an automorphism of $Y$ (as it follows from the proof of Theorem \ref{T:auto-grp}), so that the orbit of $[Y\subset \P^r]$
is well-defined (see Remark \ref{R:closedpts}).

From this explicit description it follows that the curve $X$ and the restriction $\OO_X(1)_{|\wt{X}}$ are uniquely determined by the orbit of the Chow-polystable point $[Y\subset \P^r]\in H_d$.
Since the line bundle $\OO_X(1)$ is uniquely determined by its restriction $\OO_X(1)_{|\wt{X}}$ up to automorphisms of $X$, we can recover the orbit of $[X\subset \P^r]$ from the orbit of
$[Y\subset \P^r]$ (see Remark \ref{R:closedpts}), which shows the injectivity of $\Xi$, q.e.d.

Part \eqref{P:fib3.5B}: using part \eqref{P:fib3.5A}, it is enough to prove the result for the morphism  $\Phi^{\rm ps}\circ \Xi: \ov{Q}_{d,g}^h\to \Mgp$. The proof of the statement for
$\Phi^{\rm ps}\circ \Xi$ is exactly the same as the proof of Proposition \ref{P:irreduci}\eqref{P:irreduci2} replacing Theorem \ref{T:semistable}\eqref{T:semistable2} by
Theorem \ref{T:semistable3}\eqref{T:semistable3h}.
\end{proof}

\begin{prop}\label{P:fib4}
Assume that $d=4(2g-2)$  and that $g\geq 3$.
\begin{enumerate}[(i)]
\item \label{P:fib4A} The fiber of the morphism $\Xi: \ov{Q}_{d,g}^h\to  \ov{Q}_{d,g}^c$ over the orbit of a Chow polystable point $[X\subset \P^r]\in H_d$ is equal to the number of cuspidal elliptic tails of $X$ that are special with respect to $\OO_X(1)$.
\item \label{P:fib4B} The fiber of the morphism  $\Phi^{\rm ps}:\ov{Q}_{d,g}^c\to \Mgp$ (resp. $\Phi^{\rm ps}\circ \Xi: \ov{Q}_{d,g}^h\to \Mgp$) over a p-stable curve $X$ has
dimension equal to the sum of $g$ and the number of cusps of $X$.
\end{enumerate}
\end{prop}
\begin{proof}
Part \eqref{P:fib4A}: consider the point of $\ov{Q}_{d,g}^c$ represented by the orbit of the Chow polystable point
$[X\subset \P^r]\in H_d$ (see Remark \ref{R:closedpts}). By Corollary \ref{C:polystable5}\eqref{C:polystable5c}, $X$ is quasi-wp-stable without tacnodes, $\OO_X(1)$ is strictly balanced and all the cusps of $X$ are contained in special cuspidal elliptic tails of $X$ which, furthermore, are the unique special elliptic tails or cuspidal elliptic tails of $X$.

If $X$ does not have special cuspidal elliptic tails (hence it does not have special elliptic tails at all),
then  $[X\subset \P^r]$ is also Hilbert polystable by Corollary \ref{C:polystable5}\eqref{C:polystable5h} and its orbit represents the unique point of $\Xi^{-1}([X\subset \P^r])$ and we are done.

In the general case, let $\{F_1, \ldots, F_n\}$ be the special cuspidal elliptic tails of $X$.
%and let $\wh{X}$ be the complement of the tails $F_i$ in $X$.
Set $q_i:=F_i\cap \wh{X}$ and note that $\deg \OO_X(1)_{|F_i}=4$ by the
basic inequality \eqref{E:basineq-multideg}.
By Corollary \ref{C:polystable5}\eqref{C:polystable5h} and Theorem \ref{T:basin-cusps}, any Hilbert polystable point
$[Y\subset \P^r]\in \wt{H}_d$ such that $\Xi([Y\subset \P^r])=[X\subset \P^r]$ is of the following form
(for some $\emptyset \subseteq S\subseteq [n]=\{1,\ldots, n\}$):

\un{Type $S$}:  $Y=Y_S$ is obtained from $X$ by contracting to a cuspidal point $q_i'$ all the tails $F_i$ such that
$i\in S$; in particular, there is a natural morphism $\nu_S:\wh{X}_S:=(\cup_{i\in S}F_i)^c\to Y$ which is the partial normalization of $Y$ at the cusps $q_i'$ (with $i\in S$).
Moreover, the line bundle $\OO_{Y}(1)$ is such that $\nu_S^*\OO_Y(1)=\OO_X(1)_{|\wh{X}_S}(4\cdot \sum_{i\in S} q_i)$
and each of the cuspidal elliptic tails $F_i\subset Y$ with $i\not\in S$ is non-special with respect to $\OO_Y(1)$.
Set $\un d_S$ equal to the strictly balanced multidegree of such a line bundle $\OO_Y(1)$.

From Definition \ref{D:order-rela} (and in particular Figure \ref{CuspidalTail}),  it follows that if
$\emptyset\subseteq T\subseteq S\subseteq [n]$ then $(Y_T,\un d_T)\preceq (Y_S,\un d_S)$ which then implies that
$M_{Y_T}^{\un d_T}\subseteq \ov{M_{Y_S}^{\un d_S}}$ by Proposition \ref{P:deg-strata}. In other words, inside the fiber $\Xi^{-1}([X\subset \P^r])$, the points of Type $S=[n]$ are dense.

Observe now that for points $[Y\subset \P^r]\in \wt{H}_d$ of Type $S=[n]$, the line bundle $\OO_Y(1)$ is specified up to the choice of the gluing data for each of the cusps $q_i'$. Since each of the cusps give a one-dimensional space of gluing conditions for $\OO_Y(1)$, points of Type $S=[n]$ form an irreducible $n$-dimensional family sitting in the fiber $\Xi^{-1}([X\subset \P^r])$. This shows that the dimension of $\Xi^{-1}([X\subset \P^r])$ is equal to $n$, which was
the number of special cuspidal elliptic tails of $X$, q.e.d.

Part \eqref{P:fib4B}: the same proof of Proposition \ref{P:irreduci}\eqref{P:irreduci3} works in this case by
replacing Theorems \ref{T:semistable}\eqref{T:semistable1} and \ref{T:semistable4} with Theorem \ref{T:semistable5}.

\end{proof}

Using the above Proposition, we can prove the irreducibility of $\ov{Q}_{d,g}^c$ and $\ov{Q}_{d,g}^h$ (and hence of $H_d$ and $\wt{H}_d$).

\begin{coro}\label{C:irr-quot}
Assume that $d > 2(2g-2)$ and, moreover, that $g \geq  3$ if $d\leq 4(2g-2)$. Then $\ov{Q}_{d,g}^c$ and $\ov{Q}_{d,g}^h$ are irreducible. In particular, $H_d$ and $\wt{H}_d$ are also irreducible.
\end{coro}
\begin{proof}

Let us first prove the irreducibility of $\ov{Q}_{d,g}^c$.

 In the case $d\leq 4(2g-2)$ (and $g\geq 3$), look at the surjective morphism  $\Phi^{\rm ps}:\ov{Q}_{d,g}^c\to \Mgp$. Since $\Mgp$ is irreducible by Theorem \ref{T:stacks-curves}\eqref{T:stacks-curves3} and the generic fiber of
$\Phi^{\ps}$ is irreducible by Theorem \ref{T:comp-Pic}\eqref{T:comp-Pic3}, we get that there exists a unique irreducible component of $\ov{Q}_{d,g}^c$ that dominates $\Mgp$.
Assume, by contradiction, that there is another irreducible component of $\ov{Q}_{d,g}^c$, call it $Z$, that does not dominate $\Mgp$. Let $W:=\Phi^{\rm ps}(Z)\subsetneq \Mgp$ and denote by
$l\geq 0$ the number of cusps of the generic point $X\in W$. Since each cusp will increase the codimension of $W$ in $\Mgp$ by two, we get that
\begin{equation}\label{E:dim-W}
\dim W\leq \min\{3g-4, 3g-3-2l\}.
\end{equation}
Propositions \ref{P:irreduci}\eqref{P:irreduci2}, \ref{P:irreduci}\eqref{P:irreduci3}, \ref{P:fib3.5}\eqref{P:fib3.5B},
\ref{P:fib4}\eqref{P:fib4B} imply that the generic fiber of the map $Z\twoheadrightarrow W$ has dimension less than or equal to $g+l$. Using this and \eqref{E:dim-W}, we get
\begin{equation}\label{E:dim-Z}
\dim Z\leq \min\{4g-4+l, 4g-3-l\}<4g-3.
\end{equation}
This however contradicts the fact that $\ov{Q}_{d,g}$ is of pure dimension equal to $4g-3$ by Proposition \ref{P:sing-GITquot}\eqref{T:comp-Pic5}, q.e.d.

The case $d>4(2g-2)$ is dealt with in a similar (and easier) way by considering the map $\Phi^{\rm s}:\ov{Q}_{d,g}\to \Mg$ and using Proposition \ref{P:irreduci}\eqref{P:irreduci1}.

From the irreducibility of $\ov{Q}_{d,g}^c$ it follows that: $H_d$ is connected (hence irreducible because of its smoothness, see Proposition \ref{P:sing-GITquot}\eqref{T:comp-Pic4})
because $\ov{Q}_{d,g}^c$ is the good categorial quotient of $H_d$ by the connected algebraic group $\SL_{r+1}$; $\wt{H}_d$ is irreducible because it is an open subset of $H_d$; $\ov{Q}_{d,g}^h$ is irreducible because it is the good categorical quotient of $\wt{H}_d$ by $\SL_{r+1}$.

\end{proof}

\section{Extra components of the GIT quotient}
\label{sec:extra}

So far we have considered the action of $\GL_{r+1}$ over $\Hilb_d$, and we have restricted our attention to $\Ch^{-1}(\Chow_d^{ss})^o$ and $\Hilb_d^{ss,o}$, the Chow or Hilbert semistable loci consisting of connected curves.  It is very natural to ask if there are Chow or Hilbert semistable points $[X\subset \P^r]\in \Hilb_d$ with $X$ not connected. In this section we will  answer to this question.

First of all, as a corollary of the Potential pseudo-stability Theorem \ref{teo-pstab}, we can prove the following result.

\begin{coro}\label{cor:nonconn}
Let $[X\subset \P^r]\in \Ch^{-1}(\Chow_d^{ss})$ (resp. $\in \Hilb_d^{ss}$) where $X=X_1\cup\ldots\cup X_n$ and each $X_i$ is a connected component of $X$. Suppose that $d>2(2g-2)$, set $d_i:=\deg\OO_{X}(1)_{|X_i}$, $r_i:=\dim\langle X_i\rangle$ (where $\langle X_i\rangle$ is the linear span of $X_i$) and denote by $g_i$ the genus of $X_i$. Then the following hold:
\begin{enumerate}
	\item $h^0(X_i,\OO_{X_i}(1))=d_i-g_i+1=r_i+1$, $h^1(X_i,\OO_{X_i}(1))=0$ and
  $$
  h^0(\P^r,\OO_{\P^r}(1))=h^0(X,\OO_{X}(1))=\sum_{i=1}^n h^0(X_i,\OO_{X_i}(1)).
  $$
  In particular, $\langle X_i\rangle\cap \langle X_j\rangle=\emptyset$ for every $i\neq j$.
  \item For each $i$
  $$
  \frac{d_i}{2g_i-2}=\frac{d}{2g-2}\quad\bigg(\text{i. e.}\quad \frac{d_i}{r_i+1}=\frac{d}{r+1}\bigg).
  $$
  In particular, if $n\geq 2$, then $\gcd(d,g-1)\neq 1$.
  \item For each $i$, $[X_i\subset \langle X_i\rangle]\in \Hilb_{d_i,g_i}$ is Chow (resp. Hilbert) semistable.
  \item If $n\geq 2$, $[X\subset \P^r]$ is Chow (resp. Hilbert) strictly semistable.
\end{enumerate}
\end{coro}
\begin{proof}
(1) follows easily from Theorem \ref{teo-pstab}\eqref{teo-pstab2}. (2) is an easy consequence of the basic inequality applied to $\OO_X(1)$, which holds by Theorem \ref{teo-pstab}\eqref{teo-pstab3}.
 Indeed, if $X_i$ is a connected component of $X$, we have $k_{X_i}=0$, hence
$$
d_i=\frac{d}{2g-2}(2g_i-2)
$$
and we are done.
If $X$ is not connected and, by contradiction, $\gcd(d,g-1)=1$, the ratio $\displaystyle \frac{d}{g-1}$ is reduced, hence for each connected component $X_i\subset X$, we have $d_i=d$, which is a contradiction. Let us prove (3). Consider a 1ps $\rho:\Gm\ra \GL_{r_1+1}$ diagonalized by a system of coordinates $\{x_1,\ldots,x_{r_1+1}\}$ in $\langle X_1\rangle$ and denote by $w_1,\ldots,w_{r_1+1}$ the weights of $\rho$. Let $\{y_1,\ldots,y_{r+1}\}$ be a system of coordinates in $\P^r$ such that ${y_i}_{|X_1}={x_i}_{|X_1}$ and
\begin{eqnarray*}
\langle X_1\rangle=\bigcap_{i=r_1+2}^{r+1}\{y_i=0\}\quad \text{and}\quad \langle X_1^c\rangle=\bigcap_{i=1}^{r_1+1}\{y_i=0\}.
\end{eqnarray*}
Now consider a 1ps $\rho':\Gm\ra \GL_{r+1}$ diagonalized by $\{y_1,\ldots,y_{r+1}\}$ with weights $w_1',\ldots,w_{r+1}'$ such that
$$
w_i'=
\left\{
\begin{array}{ll}
w_i & \text{if }1 \leq i\leq r_1+1\\
0 & \text{if }i\geq r_1+2.\\
\end{array}
\right.
$$
By Proposition \ref{prop:stab-tail}, we get
$$
e_{X_1,\rho}=e_{X\rho'}\leq \frac{2d}{r+1}\,w(\rho')=\frac{2d_1}{r_1+1}\,w(\rho),
$$
so that $[X_1\subset \langle X_1\rangle]\in \Hilb_{d_1,g_1}$ is Chow semistable (the Hilbert semistability is proved in the same way). In order to prove (4), it suffices to consider $\rho$ and $\rho'$ as above with $w_i=1$ for $i=1,\ldots,r_1+1$. We get
$$
e_{X,\rho'}=e_{X_1,\rho}=2d_1=\frac{2d_1}{r_1+1}\,(r_1+1)=\frac{2d}{r+1}\,w(\rho')
$$
and we are done.
\end{proof}

Next, we are going to show that each point $[X\subset \P^r]\in \Hilb_d$,
which satisfies (1), (2) and (3) of Corollary \ref{cor:nonconn}, is Chow (resp. Hilbert) semistable.

Suppose that $d>2(2g-2)$ and let $[X\inj \P^{r}]\in \Hilb_{d}$, where $X$ is the disjoint union of two curves (possibly non connected) $X_1$ and $X_2$ (of degrees $d_1,d_2$ and genus $g_1,g_2$ respectively). Under the hypothesis that $h^1(X,\OO_X(1))=0$, we have $h^0(X,\OO_{X}(1))=h^0(X_1,\OO_{X_1}(1))+h^0(X_2,\OO_{X_2}(1))$, hence there exists a system of coordinates $\{x_1,\ldots,x_{r+1}\}$ such that
\begin{eqnarray}\label{eq:coor-nonconn}
\langle X_1\rangle=\bigcap_{i=r_1+2}^{r+1}\{x_i=0\}\quad \text{and}\quad \langle X_2\rangle=\bigcap_{i=1}^{r_1+1}\{x_i=0\}.
\end{eqnarray}
We have the following criterion (very similar to Proposition \ref{prop:stab-tail}).
\begin{prop}\label{prop:stab-comp}\textbf{\emph{(Criterion of stability for non-connected curves.)}}
Let $[X\subset \P^{r}]\in \Hilb_{d}$ as above. The following conditions are equivalent:
\begin{enumerate}
	\item $[X\subset \P^{r}]$ is Hilbert semistable (resp. polystable, stable);
	\item $[X\subset \P^{r}]$ is Hilbert semistable (resp. polystable, stable)
	with respect to any one-parameter subgroup $\rho:\Gm\ra \GL_{r+1}$ diagonalized by coordinates of type (\ref{eq:coor-nonconn});
	\item $[X\subset \P^{r}]$ is Hilbert semistable (resp. polystable, stable)
	with respect to any one-parameter subgroup $\rho:\Gm\ra \GL_{r+1}$ diagonalized by coordinates of type (\ref{eq:coor-nonconn}) with weights $w_1,\ldots,w_{r+1}$ such that
$$
w_{1}=w_{2}=\ldots=w_{r_1+1}=0\quad\text{or}\quad w_{r_1+2}=w_{r_1+3}=\ldots=w_{r+1}=0.
$$
\end{enumerate}
The same holds for the Chow semistability (resp. polystability, stability).
\end{prop}
\begin{proof}
It is analogous to the proof of Proposition \ref{prop:stab-tail}.
\end{proof}

As a corollary of the above Proposition, we have that the converse of Corollary \ref{cor:nonconn} holds true.

\begin{coro}\label{cor:critnonconn}
Let $[X\subset \P^{r}]\in \Hilb_{d}$ where $X=X_1\cup\ldots\cup X_n$ and each $X_i$ is a connected component of $X$. Set $d_i:=\deg\OO_{X}(1)_{|X_i}$ and denote by $g_i$ the genus of $X_i$. If $[X\subset \P^{r}]\in \Hilb_{d}$ satisfies (1) and (2) of Corollary \ref{cor:nonconn}, then the following conditions are equivalent:
\begin{enumerate}
	\item $[X\subset \P^{r}]$ is Hilbert semistable (resp. polystable);
	\item $[X_i\subset \langle X_i\rangle]\in \Hilb_{d_i,g_i}$ is Hilbert semistable (resp. polystable).
\end{enumerate}
The same holds for the Chow semistability (resp. polystability).
\end{coro}

Thus, the semistable locus and the polystable locus of $\Hilb_d$ for $d>2g-2$ are completely determined by applying the previous corollary and the results of Section \ref{S:semistab1} and
Section \ref{S:semistab2} about the stability of connected curves.

We are now able to determine the connected components of $\Hilb_d^{ss}$ and $\Chow_d^{ss}$ for $d>2(2g-2)$. Set
$$
d':=\frac{d}{\gcd(d,g-1)}\quad\text{and}\quad g':=\frac{g-1}{\gcd(d,g-1)}+1.
$$
Let $\mathcal{H}$ be a connected component of $\Hilb_d^{ss}$ and consider $[X\subset \P^r]\in \mathcal{H}$. Using the same notation as in Corollary \ref{cor:nonconn}, suppose that
$d_1\geq d_2\geq\ldots\geq d_n$.
 We get a well-defined
(integral) partition $\displaystyle \left(\frac{d_1}{d'},  \ldots,\frac{d_n}{d'}\right)$ of $\gcd(d,g-1)$.
%$$ \frac{d_1}{d'}+\frac{d_2}{d'}+\ldots+\frac{d_n}{d'} $$
Define the function
$$
\begin{aligned}
\phi: \{\text{\rm{connected components of }} \Hilb_d^{ss}\} & \longrightarrow \{\text{\rm{partitions of }}\gcd(d,g-1)\} \\
\mathcal{H} & \stackrel{}{\longmapsto} \left(\frac{d_1}{d'},  \ldots,\frac{d_n}{d'}\right).\\
\end{aligned}
$$
Conversely, let $(k_1,\ldots,k_n)$ be a partition of $\gcd(d,g-1)$. For each $i=1,\ldots,n$ consider a smooth curve $X_i$ of genus $g_i=g'k_i+1$ and a line bundle $L_i$ on $X_i$ of degree $d_i=d'k_i$. Define the curve $\displaystyle X=\bigsqcup_{i=1}^n X_i$
and consider the line bundle $L$ on $X$ such that $L_{|X_i}=L_i$. Using the assumption that $d>2(2g-2)$, it is easy to see that $d_i\geq 2g_i+1$, so that $L_i$ is very ample, $[X\stackrel{|L_i|}{\hookrightarrow}\P^{d_i-g_i}]\in \Hilb_{d_i,g_i}$ is Hilbert stable (notice that $g_i\geq 2$ for every $i$) and $[X\stackrel{|L|}{\hookrightarrow}\P^r]\in \Hilb_d$ is Hilbert semistable by Corollary \ref{cor:critnonconn}. Let $\mathcal{K}$ the connected component which contains $[X\stackrel{|L|}{\hookrightarrow}\P^r]$. Now define the function
$$
\begin{aligned}
\psi: \{\text{\rm{partitions of }}\gcd(d,g-1)\} & \longrightarrow \{\text{\rm{connected components of }} \Hilb_d^{ss}\} \\
(k_1,\ldots,k_n) & \stackrel{}{\longmapsto} \mathcal{K}.\\
\end{aligned}
$$
It is easy to check that $\phi\circ\psi=\id$ and $\psi\circ\phi=\id$.
Summing up, we obtain that
$$
\Hilb_d^{ss}=\bigsqcup_{\pi\text{ part. of }\gcd(d,g-1)}\psi(\pi).
$$
The same arguments works for $\Ch^{-1}(\Chow_d^{ss})$ giving  a bijection
$$
\phi': \{\text{\rm{connected components of }} \Ch^{-1}(\Chow_d^{ss}) \}  \longrightarrow \{\text{\rm{partitions of }}\gcd(d,g-1)\}.
$$
We have proved the following

\begin{thm}\label{T:conn-comp}
There is a commutative diagram
$$\xymatrix{
\{\text{\rm{connected components of }} \Hilb_d^{ss}\} \ar[r]^{\phi}\ar[d]_{\eta} & \{\text{\rm{partitions of }}\gcd(d,g-1)\}\\
\{\text{\rm{connected components of }} \Ch^{-1}(\Chow_d^{ss})\}
\ar[ru]_{\phi'}&
}$$
where all the maps are one-to-one correspondences and $\eta$
%$$
%\{\text{\rm{connected components of }} \Hilb_d^{ss}\}\lra \{\text{\rm{connected components of }} \Chow_d^{ss}\}
%$$
is induced by the inclusion  $\Hilb_d^{ss}\subseteq \Ch^{-1}(\Hilb_d^{ss})$.
\end{thm}

\section{Compactifications of the universal Jacobian}
\label{S:comp-Jac}

Fix integers $d$ and $g\geq 2$.
Consider the stack $\cJ_{d,g}$, called  the \emph{universal Jacobian stack} of genus $g$ and degree $d$, whose section over a scheme $S$ is the groupoid of families
of smooth curves of genus $g$ over $S$ together with a line bundle of relative degree $d$.
We denote by $J_{d,g}$ its coarse moduli space, and we call it the \emph{universal Jacobian variety} (or simply the universal Jacobian) of degree $d$ and genus $g$\footnote{In \cite{Cap}, this variety is called the universal Picard variety and it is denoted by $P_{d,g}$. We prefer to use the name universal Jacobian, and therefore the symbol $J_{d,g}$, because the word Jacobian variety is used only for curves while the word Picard variety is used also for varieties of higher dimensions and therefore it is more ambiguous. Accordingly, we will denote the Caporaso's
compactified universal Jacobian by $\ov{J}_{d,g}$ instead of $\ov{P}_{d,g}$ as in \cite{Cap} (see Fact \ref{F:Old-Comp-Jac}).}. The aim of this section is to show how we can use the GIT analysis carried out in the previous sections in order to obtain three different modular compactifications of the universal Jacobian stack $\cJ_{d,g}$ and of the universal Jacobian variety $J_{d,g}$.

\subsection{Caporaso's compactification}\label{S:capcomp}

The first compactification of $\cJ_{d,g}$ and of $J_{d,g}$ was constructed by L. Caporaso as an output of the GIT analysis carried out in \cite{Cap}. In this subsection, we review this compactification.

%From the work of Caporaso (\cite{Cap}), it is possible to obtain a modular compactification of the universal Jacobian stack and of the universal Jacobian variety.
Denote by $\JJst$ the category fibered in groupoids over the category of schemes whose section over a scheme $S$ is the groupoid of families of quasi-stable curves over $S$ of genus $g$
endowed with a line bundle whose restriction to each geometric fiber is a properly balanced line bundle of degree $d$.
We summarize the main properties of $\JJst$ into the following

\begin{fact}\label{F:Old-Comp-Jac}
Let $g\geq 2$ and $d\in \Z$.
\begin{enumerate}
\item \label{Old1} $\JJst$ is a smooth,  irreducible, universally closed Artin stack of finite type over $k$, having dimension $4g-4$ and containing $\cJ_{d,g}$ as an open substack.
%The category fibered in groupoids $\JJps$ representing families of quasi-p-stable curves endowed with a line bundle whose restriction to the geometric fibers is
%properly balanced is a smooth,  irreducible and universally closed Artin stack of dimension $4g-4$.

\item \label{Old2} $\JJst$  admits an adequate moduli space $\Jst$ (in the sense of \cite{alper2}), which is a normal irreducible projective variety  of dimension $4g-3$
containing $J_{d,g}$ as an open subvariety.
%Moreover, if ${\rm char}(k)=0$, then $\Jps$ has rational singularities, hence it is Cohen-Macauly.

\item \label{Old3} There exists a commutative digram
$$\xymatrix{
\JJst \ar[r] \ar[d]_{\Psi^{\rm s}} & \Jst \ar[d]^{\Phi^{\rm s}}\\
\ov{\mathcal M}_g \ar[r] & \Mg
}$$
where $\Psi^{\rm s}$ is universally closed and surjective and $\Phi^{\rm s}$ is projective, surjective with equidimensional fibers of dimension $g$.

\item \label{Old4} If ${\rm char}(k)=0$, then for any $X\in \Mg$ we have that
$$ {(\Phi^{\rm s})}^{-1}(X)\cong \ov{\Jac_d}(X)/ \Aut(X),$$
 where $\ov{\Jac_d}(X)$ is the canonical compactified Jacobian of $X$ in degree $d$,
parametrizing  rank-$1$, torsion-free sheaves on $X$  that are slope-semistable with respect to $\omega_X$ (see Remark \ref{R:Gor-ss}\eqref{R:Gor-ss2}).

\item  \label{Old5} If $4(2g-2)< d $ then we have that
$$\begin{sis}
& \JJst\cong [H_d/GL(r+1)],\\
& \Jst\cong H_d/\!\!/GL(r+1)=\ov{Q}_{d,g},\\
\end{sis}$$
where $H_d\subset \Hilb_d$ is the open subset consisting of points
$[X\subset \P^r]\in \Hilb_d$ such that $X$ is connected and $[X\subset \P^r]$ is Chow semistable (or equivalently, Hilbert semistable).

\end{enumerate}
\end{fact}

Parts \eqref{Old1}, \eqref{Old2}, \eqref{Old3} follow by combining the work of Caporaso (\cite{Cap}, \cite{capneron}) and that of Melo (\cite{melo1}).
Part \eqref{Old5} follows as well from the previous cited manuscripts if $d\geq 10(2g-2)$ and working with Hilbert semistability. The extension to $d>4(2g-2)$ and to the Chow semistability
follows straightforwardly from our Theorem \ref{T:semistable}\eqref{T:semistable1}. Part \eqref{Old4} was observed by Alexeev in \cite[Sec. 1.8]{ale}
(see also \cite[Sec. 2.9]{CMKV} for a related discussion and in particular for a discussion about the need for the assumption ${\rm char}(k)=0$).

We call $\JJst$ (resp. $\Jst$) the \emph{Caporaso's compactified universal Jacobian stack} (resp. \emph{Caporaso's compactified universal Jacobian variety}) of genus $g$ and degree $d$.

\subsection{Two new compactifications of the universal Jacobian stack $\cJ_{d,g}$}\label{S:2newcomp}

The aim of this subsection is to define and study two new compactifications of the universal Jacobian stack $\cJ_{d,g}$, one  over the stack $\MMgp$ of p-stable curves of genus $g$ and the other over the stack $\MMgwp$ of wp-stable curves of genus $g$.
%Fix throughout this section two integers $d$ and $g\geq 3$.

\begin{defi}\label{D:2comp}
Fix two integers $d$ and $g\geq 3$.
\begin{enumerate}[(i)]
\item \label{D:2comp1} Let $\JJps$ be the category fibered in groupoids over the category of $k$-schemes whose sections over a $k$-scheme $S$ are pairs $(f:\mathcal X\to S,\mathcal L)$ where $f$ is a family of quasi-p-stable curves of genus $g$ and $\mathcal L$
is a line bundle on $\mathcal X$ of relative degree $d$ that is properly balanced on the geometric fibers of $f$. Arrows between such pairs are given by cartesian diagrams
\begin{equation*}
\xymatrix{
\ar@{}[dr]|{\square}{\mathcal X} \ar[d]_{f} \ar[r]^{h} & {\mathcal X'} \ar[d]^{f'} \\
{S} \ar[r] & {S'}
}
\end{equation*}
together with a specified isomorphism $\mathcal L\stackrel{\cong}{\longrightarrow} h^*\mathcal L'$ of line bundles over $\cX$.
%Note that $\JJps$ is a category fibered in groupoids over the category of $k$-schemes.

\item \label{D:2comp2} Let $\JJwp$ be the category fibered in groupoids over the category of $k$-schemes whose sections over a $k$-scheme $S$ are pairs $(f:\mathcal X\to S,\mathcal L)$ where $f$ is a family of quasi-wp-stable curves of genus $g$ and $\mathcal L$ is a line bundle on $\mathcal X$ of relative degree $d$ that is properly balanced on the geometric fibers of $f$ and such that the geometric fibers of $f$ do not contain tacnodes with a line nor special elliptic tails relative to $\cL$. Arrows between such pairs are given as in \eqref{D:2comp1} above.
\end{enumerate}
\end{defi}

The aim of this subsection is to prove that $\JJps$ and $\JJwp$ are algebraic stacks and to study their properties.
%and that $\ov{Q}_{d,g}$ is a {\em good moduli space} for $\JJps$ in the sense of Alper (\cite{alper}).
Let us first show that $\JJps$ and $\JJwp$ are periodic w.r.t. $d$ with period $2g-2$.

\begin{lemma}\label{tensor}
For any integer $n$, there are natural isomorphisms
$$\JJps\cong \ov{\mathcal J}^{\rm ps}_{d+n(2g-2),g} \hspace{0.2cm} \text{ and } \hspace{0.2cm}
 \JJwp\cong \ov{\mathcal J}^{\rm wp}_{d+n(2g-2),g}
$$
of categories fibered in groupoids.
\end{lemma}

\begin{proof}
Note that a line bundle $L$ on a quasi-wp-stable curve $X$ is properly balanced if and only if $L\otimes\omega_X^n$ is properly balanced; moreover an elliptic tail $F$ of $X$ is special with respect to $L$ if and only if $F$ is special with respect to $L\otimes\omega_X^n$.
The required isomorphisms will then consist in associating to any section $(f:\mathcal X\to S,\mathcal L)\in  \JJps(S)$ (resp. $ \JJwp(S)$) the section $(f:\mathcal X\to S,\mathcal L\otimes
\omega_f^n)\in \ov{\mathcal J}^{\rm ps}_{d+n(2g-2),g}(S)$ (resp. $\ov{\mathcal J}^{\rm ps}_{d+n(2g-2),g}(S)$),
where by $\omega_f$ we denote the relative dualizing sheaf of the morphism $f$.
\end{proof}

Moreover, the stacks $\JJps$ and $\JJwp$ are invariant by changing the sign of degree.

\begin{lemma}\label{L:invdeg}
There are natural isomorphisms
$$\JJps\cong \ov{\mathcal J}^{\rm ps}_{-d,g} \hspace{0.2cm} \text{ and } \hspace{0.2cm}
 \JJwp\cong \ov{\mathcal J}^{\rm wp}_{-d,g},
$$
of categories fibered in groupoids.
\end{lemma}

The proof of this Lemma will be given later (after Theorem \ref{T:new-descr}), when an alternative description of $\JJps$ and $\JJwp$ will be available.

\vspace{0.1cm}

We will now show that if $2(2g-2)<d\leq \frac{7}{2}(2g-2)$ (resp. $\frac{7}{2}(2g-2)< d \leq 4(2g-2)$) then $\JJps$
(resp. $\JJwp$) is isomorphic to the quotient stack $[\wt{H}_d/\GL_{r+1}]$, where
%, as in the proof of Corollary \ref{C:semist-num}, we set
\begin{equation}
\wt{H}_d:=\Hilb_d^{ss,o}:=\{[X\subset \P^r]\in \Hilb_d^{ss}\: :\: X \text{ is connected}\}
\end{equation}
is the \emph{main component} of the Hilbert semi-stable locus and the action of $\GL_{r+1}$ on $\wt{H}_d$ is induced
by the natural action of $\GL_{r+1}$ on $\P^r$.
Note that, according to Fact \ref{HilbtoChow}, $\wt{H}_d$ is contained in the main component $H_d$ of the Chow-semistable locus defined in \eqref{E:Hd2}; moreover, if $d>2(2g-2)$ then $\wt{H}_d=H_d$  if and only if $d\neq \frac{7}{2}(2g-2)$ and
$d\neq 4(2g-2)$ (see Theorems \ref{T:semistable}, \ref{T:semistable3}, \ref{T:semistable4}, \ref{T:semistable5}).

%$\GL_{r+1}$ acts on $H_d$ via its projection onto $\PGL_{r+1}$.
%, the locus of GIT-semi\-sta\-ble points in $Hilb_{\mathbb P^r}^{dt-g+1}$, with $r=d-g$,
 %by projecting onto $\PGL_{r+1}$.
 Recall that, given a scheme $S$,  $[\wt{H}_d/\GL_{r+1}](S)$ consists of $\GL_{r+1}$-principal bundles $\phi:E\to S$ with a $\GL_{r+1}$-equivariant morphism
 $\psi:E\to \wt{H}_d$. Morphisms are given by pullback diagrams which are compatible with the morphism to $\wt{H}_d$.

\begin{thm}\label{geomdesc}
Let $g\geq 3$.
\begin{enumerate}[(i)]
\item \label{geomdesc1} If $2(2g-2)< d\leq \frac{7}{2}(2g-2)$  then $\JJps$ is isomorphic to the quotient stack
    $[\wt{H}_d /\GL_{r+1}]$.
\item \label{geomdesc2} If $\frac{7}{2}(2g-2)\leq  d\leq 4(2g-2)$  then $\JJwp$ is isomorphic to the quotient stack
    $[\wt{H}_d /\GL_{r+1}]$.
\end{enumerate}
\end{thm}
\begin{proof}
To shorten the notation, we set $G:=\GL_{r+1}$.

Let us first prove \eqref{geomdesc1}.
We must show that, for every $k$-scheme $S$, the groupoids $\JJps(S)$ and $[\wt{H}_d/G](S)$ are equivalent. Our proof  goes along the lines of the proof of
 \cite[Thm. 3.1]{melo1}, so we will explain here the main steps and refer to loc. cit. for further details.

Given $(f:\mathcal X\to S,\mathcal L)\in \JJps(S)$, we must produce a principal $G$-bundle $E$ on $S$ and a $G$-equivariant morphism $\psi:E\to \wt{H}_d$. Notice that since $d>2(2g-2)$,
Theorem \ref{bal-pos}(i)  implies that $H^1(\mathcal X_s,\mathcal L_{|\mathcal X_s})=0$ for any geometric fiber $\mathcal X_s$ of $f$, so $f_*(\mathcal L)$ is locally free of rank $r+1=d-g+1$. We
can then consider its frame bundle $E$, which is a principal $G$-bundle: call it $E$.
To find the $G$-equivariant morphism to $\wt{H}_d$, consider the family $\mathcal X_E:=\mathcal X\times_S E$ of quasi-p-stable curves together with
 the pullback of $\mathcal L$ to $\mathcal X_E$, call it $\mathcal L_E$, whose restriction to the geometric fibers is properly balanced.

%$\mathcal X_E$ is a family of quasi-p-stable curves of genus $g$ and $\mathcal L_E$ is properly balanced and relatively very ample by Theorem \ref{bal-pos}\eqref{bal-va}.
By definition of frame bundle, $f_{E*}(\mathcal L_E)$ is isomorphic to $\mathbb{A}_k^{r+1}\times_k E$. Moreover, the line bundle $\mathcal L_E$ is relatively ample by Remark
\ref{R:propbal-ample}; hence it is relatively very ample by Theorem \ref{bal-pos}\eqref{bal-va}. Therefore,
$\mathcal L_E$ gives an embedding over $E$ of $\mathcal X_E$ in $\P^r\times E$.
By the universal property of the Hilbert scheme $\Hilb_d$, this family determines a map $\psi:E\to \Hilb_d$ whose image is contained in $\wt{H}_d$ by Theorems \ref{T:semistable}\eqref{T:semistable2} and \ref{T:semistable3}\eqref{T:semistable3h}. It follows immediately from the construction that $\psi$ is a $G$-equivariant map.

%From the fact that  taking the frame bundle gives an equivalence between the category of vector bundles of rank $r+1$ over $S$ and the category of principal $\GL_{r+1}$-bundles over $S$, it is
%straightforward to check that isomorphisms in  $\JJps(S)$ yield canonically to isomorphisms in $[\wt{H}_d/G](S)$ and vice-versa.

\begin{equation*}
\xymatrix{
{\mathcal X} \ar[d]_{f}& {\mathcal X_E:=\mathcal X\times_SE} \ar[d]^{f_E}  \ar[l] \\
{S} & {E}\ar[l] \ar[r]^{\psi} & {\Hilb_d}\\
%& {f_*\mathcal L} \ar[ul]
}
\end{equation*}

Let us check that isomorphisms in  $\JJps(S)$ lead canonically to isomorphisms in $[\wt{H}_d/G](S)$.
Consider an isomorphism between two pairs $(f:\mathcal X\to S,
\mathcal L)$ and $(f':\mathcal X'\to S,\mathcal
L')$ , i.e.,
%
%An isomorphism consists of
an isomorphism $h:\mathcal X\to
\mathcal X'$ over $S$ and an isomorphism of line bundles
$\mathcal L\stackrel{\cong}{\to} h^*\mathcal L'$.
%\begin{equation*}
%\xymatrix{
%{\mathcal X} \ar[rr]^{h} \ar[dr]_{f}& & {\mathcal X'} \ar[dl]^{f'}\\
%& {S}
%}
%\end{equation*}
Since $f'h=f$, we get a unique isomorphism between the vector bundles  $f_*(\mathcal L)$ and
$f'_*(\mathcal L')$.
% as follows $$f_*(\mathcal L)\cong f_*(h^*\mathcal L') \cong f_*' ( h_*(h^*\mathcal L')))\cong f_*' (\mathcal L' ).$$
Since taking the frame bundle gives an equivalence between the category of vector bundles of rank $r+1$ over $S$ and the category of principal $G$-bundles over $S$, the isomorphism $f_*(\mathcal L)\stackrel{\cong}{\rightarrow} f_*'(\mathcal L')$ leads to a unique isomorphism between their frame bundles, call them $E$ and $E'$ respectively. It is clear that this isomorphism is compatible with the $G$-equivariant morphisms $\psi:E\to \wt{H}_d$ and $\psi': E' \to \wt{H}_d$.
% because they are determined by the induced curves $\mathcal X_E$ and $\mathcal X'_{E'}$ embedded in $\P^r$ by $\mathcal L_E$ and $\mathcal L'_{E'}$.

Conversely, given a section $(\phi:E\to S,\psi:E\to \wt{H}_d)$ of $[\wt{H}_d/G]$ over a $k$-scheme $S$, let us construct a family of quasi-$p$-stable curves of genus $g$ over $S$ and a line bundle whose
restriction to the geometric fibers is properly balanced of degree $d$.

Let $\mathcal C_d$ be the restriction to $\wt{H}_d$ of the universal family on $\Hilb_d$. By Theorem \ref{T:semistable}\eqref{T:semistable2}, the pullback of $\mathcal C_d$ by $\psi$ gives a family $
\mathcal C_E$ on $E$ of quasi-$p$-stable curves of genus $g$ and a line bundle $\mathcal L_E$ on $\mathcal C_E$ whose restriction to the geometric fibers is properly balanced.
%which embeds $\mathcal C_E$ as a family of curves in $\P^r$, so $\mathcal L_E$ is properly balanced.
As $\psi$ is $G$-invariant and $\phi$ is a $G$-bundle, the family $\mathcal C_E$ descends to a family $\mathcal C_S$ over $S$, where $\mathcal C_S=\mathcal C_E/G$. In fact, since $\mathcal
C_E$ is flat over $E$ and $E$ is faithfully flat over $S$, $\mathcal C_S$ is flat over $S$ too.
% (see \cite{ega4}, Proposition 2.5.1).

Now, since $G=\GL_{r+1}$,  the action of $G$ on $\mathcal C_d$ is naturally linearized.
% (see \cite{Cap94}, 1.4)
Therefore, the action of $G$ on $E$ can also be linearized to an action on $\mathcal L_E$,  yielding descent data for $\mathcal L_E$.
% (\cite{sga8}, Proposition 7.8). .
Since $\mathcal L_E$ is relatively very ample and $\phi$ is a principal $G$-bundle, a standard descent argument
% (see the proof of Theorem 3.1 in \cite{melo1})
shows that $\mathcal L_E$ descends to a relatively very ample line bundle on $\mathcal C_S$, call it $\mathcal L_S$, whose restriction to the geometric fibers of $\mathcal C_S\to S$
is properly balanced by construction.

It is straightforward to check that an isomorphism on $[\wt{H}_d/G](S)$ leads to an unique isomorphism in $\JJps(S)$.

We leave to the reader the task of checking that the two functors between the groupoids $[\wt{H}_d/G](S)$ and $\JJps(S)$ that we have constructed are one the inverse of the other, which
concludes the proof of part \eqref{geomdesc1}.

The proof of part \eqref{geomdesc2} proceeds along the same lines using Theorems \ref{T:semistable4} and \ref{T:semistable5}\eqref{T:semistable5h}.
\end{proof}

%An immediate consequence of Theorem, we immediately get that

From Theorem \ref{geomdesc} and Lemmas \ref{tensor} and \ref{L:invdeg}, we deduce the following consequences for $\JJps$ and $\JJwp$.

\begin{thm}\label{stackprop}
Let $g\geq 3$ and $d$ any integer.
\begin{enumerate}[(i)]
 \item \label{stackprop1} $\JJps$ is a smooth and irreducible universally closed Artin stack of finite type over $k$ and of dimension $4g-4$, endowed with a universally closed
morphism $\Psi^{\rm ps}$ onto the moduli stack of p-stable curves $\MMgp$.
\item  \label{stackprop2} $\JJwp$ is a smooth and irreducible universally closed Artin stack of finite type over $k$ and of dimension $4g-4$, endowed with a universally closed
morphism $\Psi^{\rm wp}$ onto the moduli stack of wp-stable curves $\MMgwp$.
\end{enumerate}
\end{thm}
\begin{proof}
Let us first prove part \eqref{stackprop1}. Using Lemma \ref{tensor}, we can assume that $2(2g-2)<d\leq \frac{7}{2}(2g-2)$ and hence that $\JJps\cong [\wt{H}_d/\GL_{r+1}]$ by Theorem \ref{geomdesc}\eqref{geomdesc1}.
The fact that $\JJps$ is a universally closed Artin stack of finite type over $k$ follows from Theorem \ref{geomdesc} and general properties of stacks coming from GIT problems.
$\JJps$ is smooth and irreducible since $\wt{H}_d\subseteq H_d$ is smooth by Theorem \ref{T:comp-Pic}\eqref{T:comp-Pic4} and irreducible by Proposition \ref{P:irreduci}.
Using again Theorem  \ref{T:comp-Pic}\eqref{T:comp-Pic4}, we can compute the dimension of $\JJps$ as follows:
$$\dim  \JJps=\dim \wt{H}_d-\dim \GL_{r+1}=r(r+2)+4g-3-(r+1)^2=4g-4.$$
%from the fact that $\JJps$ is the quotient stack of an open and closed subset of the locus of semistable points
%with respect to the action of a reductive group on a projective variety.
Now, given $(f:\mathcal X\to S,\mathcal L)\in\JJps(S)$, we get an element of $\ov{\mathcal M}_g^{\rm \ps}(S)$ by forgetting $\mathcal L$ and by considering the p-stable reduction
$\ps(f):\ps(\cX)\to S$ of $f$ (see Definition \ref{D:p-stab}). This defines a  morphism of stacks $\Psi^{\rm ps}:\JJps\to \MMgp$, which is universally closed since
$\JJps$ is.

Let us now prove part \eqref{stackprop2}. Using Lemmas \ref{tensor} and  \ref{L:invdeg}, we can assume that $\frac{7}{2}(2g-2)< d \leq 4(2g-2)$ and hence that $\JJwp\cong [\wt{H}_d/\GL_{r+1}]$ by
Theorem \ref{geomdesc}\eqref{geomdesc2}. Now, the proof proceeds as in part \eqref{stackprop1}. Note that the
morphism $\Psi^{\rm wp}:\JJwp\to \MMgwp$ send $(f:\mathcal X\to S,\mathcal L)\in\JJwp(S)$ into the wp-stable reduction
$\wps(f):\wps(\cX)\to S$ of $f$ (see Proposition \ref{P:wp-stab}).

\end{proof}

Note that $\Gm$ acts on $\JJps$ (resp. $\JJwp$) by scalar multiplication on the line bundles and leaving the curves fixed. Thus, $\Gm$ is contained in the stabilizers of any section of $\JJps$ (resp. $\JJwp$). This implies that $\JJps$ (resp. $\JJwp$) are never DM (= Deligne-Mumford)  stacks. However, we can quotient out $\JJps$ (resp. $\JJwp$) by the action of $\Gm$ using the rigidification procedure
defined by Abramovich, Corti and Vistoli in \cite{acv}: denote the rigidified stack by $\JJps\fatslash \Gm$
(resp. $\JJwp\fatslash \Gm$).

From the modular description of $\JJps$ (resp. $\JJwp$) it follows that
the stack $\JJps\fatslash \Gm$ (resp. $\JJwp\fatslash \Gm$) is the stackification of the prestack whose sections over a scheme $S$ are given by pairs $(f : \mathcal X \to S, \mathcal L)\in \JJps$ (resp. $\JJwp$)
and whose arrows between two such pairs are given by a cartesian diagram
\begin{equation*}
\xymatrix{
\ar@{}[dr]|{\square}{\mathcal X} \ar[d]_{f} \ar[r]^{h} & {\mathcal X'} \ar[d]^{f'} \\
{S} \ar[r] & {S'}
}
\end{equation*}
together with an isomorphism $\cL\stackrel{\cong}{\to} h^*\cL'\otimes f^*M$, for some $M\in\Pic  (S)$. We refer to \cite[Sec. 4]{melo1} for more details.

From Theorem \ref{geomdesc} it follows that
$\JJps\fatslash  \Gm$ (resp. $\JJwp\fatslash \Gm$) is isomorphic to the quotient stack $[\wt{H}_d/\PGL_{r+1}]$ if $2(2g-2)<d\leq\frac{7}{2}(2g-2)$ (resp. if $\frac{7}{2}(2g-2)< d\leq 4(2g-2)$).
Note that, using Theorem \ref{stackprop}, we get
$$\begin{sis}
& \dim \JJps\fatslash  \Gm= \dim \JJps+1=4g-3,\\
& \dim \JJwp\fatslash  \Gm= \dim \JJwp+1=4g-3.
\end{sis}
$$
Moreover,  the morphisms $\Psi^{\rm ps}:\JJps\to \MMgp$ and $\Psi^{\rm wp}:\JJwp\to \MMgwp$ of Theorem \ref{stackprop} factor as
\begin{equation}\label{E:factrig}
\begin{sis}
& \Psi^{\rm ps}: \JJps \to \JJps\fatslash \Gm \stackrel{\wh{\Psi}^{\rm ps}}{\longrightarrow} \MMgp,\\
& \Psi^{\rm wp}: \JJwp \to \JJwp\fatslash \Gm \stackrel{\wh{\Psi}^{\rm wp}}{\longrightarrow} \MMgwp,
\end{sis}
\end{equation}

\noindent We can now determine when the stacks $\JJps\fatslash \Gm$ and $\JJwp\fatslash \Gm$ are DM-stacks.

\begin{prop}\label{P:DM-stack}
Let $g\geq 3$ and $d$ be any integers.
\begin{enumerate}
\item \label{P:DM-stackA} The following conditions are equivalent:
\begin{enumerate}[(i)]
\item \label{P:DM-stackA1} $\gcd(d+1-g,2g-2)=1$;
\item \label{P:DM-stackA2} For any  $d'\equiv \pm d \mod 2g-2$ with $2(2g-2)<d'\leq \frac{7}{2}(2g-2)$, the GIT quotient $\wt{H}_{d'}/\PGL_{r+1}$ is geometric, i.e., there are no strictly semistable points;
\item \label{P:DM-stackA3} The stack $\JJps\fatslash \Gm$  is a DM-stack;
\item \label{P:DM-stackA4} The stack $\JJps\fatslash \Gm$  is proper;
\item \label{P:DM-stackA5} The morphism $\wh{\Psi}^{\rm ps}: \JJps\fatslash \Gm\to \MMgp$  is representable.
\end{enumerate}
\item \label{P:DM-stackB} The following conditions are equivalent:
\begin{enumerate}[(i)]
\item \label{P:DM-stackB1} $\gcd(d+1-g,2g-2)=1$;
\item \label{P:DM-stackB2} For any  $d'\equiv \pm d \mod 2g-2$ with
$\frac{7}{2}(2g-2)<d'\leq 4(2g-2)$, the GIT quotient $\wt{H}_{d'}/\PGL_{r+1}$ is geometric, i.e., there are no strictly semistable points;
\item \label{P:DM-stackB3} The stack $\JJwp\fatslash \Gm$ is a DM-stack;
\item \label{P:DM-stackB4} The stack  $\JJwp\fatslash \Gm$ is proper;
\item \label{P:DM-stackB5} The morphism $\wh{\Psi}^{\rm wp}: \JJwp\fatslash \Gm\to \MMgwp$ is representable.
\end{enumerate}
\end{enumerate}
\end{prop}
\begin{proof}
Let us first prove part \eqref{P:DM-stackA}.

 \eqref{P:DM-stackA1} $\Longleftrightarrow$ \eqref{P:DM-stackA2}: the GIT quotient $\wt{H}_{d'}/\PGL_{r+1}$ is geometric if and only if every Hilbert polystable point is also Hilbert stable.
 From Corollaries \ref{C:polystable}\eqref{C:polystable2}, \ref{C:stable}\eqref{C:stable2}, \ref{C:polystable3}\eqref{C:polystable3h} and \ref{C:stable3}\eqref{C:stable3h}, this happens if and only if, given a quasi-p-stable curve $X$ of genus $g$ and a line bundle $L$ on $X$ of degree $d'$, $L$ is stably balanced whenever it is strictly balanced. Lemma \ref{L:coprime} says that this happens  precisely when
 $\gcd(d'+1-g,2g-2)=1$. We conclude since $\gcd(d+1-g,2g-2)=\gcd(d'+1-g, 2g-2)$ for any $d\equiv \pm d'\mod 2g-2$.

% Recalling Definition \ref{D:balanced}, it is easy to see that this occurs if and only if, given a quasi-p-stable curve $X$ of genus $g$, any proper connected subcurve $Y\subset X$ such that
 %$$m_Y=\frac{d'}{2g-2}\deg_Y\omega_X-\frac{k_Y}{2}\in \bZ, $$
%is either an exceptional component or the complementary subcurve of an exceptional component.
% Now, the combinatorial proof of \cite[Lemma  6.3]{Cap} shows that this happens precisely when $\gcd(d'+1-g,2g-2)=1$. We conclude since $\gcd(d+1-g,2g-2)=\gcd(d'+1-g, 2g-2)$ for any $d\equiv \pm d'\mod 2g-2$.

For the remainder of the proof, using Lemma \ref{tensor}, we can (and will) assume that $2(2g-2)<d\leq \frac{7}{2}(2g-2)$.

Let us now show that the conditions \eqref{P:DM-stackA2}, \eqref{P:DM-stackA3} and \eqref{P:DM-stackA5} are equivalent.
%Using Lemma \ref{tensor}, we can assume that $2(2g-2)<d<4(2g-2)$.
From Theorem \ref{T:auto-grp} and its proof, we get that for any quasi-p-stable curve $X$ of genus $g\geq 3$ and any
properly balanced line bundle $L$ on $X$
%$[X\subset \P^r]\in \wt{H}_d$, if we set $L:=\OO_X(1)$, then
we have an exact sequence
\begin{equation}\label{E:aut-seq}
0\to \Gm^{\gamma(\w{X})-1}\to \ov{\Aut(X,L)}
%\cong \Stab_{\PGL_{r+1}}([X\subset \P^r])
\to \Aut(\ps(X)),
\end{equation}
where $\gamma(\w{X})$ denotes, as usual, the connected components of the non-exceptional subcurve $\w{X}$ of $X$.
Note that $\ov{\Aut(X,L)}$ is the automorphism group of $(X,L)\in (\JJps\fatslash \Gm)(k)$ by the definition of the $\Gm$-rigidification.

We claim that each of the conditions \eqref{P:DM-stackA2}, \eqref{P:DM-stackA3} and \eqref{P:DM-stackA5} is equivalent to the condition
\begin{equation*}
\gamma(\w{X})=1 \text{ for any } [X\subset \P^r]\in \wt{H}_d \text{ or, equivalently, for any } (X,L)\in (\JJps\fatslash \Gm)(k). \tag{*}
\end{equation*}
Indeed:
\begin{itemize}
\item  Condition \eqref{P:DM-stackA2} is equivalent to (*) by Lemma \ref{compa-bal}.
\item Condition \eqref{P:DM-stackA3} implies (*) because the geometric points of a DM-stack have a finite automorphism group scheme. Conversely, if (*) holds then
$\ov{\Aut(X,L)}\subset \Aut(\ps(X))$, which is a finite and reduced group scheme since $\ov{\mathcal M}_g^{\rm ps}$ is a DM-stack if $g\geq 3$.
Therefore, also $\ov{\Aut(X,L)}$ is a finite and reduced group scheme, which implies that $\JJps\fatslash \Gm$ is a DM-stack.
\item  Condition \eqref{P:DM-stackA5} is equivalent to the injectivity of the map $\ov{\Aut(X,L)}\to \Aut(\ps(X))$ for any $(X,L)\in (\JJps\fatslash \Gm)(k)$.
This is equivalent to condition (*) by the exact sequence \eqref{E:aut-seq}.
\end{itemize}

\eqref{P:DM-stackA2} $\Longrightarrow$ \eqref{P:DM-stackA4}: this follows from the well-known fact that the quotient stack associated to a geometric projective GIT quotient is a proper
stack.

\eqref{P:DM-stackA4} $\Longrightarrow$ \eqref{P:DM-stackA2}: the automorphism group schemes of the geometric points of a proper stack are complete group schemes.
From \eqref{E:aut-seq}, this is only possible if $\gamma(\w{X})=1$ for any  $(X,L)\in (\JJps\fatslash \Gm)(k)$, or equivalently if condition (*) is satisfied.
This implies that \eqref{P:DM-stackA2} holds by what proved above.

Let us now prove part \eqref{P:DM-stackB}.

\eqref{P:DM-stackB1} $\Longleftrightarrow$ \eqref{P:DM-stackB2}:  the proof is similar to the proof of the equivalence
\eqref{P:DM-stackA1} $\Longleftrightarrow$ \eqref{P:DM-stackA2}, using Corollaries \ref{C:polystable4}, \ref{C:stable4}, \ref{C:polystable5}\eqref{C:polystable5h}, \ref{C:stable5}\eqref{C:polystable5h} and Lemma \ref{L:coprime}.

For the remainder of the proof, using Lemmas \ref{tensor} and \ref{L:invdeg}, we can (and will) assume that $\frac{7}{2}(2g-2)< d\leq 4(2g-2)$.

Note that for any quasi-wp-stable curve $X$ of genus $g\geq 3$ and any properly balanced line bundle $L$ on $X$ such that $X$ does not have tacnodes nor special elliptic tails with respect to $L$, Theorem \ref{T:auto-grp} and its proof
provides an exact sequence
\begin{equation}\label{E:aut-seq2}
0\to \Gm^{\gamma(\w{X})-1}\to \ov{\Aut(X,L)} \to \Aut(\wps(X)).
\end{equation}
Now, the equivalences \eqref{P:DM-stackB2} $\Longleftrightarrow$ \eqref{P:DM-stackB3} $\Longleftrightarrow$  \eqref{P:DM-stackB4} $\Longleftrightarrow$ \eqref{P:DM-stackB5} are proved as in part \eqref{P:DM-stackA} using
\eqref{E:aut-seq2} instead of \eqref{E:aut-seq}.

\end{proof}

\begin{rmk}
Notice that even if the existence of strictly semistable points in $\wt{H}_d$ for $2(2g-2)<d\leq \frac{7}{2}(2g-2)$
(resp. $\frac{7}{2}(2g-2)< d\leq 4(2g-2)$) prevents $\JJps\fatslash \Gm$ (resp. $\JJwp\fatslash \Gm$) from being separated
when $\gcd(d+1-g,2g-2)\neq 1$, the fact that $\JJps\fatslash \Gm$ and $\JJwp\fatslash \Gm$
  can be realized as a GIT quotients imply that their non-separatedness is, in some sense, quite mild. Indeed, according to the recent work of Alper, Smyth and van der Wick in \cite{ASvdW}, we have that the stacks $\JJps\fatslash \Gm$
  and $\JJwp\fatslash \Gm$ are weakly separated, which roughly means that sections of $\JJps\fatslash \Gm$ (resp.
  of $\JJwp\fatslash \Gm$)  over a punctured disc have  unique completions that are closed in $\JJps\fatslash \Gm$
  (resp. $\JJwp\fatslash \Gm$); see \cite[Definition 2.1]{ASvdW} for the precise statement. Since both
$\JJps\fatslash \Gm$ and $\JJwp\fatslash \Gm$ are also universally closed, then according to loc. cit. we get that they are weakly proper.
A similar argument implies that the morphisms $\wh{\Psi}^{\rm ps}: \JJps\fatslash \Gm\to \MMgp$
  and $\wh{\Psi}^{\rm wp}: \JJwp\fatslash \Gm\to \MMgwp$ are weakly proper.

\end{rmk}

\subsection{Existence of moduli spaces for $\JJps$ and $\JJwp$}\label{S:2comp-var}

The aim of this subsection is to define (adequate or good) moduli spaces for the stacks $\JJps$ and $\JJwp$.

We start by observing that, since from Theorem \ref{geomdesc} above we have that, for $2(2g-2)<d\leq \frac{7}{2}(2g-2)$ (resp. $\frac{7}{2}(2g-2)<d\leq 4(2g-2)$), the stack  $\JJps$ (resp. $\JJwp$)
is isomorphic to the quotient stack $[\wt{H}_d/\GL_{r+1}]$; moreover, there are natural morphisms
\begin{equation}\label{E:mod-space}
\JJps\to \ov{Q}_{d,g}^h:=\wt{H}_d/\!\!/\GL_{r+1} \text{ for any } 2(2g-2)<d\leq \frac{7}{2}(2g-2),
\end{equation}
\begin{equation}\label{E:mod-spacebis}
\JJwp\to \ov{Q}_{d,g}^h:=\wt{H}_d/\!\!/\GL_{r+1} \text{ for any } \frac{7}{2}(2g-2)<d\leq 4(2g-2).
\end{equation}
From the work of Alper (see \cite{alper} and \cite{alper2}), we deduce that the morphism \eqref{E:mod-space} (resp. \eqref{E:mod-spacebis}) realizes $\ov{Q}_{d,g}^h$ as
 the \emph{adequate} moduli space of $\JJps$ (resp. $\JJwp$) and even as its \emph{good} moduli space if  the characteristic of our base field $k$ is equal to zero or bigger than the order
 of the automorphism group of every p-stable (rep. wp-stable) curve of genus $g$ (because in this case, all the stabilizers are linearly reductive subgroups of $\GL_{r+1}$,
 as it follows from Lemma \ref{L:aut-stab} and the proof of Theorem \ref{T:auto-grp}).
We do not recall here the definition of an adequate or a good moduli space (we refer to \cite{alper} and \cite{alper2} for details). We limit ourselves to point out some consequences
of the fact that \eqref{E:mod-space} and \eqref{E:mod-spacebis} is an adequate moduli space, namely:
\begin{itemize}
\item The morphisms \eqref{E:mod-space} and \eqref{E:mod-spacebis} are surjective and universally closed (see \cite[Thm. 5.3.1]{alper2});
\item The morphism  \eqref{E:mod-space} (resp. \eqref{E:mod-spacebis}) is universal for morphisms from $\JJps$ (resp. $\JJwp$) to locally separated algebraic spaces (see \cite[Thm. 7.2.1]{alper2});
\item For any algebraically closed field $k'$ containing $k$, the morphisms \eqref{E:mod-space} and \eqref{E:mod-spacebis} induce bijections
$$\JJps( k')_{/\sim} \stackrel{\cong}{\longrightarrow} \ov{Q}_{d,g}^h(k') \: \text{ if }�\: 2(2g-2)<d\leq \frac{7}{2}(2g-2),$$
$$\JJwp( k')_{/\sim} \stackrel{\cong}{\longrightarrow} \ov{Q}_{d,g}^h(k') \: \text{ if }�\: \frac{7}{2}(2g-2)<d\leq 4(2g-2),$$
where we say that two points $x_1,x_2\in \JJps( k')$ (resp. $\JJwp(k')$) are equivalent, and we write $x_1\sim x_2$, if $\ov{\{x_1\}}\cap \ov{\{x_2\}}\neq \emptyset$ in $\JJps\times_k k'$
(resp. $\JJwp\times_k k'$); see \cite[Thm. 5.3.1]{alper2}.
\end{itemize}

\noindent Moreover, if the GIT-quotient is geometric, which occurs if and only if
$\gcd(d-g+1,2g-2)=1$ by Proposition \ref{P:DM-stack},
then it follows from the work of Keel-Mori (see \cite{KeMo}) that actually $\ov{Q}_{d,g}^h$ is the \emph{coarse} moduli space for $\JJps$ (resp. $\JJwp$), which means that the morphism
\eqref{E:mod-space} (resp. \eqref{E:mod-spacebis}) is universal for morphisms of $\JJps$ (resp. $\JJwp$) into algebraic spaces and moreover that \eqref{E:mod-space}
(resp. \eqref{E:mod-spacebis}) induces bijections
$$\JJps(k')\stackrel{\cong}{\longrightarrow} \ov{Q}_{d,g}^h(k') \: (\text{resp. }�\JJwp(k')\stackrel{\cong}{\longrightarrow} \ov{Q}_{d,g}^h(k'))$$
for any algebraically close field $k'$ containing $k$.

% (and for $\JJps\fatslash \Gm$).

%In the general case it follows from \cite[Thm. 4.26]{alper} that the  surjective morphism  (\ref{coarse}) is still  universal for morphisms from $\JJps$ to schemes and induces bijections
%(for any $k'$ as before)
%$$\JJps( k')_{/\sim} \stackrel{\cong}{\longrightarrow} \ov Q_{d,g}(k')$$
%where we say that two points $x_1,x_2\in \JJps( k')$ are equivalent, and we write $x_1\sim x_2$, if $\ov{\{x_1\}}\cap \ov{\{x_2\}}\neq \emptyset$ in $\JJps\times_kk'$
%(see \cite[Thm. 5.3.1]{alper2}).

From the above universal properties of the morphism \eqref{E:mod-space}, it follows that if $2(2g-2)<d, d'\leq \frac{7}{2}(2g-1)$ are such that $\JJps\cong \ov{\mathcal J}^{\rm ps}_{d',g}$ then
$\ov{Q}_{d,g}^h\cong \ov{Q}_{d',g}^h$. Similarly, if $\frac{7}{2}(2g-2)<d, d'\leq 4(2g-1)$ are such that $\JJwp\cong \ov{\mathcal J}^{\rm wp}_{d',g}$ then
$\ov{Q}_{d,g}^h\cong \ov{Q}_{d',g}^h$. By using this fact together with Lemmas \ref{tensor} and \ref{L:invdeg}, the following definition is well-posed.

\begin{defi}\label{D:comp-Pic}
Fix $d\in \bZ$ and $g\geq 3$.
\begin{enumerate}[(i)]
\item Set $\Jps:=\ov{Q}_{d',g}^h=\wt{H}_{d'}/\!\!/\GL_{r+1}$ for any $d'\equiv \pm d \mod 2g-2$ such that $2(2g-2)< d'\leq \frac{7}{2} (2g-2)$.
\item Set $\Jwp:=\ov{Q}_{d',g}^h=\wt{H}_{d'}/\!\!/\GL_{r+1}$ for any $d'\equiv \pm d \mod 2g-2$ such that $\frac{7}{2}(2g-2)< d'\leq 4 (2g-2)$.
\end{enumerate}
\end{defi}
Note that for any $d\in \bZ$, we have natural morphisms
\begin{equation}\label{E:mod-space2}
\JJps\to \Jps \: \text{ and }�\: \JJwp\to \Jwp
\end{equation}
which are adequate moduli spaces in general and coarse moduli spaces if (and only if) $\gcd(d-g+1,2g-2)=1$.

The projective varieties $\Jps$ and $\Jwp$ are two compactifications of the universal Jacobian variety $J_{d,g}$.
We collect some of their properties  in the following theorem.

\begin{thm}\label{T:prop-Pic}
Let $g\geq 3$ and $d\in \bZ$.
\begin{enumerate}
\item \label{T:prop-Pic-ps} The variety $\Jps$ satisfies the following properties:
\begin{enumerate}[(i)]
\item \label{T:prop-Pic-ps1}
$\Jps$ is a normal integral projective variety  of dimension $4g-3$ containing $J_{d,g}$ as a dense open subset.
 Moreover, if ${\rm char}(k)=0$, then $\Jps$ has rational singularities, hence it is Cohen-Macauly.
\item \label{T:prop-Pic-ps2} There exists a surjective map $\Phi^{\rm ps}: \Jps\to \Mgp$ whose geometric fibers are equidimensional of dimension $g$.
Moreover,
%$(\Phi^{\rm ps})^{-1}(C)\cong \Pic^d(C)$ for every geometric point $C\in M_g^o\subset  \Mgp$ and
if ${\rm char}(k)=0$, then $\Phi^{\rm ps}$ is flat over the smooth locus of $\Mgp$.
\item \label{T:prop-Pic-ps3} The $k$-points of $\Jps$ are in natural bijection with isomorphism classes of pairs $(X,L)$ where $X$ is a quasi-p-stable curve of genus $g$ and $L$ is a
strictly balanced line bundle of degree $d$ on $X$.
\end{enumerate}
\item \label{T:prop-Pic-wp} The variety $\Jwp$ satisfies the following properties:
\begin{enumerate}[(i)]
\item \label{T:prop-Pic-wp1}
$\Jwp$ is a normal irreducible projective variety  of dimension $4g-3$ containing $J_{d,g}$ as a dense open subset.
Moreover, if ${\rm char}(k)=0$, then $\Jwp$ has rational singularities, hence it is Cohen-Macauly.
\item \label{T:prop-Pic-wp2} There exists a surjective map $\Phi^{\rm wp}: \Jwp\to \Mgp$ whose geometric fiber over a p-stable curve $X$ has dimension equal to the sum of $g$
and the number of cusps of $X$.
%Moreover,  $(\Phi^{\rm wp})^{-1}(C)\cong \Pic^d(C)$ for every geometric point $C\in M_g^o\subset  \Mgp$.
\item \label{T:prop-Pic-wp3} The $k$-points of $\Jwp$ are in natural bijection with isomorphism classes of pairs $(X,L)$ where $X$ is a quasi-wp-stable curve of genus $g$ without tacnodes
and $L$ is a strictly balanced line bundle of degree $d$ on $X$ such that $X$ does not have special elliptic tails with respect to $L$.
\end{enumerate}
\end{enumerate}
\end{thm}
\begin{proof}
Let us first prove \eqref{T:prop-Pic-ps}.  Clearly, the above properties are preserved by the isomorphisms  of  Lemmas \ref{tensor} and \ref{L:invdeg}.
Therefore, we can assume that $2(2g-2)<d\leq \frac{7}{2} (2g-2)$ so that $\Jps=\ov{Q}_{d,g}^h=\wt{H}_d/\GL_{r+1}$ by Definition \ref{D:comp-Pic}.

Part \eqref{T:prop-Pic-ps1}  follows by combining Proposition  \ref{P:sing-GITquot} and Corollary \ref{C:irr-quot}.

Part \eqref{T:prop-Pic-ps2} follows from Theorem \ref{T:comp-Pic}, Propositions \ref{P:irreduci}\eqref{P:irreduci2} and \ref{P:fib3.5}\eqref{P:fib3.5B}.

Part \eqref{T:prop-Pic-ps3} follows from Remark \ref{R:closedpts} together with Corollaries \ref{C:polystable}\eqref{C:polystable2} and \ref{C:polystable3}\eqref{C:polystable3h}.

Let us first prove \eqref{T:prop-Pic-wp}.  Clearly, the above properties are preserved by the isomorphisms  of  Lemmas \ref{tensor} and \ref{L:invdeg}.
Therefore, we can assume that $\frac{7}{2}(2g-2)<d\leq 4 (2g-2)$ so that $\Jwp=\ov{Q}_{d,g}^h=\wt{H}_d/\GL_{r+1}$ by Definition \ref{D:comp-Pic}.

Part \eqref{T:prop-Pic-wp1}  follows by combining Proposition  \ref{P:sing-GITquot} and Corollary \ref{C:irr-quot}.

Part \eqref{T:prop-Pic-wp2} follows from Theorem \ref{T:comp-Pic}, Propositions \ref{P:irreduci}\eqref{P:irreduci3} and \ref{P:fib4}\eqref{P:fib4B}.

Part \eqref{T:prop-Pic-wp3} follows from Remark \ref{R:closedpts} together with Corollaries \ref{C:polystable4} and \ref{C:polystable5}\eqref{C:polystable5h}.

\end{proof}

%We end this subsection with the following
%\begin{question}
%Does there exists a map $\w{T}:\Pdg^{\rm s} \to \Jps$ fitting into the following commutative diagram
%$$\xymatrix{
%\Pdg^{\rm s} \ar[r]^{\w{T}} \ar[d]_{\Phi^s} & \Jps \ar[d]^{\Phi^{ps}} \\
%\Mg \ar[r]^T & \Mgp
%}$$
%where $T$ is the map from Fact \ref{F:mod-curves}\eqref{F:mod-curves2}?
%\end{question}

\subsection{An alternative description of $\JJst$, $\JJps$ and $\JJwp$}\label{S:sheaves}

The aim of this subsection is to provide an alternative description of the stack $\JJst$ (resp. $\JJps$, resp. $\JJwp$) in terms of certain torsion-free rank-$1$ sheaves on stable (resp. p-stable,
resp. wp-stable) curves rather than line bundles on quasi-stable (resp. quasi-p-stable, resp. quasi-wp-stable) curves.
Indeed, the results of this subsection are inspired by the work of Pandharipande in \cite[Sec. 10]{Pan}, where he reinterprets Caporaso's compactified universal Jacobian variety $\Jst$
as the moduli space of  slope-semistable torsion-free, rank-$1$ sheaves of degree $d$ on stable curves of genus $g$, and by the work of  Esteves-Pacini \cite{estpac}, which give
a similar reinterpretation for the  Caporaso's compactified universal Jacobian stack   $\JJst$.
% can be interpreted as the moduli stack of slope-semistable torsion-free, rank-$1$ sheaves of degree $d$ on stable curves of genus $g$.

Let us first introduce the sheaves we will be working with.

\begin{defi}\label{S:sheaves-wp}
Let $X$ be a (reduced) curve and let $I$ be a coherent sheaf on $X$.
\begin{enumerate}[(i)]
\item We say that $I$ is \emph{torsion-free}
%(or \emph{depth $1$} or \emph{of pure dimension})
if the support of $I$ is equal to $X$  and $I$ does not have non-zero subsheaves whose support has dimension zero.
\item We say that $I$ is of \emph{rank-1} if $I$ is invertible on a dense open subset of $X$.
\item The \emph{degree}� of $I$ is equal to $\deg(I):=\chi(I)-\chi(\cO_X)$.
\end{enumerate}
Given a family of curves $f:\cX\to S$, a \emph{relative} �torsion-free rank-$1$ sheaf of degree $d$ is
a coherent sheaf $\cI$ on $\cX$, flat over $S$, such that its restriction $\cI_s$ to every geometric fiber $\cX_s:=f^{-1}(s)$ of $f$
is a torsion-free rank-$1$ sheaf of degree $d$  on $\cX_s$.
\end{defi}

Observe that a torsion-free rank-$1$ sheaf can be non locally-free only at the singular points of $X$. Clearly,
every line bundle on $X$ is a torsion-free, rank-$1$ sheaf on $X$.

For each subcurve $Y$ of $X$, let $I_Y$ be the restriction $I_{|Y}$ of $I$ to $Y$ modulo torsion.
%that is, the image of the natural map
%$$\I_{|Y}\to \sum_{i=1}^m (\I_{|Y})_{\xi_i},$$
%where $\xi_1,\cdots, \xi_m$ are the generic points of Y .
If $I$ is a torsion-free (resp. rank-$1$)
sheaf on $X$, so is $I_Y$ on $Y$.
We let $\deg_Y (I)$ denote the degree of $I_Y$, that is, $\deg_Y(I) := \chi(I_Y )-\chi(\mathcal O_Y)$.

\begin{defi}\label{sheaf-ss-qs}
%\noindent
Let $X$ be a Gorenstein curve of arithmetic genus $g\geq 2$ and $I$ a rank-1 torsion-free sheaf of degree $d$ on $X$.
We say that $I$ is $\omega_X$-\emph{semistable} if,
for every proper subcurve $Z$ of $X$, we have that
\begin{equation}\label{multdeg-sh1}
\deg_Z(I)\geq  d\frac{\deg_Z(\omega_X)}{2g-2}-\frac{k_Z}{2}
\end{equation}
where $k_Z$ denotes, as usual, the length of the scheme-theoretic intersection $Z\cap Z^c$ of $X$.
\end{defi}

\begin{rmk}\label{R:Gor-ss}
Let $X$ be a Gorenstein curve such that $\omega_X$ is ample.
\begin{enumerate}[(i)]
\item \label{R:Gor-ss1} A torsion-free rank-$1$ sheaf $I$ on $X$ is $\omega_X$-semistable in the sense of Definition \ref{sheaf-ss-qs} if and only if it is  slope-semistable with respect to the
polarization $\omega_X$: the proof of this fact for stable curves in
\cite[Sec. 2.9]{CMKV} extends to the general case.

%\begin{rmk}\label{R:cancompJac}
%Let $X$ be a Gorenstein curve such that $\omega_X$ is ample.
\item \label{R:Gor-ss2} Consider the controvariant functor
\begin{equation}\label{E:fun-J}
\ov{\mathcal J}_{d,X} :\SCH\to \SET
\end{equation}
which associates to a scheme $T$ the set of $T$-flat coherent sheaves on $X\times T$ which are rank-1 torsion-free sheaves  and $\omega_X$-semistable on the geometric fibers $X\times\{t\}$ of the second projection morphism $X\times T\to T$.
The functor $\ov{\cJ}_{d,X}$ is co-represented by a projective variety $\ov{\Jac_d}(X)$, called the \emph{canonical compactified Jacobian} of $X$ in degree $d$;
see \cite[Section 2]{CMKV} for a detailed discussion on the different constructions of compactified Jacobians available in the literature.

\end{enumerate}
%The geometric points of $\ov{\Jac_d}(X)$ parametrize $S$-equivalence  classes of rank-1 torsion-free sheaves on $X$ which are $\omega_X$-semistable.
%For the definition of $S$-equivalence of sheaves, we refer the interested reader to \cite{simpson}.

\end{rmk}

\begin{rmk}\label{R:nodext}
Assume that $X$ is a Gorenstein curve such that all its singular points lying on more than one irreducible component are nodes (e.g. $X$ is a wp-stable curve).
Then a torsion-free, rank-$1$ sheaf $I$ is $\omega_X$-semistable if and only if, for any subcurve $Y\subseteq X$, we have that
\begin{equation}\label{E:strong-ss}
d\frac{\deg_Y(\omega_X)}{2g-2}-\frac{k_Y}{2}\leq \deg_Y(I) \leq  d\frac{\deg_Y(\omega_X)}{2g-2}+\frac{k_Y}{2}-|Y\cap Y^c\cap \Sing(I)|,
\end{equation}
where $\Sing(I)$ denotes the set of singular points of $X$, and $I$ is not locally free.

Indeed, under the above assumptions on $X$, we have the exact sequence
 \begin{equation}\label{E:seqI}
0\to I_{Y^c}(-[Y\cap Y^c\setminus \Sing(I)])\to I\to I_Y\to 0.
\end{equation}
From \eqref{E:seqI}, by using that $\deg(I):=\chi(I)-\chi(\cO_X)$ by definition (and the analogous formulas for  $I_{Y}$ and $I_{Y^c}$), the additivity of the Euler characteristic and the formula
$\chi(\cO_X)=\chi(\cO_Y)+\chi(\cO_{Y^c})-|Y\cap Y^c|$, we get
\begin{equation}\label{E:add-deg}
\deg(I)=\deg_Y(I)+\deg_{Y^c}(I)+|Y\cap Y^c\cap \Sing(I)|.
\end{equation}
By substituting \eqref{E:add-deg} in \eqref{multdeg-sh1} for $Y^c$, we get the right inequality in \eqref{E:strong-ss}, q.e.d.

\end{rmk}

Torsion-free, rank-$1$ sheaves on a wp-stable curve $X$ can be described via certain line bundles on quasi-wp-stable models of $X$. First of all, starting with a suitable line bundle on a quasi-wp-stable model of $X$, we obtain a rank-$1$
torsion-free sheaf of the same degree on $X$ by taking the push-forward.

\begin{lemma}\label{L:sh-lb}
Let $X$ be a wp-stable curve. For any set $S\subset X_{\rm sing}$, denote by $\wh{X}_S$ the quasi-wp-stable curve obtained from $X$ by bubbling the nodes and cusps of
$X$ belonging to $S$ and  set $\phi^S:\wh{X}_S\to X$ equal to the wp-stable reduction (as in Proposition \ref{P:wp-stab}).
%\item \label{L:sh-lb1}
Let $L$ be a line bundle on $\wh{X}_S$ such that for every exceptional component $E$ of $\wh{X}_S$ we have that $\deg_E L\in \{-1, 0, 1\}$. Then
\begin{enumerate}
\item \label{L:sh-lb1i} $R^1 \phi^S_*(L)=0$  and $\phi^{S}_*(L)$ is a torsion-free rank-$1$ sheaf on $X$ such that $\deg \: \phi^S_*(L)=\deg \: L$.
\item \label{L:sh-lb1ii} $\phi^S_*(L)$ is $\omega_X$-semistable if and only if $L$ is balanced.
%\item \label{L:sh-lb1iii} There is a natural isomorphism of algebraic groups $\Aut(X,L)\stackrel{\cong}{\to} \Aut(Y,\phi_*(L))$.
\end{enumerate}
%\item \label{L:sh-lb2}
%Let  $I$ be a torsion-free rank-1 sheaf on $X$ and denote by $\Sing(I)\subseteq X_{\rm sing}$ the set of  points of $X$ where $I$ is not locally free. Then
%there exists a a line bundle $L$ on $\wh{X}_{\Sing(I)}$ such that
%\begin{itemize}
%\item $\deg_E L=1$ for all exceptional subcurves $E$ of $\wh{X}_{\Sing(I)}$;
%\item $I=\phi^{\Sing(I)}_*(L)$.
%\end{itemize}
%Moreover,  the restriction of $L$ to the  non-exceptional subcurve of  $\wh{X}_{\Sing(I)}$ (see Definition \ref{D:quasi-wp-stable}) is unique.
 %\end{enumerate}
\end{lemma}
\begin{proof}
%Let us first prove \eqref{L:sh-lb1}.
In order to simplify the notation, set $Y:=\wh{X}_S$ and $\phi:=\phi^S$.
As in Definition \ref{D:quasi-wp-stable}, write $Y=Y_{\rm exc}\cup \wt{Y}$, where $Y_{\rm exc}$ is given by the union of all the exceptional subcurves of $Y$
and $\w{Y}=Y_{\rm exc}^c$ is the   non-exceptional subcurve of $Y$. Let $D_{\rm exc}:=Y_{\rm exc}\cap \w{Y}$, which we can view as a Cartier divisor on both $Y_{\rm exc}$
and $\wt{Y}$. The restrictions of $L$ to $\wt{Y}$ and to $Y_{\rm exc}$ give rise to the following two exact sequences of sheaves:
\begin{equation}\label{E:resseq1}
\begin{sis}
& 0\to L_{|Y_{\rm exc}}(-D_{\rm exc})\to L \to L_{|\wt{Y}}\to 0,\\
& 0\to L_{|\wt{Y}}(-D_{\rm exc})\to L \to L_{|Y_{\rm exc}}\to 0.
\end{sis}
\end{equation}
By taking the push-forward of \eqref{E:resseq1} via $\phi$, we get the two exact sequences of vector spaces
\begin{equation}\label{E:resseq1bis}
\begin{sis}
& 0\to \phi_*( L_{|Y_{\rm exc}}(-D_{\rm exc}))\to \phi_*(L)\to \phi_*(L_{|\wt{Y}}),\\
& R^1\phi_*( L_{|\wt{Y}}(-D_{\rm exc}))\to R^1\phi_*(L)\to R^1\phi_*(L_{|Y_{\rm exc}})\to 0.\\
\end{sis}
\end{equation}
Since the restriction of $\phi$ to $\w{Y}$ is a finite birational morphism onto $X$, the
sheaf  $\phi_*(L_{|\wt{Y}})$ is torsion-free and of rank $1$ on $X$ and $R^1\phi_*(L_{|\w{Y}}(-D_{\rm exc}))=0$.
On the other hand, the sheaves $\phi_*( L_{|Y_{\rm exc}}(-D_{\rm exc}))$ and $R^1\phi_*( L_{|Y_{\rm exc}})$ are torsion sheaves supported on
$\phi(Y_{\rm exc})$. For every exceptional component $E\cong \P^1$ of $Y$, we have that $\deg_E L_{|Y_{\rm exc}}\geq -1$ and
$\deg_E \left(L_{|Y_{\rm exc}}(-D_{\rm exc})\right)=\deg_E L-\deg_E\OO_E(-D_{\rm exc})\leq 1-2=-1$,
which implies that
$$\begin{sis}
& \phi_*( L_{|Y_{\rm exc}}(-D_{\rm exc}))_{\phi(E)}=H^0(E, L_{|Y_{\rm exc}}(-D_{\rm exc}))=0, \\
& R^1\phi_*( L_{|Y_{\rm exc}})_{\phi(E)}=H^1(E, L_{|Y_{\rm exc}})=0. \\
\end{sis}
$$
Therefore, using \eqref{E:resseq1bis}, we deduce that  $\phi_*(L)\subseteq \phi_*(L_{|\wt{Y}})$ is torsion-free and of rank $1$ on $X$ and $R^1\phi_*(L)=0$.
Moreover, we have that $\chi(L)=\chi(\phi_*(L))-\chi(R^1\phi_*(L))=\chi(\phi_*(L))$, which, together with the fact that $Y$ and $X$ have the same arithmetic genus, implies
that $\deg \,L=\deg\,\phi_*(L)$. Part \eqref{L:sh-lb1i} is now proved.

Let us now prove part \eqref{L:sh-lb1ii}.
Assume first that $L$ is properly balanced. Let $Z$ be a subcurve of $X$ and let $\wh{Z}$ be the subcurve of $Y$ obtained from the subcurve $\phi^{-1}(Z)$ by removing the exceptional
subcurves $E\subset \phi^{-1}(Z)$ such that  $E\cap \phi^{-1}(Z)^c\neq \emptyset$ and $\deg_E L=1$.
From the definition of $\wh{W}$, it is easy to check that
\begin{equation}\label{E:hat-curve}
\begin{sis}
& k_{\wh{Z}}=k_{Z}, \\
& p_a(\wh{Z})=p_a(Z).
\end{sis}
\end{equation}

\noindent
\un{CLAIM}: $\deg_{\wh{Z}}(L)=\deg_Z(\phi_*(L))$.

Indeed, first of all, by the projection formula, we get
\begin{equation}\label{E:res-Z}
\phi_*(L_{|\phi^{-1}(Z)})=\phi_*(L\otimes \cO_{\phi^{-1}(Z)})=\phi_*(L\otimes \phi^*(\cO_Z))=\phi_*(L)\otimes \cO_Z=\phi_*(L)_{|Z}.
\end{equation}
Let $\cE$ be the union of the exceptional subcurves of $Y$ contained in $\phi^{-1}(Z)\cap \phi^{-1}(Z^c)$ and set
$\mathring{Z}$ to be equal to the complement of $\cE$ inside $\phi^{-1}(Z)$.
The morphism $\phi: \mathring{Z}\to Z$ is the bubbling of $Z$ at the singular points $S\setminus (Z\cap Z^c)$.  Therefore, by what proved in \eqref{L:sh-lb1i}, we get that
\begin{equation}\label{E:push-Zo}
\begin{sis}
& \phi_*(L_{|\mathring{Z}}) \: \text{�is a torsion-free, rank-$1$ sheaf on }�Z,\\
& R^1\phi_*(L_{|\mathring{Z}})=0.
\end{sis}
\end{equation}
%Clearly, we have that $\mathring{Z}\subseteq \wh{Z}$ and set $\cE^1:=\ov{\wh{Z}\setminus \mathring{Z}}$. The subcurve $\cE^1$ is the union of the exceptional components $E$ of $Y$ contained
%in $\phi^{-1}(Z)\cap \phi^{-1}(Z^c)$ and such that $\deg_E L=-1$ or $0$.
We have the following two  exact sequences of sheaves on $\phi^{-1}(Z):$
\begin{equation*}
\begin{sis}
& 0\to L_{|\cE}(-\cE\cap \mathring{Z}) \to L_{|\phi^{-1}(Z)}\to L_{|\mathring{Z}}\to 0, \\
& 0\to L_{|\mathring{Z}}(-\cE\cap \mathring{Z}) \to L_{|\phi^{-1}(Z)}\to L_{|\cE}\to 0.
\end{sis}
\end{equation*}
By taking the push-forward via $\phi$ and using \eqref{E:push-Zo} and the analogous vanishing $R^1\phi_*(L_{|\mathring{Z}}(-\cE\cap \mathring{Z})) =0$,
we get the following two exact sequence of sheaves
\begin{equation}\label{E:2seq-oZ-wZ}
\begin{sis}
0\to \phi_*(L_{|\cE}(-\cE\cap \mathring{Z})) \to  \phi_*(L_{|\phi^{-1}(Z)})\to \phi_*(L_{|\mathring{Z}}), \\
%\to R^1\phi_*(L_{|\cE}(-\cE\cap \mathring{Z}))\to R^1\phi_*(L_{|\phi^{-1}(Z)}) \to 0, \\
 0 \to R^1\phi_*(L_{|\phi^{-1}(Z)}) \to R^1\phi_*(L_{|\cE})\to 0.
\end{sis}
\end{equation}
The sheaves $\phi_*(L_{|\cE}(-\cE\cap \mathring{Z}))$ and $R^1\phi_*(L_{|\cE})$ are
torsion sheaves supported at $\phi(\cE^1)$ and for any $\P^1\cong E\subseteq \cE$ we get
\begin{equation}\label{E:tor-sh}
\begin{sis}
&  \phi_*(L_{|\cE}(-\cE\cap \mathring{Z}))_{\phi(E)}=H^0(E, L_{|E}(-E\cap \mathring{Z}))=
\begin{cases}
k & \text{ if }�\deg_E L=1, \\
0 & \text{�if }�\deg_E L=-1, 0,
\end{cases}\\
%& R^1\phi_*(L_{|\cE^1}(-\cE^1\cap \mathring{Z}))_{\phi(E)}=H^1(E, L_{|E}(-E\cap \mathring{Z})))=
%\begin{cases}
%k & \text{ if }�\deg_E L=-1, \\
%0 & \text{�if }�\deg_E L=0,
%\end{cases}\\
& R^1\phi_*(L_{|\cE^1})_{\phi(E)}=H^1(E, L_{|E})=0,\\
\end{sis}�
\end{equation}
since $\deg_E L=-1, 0, 1$ and $E$ intersects $\mathring{Z}$ in (exactly) one point.
The first equation in \eqref{E:2seq-oZ-wZ} together with \eqref{E:push-Zo}   imply that $ \phi_*(L_{|\cE}(-\cE\cap \mathring{Z}))$ is the biggest torsion subsheaf of
$ \phi_*(L_{|\phi^{-1}(Z)})$. Taking into account \eqref{E:hat-curve}, we get that �
\begin{equation}\label{E:quot-tor}
\phi_*(L)_Z=\phi_*(L_{|\phi^{-1}(Z)})/\phi_*(L_{|\cE}(-\cE\cap \mathring{Z})).
\end{equation}
In order to compute the degree of $\phi_*(L)_Z$, notice first of all that from the first equation in \eqref{E:tor-sh} it follows that �$\phi_*(L_{|\cE}(-\cE\cap \mathring{Z}))$
is a torsion sheaf of length equal to the number of exceptional components $E\subseteq \cE$ such
that $\deg_E L=1$, which is also equal to $\deg_{\phi^{-1}(Z)}(L)-\deg_{\wh{Z}}(L)$. Moreover, from the second equations in \eqref{E:tor-sh} and in \eqref{E:2seq-oZ-wZ}  it follows that �
$R^1\phi_*(L_{|\cE})=R^1\phi_*(L_{|\phi^{-1}(Z)})=0$ which implies that  $\chi(L_{|\phi^{-1}(Z)})= \chi(\phi_*(L_{|\phi^{-1}(Z)}))$. Now, we can compute the degree of $\phi_*(L)_Z$ using
\eqref{E:quot-tor}:
\begin{align*}\label{E:deg-pZ}
\deg_Z(\phi_*(L))= &\deg \, \phi_*(L)_Z= \chi(\phi_*(L)_Z)-\chi(\cO_Z)=\\
=&\chi(\phi_*(L_{|\phi^{-1}(Z)}))-\chi(\phi_*(L_{|\cE}(-\cE\cap \mathring{Z})))-\chi(\cO_Z)=\\
=&\chi(L_{|\phi^{-1}(Z)})-\deg_{\phi^{-1}(Z)}(L)+\deg_{\wh{Z}}(L)-\chi(\cO_{\phi^{-1}(Z)})=\\
=&\deg_{\wh{Z}}(L),\
\end{align*}
which concludes the proof of the CLAIM.

\vspace{0.1cm}

By using the above CLAIM and \eqref{E:hat-curve},  the basic inequality \eqref{E:basineq-multideg} for $L$ and the subcurve $\wh{Z}\subseteq Y$ translates into the inequality
\eqref{multdeg-sh1} for $\phi_*(L)$ and the subcurve  $Z\subseteq X$; hence $\phi_*(L)$ is $\omega_X$-semistable.

Assume next that $\phi_*(L)$ is $\omega_X$-semistable.
Let $W$ be a connected subcurve of $Y$. We want to compare the degree of $L$ on $W$ with its degree on the subcurve $\wh{\phi(W)}\subseteq \phi^{-1}(\phi(W))$ defined above.
With this aim,  set
\begin{itemize}
\item $\cE_W^0$ to be the collection of the  exceptional subcurves  contained in $W$ but not in $\wh{\phi(W)}$ (or equivalently, contained in $W$, intersecting $\phi^{-1}(\phi(W))^c$
and having degree $1$ with respect to $L$);
\item $\cE_W^1$ to be the collection of the exceptional subcurves contained in $\wh{\phi(W)}\cap \phi^{-1}(\phi(W)^c)$ but not in $W$.
\item $\cE_W^2$ to be the collection of the  exceptional subcurves   contained in $\wh{\phi(W)}\setminus \phi^{-1}(\phi(W)^c)$ but not in $W$.
\end{itemize}
Moreover, set $e_W^i$ to be equal to the cardinality of $\cE_W^i$ (for $i=0,1,2$).
By construction, we have that
\begin{equation}\label{E:diff-W-hW}
\wh{\phi(W)}\coprod \left[\bigcup_{E\in \cE_W^0} E\right]= W\coprod \left[\bigcup_{E\in \cE_W^1} E \right] \coprod \left[\bigcup_{E\in \cE_W^2} E\right].
\end{equation}
Moreover, the degree of $L$ on the exceptional components belonging to $\cE_W^i$ can assume the following values:
\begin{equation}\label{E:deg-E}
\deg_E L=
\begin{cases}
 1 & \text{�if }�E\in \cE_W^0, \\
 -1, 0 & \text{�if }�E\in \cE_W^1, \\
 -1,0, 1 & \text{�if }�E\in \cE_W^2. \\
\end{cases}
\end{equation}
Using \eqref{E:diff-W-hW} and \eqref{E:deg-E}, together with the above CLAIM, we get that
\begin{equation}\label{E:degW-L}
 \deg_{\phi(W)}(\phi_*(L))= \deg_{\wh{\phi(W)}}L\leq \deg_{\wh{\phi(W)}}L+e_W^0\leq \deg_W L+e_W^2.
\end{equation}
Moreover, by the definition of $\cE_W^i$ together with \eqref{E:hat-curve}, it is easily checked that
\begin{equation}\label{E:W-curve}
\begin{sis}
& k_W=k_{\wh{\phi(W)}}+2e_W^2= k_{\phi(W)}+2e_W^2, \\
& p_a(W)=p_a(\wh{\phi(W)})-e_W^2=p_a(\phi(W))-e_W^2.
\end{sis}
\end{equation}
By applying the inequality \eqref{multdeg-sh1} to the sheaf $\phi_*(L)$ �and the subcurve $\phi(W)$ and using \eqref{E:degW-L} and \eqref{E:W-curve}, we get
$$\begin{aligned}
 \deg_{W}(L)\geq  \deg_{\phi(W)}(\phi_*(L))-e_W^2\geq & d\,\frac{2[p_a(\phi(W))-2]+k_{\phi(W)}}{2g-2}-\frac{k_{\phi(W)}}{2}-e_W^2=\\
& = d\,\frac{2[p_a(W)-2]+k_{W}}{2g-2}-\frac{k_{W}}{2},
\end{aligned}$$
which shows that $L$ satisfies the basic inequality \eqref{E:basineq-multideg} with respect to the subcurve $W\subseteq Y$; hence $L$ is balanced.
\end{proof}

We now want to prove that every rank-$1$ torsion-free sheaf $I$ on a quasi-wp-stable curve $X$ is obtained in a unique way from some line bundle on some quasi-wp-stable model of $X$ via the construction of previous Lemma \ref{L:sh-lb}.

First of all, we show how to construct a quasi-stable model of $X$ starting from the sheaf $I$. This is based on the following general construction.
Given a scheme $Z$ and a quasi-coherent sheaf $\mathcal F$ on $Z$, we define (following \cite[Def. 4.1.1]{EGAII}) the projective bundle on $Z$ associated to $\mathcal F$
\begin{equation}\label{projbundle}
\mathbb P_Z(\mathcal F):=\proj(\Sym(\mathcal F))
\end{equation}
where $\Sym(\mathcal F)=\oplus_{n\geq 0}\Sym^n(\mathcal F)$ is the symmetric algebra associated to $\mathcal F$. It comes with a natural projective morphism $\pi:\mathbb P_Z(\mathcal F)\to Z$, called the structure morphism, and a tautological  invertible sheaf $\mathcal O_{\mathbb P_Z(\mathcal F)}(1)$ such that $\pi^*(\mathcal F)\to \mathcal O_{\mathbb P_Z(\mathcal F)}(1)$ is surjective.

%We now generalize Proposition 5.2 in \cite{estpac} to quasi-wp-stable curves.

\begin{lemma}\label{L:ProjWP-stab}
Let $X$ be a curve with only nodal or cuspidal singularities and $I$ a torsion-free rank-1 sheaf on $X$. Set $Y:=\mathbb P_X(I)$ and let $\pi:Y\to X$ be the structure map. Then $Y$ is a pre-wp-stable curve, which is obtained from $X$ by bubbling  (in the sense of Corollary \ref{C:quasi-vs-wp}) the nodes and the cusps of $X$ at which  $I$ is not invertible. In particular, when $X$ is wp-stable, then $Y$ is quasi-wp-stable and $\pi:Y\to X$ is the wp-stable reduction.
\end{lemma}

\begin{proof}
Clearly $\pi$ is an isomorphism over the points where $\mathcal I$ is invertible, so it suffices to analyze $\pi$ on a neighborhood of the nodes and of the cusps of $X$ where $\mathcal I$ is not invertible.
The case of nodes was dealt with in the proof of \cite[Prop. 5.2]{estpac}, we include it here for the reader's convenience.
Let $P$ be a nodal point of $X$, where $\mathcal I$ is not invertible and consider the base change of $\pi$ to the spectrum of the completion of $\mathcal O_X$ at $P$, $$\widehat{\mathcal O}_{X,P}\cong \frac{k[[x,y]]}{(y^2-x^2)}.$$
Since $\mathcal I$ fails to be invertible at $P$, then $\widehat{\mathcal I}_P\cong \mathfrak m_P$ (see \cite[p. 75]{Yos}), where $\mathfrak m_P=(x,y)$ is the maximal ideal of $\widehat{\mathcal O}_{X,P}$. We have the following presentation of the maximal ideal $\mathfrak m_P$ as a $\widehat{\mathcal O}_{X,P}$-module
$$\mathfrak m_P\cong \frac{(\widehat{\mathcal O}_{X,P}).e\oplus (\widehat{\mathcal O}_{X,P}). f}{(xe-yf, ye-xf)}.$$
We then get that
$$\Sym(\mathfrak m_P)\cong \frac{k[x,y,e,f]}{(y^2-x^2,xe-yf,ye-xf)}. $$

So, locally analytically, $Y$ is the subscheme of $\mathbb A^2\times \mathbb P^1$ defined by the equations $\{y^2=x^2, xe=yf, ye=xf\}$, where $x$ and $y$ are the coordinates of $\mathbb A^2$ and $e,f$ are the homogeneous coordinates of $\mathbb P^1$. Thus, $Y$ has two components: one is given by the projective line of equation $x=y=0$ and the other by the smooth curve of equations
$\{ye=xf, e^2=f^2\}$. Moreover, the two components intersect transversally at the two points $((0,0),[1:1])$ and $((0,0),[1:-1])$. From this description, it follows that $Y$, locally analytically around $\pi^{-1}(P)$, is isomorphic to
 the total transform of $X$ in  the blow-up of $\mathbb A^2$ at the origin, equipped with the reduced structure, i.e., to the bubbling of $X$ at the node $P$.

Let us now consider the case of cusps.
Let $P$ be a cuspidal point of $X$ where $\mathcal I$ is not invertible and consider the base change of $\pi$ to the spectrum of the completion of $\mathcal O_X$ at $P$, $$\widehat{\mathcal O}_{X,P}\cong \frac{k[[x,y]]}{(y^2-x^3)}.$$
Since $\mathcal I$ fails to be invertible at $P$, then $\widehat{\mathcal I}_P\cong \mathfrak m_P$ (see \cite[p. 39]{Yos}), where $\mathfrak m_P=(x,y)$ is the maximal ideal of $\widehat{\mathcal O}_{X,P}$.
We have the following presentation of the maximal ideal $\mathfrak m_P$ as a $\widehat{\mathcal O}_{X,P}$-module
$$\mathfrak m_P\cong \frac{(\widehat{\mathcal O}_{X,P}).e\oplus (\widehat{\mathcal O}_{X,P}). f}{(yf-x^2e,ye-xf)}.$$
We then get that
$$\Sym(\mathfrak m_P)\cong \frac{k[x,y,e,f]}{(y^2-x^3, yf-x^2e,ye-xf)}. $$

So, locally analytically, $Y$ is the subscheme of $\mathbb A^2\times \mathbb P^1$ defined by the equations $\{y^2=x^3, yf=x^2e, ye=xf\}$, where $x$ and $y$ are the coordinates of $\mathbb A^2$ and $e,f$ are the homogeneous coordinates of $\mathbb P^1$. So, $Y$ has two components: one is given by the projective line of equation $x=y=0$ and the other by the smooth curve of equation
% $x^3=y^2, ye=xf??$.
$\{ye=xf, xe^2=f^2\}$. Moreover, the two components intersect  at the point $((0,0),[1:0])$, where the projective line is tangent to the second component.
From this description, it follows that $Y$, locally analytically around $\pi^{-1}(P)$, is isomorphic to
 the total transform of $X$ in  the blow-up of $\mathbb A^2$ at the origin, endowed with the reduced structure, i.e., to the bubbling of $X$ at the cusp $P$.

\end{proof}

The construction of Lemma \ref{L:sh-lb} works well in families and it is compatible with the (relative) projective bundle construction of Lemma \ref{L:ProjWP-stab}.

\begin{prop}\label{P:lb-tf}
Let $f:\mathcal X\to S$ be a family of quasi-wp-stable curves and $\mathcal L$ an invertible sheaf on $\mathcal X$ of relative degree $d$  such that $\deg_E(\mathcal L)=1$ for every exceptional component $E$ of every geometric fiber $\mathcal X_s:=f^{ -1}(s)$ of $f$. Let $\wps(f):\wps(\mathcal X)\to S$ be the wp-stable reduction of $f$ and $\phi:\mathcal X\to \wps(\mathcal X)$ the $S$-morphism between $f$ and $\wps(f)$. Then
\begin{enumerate}
 \item \label{P:lb-tf1} $\mathcal I:=\phi_*(\mathcal L)$ is a relative  torsion-free rank-1 sheaf of degree $d$ on $\wps(\mathcal X)/S$ and its formation commutes with base-change.

 \item \label{P:lb-tf2} The canonical map $e:\phi^*\mathcal I\to \mathcal L$ is surjective.
 \item \label{P:lb-tf3} There is an isomorphism $u:\mathcal X\to \mathbb P_{\mathcal X}(\mathcal I)$ over $\wps(\mathcal X)$ such that $u^*\mathcal O_{\mathbb P_{\mathcal X}(\mathcal I)}(1)\cong \mathcal L$.

 \end{enumerate}
\end{prop}

\begin{proof}
The proof is a generalization  \cite[Proposition 5.4]{estpac}.

 Statement \eqref{P:lb-tf1} follows from Lemma \ref{L:sh-lb}\eqref{L:sh-lb1i} above together with  \cite[Cor. 1.5]{knudsen}, which ensures that $\mathcal I$ is flat and its formation commutes with base-change.

To prove statement \eqref{P:lb-tf2}, we will argue as in \cite[Theorem 3.1(3)]{estpac}. First of all, observe that it is enough to prove the surjectivity of $e$ on each geometric fiber $\cX_s=f^{-1}(s)$ of $f$.
Fix then a geometric fiber $\mathcal X_s$ of $f$ and let $E_1,\dots,E_n$ be the exceptional components of $\mathcal X_s$ that are contracted by $\phi_{|\mathcal X_s}$. Denote by $P_1,\dots, P_n$ the points of $\wps(\mathcal X_s)$ that correspond to the image of $E_1,\dots,E_n$ and by $\widetilde{\mathcal X}_s$ the complement of $E_1\cup\dots\cup E_n$ in $\mathcal X_s$. Since the restriction of $e$ to ${\phi^*\mathcal I}_{|\widetilde{\mathcal X}_s}$
%$e_{{|\psi^*\mathcal I}_{|\widetilde{\mathcal X}_s}}:{\psi^*\mathcal I}_{|\widetilde{\mathcal X}_s}\to \mathcal L_{|\wps(\mathcal X_s)\setminus\{P_1\cup \dots \cup P_n\}}$
is an isomorphism, it suffices to check that the restriction of $e$ to ${\phi^*\mathcal I}_{|E_i}$, which is a map $e_i:H^0(P_i,\mathcal I_{|P_i})\otimes \mathcal O_{E_i}\to\mathcal L_{|E_i}$,
  %$e_i:=e_{|{\psi^*\mathcal I}_{|E_i}}:{\psi^*\mathcal I}_{|E_i}\to {\mathcal {L}}_{|P_i}$
is surjective for each $i=1,\dots,n$. Notice that the morphism $e_i$ is the composition of the base-change map
$v_i':\phi_*(\mathcal L)_{|P_i}(=\mathcal I_{|P_i})\to \phi_*(\mathcal L_{|E_i})$ with the evaluation map $v_i'':H^0(E_i,\mathcal L_{|E_i})\otimes \mathcal O_{E_i}\to \mathcal L_{|E_i}$, so it suffices to show that each of them is surjective. The surjectivity of $v_i''$ follows from the fact that, since $\deg \mathcal L_{|E_i}=1$, $\mathcal L_{|E_i}$ is globally generated. In order to show that $v_i'$ is surjective, consider the exact sequence
\begin{equation}\label{E:blow-up}
0\to \mathcal L_{|\widetilde{\mathcal X}_s}(-\sum P_{i,j})\to \mathcal L\to \mathcal L_{|\{E_1\cup\dots \cup E_n\}}\to 0
\end{equation}
where $\phi^{-1}(P_i)=\{P_{i,1}, P_{i,2}\}$, for $i=1,\dots,n$ (notice that $P_{i,1}=P_{i,2}$ in the case when $P_i$ is a cusp). By applying $\phi_*$ to \eqref{E:blow-up} and using the fact that $\phi_{|\widetilde{\mathcal X}_s}$ is a finite map, we get that  $\phi_*(\mathcal L)\to \phi_*(\mathcal L_{|\{E_1\cup\dots \cup E_n\}})$ is surjective. So, for all $i=1,\dots,n$, the map $\phi_*(\mathcal L)\to \phi_*(\mathcal L_{|E_i})$ is surjective and thus $v_i'$ is surjective as well.
% are the points of intersection between
%$h^0(E_i,\mathcal L_{|E_i})$
   %follows from Theorem 3.1(3)in \cite{estpac}, observing that the same proof as in loc. cit. holds in our more general situation.

Let us prove statement \eqref{P:lb-tf3}. The existence of a morphism $u:\mathcal X\to \mathbb P_{\mathcal X}(\mathcal I)$ over $\wps(\mathcal X)$ such that $u^*\mathcal O_{\mathbb P_{\mathcal X}(\mathcal I)}(1)\cong \mathcal L$ follows from statement \eqref{P:lb-tf2}. Now, since both $\mathcal X$ and $\mathbb P_{\mathcal X}(\mathcal I)$ are $S$-flat and the formation of  $\mathcal I$ commutes with base-change, it suffices to prove that $u_s$ is an isomorphism for all geometric fibers $\mathcal X_s$ of $f$. Since $u$ is an $\mathcal X$-morphism, it suffices to show that $u_{|E}$ is an isomorphism for all exceptional component $E\subseteq \mathcal X_s$. But this follows from the fact that $u_{|E}$ is given by the surjection $e_{|E}$, which is the evaluation map of the degree-1 sheaf $\mathcal L_{|E}$.

\end{proof}

Finally, we show that every  torsion-free rank-$1$ sheaf on a wp-stable curve $X$ is the push-forward of a line bundle on a  quasi-wp-stable model of $X$, and that this works also in families.

\begin{prop}\label{P:tf-lb}
Let $f:\mathcal X\to S$ be a family of curves with only nodal or cuspidal singularities  and let $\mathcal I$ be a relative torsion-free, rank-1 sheaf of degree $d$ on $\mathcal X/S$.
Consider the relative projective bundle $g:\mathbb P_{\mathcal X}(\mathcal I)\to S$ associated with $\cI$, which is endowed with a natural projective $S$-morphism $\pi: \P_{\cX}(\cI)\to \cX$. Then
\begin{enumerate}
\item \label{P:tf-lb1} $g:\mathbb P_{\mathcal X}(\mathcal I)\to S$ is a family of pre-wp-stable curves which is obtained from the family $f:\mathcal X\to S$ by fiberwise bubbling  the nodes and the cusps of the geometric fibers of $f$
where the restriction of $\cI$ is not locally free.

In particular, if $f$ if a family of wp-stable curves, then $g$ is a family of quasi-wp-stable curves and $\pi: \P_{\cX}(\cI)\to \cX$ is the wp-stable reduction morphism.

\item \label{P:tf-lb2} The tautological invertible sheaf $\mathcal O_{\mathbb P_{\mathcal X}(\mathcal I)}(1)$ has relative degree $d$ on $\mathbb P_{\mathcal X}(\mathcal I)/S$ and degree $1$ on every exceptional component of the geometric fibers of $g:\mathbb P_{\mathcal X}(\mathcal I)\to S$ contracted by $\pi$.

\item \label{P:tf-lb3} We have an isomorphism of sheaves $\mathcal I\cong \pi_*(\mathcal O_{\mathbb P_{\mathcal X}(\mathcal I)}(1)) $.

\end{enumerate}

\end{prop}

\begin{proof}
The proof is a generalization of \cite[Prop. 5.5]{estpac}.
% holds in our situation;  we will again include it here for the reader's convenience.

We start by observing that Lemma \ref{L:estklei} below can be applied to $\mathcal I$ since the associated points of $\mathcal X$ are generic points of certain fibers of $f:\cX\to S$, where $\mathcal I$ is invertible, and since $\mathcal I$ is everywhere locally generated by two sections because the fibers of $f$ have only nodal and cuspidal singularities. So, we get that $\mathbb P_{\mathcal X}(\mathcal I)$ is $S$-flat and also that $\mathcal I=\pi_*(\mathcal L)$, which proves statement
\eqref{P:tf-lb3}.

Moreover, since the formation of $\mathbb P_{\mathcal X}(\mathcal I)$ commutes with base-change, statement \eqref{P:tf-lb1} follows, by restricting to the geometric fibers of $g$,  from Lemma \ref{L:ProjWP-stab}.

Since $\mathcal O_{\mathbb P_{\mathcal X}(\mathcal I)}(1)$ is the tautological line bundle, its restriction to each exceptional component contracted by $\pi$ is also tautological, hence it has degree one.
Moreover, the fact that $\mathcal O_{\mathbb P_{\mathcal X}(\mathcal I)}(1)$ has relative degree $d$ from  Proposition \ref{P:lb-tf}\eqref{P:lb-tf1}; this concludes the proof of Statement \eqref{P:tf-lb2}.

%Finally, the fact that $\mathcal L$ has degree $1$ on the exceptional components of the fibers of  $\mathbb P_{\mathcal X}(\mathcal I)/S$ contracted by $\psi$ follows from the fact, by Lemma \ref{L:ProjWP-stab}, those components correspond to the fibers of $\mathbb P_{\mathcal X}(\mathcal I)$ over the nodes and cusps where $\mathcal I$ is not invertible, which are rational curves where $\mathcal L$, being the tautological bundle, must restrict with degree $1$.
\end{proof}

%Let us now recall the statement of Lemma 5.3 on \cite{estpac}, which is based on Lemma 3.1 in \cite{estklei} and that we will need in the proof of Proposition \ref{P:tf-lb} below.

\begin{lemma}\label{L:estklei}
Let $p:X\to S$ be a flat morphism and $F$ an $S$-flat coherent sheaf on $X$ which is invertible at associated points of $X$ and everywhere locally generated by two sections. Let $W:=\mathbb P_X(F)\to S$ the relative projective bundle
associated to $F$ and let $w:W\to X$ be the natural projective $S$-morphism. Then $W$ is $S$-flat and the natural graded $\mathcal O_X$-algebra homomorphism
$$\Sym(F)\to \bigoplus_{n\geq 0}w_*\mathcal O_W(n)$$
is an isomorphism.
\end{lemma}
\begin{proof}
See \cite[Lemma 3.1]{estklei} and \cite[Lemma 5.3]{estpac}.
\end{proof}

We will need one last definition, namely the concept of special elliptic tail with respect to a torsion-free, rank-$1$ sheaf, generalizing Definition \ref{D:elltails-types}.

\begin{defi}\label{D:specell-sh}
Let $X$ be a quasi-wp-stable curve and let $I$ be a torsion-free, rank-$1$ sheaf on $X$.
Let $F$ be an irreducible elliptic tail of $X$  and let $p$ denote the intersection point between $F$ and the complementary subcurve $F^c$. Denote, as usual,
by $I_{F}$ the restriction of $I$ to $F$ modulo the torsion subsheaf.
We say that $F$ is \emph{special} with respect to $I$ if $I_F$ is a line bundle and $I_F=\OO_F(d_F\cdot p)$, where $d_F:=\deg(I_F)$.
Otherwise, we say that $F$ is \emph{non-special} with respect to $I$.
\end{defi}

We can now introduce three new categories fibered in groupoids over the category of schemes parametrizing certain torsion-free, rank-$1$ sheaves on stable
(resp. p-stable, resp. wp-stable) curves.

\begin{defi}\label{D:stack-sh}
Fix two integers $d$ and $g\geq 2$.
\begin{enumerate}[(i)]
\item  \label{D:stack-sh1} Let $\SSst$ be the category fibered in groupoids over the category of $k$-schemes whose sections over a $k$-scheme $S$ are pairs $(f:\mathcal X\to S,\mathcal I)$, where
$f$ is a family of stable curves of genus $g$ and $\mathcal I$ is a coherent sheaf on $\cX$, flat over $S$, such that its restriction $\cX_s$ to every geometric fiber $\cX_s:=f^{-1}(s)$ of
$f$ is a torsion-free, rank-$1$, $\omega_{\cX_s}$-semistable sheaf of degree $d$.
 Arrows between such pairs are given by cartesian diagrams
\begin{equation*}
\xymatrix{
\ar@{}[dr]|{\square}{\mathcal X} \ar[d]_{f} \ar[r]^{h} & {\mathcal X'} \ar[d]^{f'} \\
{S} \ar[r] & {S'}
}
\end{equation*}
together with a specified isomorphism $\mathcal I\stackrel{\cong}{\longrightarrow} h^*\mathcal I'$ of coherent sheaves over $\cX$.

\item \label{D:stack-sh2}  Let $\SSps$ be the category fibered in groupoids over the category of $k$-schemes whose sections over a $k$-scheme $S$ are pairs $(f:\mathcal X\to S,\mathcal I)$, where
$f$ is a family of  p-stable curves of genus $g$ and $\mathcal I$ is a coherent sheaf on $\cX$, flat over $S$, such that its restriction $\cI_s$ to every geometric fiber $\cX_s:=f^{-1}(s)$ of $f$
is a torsion-free, rank-$1$, $\omega_{\cX_s}$-semistable sheaf of degree $d$. Arrows between such pairs are given as in \eqref{D:stack-sh1} above.

\item \label{D:stack-sh3} Let $\SSwp$ be the category fibered in groupoids over the category of $k$-schemes whose sections over a $k$-scheme $S$ are pairs $(f:\mathcal X\to S,\mathcal I)$, where
$f$ is a family of wp-stable curves of genus $g$ and $\mathcal I$ is a coherent sheaf on $\mathcal X$, flat over $S$, such that its restriction $\cI_s$ to every geometric fiber
$\cX_s:=f^{-1}(\cX_s)$ of $f$ is torsion-free, rank-$1$, $\omega_{\cX_s}$-semistable with the property that $\cI_s$ is locally free at the cusps of $\cX_s$ and each elliptic tail of $\cX_s$
is non-special with respect to  $\cI_s$. Arrows between such pairs are given as in \eqref{D:stack-sh1} above.
\end{enumerate}
\end{defi}

We can now prove that the stacks $\SSst$, $\SSps$ and $\SSwp$  are isomorphic to, respectively, the stacks $\JJst$, $\JJps$ and $\JJwp$; thus, they provide an alternative modular description of them.
%Notice that the same statement for the moduli schemes in the case of $\JJst$ and $\SSst$ was proved by Pandharipande in \cite[Sec. 10]{Pan}.

\begin{thm}\label{T:new-descr}
Fix two integers $d$ and $g\geq 2$.
There are isomorphisms of categories fibered in groupoids
\begin{equation*}
\begin{sis}
& \Phi:\JJst \stackrel{\cong}{\longrightarrow} \SSst \\
& \Phi^{\rm ps}:\JJps \stackrel{\cong}{\longrightarrow} \SSps \\
& \Phi^{\rm wp}:\JJwp \stackrel{\cong}{\longrightarrow} \SSwp \\
\end{sis}
\end{equation*}
which are obtained by sending $(f:\cX\to S, \cL)\in \JJst(S)$ (resp. $\JJps(S)$, resp. $\JJwp(S)$) into $(\wps(f): \wps(\cX)\to S, \phi_*(\cL))\in \SSst(S)$ (resp. $\SSps(S)$, resp. $\SSwp(S)$),
where $\phi:\cX\to \wps(\cX)$ is the morphism between the family $f:\cX\to S$ and its wp-stable reduction $\wps(f):\wps(\cX)\to S$ (as in Proposition \ref{P:wp-stab}).

The inverses of the above isomorphisms are given by
sending $(f:\mathcal Y\to S, \mathcal I)\in \SSst(S)$ (resp. $\SSps(S)$, resp. $\SSwp(S)$) to $(\mathbb P_{\mathcal Y}(\mathcal I)\to S, \mathcal O_{\mathbb P_{\mathcal Y}(\mathcal I)}(1))\in \JJst(S)$ (resp. $\JJps(S)$, resp. $\JJwp(S)$).
\end{thm}
\begin{proof}
First of all, the fact that the above defined fibered functors
$\Phi^{\star}: \JJst^{\star} \longrightarrow\SSst^{\star}$ (for $\star=\emptyset, {\rm ps}, {\rm wp}$)
are well-defined follows from Proposition \ref{P:lb-tf}\eqref{P:lb-tf1}
%is proved similarly to \cite[Sec. 10]{Pan}:
 %the flatness over $S$ of the coherent sheaf $\phi_*(\cL)$ follows from the fact
%that $R^1\phi_*(\cL_s)=0$ for every $s\in S$ (by Lemma \ref{L:sh-lb}\eqref{L:sh-lb1}) using the flatness criterion of \cite[Lemma 10.5.2]{Pan};
%the fact that the geometric fibers $\phi_*(\cL)_s$ of $\phi_*(\cL)$ are torsion-free, rank-$1$ of degree $d$ and $\omega_{\wps(\cX)_s}$-semistable
%follows from
together with Lemma \ref{L:sh-lb}\eqref{L:sh-lb1ii}; moreover the extra-properties of the geometric fibers of $\phi_*(\cL)$ for elements of $\SSwp(S)$
as in Definition \ref{D:stack-sh}\eqref{D:stack-sh3} follow from the analogous extra-properties of the geometric fibers of $\cL$ for elements of $\JJwp(S)$
as in  Definition \ref{D:2comp}\eqref{D:2comp2}.

Consider now the fibered functor $\Psi: \SSst^{\star} \longrightarrow\JJst^{\star}$ (for $\star=\emptyset, {\rm ps}, {\rm wp}$) obtained by sending $(f:\mathcal Y\to S, \mathcal I)\in \SSst^{\star}(S)$
into $(\mathbb P_{\mathcal Y}(\mathcal I)\to S, \mathcal O_{\mathbb P_{\mathcal Y}(\mathcal I)}(1))\in \JJst^{\star}(S)$. The fact that $\Psi^{\star}$ is well-defined follows from Proposition \ref{P:tf-lb} together with
Lemma \ref{L:sh-lb}\eqref{L:sh-lb1ii}.

%Now, we have to check that the given natural transformations of categories fibered in groupoids are equivalences for every scheme $S$. We will do it by showing that by assigning to a pair $(f:\mathcal Y\to S, \mathcal I)\in \SSst(S)$ (resp. $\SSps(S)$, resp. $\SSwp(S)$) the pair $(\mathbb P_{\mathcal Y}(\mathcal I)\to S, \mathcal O_{\mathbb P_{\mathcal Y}(\mathcal I)}(1))$, we get a natural transformation that is an inverse for $\Phi^{\star}: \JJst^{\star} \longrightarrow\SSst^{\star}$ (for $\star=\emptyset, {\rm ps}, {\rm wp}$).

%From Proposition \ref{P:tf-lb} we get that if  $\mathcal Y\to S$ is a family of wp-stable curves, then $\mathbb P_{\mathcal Y}(\mathcal I)$ is a family of pre-wp-stable curves, which is moreover quasi-wp-stable (resp. quasi-pseudo-stable, quasi-stable) if $\mathcal I\to S$ is wp-stable (resp. pseudo-stable, stable) since the structure map $\pi: \mathbb P_{\mathcal Y}(\mathcal I)\to \mathcal Y$ is an isomorphism away from the points where $\mathcal I$ is invertible. From Proposition \ref{P:tf-lb} jointly with Lemma \ref{L:sh-lb}\eqref{L:sh-lb1ii}, we conclude that $(\mathbb P_{\mathcal Y}(\mathcal I), \mathcal O_{\mathbb P_{\mathcal Y}(\mathcal I)}(1))\in  \JJst(S)$ (resp. $\JJps(S)$, resp. $\JJwp(S)$) if $(f:\mathcal Y\to S, \mathcal I)\in \SSst(S)$ (resp. $\SSps(S)$, resp. $\SSwp(S)$), so we get a well defined functor $\Psi: \SSst^{\star} \longrightarrow\JJst^{\star}$ (for $\star=\emptyset, {\rm ps}, {\rm wp}$) (the fact that it is well-defined for morphisms is easy to check so we leave it to the reader).

Finally, from Proposition \ref{P:tf-lb}\eqref{P:tf-lb3} we get that $\Phi^\star \circ \Psi^\star\cong \id_{\SSst{^\star}}$, whereas Proposition \ref{P:lb-tf}\eqref{P:lb-tf3} gives that $\Psi^\star \circ \Phi^\star\cong \id_{\JJst{^\star}}$. Therefore, we  conclude that $\Phi^\star$  is an isomorphism with inverse equal to $\Psi^\star$  (for $\star=\emptyset, {\rm ps}, {\rm wp}$), q.e.d.
\end{proof}

Using Theorem \ref{T:new-descr}, we can now prove Lemma \ref{L:invdeg}.

\begin{proof}[Proof of Lemma \ref{L:invdeg}]
For every scheme $S$ (which we can assume to be of finite type over $k$) and for every $\star=\emptyset, {\rm ps}, {\rm wp}$, consider the natural transformation of functors
\begin{equation}\label{E:invS}
\begin{aligned}
\Lambda_S^{\star}: \SSst^{\star}(S) & \longrightarrow \ov{\mathcal S}^{\star}_{-d,g}(S) \\
(f:\cX\to S, \cI) & \mapsto (f:\cX\to S, \cI^{\vee}:={\mcl RHom}(\cI, {\mathbb D}_{\cX}\otimes \omega_f ^{-1})),
\end{aligned}
\end{equation}
where $ {\mathbb D}_{\cX}$ is the dualizing complex of $\cX$ and $\omega_f$ is the relative dualizing sheaf of $f$
(which is a line bundle because the fibers of $f$ are Gorenstein curves).

Let us check that $\Lambda_S^{\star}$ is well-defined and an equivalence of groupoids.
The coherent sheaf $\cI$ is flat over $S$ by assumption
and its fibers are Cohen-Macaulay sheaves (because a torsion-free sheaf on a curve is automatically Cohen-Macaulay).
Therefore, standard results for families of Cohen-Macaulay sheaves (see e.g. \cite[Lemma 2.1]{arin}) show that the coherent sheaf $\cI^{\vee}$ is  flat over $S$ and that
$$(\cI^{\vee})_{|f^{-1}(s)}={\mcl RHom}(\cI_{|f^{-1}(s)}, ({\mathbb D}_{\cX}\otimes \omega_f ^{-1})_{|f^{-1}(s)})=
{\mcl Hom}(\cI_{|f^{-1}(s)}, \cO_{f^{-1}(s)})=(\cI_{|f^{-1}(s)})^{\vee}.
$$
Using this and Lemma \ref{L:dual-I} below, we get that $\Lambda_S^{\star}$ is well-defined.
Similarly, we have that $(\cI^{\vee})^{\vee}=\cI$ which implies that $\Lambda_S^{\star}\circ \Lambda_S^{\star}={\rm id}$, hence $\Lambda_S^{\star}$ is an an equivalence of groupoids.

Therefore, we get that
$$\SSst^{\star}  \cong \ov{\mathcal S}^{\star}_{-d,g},$$
which, together with Theorem \ref{T:new-descr}, concludes the proof of the Lemma.

\end{proof}

\begin{lemma}\label{L:dual-I}
Let $X$ be a Gorenstein curve and let $I$ be a rank-$1$, torsion-free sheaf on $X$ of degree $d$. Then the following hold:
\begin{enumerate}[(i)]
\item \label{L:dual-I1} The dual $I^{\vee}:={\mcl Hom}(I,\OO_X)$ of $I$ is  a rank-$1$, torsion-free sheaf on $X$ of degree $-d$.
\item \label{L:dual-I2} $I$ is reflexive, i.e. $(I^{\vee})^{\vee}=I$.
\item \label{L:dual-I3} If moreover $X$ is wp-stable, then $I$ is $\omega_X$-semistable if and only if $I^{\vee}$ is $\omega_X$-semistable.
\end{enumerate}
\end{lemma}
\begin{proof}
Part \eqref{L:dual-I2} follows from \cite[Prop. 1.6]{Har2}.

Let us now prove \eqref{L:dual-I1}. Rank-$1$, torsion-free (or equivalently reflexive by \eqref{L:dual-I2}) sheaves of degree $d$ on $X$ are in bijection with generalized divisors of degree $d$
on $X$ up to linear equivalence, see \cite[Prop. 2.8]{Har2}. Taking the dual of such a sheaf correspond to taking the dual of the corresponding linear equivalence class of generalized divisors by
\cite[Prop. 2.8(d)]{Har2}. Therefore, $I^{\vee}$ is a rank-$1$, reflexive (hence torsion-free) sheaf of degree $-d$ on $X$.

Part \eqref{L:dual-I3}: by \eqref{L:dual-I2}, it is enough to prove the only if part. So assume that $I$ is $\omega_X$-semistable and let us show that $I^{\vee}$ is $\omega_X$-semistable.
From Proposition \ref{P:tf-lb}, it follows that, setting $\wh{X}:=\wt{X}_{\Sing(I)}$ and $\phi:=\phi^{\Sing(I)}$, there exists a properly balanced bundle $L$ on $\wh{X}$ such that
$\phi_*(L)=I$.
The line bundle $L^{-1}$ is also balanced (although not necessarily properly balanced!) since, given a proper subcurve $Z\subseteq X$, we have that
\begin{equation*}
\left|\un d_Z-\frac{d}{2g-2}\deg_Z \omega_X\right|\leq \frac{k_Z}{2} \Leftrightarrow \left|-\un d_Z-\frac{-d}{2g-2}\deg_Z \omega_X\right|\leq \frac{k_Z}{2}.
\end{equation*}
Moreover, by the definitions of ${\mcl Hom}(-, -)$ and $\phi_*$ and using that $\phi_*(L)=I$ and $\phi_*(\cO_{\wh{X}})=\cO_X$, we have that
\begin{equation*}
\phi_* (L^{-1})=\phi_*{\mcl Hom}(L,\cO_{\wh{X}})={\mcl Hom}(\phi_*(L),\phi_*(\cO_{\wh{X}}))={\mcl Hom}(I,\cO_X)=I^{\vee}.
\end{equation*}
We conclude that $I^{\vee}$ is $\omega_X$-semistable by Lemma \ref{L:sh-lb}\eqref{L:sh-lb1ii}.
\end{proof}

From Theorem \ref{T:new-descr}, using Fact \ref{F:Old-Comp-Jac}, Theorem \ref{stackprop} and the arguments in \S\ref{S:2comp-var}, we deduce the following corollary.

\begin{coro}\label{C:alg-stacks}
Let $d\in \Z$ and $g\geq 2$ (resp. $g\geq 3$, resp. $g\geq 3$).
\begin{enumerate}[(i)]
\item \label{C:alg-stacks1} $\SSst$ (resp. $\SSps$, resp. $\SSwp$)  is a smooth and irreducible universally closed Artin stack of finite type over $k$ and of dimension $4g-4$, endowed with a
universally closed morphism $\Psi^{\rm s}$ (resp. $\Psi^{\rm ps}$, resp. $\Psi^{\rm wp}$) onto the moduli stack of stable (resp. p-stable, resp. wp-stable) curves $\MMg$ (resp. $\MMgp$, resp.
$\MMgwp$).
\item \label{C:alg-stacks2} The projective variety $\Jst$ (resp. $\Jps$, resp. $\Jwp$)  is an adequate moduli space (and even a good moduli space if ${\rm char}(k)=0$) for
$\SSst$ (resp. $\SSps$, resp. $\SSwp$).
\end{enumerate}
\end{coro}

%The aim of this subsection is to give a description of the fiber of $\Phi^{\rm ps}$ over a $p$-stable curve $X\in \Mgp$ in terms of the Simpson's compactified Jacobian  of $X$, $\ov{\Jac_d}(X)$, that we will now describe.

%Consider the controvariant functor
%\begin{equation}\label{E:fun-J}
%\ov{\mathcal J}_{d,X} :\SCH\to \SET
%\end{equation}
%which associates to a scheme $T$ the set of $T$-flat coherent sheaves on $X\times T$ which are rank-1 torsion-free sheaves  and $\omega_X$-semistable on the geometric fibers $X\times\{t\}$ of the second projection morphism $X\times T\to T$.

%From the work on Simpson in \cite{simpson}, we have the following result concerning the co-representability of the moduli functor $\ov{\mathcal J}_{d,X}$ (in the sense of Convention \ref{Con:functor}).

%\begin{fact}[Simpson]\label{F:Simp}
%For any integer $d$, there is a projective variety $\ov{\Jac_d}(X)$ which co-represents the functor $\ov{\mathcal J}_{d,X}$. The geometric points of $\ov{\Jac_d}(X)$ parametrize $S$-equivalence classes of rank-1 torsion-free sheaves on $X$ which are $\omega_X$-semistable.
%\end{fact}

%For the definition of $S$-equivalence of sheaves, we refer the interested reader to \cite{simpson}.

Another corollary of Theorem \ref{T:new-descr} is  a modular description of the fibers of the map $\Phi^{\rm ps}: \Jps\to \ov M_g^{\rm ps}$ in terms of
canonical compactified Jacobians (see Remark \ref{R:Gor-ss}\eqref{R:Gor-ss2}),
extending the description of the fibers of the map $\Phi^{\rm s}:\Jst\to \Mg$ given in Fact \ref{F:Old-Comp-Jac}\eqref{Old4}.

\begin{coro}\label{C:fiber-ps}
Let $g\geq 3$ and $d\in \Z$. Assume that ${\rm char}(k)=0$.
%or that ${\rm char}(k)=p>0$  is bigger than the order of the automorphism group of every p-stable curve of genus $g$.
Then the fiber ${(\Phi^{\rm ps})}^{-1}(X)$ of  the morphism $\Phi^{\rm ps}: \Jps\to \ov M_g^{\rm ps}$ over $X\in \Mgp$
is isomorphic to  $\ov{\Jac_d}(X)/\Aut(X)$.
\end{coro}
\begin{proof}
The proof is the same as the proof of the analogous result for the morphism $\Phi^{\rm s}: \Jst\to \Mg$ (see Fact \ref{F:Old-Comp-Jac}\eqref{Old4})
using that $\Jst$ is a good moduli space for $\SSst$ by Corollary
\ref{C:alg-stacks}\eqref{C:alg-stacks2}; see e.g. \cite[Proof of Fact 2.6(3)]{CMKV} for more details.
\end{proof}

It would be interesting to know if the above Corollary \ref{C:fiber-ps} is true regardless of the characteristic of the base field $k$.
This would follow if one could prove that  $\Jps$ is a good moduli scheme for the stack $\JJps$ (or equivalently for the stack $\SSps$) also in positive characteristic.

\section{Appendix: Positivity properties of balanced line bundles}\label{S:appendix}

The aim of this Appendix is to investigate positivity properties of balanced line bundles of sufficiently high degree on (reduced) Gorenstein curves. The results obtained here are applied in this manuscript only for quasi-wp-stable curves. However we decided to present these results in the Gorenstein case for two reasons: firstly, we think that these results are interesting in their own (in particular we will generalize several results of \cite{Capseries} and \cite[Sec. 5]{melo} in the case of nodal curves); secondly,
our proof extends without any modifications to the Gorenstein case.

So, throughout this Appendix, we let $X$ be a connected (reduced) Gorenstein curve
%(i.e., scheme of finite type over an algebraically closed field $k$ of pure dimension one)
of genus $g\geq 2$
%\footnote{We assume $g\geq 2$ for simplicity: the cases $g=0, 1$ are easy.}
and $L$ be a balanced line bundle on $X$ of degree $d$, i.e., a line bundle $L$ of degree $d$ satisfying the basic
inequality
\begin{equation}\label{bas-ineq-bis}
\left|\deg_Z L -\frac{d}{2g-2}\deg_Z \omega_X\right|\leq \frac{k_Z}{2},
\end{equation}
for any (connected) subcurve $Z\subseteq X$, where $k_Z$ is as usual the length of the scheme-theoretic intersection of
$Z$ with the complementary subcurve $Z^c:=\ov{X\setminus Z}$ and $\omega_X$ is the dualizing invertible (since $X$ is Gorenstein) sheaf.

The following definitions are natural generalizations to the Gorenstein case of the familiar concepts for nodal curves.

\begin{defi}\label{D:G-sing}
Let $X$ be a connected  Gorenstein curve of genus $g\geq 2$. We say that
\begin{enumerate}[(i)]
\item $X$ is \emph{G-semistable}\footnote{The letter G stands for Gorenstein to suggest that these notions are the natural generalizations of the usual notions from nodal to Gorenstein curves.}
if $\omega_X$ is nef, i.e. $\deg_Z\omega_X\geq 0$ for any (connected) subcurve $Z$.
The connected subcurves $Z$ such that $\deg_Z \omega_X=0$ are called \emph{exceptional} subcurves.
\item $X$ is \emph{G-quasistable} if $X$ is G-semistable and every exceptional subcurve $Z$ is isomorphic to $\P^1$.
\item $X$ is \emph{G-stable} if $\omega_X$ is ample, i.e.  $\deg_Z\omega_X> 0$
for any (connected) subcurve $Z$.
\end{enumerate}
\end{defi}

Note that G-semistable (resp. G-stable) curves are called semi-canonically positive (resp. canonically positive) in \cite[Def. 0.1]{Cat}.
The terminology G-stable was introduced in \cite[Def. 2.2]{CCE}. We refer to \cite[Sec. 1]{Cat} for more details on G-stable and G-semistable curves.

Observe also that quasi-wp-stable, quasi-p-stable and quasi-stable curves are G-quasistable; similarly wp-stable, p-stable and stable curves are G-stable.

\begin{rmk}\label{rmk:adj-form}
Given a subcurve \footnote{Note that a subcurve of Gorenstein curve need not to be Gorenstein. For example, the curve $X$ given by the union of $4$ generic lines through the origin in
$\mathbb A_k^3$ is Gorenstein, but each subcurve of $X$ given by the union of three lines is not Gorenstein.} $i:Z\subseteq X$
with complementary subcurve $Z^c$, consider the exact sequence
$$0\to \omega_X\otimes I_{Z^c}\to \omega_X\to (\omega_X)_{|Z^c}\to 0,$$
where $I_{Z^c}$ is the ideal sheaf of $Z^c$ in $X$.
By the definition of the dualizing sheaf $\omega_Z$ of $Z$, it is easy to check that $i_*(\omega_Z)=\omega_X\otimes I_{Z^c}$ which, by restricting to $Z$, gives
$$\omega_Z=(\omega_X\otimes I_{Z^c})_{|Z}=(\omega_X)_{|Z}\otimes I_{Z\cap Z^c/Z},$$
where $I_{Z\cap Z^c/Z}$ is the ideal sheaf of the scheme theoretic intersection $Z\cap Z^c$ seen as a subscheme of $Z$.
By taking degrees, we get the adjunction formula
\begin{equation}\label{E:adj-form}
\deg_Z\omega_X=2g_Z-2+k_Z.
\end{equation}
Using the above adjunction formula and  recalling that $g_Z\geq 0$ if $Z$ is connected, it is easy
to see that:
\begin{enumerate}[(i)]
\item $X$ is $G$-semistable if and only if for any connected subcurve $Z$ such that $g_Z=0$ we have that $k_Z\geq 2$.
\item $X$ is $G$-stable if and only if  for any connected subcurve $Z$ such that $g_Z=0$ we have that $k_Z\geq 3$.
\end{enumerate}
\end{rmk}

Our first result says when a balanced line bundle of sufficiently high degree is nef or ample.

\begin{prop}\label{nefample}
Let $X$ be a connected  Gorenstein curve of genus $g\geq 2$ and let $L$ be a balanced line bundle on $X$ of degree $d$.
The following is true:
\begin{enumerate}[(i)]
\item \label{nef} If $d>\frac{1}{2}(2g-2)=g-1$ then $L$ is nef if and only if $X$ is G-semistable and for every exceptional subcurve $Z$ it holds that $\deg_ZL=0$ or $1$.
\item \label{ample} If $d>\frac{3}{2}(2g-2)=3(g-1)$ then $L$ is ample  if and only if $X$ is G-quasistable and for every exceptional subcurve $Z$ it holds that $\deg_Z L=1$.

\end{enumerate}
\end{prop}
\begin{proof}
%Note that $L$ is ample on $X$ if and only if $L$ has positive degree on all the irreducible (or equivalently connected)
%subcurves of $X$.
Let us first prove part \eqref{nef}. Let $Z\subseteq X$ be a connected subcurve of $X$. If $Z=X$ then $\deg_Z L=\deg L=d>(g-1)>0$ by assumption. So we can assume that $Z\subsetneq X$. Notice that, since $X$ is connected, this implies that $k_Z\geq 1$.

If $\deg_Z \omega_X=2g_Z-2+k_Z>0$ then, using the basic inequality  \eqref{bas-ineq-bis} and the assumption
$d>\frac{1}{2}(2g-2)$, we get
$$\deg_ZL\geq d\cdot \frac{2g_Z-2+k_Z}{2g-2}-\frac{k_Z}{2}> \frac{2g_Z-2+k_Z}{2}-\frac{k_Z}{2}\geq
\begin{sis}
0 & \text{\: if } g_Z\geq 1, \\
-1 & \text{\: if } g_Z=0,
 \end{sis}$$
hence $\deg_ZL\geq 0$. If $g_Z=0$ and $k_Z=1$ then,
using the basic inequality and the assumption on $d$, we get that
\begin{equation*}\label{Gss}
\deg_Z L\leq \frac{d}{2g-2}(-1)+\frac{1}{2}<0.
\end{equation*}
Therefore, if $L$ is nef then $X$ must be G-semistable. Finally, if $Z$ is any exceptional subcurve of $X$, then the basic inequality gives
\begin{equation}\label{Gexc}
\left|\deg_ZL\right|\leq 1,
\end{equation}
from which we deduce that if $L$ is nef then $\deg_Z L=0$ or $1$. Conversely, it is also clear that if $X$ is G-semistable and
$\deg_Z L=0$ or $1$ for every exceptional subcurve $Z$ of $X$ then $L$ is nef.

Let us now prove part \eqref{ample}. Let $Z\subseteq X$ be a connected subcurve of $X$. If $Z=X$ then $\deg_Z L=\deg L=d>3(g-1)>0$ by assumption. So we can assume that $Z\subsetneq X$. Notice that, since $X$ is connected, this implies that $k_Z\geq 1$.

If $\deg_Z\omega_X=2g_Z-2+k_Z>0$ then, using the basic inequality  \eqref{bas-ineq-bis} and the
inequality $d>\frac{3}{2}(2g-2)$, we get
$$\deg_ZL\geq d\cdot \frac{2g_Z-2+k_Z}{2g-2}-\frac{k_Z}{2}> \frac{3(2g_Z-2+k_Z)}{2}-\frac{k_Z}{2}\geq
\begin{sis}
k_Z\geq 1 & \text{ if } g_Z\geq 1, \\
\frac{2k_Z-6}{2}\geq 0 & \text{ if } g_Z=0 \text{ and } k_Z\geq 3,
\end{sis}$$
hence $\deg_Z L>0$. From part \eqref{nef} and equation \eqref{Gexc}, we get that if $L$ is ample then $X$ is G-semistable and for every exceptional subcurve $Z$ we have that $\deg_Z L=1$. Note that every exceptional subcurve $Z$ of $X$
is a chain of $\P^1$. Assume that this chain has
length $l\geq 2$ and denote by $W_i$ (for $i=1,\dots, l$) the irreducible components of $Z$. Then each of the $W_i$'s
is an exceptional subcurve of $X$. Therefore, the same inequality as before gives that if $L$ is ample then
$\deg_{W_i}L=1$. This is a contradiction since $1=\deg_ZL =\sum_i \deg_{W_i}L=l>1$.
 Hence $Z\cong \P^1$ and $X$ is G-quasistable. Conversely, it is clear that  if $X$ is G-semistable and
$\deg_Z L=1$ for every exceptional subcurve $Z$ of $X$ then $L$ is ample.

\end{proof}

We next investigate when a balanced line bundle on a Gorenstein curve is non-special, globally generated, very ample or normally generated.
%or, more generally, $k$-very ample.
To this aim, we will use the following criteria, due to Catanese-Franciosi \cite{CF}, Catanese-Franciosi-Hulek-Reid \cite{CFHR} and  Franciosi-Tenni \cite{FT}
(see also \cite{Fra1} and \cite{Fra2}) which generalize the classical criteria for smooth curves.

%we get the following criteria to test the very ampleness and the non-speciality for any line bundle $L$ on a Gorenstein curve %$X$.

\begin{fact}\label{numcrit} (\cite{CF}, \cite{CFHR}, \cite{FT})
Let $L$ be a line bundle on a  Gorenstein curve $X$. Then the following holds:
\begin{enumerate}[(i)]
\item \label{numcrit1} If $\deg_ZL>2g_Z-2$ for all (connected) subcurves $Z$ of $X$, then $L$ is non-special, i.e.,
    $H^1(X, L)=0$.
\item \label{numcritbis} If $\deg_ZL> 2g_Z-1$ for all (connected) subcurves $Z$ of $X$, then $L$ is globally generated;
\item \label{numcrit2} If $\deg_ZL> 2g_Z$ for all (connected) subcurves $Z$ of $X$, then $L$ is very ample.
%\item \label{numcrit2bis} More generally, for any $k\geq 0$, if $\deg_Z L>2g_Z+k-1$ for all (connected) subcurves $Z$ %of $X$, then $L$ is $k$-very ample.
\item \label{numcrit4} If $\deg_ZL> 2g_Z$ for all (connected) subcurves $Z$ of $X$, then $L$ is normally generated, i.e. the multiplication maps
$$\rho_k: H^0(X, L)^{\otimes k}\to H^0(X, L^k)$$
are surjective for every $k\geq 2$.
\end{enumerate}
\end{fact}
Recall that if $Z$ is a subcurve that is a disjoint union of two subcurves $Z_1$ and $Z_2$
then $g_Z=g_{Z_1}+g_{Z_2}-1$. From this, it is easily checked that if the numerical assumptions of \eqref{numcrit1}, \eqref{numcritbis}, \eqref{numcrit2} and \eqref{numcrit4}
are satisfied for all connected subcurves $Z$ then they are satisfied for all subcurves $Z$.
With this in mind, part \eqref{numcrit1} follows from \cite[Lemma 2.1]{CF}.
%(see also \cite[Lemma 2.5]{Capseries} in the case of nodal curves).
Note that in loc. cit. this result is only stated for a curve $C$  embedded  in a smooth surface; however, a closer inspection of the proof reveals that the same result is true for any Gorenstein curve
$C$. Parts \eqref{numcritbis} and \eqref{numcrit2} follow from \cite[Thm. 1.1]{CFHR}. Part \eqref{numcrit4} follows from \cite[Thm. 4.2]{FT}, which generalizes the  previous results of Franciosi
(see \cite[Thm. B]{Fra1} and \cite[Thm. 1]{Fra2}) for reduced curves with locally planar singularities.

%\begin{fact}\label{F:crit-norm-gen}(\cite{Fra1}, \cite{Fra2})
%Assume that $X$ has locally planar singularities and let $L$ be a line bundle on $X$.
%If $\deg_ZL>\max\{2g_Z, \deg_Z \omega_X\}$ for all (connected) subcurves $Z$ of $X$ then $L$ is normally generated.
%\end{fact}
%This follows from \cite[Thm. B]{Fra1} and \cite[Thm. 1]{Fra2}. Note that in loc. cit. the assumptions on $L$ are that $\deg_Z L>2g_Z+1$ for all subcurves $Z$ of $X$ and
%$L$ is numerically equivalent to $\omega_X\otimes A$ for a certain ample line bundle $A$ on $X$. However, this is seen to be equivalent to our assumptions since a line bundle
%$A$ is ample on $X$ if and only if $\deg_Z A>0$ for every (connected) subcurve $Z$ of $X$.

%Indeed, the inequality $\deg_Z L> 2g_Z$ is satisfied for all subcurves $Z$ of $X$ if and only if it is satisfied for all connected subcurves $Z$ of $X$.
%Moreover, if $L$ is numerically equivalent to $\omega_X\otimes A$ for some ample line bundle $A$ on $X$ then for every subcurve $Z$ of $X$ we have that
%$\deg_ZL=\deg_Z(\omega_X\otimes A)=\deg_Z \omega_X+\deg_Z A>\deg_Z\omega_X$ since $A$

Using the above criteria, we can now investigate when balanced line bundles are non-special, globally generated, very ample or normally generated.

\begin{thm}\label{bal-pos}
Let $L$ be a balanced line bundle of degree $d$ on a connected  Gorenstein curve $X$ of genus $g\geq 2$.
Then the following properties hold:
\begin{enumerate}[(i)]
\item \label{bal-ns} If $X$ is G-semistable  and $d>2g-2$
then $L$ is non-special.
\item \label{bal-gg} Assume that $L$ is nef. If $d>\frac{3}{2}(2g-2)=3(g-1)$ then  $L$ is globally generated.
\item \label{bal-va} Assume that $L$ is ample. Then:
\begin{enumerate}
\item \label{bal-va1} If  $d>\frac{5}{2}(2g-2)=5(g-1)$ then $L$ is very ample and normally generated.
\item \label{bal-va2} If $d>\max\{\frac{3}{2}(2g-2)=3(g-1), 2g\}$ and $X$ does not have elliptic tails (i.e., connected subcurves $Z$ such that $g_Z=1$ and $k_Z=1$)
then $L$ is very ample and normally generated.
\end{enumerate}
%\item \label{bal-normgen}
%Assume that $X$ has locally planar singularities. If $X$ is G-stable and $d>\frac{5}{2}(2g-2)=5(g-1)$ then $L$ is normally generated.
\end{enumerate}
\end{thm}
\begin{proof}
In order to prove part \eqref{bal-ns}, we apply Fact \ref{numcrit}\eqref{numcrit1}.
Let $Z\subseteq X$ be a connected subcurve. If $Z=X$ then $\deg_Z L=d > 2g-2$ by assumption. Assume now that $Z\subsetneq X$
(hence that $k_Z\geq 1$). Since $X$ is G-semistable, we have that $\deg_Z(\omega_X)=2g_Z-2+k_Z\geq 0$.
If $\deg_Z(\omega_X)>0$ then the basic inequality \eqref{bas-ineq-bis} together with the hypothesis on $d$ gives that
$$\deg_Z L\geq \frac{d}{2g-2}(2g_Z-2+k_Z)-\frac{k_Z}{2}>2g_Z-2+\frac{k_Z}{2}>2g_Z-2.$$
If $\deg_Z(\omega_X)=0$ (which happens if and only if $Z$ is exceptional, i.e., $g_Z=0$ and $k_Z=2$) then the basic inequality gives that
$$\deg_Z L\geq \frac{d}{2g-2}(2g_Z-2+k_Z)-\frac{k_Z}{2}=-1 > -2=2g_Z-2.$$
In order to prove part \eqref{bal-gg}, we apply Fact \ref{numcrit}\eqref{numcritbis}.  Let $Z\subseteq X$ be a connected subcurve. If $Z=X$ then
 we have that  $\deg_Z L=d > 3(g-1)\geq 2g-1$ by the assumption on $d$.
Assume now that $Z\subsetneq X$ (hence that $k_Z\geq 1$). If $g_Z=0$ then $\deg_Z L>-1=2g_Z-1$ since $L$ is nef. Therefore, we can assume that $g_Z\geq 1$. By applying the basic inequality \eqref{bas-ineq-bis} and using our
assumption on $d$, we get that
$$\deg_ZL\geq \frac{d}{2g-2}(2g_Z-2+k_Z)-\frac{k_Z}{2}> \frac{3}{2}(2g_Z-2+k_Z)-\frac{k_Z}{2}=3(g_Z-1)+k_Z
\geq 2g_Z-1.$$
In order to prove parts \eqref{bal-va1} and \eqref{bal-va2}, we apply Facts \ref{numcrit}\eqref{numcrit2} and \ref{numcrit}\eqref{numcrit4}.  Let $Z\subseteq X$ be a connected subcurve. If $Z=X$ then, in each of the cases \eqref{bal-va1} and
\eqref{bal-va2}, we have that  $\deg_Z L=d > 2g$ by the assumption on $d$ (note that $5(g-1)>2g$ since $g\geq 2$).
Assume now that $Z\subsetneq X$ (hence that $k_Z\geq 1$). If $g_Z=0$ then $\deg_Z L>0=2g_Z$ since $L$ is ample. Therefore, we can assume that $g_Z\geq 1$.

In the first case \eqref{bal-va1}, by applying the basic inequality \eqref{bas-ineq-bis} and the numerical
assumption on $d$, we get that
$$\deg_ZL\geq \frac{d}{2g-2}(2g_Z-2+k_Z)-\frac{k_Z}{2}> \frac{5}{2}(2g_Z-2+k_Z)-\frac{k_Z}{2}=5(g_Z-1)+2k_Z
\geq 2g_Z.$$

In the second case \eqref{bal-va2}, from the basic inequality \eqref{bas-ineq-bis} and the numerical
assumption on $d$, we get that
$$\deg_ZL\geq \frac{d}{2g-2}(2g_Z-2+k_Z)-\frac{k_Z}{2}> \frac{3}{2}(2g_Z-2+k_Z)-\frac{k_Z}{2}=3(g_Z-1)+k_Z
\geq 2g_Z,$$
where in the last inequality we used that $g_Z, k_Z\geq 1$ and $(g_Z,k_Z)\neq (1,1)$ because $X$ does not contain
elliptic tails.

\end{proof}

%Let us briefly indicate how our results relate to other results present in the literature.

\begin{rmk}
Theorem \ref{bal-pos}\eqref{bal-ns} recovers \cite[Thm. 2.3(i)]{Capseries} in the case of nodal curves.
Theorem \ref{bal-pos}\eqref{bal-gg} combined with Proposition \ref{nefample}\eqref{nef} recovers and improves
\cite[Thm. 2.3(iii)]{Capseries} in the case of nodal curves. Theorem \ref{bal-pos}\eqref{bal-va} improves \cite[Cor. 5.11]{melo} in the case of nodal curves.
See also \cite{Bal}, where the author gives some criteria for the global generation and very ampleness
of balanced line bundles on quasi-stable curves.
\end{rmk}

The previous results can be applied to study the positivity properties of powers of the canonical line bundle on a  Gorenstein curve, which is clearly a balanced line bundle.

\begin{coro}\label{pospropdualizing}
Let $X$ be a connected Gorenstein curve of genus $g\geq 2$. Then the following holds:
\begin{enumerate}[(i)]
\item \label{can-gg} If $X$ is G-semistable then $\omega_X^i$ is non-special and globally generated for all $i\geq 2$;
\item \label{can-va} If $X$ is G-stable then $\omega_X^i$ is very ample for all $i\geq 3$;
\item \label{can-ng} If $X$ is G-quasistable then $\omega_X^i$ is normally generated for all $i\geq 3$.
\end{enumerate}
\end{coro}
\begin{proof}
Part \eqref{can-gg} follows from Theorem \ref{bal-pos}\eqref{bal-ns} and Theorem \ref{bal-pos}\eqref{bal-gg}.

Part \eqref{can-va} follows from Theorem \ref{bal-pos}\eqref{bal-va1}.

%Parts (i) and (ii) follow immediately by applying Fact \ref{numcrit} to $\omega_X^i$, since for $i\geq 2$ we get that for any subcurve $Z$ of $X$
%$$\deg_Z(\omega_X^i)=i(2g_Z-2+k_Z)=(2g_Z-1)+(k_Z-1)+(i-1)(2g_Z-2+k_Z).$$
%Since $X$ is $G-$semistable, $\omega_X$ is nef so $2g_Z-2+k_Z\geq 0$ for all $Z\subseteq X$. Moreover, when it is equal to $0$ then $Z$ is rational, which implies that $k_Z\geq 2$, so $k_Z-1\geq
%1$ and we conclude.

Let us now prove part \eqref{can-ng}. If $X$ is G-stable, then this follows from Theorem  \ref{bal-pos}\eqref{bal-va1}. In the general case, since $\omega_X^i$ is globally generated
by part \eqref{can-gg}, it defines a  morphism
$$q:X\to \mathbb P:=\mathbb P(H^0(X,\omega_X^i)^{\vee}),$$
whose image we denote by $Y:=q(X)$. Since $X$ is $G$-quasistable, the degree of $\omega_X^i$ on a connected subcurve $Z$ of $X$ is zero if and only if $Z=E$ is an exceptional
subcurve, i.e., if $E\cong \P^1$ and $k_E=2$.
The map $q$ will contract such an exceptional subcurve $E$ to a node if $E$ meets the complementary subcurve $E^c$ in two distinct points and to a cusp if $E$ meets $E^c$
in one point with multiplicity two. Moreover, using Fact \ref{numcrit}\eqref{numcrit2}, it is easy to check that $\omega_X$ is very ample on $X\setminus \cup E$, where the union runs over all
exceptional subcurves $E$ of $X$. We deduce that  $Y$ is G-stable. By what proved above, $\omega_Y^i$ is normally generated.
Clearly, $q^*\omega_Y^i=\omega_X^i$ and moreover, since $q$ has connected fibers, we have that $q_*\OO_X=\OO_Y$. This implies that $H^0(X,(\omega_X^i)^k)=H^0(Y,(\omega_Y^i)^k)$
from which we deduce that $\omega_X^i$ is normally generated.
\end{proof}

\begin{rmk}
Part \eqref{can-gg} of the above Corollary \ref{pospropdualizing}  recovers \cite[Thm. A and p. 68]{Cat}, while part \eqref{can-va} recovers \cite[Thm B]{Cat}. Part \eqref{can-ng}
was proved for nodal curves in \cite[Cor. 5.9]{melo}.

A closer inspection of the proof reveals that parts \eqref{can-va} and \eqref{can-ng}  continue to hold for $\omega_X^2$ if,
moreover, $g\geq 3$ and $X$ does not have elliptic tails (see also \cite[Thm. C]{Cat} and \cite[Thm. C]{Fra1}).
\end{rmk}

Let us end this Appendix by mentioning that it is possible to generalize the above results in order to prove that a balanced line bundle of sufficiently high degree
is $k$-very ample in the sense of Beltrametti-Francia-Sommese (\cite{BFS}). Recall first the definition of $k$-very ampleness.

\begin{defi}\label{k-va}
Let $L$ be a line bundle on $X$ and let $k\geq 0$ be a integer. We say that $L$ is \emph{$k$-very ample} if for any $0$-dimensional subscheme $S\subset X$ of length at most $k+1$ we have that the natural restriction map
$$H^0(X, L)\to H^0(S, L_{|S})$$
is surjective. In particular $0$-very ample is equivalent to being globally generated and $1$-very ample is equivalent
to being very ample.
\end{defi}

The proof of the following Theorem is very similar to the proof of the Theorem \ref{bal-pos} above, using again \cite[Thm. 1.1]{CFHR}, and therefore we omit it.

\begin{thm}\label{bal-k-va}
Let $k\geq 2$ and assume that $X$ is G-stable. Then:
\begin{enumerate}[(i)]
\item \label{bal-k-va1} If  $d>\frac{2k+3}{2}(2g-2)=(2k+3)(g-1)$ then $L$ is $k$-very ample.
\item \label{bal-k-va2} If $d>\frac{2k+1}{2}(2g-2)=(2k+1)(g-1)$ and $X$ does not have elliptic tails then $L$ is $k$-very ample.
\end{enumerate}
\end{thm}

%\begin{prop}
%Let $L$ be a balanced line bundle of degree $d>2(2g-2)$ on a quasi-wp-stable curve $X$ of genus $g\geq 2$. Then $L$ is non-special. Assume moreover that $L$ is ample and that $X$ has no elliptic
%tails. Then $L$ is very ample.
%\end{prop}
%\begin{proof}
%Let $Z$ be a subcurve of $X$. Then, since $L$ is balanced and $d>2(2g-2)$, we get that
%\begin{equation}\label{bas-ineq2}
%\deg_Z(L)\geq d\frac{2g_Z-2+\delta_Z}{2g-2}-\frac{\delta_Z}{2}>2(2g_Z-2)+\frac{3}{2}\delta_Z.
%\end{equation}
%If $g_Z\geq 2$, from formula (\ref{bas-ineq2}) we get that $\deg_ZL>2g_Z+2g_Z-4+\frac{3}{2}\delta_Z>2g_Z+1$.
%If $g_Z=1$, (\ref{bas-ineq2}) gives that $\deg_ZL>\frac{3}{2}\delta_Z$. So, if and $\delta_Z\geq 2$, then $\deg_ZL\geq 3=2g_Z+1$. If instead $\delta_Z=1$, which means that $Z$ is an elliptic tail of $X$,
%we only get that $\deg_ZL\geq 2=2g_Z$.
%If $g_Z=0$ and $\delta_Z\geq 2$ then we get that $\deg_ZL\geq -1=2g_Z-1$. If moreover $L$ is ample then by Lemma \ref{ample} above, we get that necessarily $\deg_ZL=1=2g_Z+1$.  Finally, if
%$g_Z=0$ and $\delta_Z\geq 3$, then inequality \ref{bas-ineq2} above implies that $\deg_ZL\geq 1=2g_Z+1$.
%The result follows now by direct application of Fact \ref{numcrit} to $L$.
%\end{proof}


\begin{thebibliography}{bla-bla}

\bibitem[ACV01]{acv} D. Abramovich, A. Corti, A. Vistoli: {\em Twisted bundles and admissible covers.}  Comm. Algebra {\bf 31} (2003), no. 8,  3547--3618.

\bibitem[Ale04]{ale} V. Alexeev: {\it Compactified Jacobians and Torelli map.} Publ. RIMS, Kyoto Univ. {\bf 40} (2004), 1241--1265.


\bibitem[Alp]{alper} J. Alper:  \emph{Good moduli spaces for Artin stacks}. To appear in Annales de l'Institut de Fourier (preprint available at arXiv:0804.2242v2).

\bibitem[Alp2]{alper2} J. Alper: \emph{Adequate moduli spaces and geometrically reductive group schemes}. To appear  in Algebraic Geometry (preprint available at arXiv:1005.2398).

\bibitem[AFS1]{AFS0} J. Alper, M. Fedorchuk, D. Smyth: \emph{Singularities with $\Gm$-action and the log minimal model program for $\Mg$.} To appear Journal f\"ur die reine und angewandte Mathematik (preprint available at arXiv:1010.3751).

\bibitem[AFS13]{AFS} J. Alper, M. Fedorchuk, D. Smyth: \emph{Finite Hilbert stability of (bi)canonical curves.}
Inventiones mathematicae 191 (2013), 671--718.

\bibitem[AH12]{AH} J. Alper, D. Hyeon: \emph{GIT construction of log canonical models of $\ov{M}_g$}.
Compact moduli spaces and vector bundles, 87�106, Contemp. Math., 564, Amer. Math. Soc., Providence, RI, 2012.

\bibitem[ASvdW]{ASvdW} J.Alper, D. I. Smyth, F. van der Wyck: \emph{Weakly proper moduli stacks of curves}. Preprint available at arXiv:1012.0538.

\bibitem[AK79]{AK} A. B. Altman, S. L. Kleiman: \emph{Bertini theorems for hypersurface sections containing a subscheme.} Comm. Algebra  {\bf 7}  (1979), 775--790.

\bibitem[ACG11]{ACG} E. Arbarello, M. Cornalba, P. A. Griffiths: \emph{Geometry of algebraic curves. Volume II.} With a contribution by Joseph Daniel Harris. Grundlehren der Mathematischen Wissenschaften 268. Springer, Heidelberg, 2011.

\bibitem[Ari13]{arin} D. Arinkin: {\em Autoduality of compactified Jacobians for curves with plane singularities}. J. Algebraic Geometry 22 (2013), 363--388.


\bibitem[Bal09]{Bal} E. Ballico: \emph{Very ampleness of balanced line bundles on stable curves.}
Riv. Mat. Univ. Parma (8) {\bf 2} (2009), 81--90.

\bibitem[BFS89]{BFS} M. Beltrametti, P. Francia, A. J. Sommese: \emph{On Reider's method and
higher order embeddings.} Duke Math. J. {\bf 58} (1989), 425--439.

\bibitem[BFV12]{BFV} G. Bini, C. Fontanari, F. Viviani: \emph{On the birational geometry of the universal Picard variety.}
Int. Math. Res. Notices (2012) Vol. 2012, No. 4, 740--780.

\bibitem[BMV12]{BMV} G. Bini, M. Melo, F. Viviani: \emph{On GIT quotients of Hilbert and Chow schemes of curves.}
Electron. Res. Announc. Math. Sci. 19 (2012), 33--40.


\bibitem[BLR90]{BLR} S. Bosch, W. L\"utkebohmert, M. Raynaud: \emph{N\'eron models.} Ergebnisse der Mathematik und ihrer Grenzgebiete (3), Vol. {\bf 21}. Springer-Verlag, Berlin, 1990.

\bibitem[Bou87]{Bou} J.F. Boutot: \emph{Singularit\'es rationnelles et quotients par les groupes r\'eductifs.}
Invent. Math.  {\bf 88}  (1987), 65--68.

\bibitem[Cap94]{Cap} L. Caporaso: \emph{A compactification of the universal Picard variety
over the moduli space of stable curves}. J. Amer. Math. Soc. {\bfseries 7} (1994), 589--660.

\bibitem[Cap05]{capneron} L. Caporaso: \emph{N\'eron models and compactified Picard schemes over the moduli stack of stable curves}. Amer. J. Math. {\bf 130} (2008), no. 1, 1--47.

\bibitem[Cap10]{Capseries}
L. Caporaso: \emph{Linear series on semistable curves.} International Mathematics Research Notices (2010), doi: 10.1093/imrn/rnq188.

\bibitem[CCE08]{CCE} L. Caporaso; J. Coelho; E. Esteves: \emph{Abel maps of Gorenstein curves.}  Rend. Circ. Mat. Palermo (2) {\bf 57} (2008), no. 1, 33--59.

\bibitem[CMKV]{CMKV} S. Casalaina-Martin, J. Kass, F. Viviani: \emph{The Local Structure of Compactified Jacobians}. Preprint available at arXiv:1107.4166v2.

\bibitem[CMKV2]{CMKV2} S. Casalaina-Martin, J. Kass, F. Viviani: \emph{The singularities and birational geometry of the universal compactified Jacobian.}� In preparation.

\bibitem[Cat82]{Cat} F. Catanese: \emph{Pluricanonical-Gorenstein-curves.} Enumerative geometry and classical algebraic geometry (Nice, 1981), pp. 51--95,
Progr. Math. Vol. {\bf 24}, Birkh\"auser Boston, Boston, MA, 1982.

\bibitem[CF96]{CF} F. Catanese, M.  Franciosi: \emph{Divisors of small genus on algebraic surfaces and projective embeddings.}
Proceedings of the Hirzebruch 65 Conference on Algebraic Geometry (Ramat Gan, 1993),  109--140, Israel Math. Conf. Proc., 9, Bar-Ilan Univ., Ramat Gan, 1996.

\bibitem[CFHR99]{CFHR} F. Catanese; M. Franciosi; K. Hulek; M. Reid:
\emph{Embeddings of curves and surfaces.}
Nagoya Math. J. {\bf 154} (1999), 185--220.

%\bibitem[CLO98]{CLO} D. Cox, J. Little, D. O'Shea: \emph{Using algebraic geometry.} Graduate Texts in Mathematics, 185. Springer-Verlag, New York, 1998.

\bibitem[Dol03]{Dol} I. Dolgachev: \emph{Lectures on invariant theory.} London Mathematical Society Lecture Note Series Vol. {\bf 296}. Cambridge University Press, Cambridge, 2003.

\bibitem[DH98]{DolHu} I. V. Dolgachev, Y. Hu:
\emph{Variation of geometric invariant theory quotients.}
Inst. Hautes \'Etudes Sci. Publ. Math. {\bf 87} (1998), 5--56.
With an Appendix by Nicolas Ressayre.

\bibitem[DM69]{DM} P. Deligne, D. Mumford:
\emph{The irreducibility of the space of curves of given genus.}
Inst. Hautes \'Etudes Sci. Publ. Math. {\bf 36} (1969), 75--109.


\bibitem[Edi00]{Edi} D. Edidin: \emph{Notes on the construction of the moduli space of curves}.
Recent progress in intersection theory (Bologna, 1997), 85�113, Trends Math., Birkh\:auser Boston, Boston, MA, 2000.

\bibitem[EK05]{estklei} E. Esteves, S. Kleiman: \emph{The compactified Picard scheme of the compactified Jacobian.} Adv. Math. {\bf 198} (2005), 484--503.

\bibitem[EP]{estpac} E. Esteves, M. Pacini:\emph{ Semistable modifications of families of curves and compactified Jacobians}. Preprint arXiv:1406.1239.

\bibitem[FGvS99]{FGvS} B. Fantechi, L. G\"ottsche, D. van Straten: {\it Euler number of the compactified Jacobian and multiplicity of rational curves.} J. Algebraic Geom. {\bf 8} (1999), no. 1, 115--133.

\bibitem[EGAII]{EGAII} A. Grothendieck, J. Dieudonn\'e: \emph{El\'ements de G\'eom\'etrie Alg\'ebrique II}. Publ. Math. Inst. Hautes \'Etudes Sci. 8 (1961).

\bibitem[EGAIII1]{EGAIII1} A. Grothendieck, J. Dieudonn\'e: \emph{El\'ements de G\'eom\'etrie Alg\'ebrique III-Part 1}. Publ. Math. Inst. Hautes \'Etudes Sci. 11 (1961).


%\bibitem[EH86]{EH}
%David Eisenbud and Joe Harris.
%\newblock Limit linear series: basic theory.
%\newblock {\em Invent. Math.}, 85(2):337--371, 1986.


%\bibitem[Far06]{Far}
%Gavril Farkas.
%\newblock The global geometry of the moduli space of curves, 2006.
%\newblock arXiv:math/0612251v1.


\bibitem[FS13]{FS} M. Fedorchuk,  D. I. Smyth: \emph{Alternate compactifications of moduli space of curves}. Handbook of Moduli, Volume I. G. Farkas, I. Morrison (Eds.), Advanced Lectures in Mathematics, Volume XXIV (2013),
331--414.

\bibitem[Fel14]{Fel} F. Felici: \emph{GIT for Hilbert and Chow schemes of curves}. PhD thesis. Roma Tre University, 2014.

\bibitem[Fra04]{Fra1} M. Franciosi: \emph{Adjoint divisors on algebraic curves (with an Appendix of F. Catanese).}
Adv. Math. {\bf 186} (2004), 317--333.

\bibitem[Fra07]{Fra2} M. Franciosi: \emph{Arithmetically Cohen-Macaulay algebraic curves.}  Int. J. Pure Appl. Math.  {\bf 34}  (2007),  69--86.

\bibitem[FT14]{FT} M. Franciosi, E. Tenni: \emph{The canonical ring of a 3-connected curve.}  Atti Accad. Naz. Lincei,  Rend. Lincei, Mat. Appl. 25 (2014), 37--51.

\bibitem[Gie82]{Gie} D. Gieseker: \emph{Lectures on moduli of curves}. Tata Institute
  of Fundamental Research Lectures on Mathematics and Physics, Volume {\bf 69}.
Tata Institute of Fundamental Research, Bombay, 1982.

%\bibitem[GKM02]{GKM}
%Angela Gibney, Sean Keel, and Ian Morrison.
%\newblock Towards the ample cone of {$\overline M\sb {g,n}$}.
%\newblock {\em J. Amer. Math. Soc.}, 15(2):273--294 (electronic), 2002.

\bibitem[Hal]{hall} J. Hall:  \emph{Moduli of Singular Curves}.
Preprint available at arXiv:1011.6007v1.

\bibitem[HM98]{HM} J. Harris, I. Morrison: \emph{Moduli of curves.} Graduate text in mathematics {\bf 187}. Springer-Verlag, New York-Heidelberg, 1998.

\bibitem[Har77]{Har} R. Hartshorne: \emph{Algebraic geometry.} Graduate Texts in Mathematics {\bf 52}. Springer-Verlag, New York-Heidelberg, 1977.

\bibitem[Har94]{Har2} R. Hartshorne: \emph{Generalized divisors on Gorenstein schemes.} K-theory {\bf 8} (1994), 287--339.

%\bibitem[Has05]{Has}
%Brendan Hassett.
%\newblock Classical and minimal models of the moduli space of curves of genus two.
%\newblock In {\em Geometric methods in algebra and number theory}, volume 235
%  of {\em Progress in mathematics}, pages 160--192. Birkh\"auser, Boston, 2005.

\bibitem[HHL07]{HHL}
B. Hassett, D. Hyeon, and Y. Lee:
\emph{Stability computation via {G}r\"obner basis}.  J. Korean Math. Soc. {\bf 47} (2010), no. 1, 41�-62.


\bibitem[HH09]{HH1}
B. Hassett, D. Hyeon: \emph{Log canonical models for the moduli space of curves:
first divisorial contraction.}  Trans. Amer. Math. Soc.  {\bf 361}  (2009), 4471--4489.

\bibitem[HH13]{HH2}
B. Hassett, D. Hyeon: \emph{Log canonical models for the moduli space of curves: the first flip.}  Ann. Math. (2) 177 (2013), 911--968.


%\bibitem[Hye]{Hye} D. Hyeon: \emph{An outline of the log minimal model program for the moduli space of curves.}
%Preprint available at arXiv:1006.1094.

\bibitem[HeHi11]{HeHi}
J. Herzog, T. Hibi, \emph{Monomial ideals.} Graduate Texts in Mathematics, \textbf{260}. Springer-Verlag London,
Ltd., London, 2011.

\bibitem[HL07]{HL} D. Hyeon, Y. Lee: \emph{Stability of tri-canonical curves of genus two.} Math. Ann. {\bf 337} (2007), 479--488.

\bibitem[HR74]{HR} M. Hochster, J. L. Roberts: \emph{Rings of invariants of reductive groups acting on regular
rings are Cohen-Macaulay.} Adv. Math. {\bf 13} (1974), 115--175.

%\bibitem[HL07a]{HL2}
%Donghoon Hyeon and Yongnam Lee.
%\newblock Log minimal model program for the moduli space of stable curves of
%  genus three, 2007.
%\newblock arXiv.org:math/0703093.

%\bibitem[HL07b]{HL1}
%Donghoon Hyeon and Yongnam Lee.
%\newblock Stability of tri-canonical curves of genus two.
%\newblock {\em Math. Ann.}, 337(2):479--488, 2007.

\bibitem[HM10]{HMo}
D. Hyeon, I. Morrison:
\emph{Stability of Tails and 4-Canonical Models.}  Math. Res. Lett.  {\bf 17}  (2010),  no. 4, 721--729.
%Preprint available at arXiv:0806.1269.

\bibitem[Kas]{Kas} J. Kass: \emph{Good completions of N\'eron models.}  PhD Thesis. Harvard University, 2009.

\bibitem[KeM97]{KeMo} S. Keel, S. Mori: \emph{Quotients by groupoids}, Ann. of Math. (2) {\bf 145} (1997), no. 1, 193--213.

\bibitem[Kle05]{Kle} S. L. Kleiman: \emph{The Picard scheme.} Fundamental algebraic geometry, 235--321, Math. Surveys Monogr. Vol. {\bf 123}, Amer. Math. Soc., Providence, RI, 2005.

\bibitem[Knu83]{Knu} F. F. Knudsen: \emph{The projectivity of the moduli space of stable curves. II. The stacks $M_{g,n}$.} Math. Scand. {\bf 52} (1983), no. 2, 161--199.

\bibitem[KoM98]{KM} J. Koll\'ar, S. Mori: \emph{Birational geometry of algebraic varieties.}
With the collaboration of C. H. Clemens and A. Corti. Translated from the 1998 Japanese original.
Cambridge Tracts in Mathematics, Vol. {\bf 134}. Cambridge University Press, Cambridge, 1998.

\bibitem[LW]{LW} J. Li, X. Wang: \emph{Hilbert-Mumford criterion for nodal curves.} Preprint available at arXiv: 1108.1727v1.

%\bibitem[Lie06]{Lie1} M. Lieblich: \emph{Remarks on the stack of coherent algebras.} Int. Math. Res. Not. 2006, Art. ID 75273, 12 pp.

%\bibitem[Lie07]{Lie2} M. Lieblich: \emph{Moduli of twisted sheaves.} Duke Math. J. {\bf 138} (2007), no. 1, 23--118.

\bibitem[Mat89]{Mat} H. Matsumura: \emph{Commutative ring theory.} Translated from the Japanese by M. Reid. Second edition. Cambridge Studies in Advanced Mathematics,
Vol. {\bf 8}. Cambridge University Press, Cambridge, 1989.

\bibitem[Mel09]{melo1} M. Melo: {\em Compactified Picard stacks over ${\overline{\mathcal M}_g}$.}
Math. Zeit. {\bf 263} (2009), No. 4, 939--957.

\bibitem[Mel11]{melo} M. Melo: {\em Compactified Picard stacks over the moduli stack of stable curves with marked points.} Adv. Math {\bf 226} (2011), 727--763.

\bibitem[MV12]{MV} M. Melo, F. Viviani: {\em Fine compactified Jacobians.} Math. Nach. {\bf 285} (2012), no. 8-9, 997--1031.

\bibitem[Mor10]{Mo}
I. Morrison: \emph{GIT Constructions of Moduli Spaces of Stable Curves and Maps.}
Ji, Lizhen (ed.) et al., Geometry of Riemann surfaces and their moduli spaces. Somerville, MA: International Press. Surveys in Differential Geometry {\bf 14}, 315--369 (2010).

\bibitem[MS11]{MS}
I. Morrison, D. Swinarski: \emph{Gr\"obner techniques for low degree Hilbert stability.}
Exp. Math. {\bf 20} (2011), no. 1, 34--56.


%\bibitem[HM82]{HM}
%Joe Harris and David Mumford.
%\newblock On the {K}odaira dimension of the moduli space of curves.
%\newblock {\em Invent. Math.}, 67(1):23--88, 1982.
%\newblock With an Appendix by William Fulton.

%\bibitem[KM76]{Kn}
%Finn~Faye Knudsen and David Mumford.
%\newblock The projectivity of the moduli space of stable curves. {I}.
%  {P}reliminaries on ``det'' and ``{D}iv''.
%\newblock {\em Math. Scand.}, 39(1):19--55, 1976.

\bibitem[Knu83]{knudsen} F. Knudsen: {\it The projectivity of the moduli space of stable curves. II. The stacks $M\sb{g,n}$.} Math. Scand. {\bf 52}  (1983),  no. 2, 161--199.

\bibitem[Kol96]{Kol}
J. Koll{\'a}r: {\em Rational curves on algebraic varieties.}
Ergebnisse der Mathematik und ihrer Grenzgebiete, Vol. {\bf 32}.
%3. Folge.  A Series of
%  Modern Surveys in Mathematics [Results in Mathematics and Related Areas. 3rd
%  Series. A Series of Modern Surveys in Mathematics]}.
Springer-Verlag, Berlin, 1996.


\bibitem[MFK94]{GIT}
D. Mumford, J. Fogarty, F. Kirwan: \emph{Geometric invariant theory}.
Ergebnisse der Mathematik und ihrer Grenzgebiete (2), Vol. {\bf 34}.
Springer-Verlag, Berlin, third edition, 1994.

\bibitem[Mum66]{MumCAS}
D. Mumford: \emph{Lectures on curves on an algebraic surface}.
Annals of Mathematics Studies, No. {\bf 59}. Princeton University Press, Princeton, N.J., 1966.

\bibitem[Mum77]{Mum}
D. Mumford: \emph{Stability of projective varieties.}
Enseignement Math. (2) {\bf 23} (1977), 39--110.

\bibitem[Pan96]{Pan} R. Pandharipande: \emph{A compactification over $\Mg$ of the universal moduli space of slope-semi-stable vector bundles.}
J. Amer. Math. Soc. {\bf 9} (1996), 425--471.

\bibitem[Ray70]{Ray} M. Raynaud: \emph{Sp\' ecialisation du foncteur de Picard.}  Inst. Hautes \'Etudes Sci. Publ. Math. No. {\bf 38} (1970), 27--76.

%\bibitem[Rei89]{Rei} Z. Reichstein: \emph{Stability and equivariant maps.}
%Invent. Math. {\bf 96} (1989), 349--383.

%\bibitem[Sho05]{Sch} H. Schoutens: \emph{Log-terminal singularities and vanishing theorems via non-standard tight closure.}
%J. Algebraic Geom. 14 (2005), 357-–390.

\bibitem[Sch91]{Sch}
D. Schubert: \emph{A new compactification of the moduli space of curves.}
Compositio Math. {\bf 78} (1991), 297--313.

\bibitem[Ser06]{Ser}E. Sernesi: {\it Deformations of algebraic schemes}. Grundlehren der mathematischen Wissenschaften {\bf 334}, Springer, New York, 2006.

%\bibitem[Ses72]{Ses1}
%C.~S. Seshadri.
%\newblock Quotient spaces modulo reductive algebraic groups.
%\newblock {\em Ann. of Math. (2)}, 95:511--556; errata, ibid. (2) 96 (1972), 599, 1972.

%\bibitem[Ses77]{Ses}
%C.~S. Seshadri.
%\newblock Geometric reductivity over arbitrary base.
%\newblock {\em Advances in Math.}, 26(3):225--274, 1977.

\bibitem[Sim94]{simpson} C. T.   Simpson.
\newblock {\em Moduli of representations of the fundamental group of a smooth projective variety. }
\newblock  Inst. Hautes \'Etudes Sci. Publ. Math., No. {\bf 80} (1994), 5--79.


\bibitem[Smy13]{Smy} D. I. Smyth: \emph{Towards a classification of modular compactifications of $M_{g,n}$.}
Invent. Math. 192 (2013), 459--503.

\bibitem[Yos90]{Yos} Y. Yoshino: \emph{Cohen-Macaulay modules over Cohen-Macaulay rings.}
London Mathematical Society Lecture Note Series, 146. Cambridge University Press, Cambridge, 1990.


%\bibitem[Tha96]{thaddeus}
%Michael Thaddeus.
%\newblock Geometric invariant theory and flips.
%\newblock {\em J. Amer. Math. Soc.}, 9(3):691--723, 1996.

%\bibitem[Vie95]{Vbook}
%Eckart Viehweg.
%\newblock {\em Quasi-projective moduli for polarized manifolds}, volume~30 of
%  {\em Ergebnisse der Mathematik und ihrer Grenzgebiete (3) [Results in
%  Mathematics and Related Areas (3)]}.
%\newblock Springer-Verlag, Berlin, 1995.


%\bibitem[Wan10]{wang} J. Wang: \emph{Deformations of pairs $(C,L)$ when $C$ is singular}, arXiv:1003.6073.


\end{thebibliography}
\end{document}